\documentclass[12pt]{amsart}

\usepackage{amssymb}
\usepackage{amsthm}
\usepackage{amsmath}
\usepackage{mathtools}
\usepackage{longtable}
\usepackage{mathrsfs}
\usepackage{graphicx}
\usepackage{hyperref}
\usepackage{enumitem}

\usepackage[left=1.1in,top=0.85in,right=1.1in,bottom=0.7in]{geometry}

\usepackage[all]{xy}
\usepackage{tikz-cd}
\usepackage[title]{appendix}
\usepackage{color}
\usepackage[utf8]{inputenc}
\usepackage[T1]{fontenc}
\hypersetup{colorlinks=true,linkcolor=blue,citecolor=blue}

\newcommand{\calF}{{\mathcal F}}
\newcommand{\calG}{{\mathcal G}}
\newcommand{\calC}{{\mathcal C}}

\newcommand{\calD}{{\mathcal D}}
\newcommand{\Oo}{{\mathcal O}}
\newcommand{\calP}{{\mathcal P}}

\newcommand{\Ee}{{\mathscr E}}

\newcommand{\Dd}{{\mathscr D}}

\newcommand{\Uu}{{\mathscr U}}

\newcommand{\Vv}{{\mathscr V}}

\newcommand{\Xx}{{\mathscr X}}

\newcommand{\hdot}{{\bullet}}

\newcommand{\Filt}{{\bf Filt}}

\newcommand{\cc}{\mathbb{C}}
\newcommand{\rr}{\mathbb{R}}
\newcommand{\zz}{\mathbb{Z}}
\newcommand{\pp}{\mathbb{P}}
\newcommand{\iiii}{\mathbb{I}}
\newcommand{\tttt}{\mathbb{T}}

\newcommand{\aaa}{\mathbb{A}}

\newcommand{\HH}{\mathbb{H}}
\newcommand{\qq}{\mathbb{Q}}

\newcommand{\delbar}{\overline{\partial}}
\newcommand{\Gr}{\mathrm{Gr}}
\newcommand{\Spec}{\mathrm{Spec}}

\newcommand{\vol}{\mathrm{vol}}
\newcommand{\ch}{\mathrm{ch}}
\newcommand{\Tr}{\mathrm{Tr}}
\newcommand{\End}{\mathrm{End}}
\newcommand{\Hom}{\mathrm{Hom}}
\newcommand{\im}{\mathrm{im}}
\newcommand{\Id}{\mathrm{Id}}
\newcommand{\CS}{CS}
\newcommand{\Del}{\mathrm{Del}}
\newcommand{\Cube}{\mathrm{Cube}}

\DeclareMathOperator{\dlog}{d\!\log}

\newtheorem{theorem}{Theorem}[section]
\newtheorem{corollary}[theorem]{Corollary}
\newtheorem{lemma}[theorem]{Lemma}
\newtheorem{proposition}[theorem]{Proposition}

\newtheorem{construction}[theorem]{Construction}
\newtheorem{conjecture}[theorem]{Conjecture}

\theoremstyle{definition}
\newtheorem{definition}[theorem]{Definition}
\newtheorem{example}[theorem]{Example}

\theoremstyle{remark}
\newtheorem{remark}[theorem]{Remark}
\newtheorem{notation}[theorem]{Notation}
\newtheorem{question}[theorem]{Question}

\begin{document}

\author[J.\,N.\,Iyer]{Jaya NN Iyer}
\address{Institute for Mathematical Sciences (HBNI),Tarmani, Chennai 600113, India}
\email{jniyer@imsc.res.in}
\address{Max-Planck Institute for Mathematics, 7, Vivatsgasse, Bonn 53121, Germany}
\email{jniyer@mpim-bonn.mpg.de}

\author[C.\,Simpson]{Carlos Simpson}
\address{CNRS, Laboratoire J.\,A.\,Dieudonn\'{e}, UMR 6621\\
Universit\'{e}  C\"ote d'Azur,\\
06108 Nice, Cedex 2, France}
\email{carlos.simpson@univ-cotedazur.fr}

\title[Torsion of extended Chern-Simons classes]{Torsion of extended
Chern-Simons classes for canonical extensions of flat bundles}



\begin{abstract}
Let $X$ be a smooth complex projective variety and
$D = D_1+\cdots+D_k\subset X$ a divisor with simple normal crossings. Consider Deligne's canonical extension $(F,\nabla)$ of a flat algebraic
vector bundle on $X^*:=X\setminus D$ with unipotent monodromy around every
component of~$D$.
We define and compare the various constructions of  the \emph{extended Chern-Simons classes}
\[
  \CS_p(\nabla^{\Del})\;\in\; H^{2p-1}(X,\cc/\zz),\quad p\geq 1,
\]
attached to  $(F,\nabla)$.

Our main theorem states that $\CS_p(\nabla^{\Del})$ is \emph{torsion} in
$H^{2p-1}(X,\cc/\zz)$, for every $p\geq 2$, extending \cite{Reznikov}, \cite{Reznikov2}, \cite{IS-arXiv},\cite{IS-2div}).  In \S \ref{sec:parabolic}, we treat the case of quasi-unipotent local monodromies via \textit{locally abelian parabolic bundles}, and deduce the torsion of Chern-Simons classes. In Appendix \ref{sec:deligne-sullivan}, we extend the Deligne-Sullivan  theorem \cite{DeSu} on triviality of flat bundle on a finite covering of a smooth manifold, to that of a canonical extension, and provide torsion-bounds on the extended characteristic  classes. 
\end{abstract}

\maketitle

\setcounter{tocdepth}{2}
\tableofcontents

\section{Introduction}
\label{sec:intro}

S.S. Chern, J. Cheeger and J. Simons \cite{ChrS}, \cite{CS} introduced
a theory of \emph{differential cohomology} on smooth manifolds.  For vector
bundles equipped with connections, they defined classes in the ring of
differential characters, now called \emph{secondary invariants} or
\emph{Chern-Simons classes}.  These classes lift the closed differential form
defined by the curvature of the given connection.  In particular, when the
connection is flat, the secondary invariants yield classes in cohomology with
$\rr/\zz$-coefficients.  These are the \emph{Chern-Simons classes of flat
	connections}.

The following question was raised in \cite[p.~70--71]{CS} (see also
\cite[p.~104]{Bloch}) by Cheeger and Simons):

\begin{question}\label{q:CS}
	Suppose $X$ is a smooth manifold and $(E,\nabla)$ is a flat connection on $X$.
	Are the Chern-Simons classes $\widehat{c}_p(E,\nabla)$ of $(E,\nabla)$ torsion
	in $H^{2p-1}(X,\rr/\zz)$ for $p\geq 2$?
\end{question}

Let $X$ be a smooth projective variety over $\cc$ and $(E,\nabla)$ a vector
bundle with flat connection.  Bloch \cite{Bloch} showed that for a unitary
connection the Chern-Simons classes map to the Chern classes of $E$ in
Deligne cohomology.  This observation, combined with Question~\ref{q:CS},
led him to conjecture:

\begin{conjecture}[Bloch]\label{conj:bloch}
	The Chern classes of flat bundles are torsion in the Deligne cohomology
	$H^{2p}_{\calD}(X,\zz(k))$ of $X$, for $p\geq 2$.
\end{conjecture}

Several constructions of secondary classes for flat bundles on smooth
projective varieties followed.  Beilinson defined universal secondary classes,
Esnault \cite{Es} constructed secondary classes using a modified
\emph{splitting principle} in $\cc/\zz$-cohomology, and Karoubi \cite{Karoubi} using K-theory.  These classes 
lift the Chern classes in the Deligne cohomology.
  They also admit an
interpretation in terms of differential characters, and the original
$\rr/\zz$-classes of Chern-Simons are recovered by the projection
$\cc/\zz\to\rr/\zz$.  The imaginary parts of the $\cc/\zz$-classes are
Borel's \emph{volume regulators}
\[
\mathrm{Vol}_{2p-1}(E,\nabla)\;\in\; H^{2p-1}(X,\rr).
\]
All of these constructions give the same class in odd degrees, called the
\emph{secondary classes} (or \emph{regulators}) of $(E,\nabla)$ on $X$; see
\cite{DHZ,Es2} for a detailed discussion.

Reznikov \cite{Rez2,Reznikov} proved that the secondary classes
of $(E,\nabla)$ are torsion in $H^{2i-1}(X,\cc/\zz)$ for $i\geq 2$,
thereby confirming Bloch's conjecture.

We extend Bloch's question to a quasi projective variety, as follows.

Let $X$ be a smooth projective variety over $\cc$ of dimension $d$, and let
$D = D_1+\cdots+D_k\subset X$ be a divisor with simple normal crossings.
Write $X^* := X\setminus D$.  Fix a base-point $x_0\in X^*$ and a
representation
\[
  \zeta : \pi_1(X^*,x_0) \longrightarrow GL_r(\cc)
\]
such that for every loop $\gamma_i$ encircling the divisor component $D_i$
the image $\zeta(\gamma_i)$ is \emph{unipotent}.  Let $F^{\zeta}$ denote the
corresponding locally constant sheaf on $X^*$.

Deligne \cite{Deligne70} constructed a canonical extension: a holomorphic
vector bundle $F$ on $X$ equipped with a logarithmic connection
\[
  \nabla : F \longrightarrow F\otimes\Omega^1_X(\log D),
\]
whose residues $\mathrm{Res}_{D_i}(\nabla)$ are nilpotent endomorphisms of
$F|_{D_i}$.  The flat connection $\nabla|_{X^*}$ corresponds to $\zeta$ under
the Riemann-Hilbert correspondence.

Following the approach of Cheeger-Simons \cite{CS} and its extension to the
logarithmic setting, one attaches secondary characteristic classes:

\[
  \CS_p(\nabla^{\Del}) \;\in\; H^{2p-1}(X,\cc/\zz),\quad p\geq 1,
\]
where $\nabla^{\Del}$ is Deligne's \emph{patched connection}: a
$\mathcal{C}^\infty$ locally nil-flat connection on $F$ constructed from
$\nabla$ by subtracting the singular $\dlog$ terms and averaging with a
partition of unity (Appendix~\ref{app:deligne-patch}).

\begin{definition}[Extended Chern-Simons classes]\label{def:extended-CS}
We call $\CS_p(\nabla^{\Del})\in H^{2p-1}(X,\cc/\zz)$ the \emph{extended
Chern-Simons classes} of the flat bundle $(V,\nabla_{\mathrm{flat}})$ on $X^*$,
or of the representation $\zeta$.
\end{definition}

The word ``extended'' refers to the canonical extension: these are classes on
the \emph{projective} variety $X$, attached to Deligne's extension $F$ of $V$
across $D$, rather than classes on the open part $X^*$.   Note
that $\nabla^{\Del}$ is locally nil-flat, so
$\CS_p(\nabla^{\Del})$ is not a Chern-Simons class in the classical
Cheeger-Simons sense. It is nevertheless defined via the vanishing
of all Chern-Weil forms of a locally nil-flat connection
(Proposition~\ref{prop:nilflat-ch-zero}).

The main result of this paper is:

\begin{theorem}[Main theorem]
\label{thm:main}
In the situation above, the extended Chern-Simons classes
$\CS_p(\nabla^{\Del})\in H^{2p-1}(X,\cc/\zz)$ are \emph{torsion}, for every
$p\geq 2$.
\end{theorem}

\medskip

The case of a smooth irreducible divisor was treated in \cite{IS-arXiv}, and
the case of two smooth components meeting transversally in \cite{IS-2div}.
Both proceed by  a \emph{bifiltered method}: the extended class is constructed
directly on $X$ as the Cheeger-Simons character of an explicit
$\calC^\infty$ connection, patched over a cubical decomposition of $X$ given
by corner coordinates; at a corner the datum is the \emph{pair} of monodromy
weight filtrations $W(N_1),W(N_2)$; the comparison with the regulator is made
through a $2$-cubical pushout with $BGL(F[t_1,t_2])^+$ at the $2$-face; and
the volume regulators are killed universally, in hermitian $K$-theory, using
Karoubi's homotopy invariance and Reznikov's computation on
$BO_{\infty,\infty}(\cc)^+$.


The method used in \cite{IS-arXiv,IS-2div} does not 
extend directly to a general normal crossings divisor, indeed we use there
the existence of simultaneous splittings for two filtrations, given essentially by the 
Bruhat decomposition. In the case of three or more filtrations, there is not in general a simultaneous splitting, as may be seen already for three
distinct lines in $\cc^2$, in other words $GL(V)$ has no normal form on triples of
flags. At a triple
point of $D$ the tuple $(W(N_1),W(N_2),W(N_3))$ can not be used as a
corner datum: the iterated graded depends on the order in which it is taken,
and the corner model connection, the independence argument, and the
polarization-compatible splitting of the hermitian step do not work.

The route taken here replaces the tuple of filtrations at a stratum by the
\emph{single} monodromy weight filtration $W(N_I)$ of the total operator
$N_I=\sum_{i\in I}N_i$, whose compatibility along chains of strata is supplied
by Mochizuki's theorem.  Everything else follows from that choice: patching
data indexed by chains rather than tuples
(Section~\ref{sec:patching}), a multi-Rees construction interpolating between
strata (Section~\ref{sec:multrees}), and a hermitian argument that must now be
solved globally on $X$ rather than universally, by way of the contractible
poset of $N_\bullet$-isotropic filtrations and a \v{C}ech section theorem
(Sections~\ref{sec:nisotropic}, \ref{sec:volume}).  It should be said that the
older method is not merely a special case: where it applies it uses no
harmonic bundle theory in the construction of the classes, no algebraic model,
and no arithmetic Chern character, and its vanishing statement is universal.
Moreover its real hypothesis is that at most two branches of $D$ pass through
any point. This comparison is
carried out in Appendix~\ref{app:bifiltered}.
\medskip

We record the consequences of Theorem \ref{thm:main}.
\begin{corollary}
\label{cor:deligne-torsion}
The $p$-th Deligne-Beilinson Chern character class
$\ch^{\calD}_p(F)\in H^{2p}_{\calD}(X,\zz(p))$ of the canonical extension is
torsion, for every $p\geq 2$.
\end{corollary}

\begin{corollary}
\label{cor:deformation}
The extended Chern-Simons classes are locally constant on the parameter space
$R^{\mathrm{nil}}$ of representations with unipotent monodromy.
\end{corollary}

\begin{corollary}[$\ell$-adic Chern-Simons classes are torsion]
\label{cor:ladic-torsion}
Let $\ell$ be a prime and let $L_\ell$ be the $\ell$-adic local system on
$X^*_{\mathrm{\acute{e}t}}$ attached to $L$ by the Artin comparison
isomorphism.  Then the classes
$\CS_p^{(\ell)}(L_\ell)\in H^{2p-1}_{\mathrm{\acute{e}t}}(X,\qq_\ell/\zz_\ell(p))$
of Definition~\ref{def:ladic-CS} are torsion for every $p\geq2$.
\end{corollary}

A word is needed on what these classes are, since the $\ell$-adic local system
lives on $X^*$ and \emph{not} on $X$: there is no $\ell$-adic sheaf on $X$
extending $L_\ell$, and no \'{e}tale analogue of the canonical extension.  The
classes are therefore not defined by transporting an $\ell$-adic construction
across $D$.  What is used instead is the Artin comparison isomorphism in the
opposite direction: for the finite coefficient rings $\zz/\ell^k$ one has
\[
  H^{n}_{\mathrm{\acute{e}t}}\bigl(X,\zz/\ell^k(p)\bigr)
  \;\cong\;H^{n}\bigl(X(\cc),\zz/\ell^k(p)\bigr),
\]
an isomorphism of the \'{e}tale cohomology of the \emph{projective} variety
$X$ with the singular cohomology of $X(\cc)$.  The transcendental extended
class $\CS_p(\nabla^{\Del})\in H^{2p-1}(X,\cc/\zz)$, being torsion by
Theorem~\ref{thm:main}, lies in the image of
$H^{2p-1}(X(\cc),\qq/\zz)$; one then takes its $\ell$-primary component and
transports it through the comparison isomorphism.  So
$\CS^{(\ell)}_p(L_\ell)$ is by construction a class on $X$, defined only
because the transcendental class is already known to be torsion, and the
notation $\CS^{(\ell)}_p(L_\ell)$ records the local system it comes from
rather than an independent $\ell$-adic construction.  In particular the
corollary is a consequence of Theorem~\ref{thm:main} and not an independent
statement.  See Definition~\ref{def:ladic-CS} for the details.

The analogous result for flat bundles on compact K\"{a}hler manifolds/projective
\emph{without} boundary divisor was proved by Reznikov \cite{Rez2}, \cite{Reznikov}:
every Chern-Simons class of a flat bundle is torsion.  His argument uses a
deformation to a representation underlying a polarized VHS (carrying an
indefinite Hermitian form preserved by the connection) and an explicit
homotopy. He also proved torsion -property over a complex quasi-projective variety in \cite[section 6.1]{Reznikov2}. Our proof extends this to the logarithmic setting  
yielding the torsion result for classes on $X$ rather than $X^*$,
and makes the
role of the hermitian geometry transparent. 

We define and work with classes on the
projective completion $X$, attached to the canonical extension.  Since
restriction $H^{2p-1}(X,\cc/\zz)\to H^{2p-1}(X^*,\cc/\zz)$ has kernel in
general, torsion of the restricted class does not imply torsion of the
extended one, and Theorem~\ref{thm:main} is not a formal consequence of the
quasi-projective statement.  The extended classes also carry information that
the open ones do not: they are the classes that enter the arithmetic
applications, through Corollary~\ref{cor:deligne-torsion} on $X$.

Earlier work of Esnault\cite{Es} established torsion of certain
characteristic classes, when residues are nilpotent, see also \cite{Es3} for a definition of extended Chern-Simons classes, and the case of cVHS on a complex smooth quasi-projective variety \cite{Es4}.   Mochizuki
\cite{Mochizuki} proved the sequential compatibility of weight filtrations for
wild harmonic bundles; in the tame case (unipotent monodromy) which we treat.

The proof proceeds in the following 
steps.

\smallskip\noindent
(1)~We construct, following an idea in 
\cite{IS-arXiv}, a \emph{good adapted open covering}
$X=\bigcup_I U_I$ indexed by multi-indices for the divisor components, in
which non-empty multiple intersections are linearly ordered by inclusion of
index sets, and divisor components do not cross between distinct index classes.

\smallskip\noindent
(2)~We equip the Deligne canonical extension $F$ with a
\emph{strict collection of patching data} $(W(I),\tau(I))$, derived from the
monodromy weight filtrations supplied by Mochizuki's theorem \cite{Mochizuki}.
The data assemble into a functor from the nerve poset $A$ to the category
$\Xi_K$ of graded $K$-vector spaces with refinement morphisms.

\smallskip\noindent
(3)~Via an iterated \emph{multi-Rees construction}, we build an algebraic
vector bundle $F_K$ on the projective cubical realization
$\pp_K\,\Cube\,A$, carrying a stratified patching structure.  This step then
splits into two comparisons which are logically independent and are carried
out by quite different means; both are needed, and their meeting point is the
Deligne-Beilinson Chern class $\ch^\calD_p(F_K)$.

The comparison \emph{Regulator $=$ Deligne class} is treated in Appendix~\ref{app:rees},
    \S\ref{subsec:classmap}.  A classifying map
    $r_L:X\simeq\iiii\Cube A\to BGL^+(K)$ is constructed, and the pullback of
    the universal regulator class is identified with $\ch^\calD_p(F_K)$.  This
    is a statement about \emph{bundles}: no connection appears, and the proof
    is topological and $K$-theoretic, resting on homotopy invariance of the
    $K$-theory of $F[t_1,\dots,t_n]$ and acyclicity of the plus construction.

The comparison \emph{Deligne class $=$ Chern-Simons class} is treated in
    Section~\ref{sec:F1} and Appendix~\ref{app:rees},
    \S\ref{subsec:degeneration}.  The Deligne class is \emph{evaluated} using
    the extra structure of the Rees bundle: on the projective completion
    $\pp_K\Cube A$ it carries $F^1$-connections with logarithmic poles, and
    the Dupont-Hain-Zucker independence theorem together with the Burgos-Gil
    comparison identifies $\ch^\calD_p(F_K)$ with $\CS_p(\nabla^{\Del})$.
    Both theories are applied cube by cube, on the smooth projective pieces
    $(\pp^1)^{|c|}$ of the realization rather than on the singular glued
    scheme; see Appendix~\ref{app:burgosgil}, \S\ref{subsec:BG-on-cubical}.
    

Composing these two comparisons 
gives $r_L^*(r_{2p-1})=\CS_p(\nabla^{\Del})$, which is what
Steps (4) and (5) act on.

\smallskip\noindent
(4)~When $\zeta$ underlies a complex variation of Hodge structure (VHS), the
flat connection preserves an indefinite Hermitian form.  We prove that any
locally nil-flat connection preserving such a form has zero \emph{volume
	regulator}.  The crucial geometric ingredient is the contractibility of the
poset of $N$-isotropic filtrations, proved using Quillen's Theorem~A.

\smallskip\noindent
(5)~All volume regulators vanishing, Borel's theorem forces the real
cohomology of $BGL^+(K)$ to map to zero on the cubical nerve, making every
regulator---and hence every extended Chern-Simons class---torsion in degree
$\geq 2$.

\medskip
 Since the framework of Burgos-Gil \cite{BG}
is invoked at several points, it may help to say precisely which statements
are needed and where.  
They both
concern the second comparison in Step~(3).
\begin{itemize}[leftmargin=2em]
  \item \emph{Independence of the $F^1$-connection.}  Any two
    $F^1$-connections on a bundle over a smooth projective variety, with
    logarithmic poles along a normal crossings divisor, define the same class
    in Deligne-Beilinson cohomology, the Bott-Chern secondary form of the pair
    lying in $F^1$.  This is stated as Theorem~\ref{thm:DHZ} and is what makes
    $\ch^\calD_p$ computable by whichever connection is convenient---the Rees
    connection, Deligne's patched connection, or a Hermitian connection
    pulled back from a Grassmannian.  It is used in
    Section~\ref{sec:F1} and again in
    Appendix~\ref{app:rees}, \S\ref{subsec:degeneration}.
  \item \emph{Arithmetic Chern character and the comparison theorem.}  The
    arithmetic Chern character $\widehat{\ch}_p(E,h)$ maps to $\ch^\calD_p(E)$
    independently of the metric, and its image under the edge map of the
    Deligne exact sequence is the Chern-Simons class
    (Theorem~\ref{thm:BG-comparison}).  This is what converts a Deligne class
    into a class in $H^{2p-1}(-,\cc/\zz)$, and it is the step that requires the
    Chern-Weil forms to vanish---supplied here by local nil-flatness
    (Proposition~\ref{prop:nilflat-ch-zero}).
\end{itemize}
Both statements are theorems about smooth \emph{projective} varieties, which
is why the Rees bundle is constructed on the projective completion
$\pp_K\Cube A$ and not merely on the affine realization: on affine space the
$F^1$-condition carries no information and the independence statement fails.
Since the glued scheme $\pp_K\Cube A$ is projective but not smooth, both are
applied to it levelwise, on the smooth projective cubes $\pp(c)$, and descend.
For the question of smoothness, we need to notice that they are pulled back from a Grassmannian along a face-compatible
classifying map.  This reduction is discussed in
\S\ref{subsec:BG-on-cubical}.
By contrast, the regulator comparison of
Appendix~\ref{app:rees}, \S\ref{subsec:classmap} does not use 
\cite{BG} or any Hodge theory.

\subsection{Outline of the paper}

The paper is organized as follows.

\begin{itemize}[leftmargin=2em]
\item Section~\ref{sec:covering} constructs the adapted open covering and
  establishes its key combinatorial properties.

\item Section~\ref{sec:cubical} reviews cubical sets, the cubical barycentric
  subdivision of a simplicial space, and the associated affine and projective
  realizations.

\item Section~\ref{sec:patching} develops the formalism of patching data and
  the category $\Xi_K$, and proves the crucial contractibility of the poset
  $\Filt^{A,B}$.

\item Section~\ref{sec:multrees} gives the multi-Rees construction and proves
  its basic properties.

\item Section~\ref{sec:rees-global} constructs the global Rees bundle
  $F_K$ on $\pp_K\Cube A$ and establishes its stratified patching data.

\item Section~\ref{sec:F1} constructs a compatible $F^1$-connection and
  compares the resulting Chern-Simons class with the \emph{Deligne-Beilinson
  Chern class} of the Rees bundle, using the independence theorem for
  $F^1$-connections and the Burgos-Gil comparison theorem.  The complementary
  half of the argument---the comparison of the Deligne-Beilinson Chern class
  with the \emph{regulator} class pulled back from $BGL^+(K)$---is
  independent of Hodge theory and is carried out in
  Appendix~\ref{app:rees}, \S\ref{subsec:classmap}.  Neither comparison
  subsumes the other; together they give
  $r_L^*(r_{2p-1})=\CS_p(\nabla^{\Del})$.

\item Section~\ref{sec:volume} defines the volume regulator, proves it
  vanishes for connections preserving an indefinite Hermitian form, and
  deduces vanishing for Deligne's patched connection via the theory of
  $N$-isotropic filtrations.

\item Section~\ref{sec:nisotropic} proves contractibility of the poset of
  $N$-isotropic filtrations using Quillen's Theorem~A.

\item Section~\ref{sec:poset-presheaf} develops the \v{C}ech section theorem for
  contractible poset presheaves.

\item Section~\ref{sec:hermitian-patching} proves the existence theorem for
  connections preserving an indefinite hermitian form.

\item Section~\ref{sec:torsion} assembles the proof of Theorem~\ref{thm:main}
  and its corollaries.

\item Section~\ref{sec:parabolic} treat the case of quasi-unipotent local monodromies via locally abelian parabolic bundles.

\item Section~\ref{app:deligne-patch} gives the detailed construction of
  Deligne's patched connection and proves its properties.

\item Section~\ref{app:deligne-patch} proves deformation invariance.

\item Section~\ref{app:further} treats the case of a variation of Hodge
  structure explicitly.

\item Section~\ref{sec:deligne-sullivan} treats the triviality of the canonical extension on a finite covering.

\item Section~\ref{app:bifiltered} makes a comparison of the bifiltered two-divisor case, and the need for the present Rees construction.

\item Section~\ref{app:further} poses Open Questions.
\end{itemize}

\subsection{Conventions}

Throughout the paper:
\begin{itemize}[leftmargin=2em]
\item $K$ denotes a field of characteristic zero, often a number field or
  $K=\cc$.
\item $X$ is a smooth connected projective variety over $\cc$.
\item $D = D_1+\cdots+D_k\subset X$ is a simple normal crossings divisor.
\item $X^* := X\setminus D$.
\item Multi-indices $I = (i_1,\ldots,i_a)$ have distinct entries in
  $\{1,\ldots,k\}$; we write $|I|=a$, $D_I := D_{i_1}\cap\cdots\cap D_{i_a}$,
  and $I\subset J$ if the entries of $I$ form a subset of those of $J$.
\item For a poset $A$, the nerve $NA$ is the simplicial set whose
  $n$-simplices are chains $a_0\leq\cdots\leq a_n$ in $A$.
\item All open coverings are assumed locally finite; all manifolds are second
  countable.
\end{itemize}


\subsection{Glossary of notation}
\label{subsec:glossary}

We collect here the notation used throughout, with the place
where each item is introduced.  
{\small
\begin{longtable}{@{}p{0.20\textwidth}p{0.76\textwidth}@{}}

\multicolumn{2}{@{}l}{\textit{Geometry and combinatorics}}\\[2pt]
$X$, $D$, $X^*$ &
  $X$ a smooth connected projective variety over $\cc$, $D=D_1+\cdots+D_k$ a
  simple normal crossings divisor, $X^*=X\setminus D$ (\S\ref{sec:intro},
  Conventions).\\
$D_I$, $|I|$ &
  For a multi-index $I$ with distinct entries in $\{1,\dots,k\}$, the closed
  stratum $D_I=\bigcap_{i\in I}D_i$, of codimension $|I|$; $D_\emptyset=X$
  (Conventions).\\
$A$ &
  The poset of multi-indices indexing the adapted covering, ordered by
  inclusion (Proposition~\ref{prop:covering}).\\
$U_I$, $U_I^*$ &
  The open set of the adapted covering attached to $I\in A$, and
  $U_I^*=U_I\cap X^*$ (Proposition~\ref{prop:covering}).\\
$NA$ &
  Nerve of the poset $A$; $|NA|\simeq X$
  (Proposition~\ref{prop:nerve-htpy}).\\
$\Cube A$ &
  The cubical set attached to $A$ (\S\ref{sec:cubical}).\\
$\iiii\Cube A$, $S$ &
  Interval realization, glued from cubes $\iiii(c)=[0,1]^{|c|}$; homotopy
  equivalent to $X$.  We write $S=\iiii\Cube A$ (\S\ref{sec:cubical}).\\
$\aaa_K\Cube A$, $\pp_K\Cube A$ &
  Affine and projective realizations, glued from $\aaa(c)\cong\aaa^{|c|}$ and
  $\pp(c)\cong(\pp^1)^{|c|}$.  The projective one is where the Hodge-theoretic
  arguments run (\S\ref{sec:cubical}, \S\ref{subsec:degeneration}).\\
$\underline{\pp}_K\Cube A$ &
  The same data read as a \emph{cubical scheme} $c\mapsto\pp(c)$, with smooth
  projective terms; this is what Burgos-Gil and Dupont-Hain-Zucker are
  applied to (\S\ref{subsec:BG-on-cubical}).\\
$\Re_\aaa X_\hdot$ &
  The $\aaa$-realization of a semi-simplicial set: algebraic bundles over
  $\aaa^n$ varying simplicially (Definition~\ref{def:A-realization}).\\[6pt]

\multicolumn{2}{@{}l}{\textit{Monodromy, filtrations, patching data}}\\[2pt]
$\zeta$, $\rho$, $L$ &
  The monodromy representation $\pi_1(X^*,x_0)\to GL_r(\cc)$ and the
  associated local system on $X^*$ (\S\ref{sec:intro}).\\
$R^{\mathrm{nil}}$ &
  The variety of rank-$r$ representations with unipotent monodromy around
  every $D_i$ (\S\ref{sec:deform-to-vhs-detail}).\\
$T_i$, $N_i$ &
  Local monodromy around $D_i$, and its nilpotent logarithm
  $N_i=\tfrac{1}{2\pi i}\log T_i$ (\S\ref{app:deligne-patch}).\\
$N_I$, $W(N_I)$ &
  $N_I=\sum_{i\in I}N_i$ and its monodromy weight filtration; independent of
  positive coefficients by \cite{CKS} (\S\ref{sec:patching}).\\
$W(I)$, $\tau(I)$ &
  The patching datum at $I\in A$: a filtration together with a flat
  trivialization of its graded (Theorem~\ref{thm:patching-exists}).\\
$\Xi_K$ &
  The category of patching data over $K$; its objects are graded vector
  spaces, its morphisms filtrations with identifications of the graded
  (\S\ref{sec:patching}).\\
$A\to\Xi_K$ &
  The functor assembled from $\{(W(I),\tau(I))\}$, the combinatorial heart of
  the construction (Construction~\ref{constr:A-to-Xi}).\\
$\mathbf{ifilt}^{N_\hdot}$ &
  Poset of $N_\hdot$-isotropic filtrations; contractible, which is what makes
  a global choice possible
  (\S\ref{sec:nisotropic}, Proposition~\ref{prop:nisofilcontr}).\\[6pt]

\multicolumn{2}{@{}l}{\textit{Bundles and connections}}\\[2pt]
$F$, $\nabla$ &
  Deligne's canonical extension of $L$ to $X$, with its logarithmic connection
  and nilpotent residues (\S\ref{app:deligne-patch}).\\
$\nabla^{\Del}$ &
  Deligne's patched connection: the $\calC^\infty$ connection obtained by
  removing the singular $\dlog$ terms and averaging
  (Definition~\ref{def:deligne-patch}).\\
locally nil-flat &
  A connection preserving local filtrations with flat graded; equivalently
  with locally nilpotent curvature.  All Chern-Weil forms vanish
  (Definition~\ref{def:locally-nilflat},
  Proposition~\ref{prop:nilflat-ch-zero}).\\
$F_K$ &
  The global Rees bundle on $\pp_K\Cube A$, interpolating the filtrations
  $W(I)$ in $|c|$ parameters (Theorem~\ref{thm:global-rees}).\\
$\nabla^{\mathrm{can}}$ &
  The canonical Rees connection, with form $\sum_jn_j\dlog(1-y_j)$ on each
  graded piece; locally nil-flat but \emph{not} flat
  (Proposition~\ref{prop:rees-can-nilflat}).\\
$F^1$-connection &
  A connection whose matrix lies in $F^1$ of the logarithmic de Rham complex.
  Any two define the same Deligne class
  (Theorems~\ref{thm:DHZ}, \ref{thm:F1-exists}).\\[6pt]

\multicolumn{2}{@{}l}{\textit{Characteristic classes and regulators}}\\[2pt]
$\CS_p(\nabla^{\Del})$ &
  The extended Chern-Simons classes, in $H^{2p-1}(X,\cc/\zz)$; the subject of
  Theorem~\ref{thm:main} (Definition~\ref{def:extended-CS}).\\
$\widehat{\ch}_p$, $\widehat H^{\hdot}$ &
  Cheeger-Simons differential character and its group; flat characters are
  exactly $H^{2p-1}(-,\cc/\zz)$ (\S\ref{app:diffchar}).\\
$\ch^\calD_p$, $c^\calD_p$ &
  Deligne-Beilinson Chern character and Chern class of a bundle, in
  $H^{2p}_\calD(-,\zz(p))$; defined without any connection
  (\S\ref{app:burgosgil}).\\
$c^B_p$ &
  The integral Betti Chern class in $H^{2p}(X,\zz)$.\\
$\alpha$ &
  The natural map $H^{2p-1}(-,\cc/\zz)\to H^{2p}_\calD(-,\zz(p))$, with kernel
  the image of $F^pH^{2p-1}$; an isomorphism on $\aaa$-realizations
  (\eqref{eq:alpha-map}, Remark~\ref{rem:alpha-kernel},
  Proposition~\ref{prop:DB-A-realization}).\\
$\partial$ &
  The Bockstein $H^{2p-1}(X,\cc/\zz)\to H^{2p}(X,\zz)$, sending $\CS_p$ to
  $c^B_p(F)$; its kernel is divisible
  (Remark~\ref{rem:betti-vs-CS}).\\
$\vol_p$ &
  The volume regulator: the imaginary component of $\CS_p$ under
  $\cc/\zz\cong\rr/\zz\oplus i\rr$, lying in $H^{2p-1}(X,\rr)$ itself, with
  no reduction modulo $\zz$ (Lemma~\ref{lem:CZ-splitting}).\\
$BGL(R)^+$, $BGL(R)^\delta$ &
  Plus construction and discrete classifying space; the map between them is
  acyclic (\S\ref{subsec:classmap}).\\
$r_{2p-1}$ &
  The universal regulator class in $H^{2p-1}(BGL(F)^+,\cc/\zz)$, defined as
  $\CS_p$ of the tautological flat bundle
  (Definition~\ref{def:universal-regulator}).\\
$r_L$, $r_{\zeta_0}$ &
  The classifying map $X\simeq|NA|\to BGL^+(K)$ built from the patching data
  (Construction~\ref{constr:rL}).\\
$\phi_q$ &
  The Grassmannian classifying map of a face-compatible presentation of $F_K$
  (Proposition~\ref{prop:grassmann-definition}).\\[6pt]

\multicolumn{2}{@{}l}{\textit{Hodge theory}}\\[2pt]
$Q$, $\tilde g$ &
  The polarization of a variation of Hodge structure, and the indefinite
  hermitian form it induces on the canonical extension
  (\S\ref{app:hermitian}, Theorem~\ref{thm:main-hermitian}).\\
$\zeta_0$ &
  A representation underlying a polarized variation of Hodge structure, to
  which $\zeta$ is deformed within $R^{\mathrm{nil}}$
  (Theorem~\ref{thm:deform-to-vhs}).\\
$(E,\theta,h)$ &
  A tame harmonic bundle: holomorphic bundle, Higgs field, pluri-harmonic
  metric (\S\ref{sec:deform-to-vhs-detail}).\\

\end{longtable}
}


\textbf{Acknowledgements}: The first author thanks University of Nice, France,  for hosting us during several visits between 2005-2022, which enabled this collaboration. The first author was partly supported by IMSc-DAE-project "Complex Algebraic Geometry". This work was completed during her stay at Max-Planck Institute for Mathematics, Bonn, 2025-2027, and their support and hospitality is gratefully acknowledged. The second author was partly supported by the ERC Horizon Synergy grant 101167526 (MALINCA), the Horizon 2020 grant 670624 (Mai Gehrke's DuaLL project), and the ANR program IsoMoDyn (ANR-25-CE40-1360).

\textit{AI Disclosure}: The proof-checking of the article used Claude AI.
\section{Preliminaries}

\subsubsection{Chern-Weil theory and secondary classes}

The study of characteristic classes of vector bundles has a long history.
For a complex vector bundle $E$ on a smooth manifold $X$, Chern-Weil theory
associates to any invariant polynomial $P$ on $\mathfrak{gl}_r$ and connection
$\nabla$ on $E$ a closed differential form $P(\nabla)$ whose de Rham class is
independent of the choice of $\nabla$.  The resulting cohomology classes are
the \emph{Chern character forms}.

When the bundle $E$ is topologically trivial (or, more generally, when a
specific Chern class vanishes), one can ask for a \emph{secondary} or
\emph{differential} refinement.  Given two connections $\nabla_0,\nabla_1$
on $E$, the Chern-Weil theory produces a canonically defined cohomology class
in $H^{2p-1}(X,\rr)$ (modulo $H^{2p-1}(X,\zz)$) measuring the difference.
This is the \emph{Chern-Simons class}.

More precisely, Cheeger and Simons \cite{CS} defined \emph{differential
characters}: homomorphisms $\hat h: Z_{2p-1}(X)\to\rr/\zz$ from smooth
$(2p-1)$-cycles, satisfying a compatibility with the boundary.  The secondary
characteristic classes are specific differential characters, and they are
refinements of the integral cohomology classes lying in the Deligne-Beilinson
cohomology $H^{2p}_\calD(X,\zz(p))$.

\subsubsection{Flat bundles and the Reznikov theorem}

A \emph{flat bundle} $(E,\nabla)$ is one where the curvature
$F(\nabla) = \nabla^2 = 0$.  Flat bundles are locally trivial, and their
monodromy gives a representation of the fundamental group.  Since the Chern
character forms $\ch_p(\nabla) = \Tr(F(\nabla)^p) = 0$ vanish for a flat
connection, the primary characteristic classes in de Rham cohomology are zero.
The secondary classes $\CS_p\in H^{2p-1}(X,\cc/\zz)$ may however be non-zero.

The question of whether the Chern-Simons classes of flat bundles are torsion
has a rich history.  Chern-Cheeger--Simons \cite{CS} showed that secondary
classes of representations over algebraically closed fields of characteristic
zero take values in $\cc/\zz$ that map to $\rr/\zz$ (the ``volume'' part).
Reznikov \cite{Reznikov} proved the fundamental theorem: for compact K\"{a}hler
manifolds/projective, all Chern-Simons classes of flat bundles are torsion.   
He gave two proofs, the first using Corlette's theorem on existence of harmonic metrics
and the abelian property of the Higgs field.  The second proof used:
\begin{enumerate}[leftmargin=2em,label=\rm(\roman*)]
  \item deformation to a representation underlying a polarized variation of
    Hodge structure (VHS), which carries an \emph{indefinite Hermitian form}
    (the polarization) preserved by the flat connection 
    by the non-abelian
    Hodge correspondence;   
    
  \item for flat bundles preserving an indefinite Hermitian form (in
    particular, those underlying a VHS), the volume regulator vanishes
    identically, for \emph{every} $p\geq 1$
    (Proposition~\ref{prop:hodge-vanish}): a purely algebraic trace identity
    for the correctly normalized transgression form shows directly that a
    $g$-preserving path from a compatible unitary connection to $\nabla$
    gives a real transgression $v_p^{\mathrm{nm}}$, so that $\vol_p(\nabla) =
    [\Im\,v_p^{\mathrm{nm}}]=0$, this is the
    proof of Proposition~\ref{prop:hodge-vanish}, which records the
    resulting vanishing for every $p\geq1$, in particular in the range
    $p\geq2$ used for Theorem~\ref{thm:main}.
\end{enumerate}

In the present paper we implement this directly in the logarithmic setting:
since the canonical extension of a connection preserving an indefinite
Hermitian form can be joined to a compatible unitary connection by a path of
$g$-preserving connections defined smoothly on all of $X$
(Lemma~\ref{lem:decomp-connection}), no deformation to a VHS, semisimplicity
hypothesis, or harmonic/pluri-harmonic metric is needed to conclude
$\vol_p=0$ for every $p\geq 1$.

\subsubsection{The logarithmic setting}

The extension to the logarithmic setting---flat bundles on a quasi-projective
variety $X^* = X\setminus D$---requires new ideas.  The Deligne canonical
extension provides a mechanism for extending the flat bundle to a bundle with
logarithmic connection on the compactification, but the non-abelian Hodge
correspondence is more delicate.  Mochizuki's deep results \cite{Mochizuki}
provide the asymptotic analysis needed near the boundary divisor.

The present paper takes a different approach, closer in spirit to the theory
of admissible variations of Hodge structure.  Rather than deforming to a
Hermitian-flat connection, we:
\begin{itemize}[leftmargin=2em]
  \item construct, via patching data, a connection that is compatible with
    the filtrations arising from the monodromy weight filtration;
  \item show that whenever a flat connection underlies a VHS, these filtrations
    have a Hermitian compatibility that forces the volume regulator to vanish;
  \item use the Borel-Reznikov argument to conclude torsion from vanishing.
\end{itemize}

\subsection{The de Rham complex on a simplicial space}
\label{subsec:deRham-simplicial}

We record here the de Rham formalism on the simplicial spaces that arise from
our covering construction. See Appendix \ref{app:simplicial}, for details.

\begin{definition}
Let $C^\bullet[\phi]$ be a strict cosimplicial cochain complex of complex
vector spaces, with cochain complex corresponding to $[\phi:m\to n]$ in
$\Delta$.  The \emph{de Rham complex} $DC^\bullet$ is the total complex of the
double complex $D^{s,t}C^\bullet$ consisting of elements
$(w_\phi)\in\prod_\phi A^s(A([n]\setminus\phi[m]))\otimes C^t[\phi]$
satisfying the compatibility: for any morphism $f:\phi\to\phi'$ as in
\eqref{eq:facemap}, one has $(|f|\times\mathrm{id})^*w_\phi = (\mathrm{id}\times f)^*w_{\phi'}$.
\end{definition}

This is modelled on the Dupont \cite{Dupont} construction for simplicial
manifolds.  The double complex gives a spectral sequence converging to the
cohomology of the total space.

For our purposes the key property is:

\begin{proposition}\label{prop:deRham-cube}
The de Rham complex on $\Sigma_F(\Uu_\hdot)$ computes the cohomology of the
realization $|\Uu_\hdot|$ with coefficients in $\cc$.
\end{proposition}

\begin{proof}
The affine pieces $A([n]\setminus\phi[m]) = \aaa^{n-m}_F$ are contractible
in the $\aaa^1$-homotopy category.  The de Rham complex on each piece is a
resolution of the constant sheaf by acyclics (by the algebraic Poincar\'{e}
lemma).  The result follows from a standard spectral sequence argument.
\end{proof}

\subsection{Overview of the regulator construction}

Let us give an informal overview of how the regulator classes are constructed
from the data of the flat bundle.

Given a local system $L$ on $X^*$ with unipotent monodromy, and a number field
$K$ over which the monodromy representation is defined, the construction
proceeds in three parallel tracks:

\medskip\noindent
\textit{Track 1: Algebraic $K$-theory-- Appendix \ref{app:rees}:}
The monodromy representation $\rho:\pi_1(X^*)\to GL_r(K)$ defines an element
of $H^0(X^*,BGL_r(K)^+)$ (roughly: a flat bundle with $K$-coefficients).
This gives, via the map $X^*\simeq S = \iiii\Cube A\to BGL_r(K)^+$, a
cohomology class $\phi^*(r_{2p-1})\in H^{2p-1}(X,K_p(K)\otimes\cc/\zz)$.

\medskip\noindent
\textit{Track 2: Analytic secondary classes--Appendix \ref{app:deligne-patch}:}
The Deligne canonical extension $(F,\nabla)$ and Deligne's patched connection
$\nabla^{\Del}$ give a locally nil-flat connection, whose Chern-Simons form
represents a class in $H^{2p-1}(X,\cc/\zz)$.

\medskip\noindent
\textit{Track 3: Deligne-Beilinson cohomology--Appendix \ref{app:burgosgil}:}
The algebraic bundle $F$ on $X$ has a Deligne-Beilinson Chern class
$\ch^{\calD}_p(F)\in H^{2p}_\calD(X,\zz(p))$.
Recall the long exact sequence of the Deligne complex:
\[
  \cdots \longrightarrow H^{2p-1}(X,\cc/\zz(p))
  \xrightarrow{\;\delta\;}
  H^{2p}_\calD(X,\zz(p))
  \longrightarrow H^{2p}(X,\zz) \oplus F^p H^{2p}(X,\cc)
  \longrightarrow \cdots
\]
For a flat bundle, the primary Chern character in de Rham cohomology vanishes:
$\ch_p^{\mathrm{dR}}(F) = 0\in F^p H^{2p}(X,\cc)$.  It follows that
$\ch^{\calD}_p(F)$ lies in the image of the \emph{edge map}
\[
  \delta : H^{2p-1}(X,\cc/\zz(p)) \;\longrightarrow\; H^{2p}_\calD(X,\zz(p)).
\]

Hence Track 3 is: the \emph{Chern-Simons class lifts the
Deligne-Beilinson Chern class} via the edge map $\delta$.

The main technical content of this paper is to show that all three tracks give
the same class in $H^{2p-1}(X,\cc/\zz)$, and that this class is torsion.

\subsection{Differential forms and currents on logarithmic spaces}

We fix notation for the analytic objects used throughout the paper.

\subsubsection{The logarithmic de Rham algebra}

Let $\bar Z$ be a smooth compact complex manifold and $D\subset\bar Z$ a
normal crossings divisor with smooth irreducible components $D_1,\ldots,D_k$.
Let $j: Z = \bar Z\setminus D\hookrightarrow\bar Z$ be the open inclusion.

\begin{definition}
The sheaf of \emph{holomorphic logarithmic $p$-forms} $\Omega^p_{\bar Z}(\log D)$
is the subsheaf of $j_*\Omega^p_Z$ locally generated, near a point where
$D = \{z_1\cdots z_r = 0\}$, by
\[
  \frac{dz_1}{z_1}\wedge\cdots\wedge\frac{dz_r}{z_r}\wedge dz_{r+1}
  \wedge\cdots\wedge dz_p.
\]
\end{definition}

The logarithmic de Rham complex $\Omega^\hdot_{\bar Z}(\log D)$ is a
sub-dga of $j_*\Omega^\hdot_Z$.  Its key property is:

\begin{theorem}[Deligne {\cite{Deligne70}}]
The inclusion $\Omega^\hdot_{\bar Z}(\log D)\hookrightarrow j_*\Omega^\hdot_Z$
is a quasi-isomorphism.  Consequently,
$H^n(\bar Z, \Omega^\hdot_{\bar Z}(\log D))\cong H^n(Z,\cc)$.
\end{theorem}

\subsubsection{The $\dlog$ map and the residue sequence}

For a smooth component $D_i\subset D$, there is a residue map
$\mathrm{Res}_{D_i}: \Omega^p_{\bar Z}(\log D)\to\Omega^{p-1}_{D_i}(\log D|_{D_i})$.
These fit into the residue exact sequence
\[
  0 \to \Omega^p_{\bar Z} \to \Omega^p_{\bar Z}(\log D)
  \xrightarrow{\sum_i\mathrm{Res}_{D_i}} \bigoplus_i\Omega^{p-1}_{D_i}(\log D|_{D_i})
  \to 0.
\]

For the Deligne canonical extension $(F,\nabla)$, the residue
$\mathrm{Res}_{D_i}(\nabla)\in\End(F|_{D_i})$ is the nilpotent endomorphism
\[
  N_i \;:=\; \frac{1}{2\pi i}\log T_i, \qquad\text{so that}\qquad
  T_i = \exp(2\pi i\,\mathrm{Res}_{D_i}\nabla),
\]
where $T_i$ is the local monodromy around $D_i$; this is the normalization
used consistently throughout the paper (in particular in
Appendix~\ref{app:deligne-patch}, \S\ref{subsec:del-patch-construct}).

\subsubsection{The mixed Hodge structure and the weight filtration}

The logarithmic de Rham complex carries a weight filtration
$W_k\Omega^\hdot(\log D)$: the subcomplex generated in degree $n$ by forms with
at most $k$ logarithmic poles:
\[
  W_k\Omega^n(\log D) = \Omega^{n-k}_{\bar Z}\wedge\Omega^k_{\bar Z}(\log D).
\]

Together with the Hodge filtration, this gives the mixed Hodge structure on
$H^n(Z,\zz)$.  The weight filtration on cohomology is the monodromy weight
filtration of the local system $\zz_{Z}$, shifted appropriately.

\subsection{The Deligne-Beilinson cohomology: explicit description}

Following \cite{EV}, we give an explicit cochain-level description of
Deligne-Beilinson cohomology.

\begin{definition}
The \emph{smooth Deligne complex} $\mathcal{D}^\hdot_{\mathrm{sm}}(p)$ on
$Z$ is the cone complex
\[
  \mathcal{D}^n_{\mathrm{sm}}(p) = F^p A^n(\bar Z,\cc)\oplus A^{n-1}(\bar Z,\cc),
\]
with differential $d_{\mathcal{D}}(\alpha,\beta) = (-d\alpha, d\beta + \alpha)$.
The Deligne-Beilinson cohomology $H^n_{\calD}(Z,\rr(p))$ is the hypercohomology
of the corresponding complex with $\rr(p) = (2\pi i)^p\rr$.
\end{definition}

A class in $H^{2p}_{\calD}(Z,\zz(p))$ can be represented by a pair $(\omega,h)$
where $\omega\in F^pA^{2p}(\bar Z)$ is a closed form representing a class in
$H^{2p}(Z,\cc)$, and $h\in A^{2p-1}(Z,\cc)/A^{2p-1}(\bar Z,\cc)$ satisfies
$dh = \omega - c_p$ for an integral cocycle $c_p$ representing the integral
Chern class.

For a vector bundle $E$ on $Z$:
\begin{itemize}[leftmargin=2em]
  \item The \emph{Deligne Chern class} $c_p^{\calD}(E)\in H^{2p}_{\calD}(Z,\zz(p))$
    maps to the topological Chern class $c_p(E)\in H^{2p}(Z,\zz)$ and to the
    de Rham class $c_p^{\mathrm{dR}}(E)\in H^{2p}(Z,\cc)$.
  \item If $E$ is flat, $c_p^{\mathrm{dR}}(E) = 0$, and $c_p^{\calD}(E)$
    lies in $H^{2p-1}(Z,\cc/\zz)\subset H^{2p}_{\calD}(Z,\zz(p))$.
\end{itemize}

\subsection{The Chern-Weil map and secondary invariants}

\begin{definition}
For a vector bundle $E$ with connection $\nabla$ on $Z$, the
\emph{Chern-Weil homomorphism} assigns to each symmetric polynomial
$P\in\mathrm{Sym}^p(\mathfrak{gl}_r)^{GL_r}$ a closed differential form
$P(\nabla)\in A^{2p}(Z)$.  The de Rham class $[P(\nabla)]$ is independent of
$\nabla$.
\end{definition}

The connections arising in this paper are not flat, and the relevant weakening
of flatness is the following.

\begin{definition}[Locally nil-flat connection]\label{def:locally-nilflat}
A connection $\nabla$ on a bundle $E\to Z$ is \emph{locally nil-flat} if there
is an open covering $Z=\bigcup_iU_i$ such that, on each $U_i$, the bundle
$E|_{U_i}$ carries a filtration $W_i^\hdot$ preserved by $\nabla|_{U_i}$ whose
induced connection on each graded quotient $\Gr^{W_i}(E|_{U_i})$ is flat.
Equivalently, the curvature $F(\nabla)$ is locally nilpotent: it is strictly
upper-triangular with respect to $W_i^\hdot$ on $U_i$.
\end{definition}

A locally nil-flat connection is in general \emph{not} flat: the curvature is
nilpotent, not zero.  The two principal examples are Deligne's patched
connection $\nabla^{\Del}$ on the canonical extension, constructed in
Appendix~\ref{app:deligne-patch}, where the filtrations are the local
monodromy weight filtrations at each stratum of $D$; and the canonical Rees
connection of Section~\ref{sec:rees-global}.  The point of the definition is
the following vanishing, which is what allows secondary invariants to be
defined at all.

\begin{proposition}\label{prop:nilflat-ch-zero}
For a locally nil-flat connection $\nabla$ and the $p$-th Chern character
polynomial $\ch_p = \frac{1}{p!}\Tr(-)^p$, we have $\ch_p(\nabla) = 0$
as a differential form, for all $p\geq1$.
\end{proposition}

\begin{proof}
Locally on each open set $U_i$ of the nil-flat structure, there exists a
filtration $W^\hdot_i$ of $E|_{U_i}$ preserved by $\nabla|_{U_i}$ with
$\Gr^{W_i}(\nabla|_{U_i})$ flat.  The curvature $F(\nabla|_{U_i})$ is
strictly upper-triangular with respect to $W^\hdot_i$ (not merely
nilpotent as an abstract endomorphism: nilpotency of the individual factor
is not by itself enough, since a product of nilpotent matrices need not be
traceless unless they are simultaneously strictly triangularized). Since
$F(\nabla|_{U_i})^p$ is then itself a product of $p$ strictly
upper-triangular matrices with respect to the \emph{same} filtration
$W^\hdot_i$, it is again strictly upper-triangular with respect to
$W^\hdot_i$, hence has zero diagonal and so $\Tr(F(\nabla|_{U_i})^p)=0$.
\end{proof}

\begin{remark}[A note on normalization]
The polynomial $\ch_p=\frac{1}{p!}\Tr(-)^p$ above is the ordinary,
un-normalized Chern-Weil polynomial, used here only to record the vanishing
$\ch_p(\nabla)=0$ as a differential form for a locally nil-flat $\nabla$ ---
a statement unaffected by any choice of overall nonzero normalizing
constant. From \S\ref{sec:volume} on, the closely related but
differently-scaled quantities $\ch_p(\nabla):=\Tr(F(\nabla)^p)$ of
Lemma~\ref{lem:bott-chern} and the normalized
$\ch_p^{\mathrm{nm}}(\nabla):=\frac{1}{p!(2\pi i)^p}\Tr(F(\nabla)^p)$ of
Definition~\ref{def:normalized-ch-vp} are used instead, the latter being the
one that corresponds to the integral Chern class $c_p(E)\in H^{2p}(X,\zz)$
under the Chern-Weil map. All three agree, up to the nonzero constant
$p!(2\pi i)^p$ or $p!$, and in particular vanish together; only
$\ch_p^{\mathrm{nm}}$ is used where the precise integral normalization
matters (e.g.\ Step 5 of the proof of Theorem~\ref{thm:CS-regulator}).
\end{remark}

The Chern-Simons class is the primary obstruction to the flatness being
``globally trivial'': it measures the twisting needed to patch the local
flat trivializations into a global one.

\section{The adapted open covering}
\label{sec:covering}

\subsection{Multi-indices and divisor strata}

Recall that $D = D_1+\cdots+D_k$ is a simple normal crossings divisor on $X$.
By definition, every irreducible component $D_i$ is smooth, and for any
multi-index $I = (i_1,\ldots,i_a)$ with distinct entries in $\{1,\ldots,k\}$
the stratum
\[
  D_I := D_{i_1}\cap\cdots\cap D_{i_a}
\]
is smooth of codimension $a$ (or empty).  We write $|I|:=a$ for the length,
and $I\subset J$ if the set of entries of $I$ is a subset of those of $J$.

\begin{notation}
The empty multi-index is $\emptyset$, with $D_\emptyset := X$.
We include $\emptyset$ as a valid multi-index, corresponding to the open set
that avoids all divisor components.
\end{notation}

\subsection{Construction of the covering}

Choose a Riemannian metric on $X$.  For each $i$, let $p_i: Y_i\to D_i$ be
the nearest-point projection from a tubular neighbourhood $Y_i$ of $D_i$.  We
may assume:
\begin{itemize}[leftmargin=2em]
  \item $p_i$ and $p_j$ commute on $Y_{ij} := Y_i\cap Y_j$;
  \item for any multi-index $I = (i_1,\ldots,i_a)$ the product
    $p_I := p_{i_1}\circ\cdots\circ p_{i_a}$ is a retraction from
    $Y_I := Y_{i_1}\cap\cdots\cap Y_{i_a}$ to $D_I$;
  \item the retractions satisfy the distance estimate
    \begin{equation}\label{eq:dist}
      d(p_I(x),p_I(y)) \leq 2\,d(x,y).
    \end{equation}
\end{itemize}

Fix small constants $0 < \eta < \epsilon$ with $3\epsilon < 4\eta$.
For each multi-index $I$ set
\[
  \epsilon_I := 4^{|I|}\epsilon,\qquad \eta_I := 4^{|I|}\eta.
\]
Define open and closed tubular neighbourhoods
\[
  T_I := \{ x\in X : d(x,p_I(x)) < \epsilon_I \},
\]
\[
  T'_I := \{ x\in X : d(x,p_I(x)) \leq \eta_I \},
\]
and set
\begin{equation}\label{eq:UI-def}
  U_I := T_I \;\setminus\; \Bigl( T_I \cap \bigcup_{J\not\subset I} T'_J \Bigr).
\end{equation}

\begin{proposition}[Adapted covering]
\label{prop:covering}
The open sets $\{U_I\}$, indexed by multi-indices, cover $X$ and satisfy:
\begin{enumerate}[label={\rm(\arabic*)}]
  \item $U_I \subset T_I$ (tubular neighbourhood condition).
  \item If $U_{I_0}\cap\cdots\cap U_{I_m}\neq\emptyset$ then, up to a
    permutation of $\{0,\ldots,m\}$, we have
    $I_0\subset I_1\subset\cdots\subset I_m$.
  \item $D_i\cap U_J\neq\emptyset \Rightarrow i\in J$.
\end{enumerate}
\end{proposition}

\begin{proof}
Property (1) holds by construction.  We prove the key intersection lemma.

\begin{lemma}\label{lem:TI-TJ}
If $I\not\subset J$ and $J\not\subset I$ then $T_I\cap T_J\subset T'_{I\cup J}$.
\end{lemma}
\begin{proof}
Since $I\not\subset J$ and $J\not\subset I$ we have
$|I\cup J| \geq \max(|I|+1,|J|+1)$, so
\[
  \eta_{I\cup J} = 4^{|I\cup J|}\eta \geq 4^{|I|}(4\eta)
  > 3\cdot 4^{|I|}\epsilon = 2\epsilon_I + \epsilon_I.
\]
More precisely, by symmetry between $I$ and $J$:
\begin{equation}\label{eq:etabound}
  \eta_{I\cup J} > \epsilon_I + 2\epsilon_J.
\end{equation}
Now suppose $x\in T_I\cap T_J$, so $d(x,p_I(x)) < \epsilon_I$ and
$d(x,p_J(x)) < \epsilon_J$.  Since
$p_{I\cup J}(x) = p_I(p_J(x))$, the distance estimate~\eqref{eq:dist} gives
$d(p_I(x),p_{I\cup J}(x))\leq 2d(x,p_J(x))$, hence
\[
  d(x,p_{I\cup J}(x)) \leq d(x,p_I(x)) + 2d(x,p_J(x))
  < \epsilon_I + 2\epsilon_J < \eta_{I\cup J},
\]
so $x\in T'_{I\cup J}$.
\end{proof}

\noindent\textit{Proof of~(2).}
Suppose $x\in U_{I_0}\cap U_{I_1}$ with $I_0\not\subset I_1$ and
$I_1\not\subset I_0$.  Then $x\in T_{I_0}\cap T_{I_1}$, so by
Lemma~\ref{lem:TI-TJ} we have $x\in T'_{I_0\cup I_1}$.  Setting
$K := I_0\cup I_1$, note that $I_0\not\subset K$ is false (since
$I_0\subset K$ by definition), so the definition~\eqref{eq:UI-def} of $U_{I_0}$
would exclude $x$ unless $I_0\subset K$ is not ``$J\not\subset I_0$''.  More
carefully: the definition says $U_{I_0} = T_{I_0}\setminus(T_{I_0}\cap
\bigcup_{J\not\subset I_0}T'_J)$.  Since $K = I_0\cup I_1$ satisfies
$K\not\subset I_0$ (because $I_1\not\subset I_0$), we have
$x\in T'_K\cap T_{I_0}\cap T_{I_1}$, which contradicts $x\in U_{I_0}$.

It follows that whenever $x\in U_{I_0}\cap U_{I_1}$, one of $I_0\subset I_1$
or $I_1\subset I_0$ holds.  By transitivity, if $x\in\bigcap_j U_{I_j}$ then
the indices form a totally ordered set; reindexing gives~(2).

\noindent\textit{Proof of~(3).}
Suppose $x\in D_i\cap U_J$ with $i\notin J$.  Then $K := J\cup\{i\}$ satisfies
$K\not\subset J$.  Since $x\in D_i$ we have $d(x,p_{\{i\}}(x))=0\leq\eta_i$,
and because $J\subset K$ the retractions $p_J$ and $p_{\{i\}}$ commute near
$x$, giving $d(x,p_K(x))=0\leq\eta_K$.  So $x\in T'_K$.  But $K\not\subset J$,
so the definition of $U_J$ would exclude $x$, a contradiction.
\end{proof}

\begin{remark}\label{rem:catlocsys}
On $U_{I_0\cdots I_m} := U_{I_0}\cap\cdots\cap U_{I_m}$ (which by (2) has
$I_0\subset\cdots\subset I_m$), the components of $D$ that can appear are
exactly those $D_i$ with $i\in I_m$.  Thus the category of local systems on
$U^*_{I_0\cdots I_m} := U_{I_0\cdots I_m}\cap X^*$ equals the category of
local systems on $U_{I_0\cdots I_m}$ whose monodromy is trivial around every
$D_i$ with $i\in I_m$ that meets $U_{I_0\cdots I_m}$.
\end{remark}

\subsection{The simplicial space of the covering}

\begin{definition}
Given an open covering $X = \bigcup_I U_I$ as above, define the associated
\emph{simplicial space} $\Uu_\hdot$ by
\[
  \Uu_n := \coprod_{I_0,\ldots,I_n} U_{I_0\cdots I_n},
\]
where the coproduct runs over all $(n+1)$-tuples of multi-indices.
\end{definition}

By Proposition~\ref{prop:covering}(2), the realization $|\Uu_\hdot|$ is
homotopy equivalent to $X$.  On $\Uu_n$ the various restriction maps and
permutation maps from the simplicial structure are defined in the standard way.


\subsection{Examples and the picture for two divisor components}

To illustrate the structure of the adapted covering, let us describe it
explicitly in the case $k=2$ (two divisor components $D_1$ and $D_2$
intersecting along $D_{12} = D_1\cap D_2$). 
We note that this picture is intrinsic to the approach taken in \cite{IS-2div}. Even though, as we said above, the case of three or more components is more complicated, the two-divisor case is a good place to start.

There are four multi-indices: $\emptyset$, $(1)$, $(2)$, $(12)$.
The corresponding open sets are:
\begin{align*}
  U_\emptyset &= X \setminus (T'_{(1)}\cup T'_{(2)}),\\
  U_{(1)} &= T_{(1)} \setminus T'_{(2)},\\
  U_{(2)} &= T_{(2)} \setminus T'_{(1)},\\
  U_{(12)} &= T_{(12)}.
\end{align*}

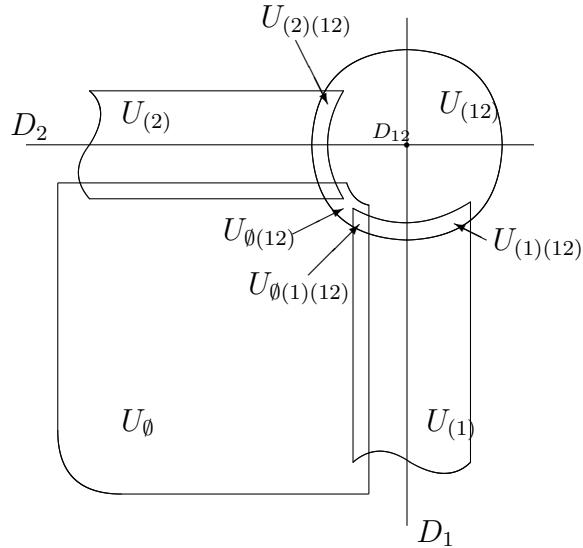
\begin{figure}[h]
\centering
{\setlength{\unitlength}{.42mm}
\begin{picture}(200,200)
\put(20,140){\line(1,0){160}}
\put(140,20){\line(0,1){160}}
\put(140,140){\circle*{2}}

\qbezier(110,140)(111,169)(140,170)
\qbezier(140,170)(169,169)(170,140)
\qbezier(170,140)(169,111)(140,110)
\qbezier(140,110)(111,111)(110,140)

\put(123,40){\line(0,1){80}}
\put(160,40){\line(0,1){82}}
\qbezier(123,40)(131,47)(141,40)
\qbezier(141,40)(153,33)(160,40)
\qbezier(123,120)(141,110)(160,122)

\put(40,123){\line(1,0){80}}
\put(40,157){\line(1,0){80}}
\qbezier(40,123)(33,131)(40,140)
\qbezier(40,140)(47,151)(40,157)
\qbezier(120,123)(110,140)(120,157)

\put(30,128){\line(1,0){91}}
\put(128,30){\line(0,1){91}}
\qbezier(121,128)(123,122)(128,121)

\put(50,30){\line(1,0){78}}
\put(30,50){\line(0,1){78}}
\qbezier(30,50)(30,30)(50,30)

\put(150,150){\ensuremath{U_{(12)}}}
\put(146,50){\ensuremath{U_{(1)}}}
\put(50,147){\ensuremath{U_{(2)}}}
\put(50,50){\ensuremath{U_{\emptyset}}}

\put(143,15){\ensuremath{D_{1}}}
\put(15,143){\ensuremath{D_{2}}}

\put(129,142){\ensuremath{{\scriptstyle D_{12}}}}

\put(109,99){\vector(1,1){16}}
\put(90,93){\ensuremath{{U_{\emptyset (1)(12)}}}}

\put(104,112){\vector(2,1){16}}
\put(82,110){\ensuremath{{U_{\emptyset (12)}}}}

\put(165,110){\vector(-2,1){10}}
\put(167,107){\ensuremath{U_{(1)(12)}}}

\put(105,173){\vector(1,-2){10}}
\put(94,177){\ensuremath{U_{(2)(12)}}}

\end{picture}}
\caption{The adapted covering near the crossing $D_{12}=D_1\cap D_2$ of two
divisor components.  The central regions $U_\emptyset$, $U_{(1)}$, $U_{(2)}$,
$U_{(12)}$ are the four open sets of \S\ref{sec:covering}; the curved
boundary near the corner is the locus $T'_{(12)}$.  The arrows indicate the
non-empty triple and double overlaps
$U_{\emptyset(1)(12)}$, $U_{\emptyset(12)}$, $U_{(1)(12)}$, $U_{(2)(12)}$
appearing in the intersection pattern described below.}
\label{fig:opencover}
\end{figure}

The intersection pattern is:
\begin{itemize}[leftmargin=2em]
  \item $U_\emptyset\cap U_{(1)}$: the transition region between the bulk and
    the neighbourhood of $D_1$, away from $D_2$.
  \item $U_\emptyset\cap U_{(12)}$: this is \emph{empty}, since by
    Proposition~\ref{prop:covering}(2) any intersection involves a chain
    $I_0\subset I_1\subset\cdots$, and $\emptyset\not\subset(12)$ is the only
    option without going through $(1)$ or $(2)$.
  \item $U_{(1)}\cap U_{(12)}$: the transition towards the corner $D_{12}$
    from the side of $D_1$.
  \item $U_\emptyset\cap U_{(1)}\cap U_{(2)}$: \emph{empty} by (2) since
    $\{(1),(2)\}$ is not linearly ordered.
  \item $U_{(1)}\cap U_{(2)}\cap U_{(12)}$: \emph{empty} for the same reason.
\end{itemize}

The non-empty multiple intersections form a simplicial set whose maximal
simplices are:
\[
  U_\emptyset,\; U_{(1)},\; U_{(2)},\; U_{(12)},\;
  U_{\emptyset(1)},\; U_{\emptyset(2)},\; U_{(1)(12)},\; U_{(2)(12)},\;
  U_{\emptyset(1)(12)},\; U_{\emptyset(2)(12)}.
\]
The cubical barycentric subdivision picture for the triple intersection
$U_{\emptyset(1)(12)}$ corresponds to three squares glued along their sides,
as described in the introduction of the companion paper.

\subsection{Local systems on intersections and monodromy}

We record here the key fact about local systems on the open sets of the
adapted covering.

\begin{lemma}\label{lem:locsys-intersect}
On $U_{I_0\cdots I_m}^* = U_{I_0\cdots I_m}\cap X^*$ (with
$I_0\subset\cdots\subset I_m$), the local system $L = V^*$ has monodromy
operators $T_i$ around $D_i$ for every $i\in I_m$.  The monodromy operators
$T_i$ for $i\in I_j\setminus I_{j-1}$ are grouped together and commute.
\end{lemma}

\begin{proof}
By Remark~\ref{rem:catlocsys}, the only divisor components that can appear
in $U_{I_0\cdots I_m}$ are those indexed by elements of $I_m$ (the largest
multiindex in the chain).  The monodromy operators around these components
commute because the corresponding divisors meet transversally in normal
crossings, and small loops around different components of a normal crossings
divisor commute in $\pi_1(X^*)$.
\end{proof}

\subsection{Homotopy type of the nerve}

\begin{proposition}\label{prop:nerve-htpy}
The geometric realization $|\Uu_\hdot|$ of the simplicial space associated
to the adapted covering is homotopy equivalent to $X$.
\end{proposition}

\begin{proof}
This is a standard result for open coverings of paracompact spaces: the
realization of the nerve of a covering by contractible open sets is homotopy
equivalent to the base space.  In our case, each $U_I$ is contractible
(it is a tubular neighbourhood, hence has the homotopy type of $D_I$, which
is smooth, and by choosing the radii small enough each $U_I$ can be assumed
contractible).  The higher intersections $U_{I_0\cdots I_m}$ are also
contractible (they are tubular neighbourhoods of $D_{I_m}$ within the
intersection, by the geometry of the covering).

More precisely: the simplicial space $\Uu_\hdot$ is a hypercover of $X$ in
the sense of \v{C}ech cohomology.  For a hypercover by contractible sets, the
map from the geometric realization to $X$ is a homotopy equivalence by the
basic theorem of \v{C}ech theory.
\end{proof}

\subsection{The affine realization and $\aaa^1$-homotopy}

\begin{proposition}\label{prop:A1-equiv}
The affine realization $\aaa\Cube A\subset(\aaa^1)^A$ is $\aaa^1$-homotopy
equivalent to the interval realization $\iiii\Cube A$, and hence to $X$.
\end{proposition}

\begin{proof}
Each cube $\aaa(c)\cong\aaa^{\dim c}$ is $\aaa^1$-contractible (the affine
$n$-space $\aaa^n$ is $\aaa^1$-contractible via the homotopy $(x,t)\mapsto tx$).
The glueing is along affine faces, which are also $\aaa^1$-contractible.
By an inductive argument on the dimension of the cubes and the standard
machinery of $\aaa^1$-homotopy theory \cite{MV}, the realization $\aaa\Cube A$
is $\aaa^1$-homotopy equivalent to $\iiii\Cube A$.
\end{proof}

\subsection{The sequential compatibility condition}

In the non-compact setting, the monodromy weight filtrations from different
boundary components must be compatible when multiple components meet.

\begin{definition}
Let $I_0,\ldots,I_n$ be a sequence of multiindices that is linearly ordered by
inclusion (up to permutation).  The \emph{sequential compatibility condition}
states that the weight filtrations $W(N_{I_j})$ on the local system near any
point of $D_{I_n}$ commute: they admit a common splitting.

Equivalently, the associated-graded operations $\Gr^{W(N_{I_j})}$ commute for
all $j=0,\ldots,n$.
\end{definition}

\begin{theorem}[Mochizuki {\cite{Mochizuki}}]
\label{thm:seq-compat}
If $X$ is a smooth projective variety, $D$ is a normal crossings algebraic
divisor, and $L$ is an admissible variation of Hodge structure on $X^*$, then
the sequential compatibility condition holds near every point of $D$.
\end{theorem}

In the present paper we use the sequential compatibility condition in the
following way: on each non-empty $U_{I_0\cdots I_m}^*$ (with
$I_0\subset\cdots\subset I_m$), the weight filtrations $W(N_{I_j})$ on the
local system $L|_{U_{I_0\cdots I_m}^*}$ commute, so we may take their
associated-gradeds in any order.

\subsection{The Rees module: detailed example in dimension two}

To clarify the multi-Rees construction, we work through the case $k=2$
explicitly.

Let $V$ be a $K$-vector space with two commuting strict filtrations
$W^1_\hdot, W^2_\hdot$.  After choosing a common splitting
$V = \bigoplus_{(a_1,a_2)\in\zz^2} V_{a_1,a_2}$, the multi-Rees module is
\[
  \xi(V;W^1,W^2) = \bigoplus_{(a_1,a_2)} x_1^{-a_1}x_2^{-a_2}V_{a_1,a_2}
  \cdot K[x_1,x_2].
\]
As an abstract $K[x_1,x_2]$-module,
\[
  \xi(V;W^1,W^2) \cong \bigoplus_{(a_1,a_2)} V_{a_1,a_2}\otimes_K K[x_1,x_2].
\]

The restrictions are:
\begin{align*}
  \xi(V;W^1,W^2)|_{x_1=1} &\cong \xi(V;W^2),\\
  \xi(V;W^1,W^2)|_{x_2=1} &\cong \xi(V;W^1),\\
  \xi(V;W^1,W^2)|_{x_1=0} &\cong \xi(\Gr^{W^1}(V);W^2),\\
  \xi(V;W^1,W^2)|_{x_2=0} &\cong \xi(\Gr^{W^2}(V);W^1),\\
  \xi(V;W^1,W^2)|_{x_1=0,x_2=0} &\cong \Gr^{W^1}\Gr^{W^2}(V)
  \cong \Gr^{W^2}\Gr^{W^1}(V).
\end{align*}

The bundle over $(\pp^1)^2$ that this defines is a twisted direct sum:
\[
  F(c) \cong \bigoplus_{(a_1,a_2)} V_{a_1,a_2}\otimes
  \Oo_{(\pp^1)^2}(-a_1Q_1-a_2Q_2),
\]
where $Q_i = \{y_i=1\}$.  Over $(\aaa^1)^2$, each factor
$\Oo(-a_iQ_i)|_{\aaa^1}$ is trivialized by the section $(1-y_i)^{a_i}$,
giving the canonical trivialization $\tau$.

\subsection{Compatibility of Rees modules with morphisms in $\Xi_K$}

A key property of the multi-Rees construction is its compatibility with the
morphisms in $\Xi_K$.

\begin{proposition}\label{prop:rees-morphism}
Let $(W,g):(V,I)\to(V',I')$ be a morphism in $\Xi_K$ (a filtration $W$ on $V$
with grading isomorphism $g:\Gr^W(V)\cong V'$).  Suppose $W^1,\ldots,W^k$ are
additional filtrations on $V$ that are coarser than $W$ (every step of $W^j$
is a step of $W$), and let $W'^1,\ldots,W'^k$ be the induced filtrations on
$\Gr^W(V)\cong V'$ via $g$.  Then there is a canonical isomorphism
\[
  \xi(V;W,W^1,\ldots,W^k)|_{x_0=0}
  \cong \xi(V';W'^1,\ldots,W'^k),
\]
where $x_0$ is the variable for the filtration $W$.
\end{proposition}

\begin{proof}
Setting $x_0=0$ in the Rees construction with respect to $W$ takes
$\Gr^W(V) = V'$.  The remaining variables $x_1,\ldots,x_k$ then parametrize
the Rees construction for the filtrations $W^1,\ldots,W^k$ on $\Gr^W(V) = V'$,
giving exactly $\xi(V';W'^1,\ldots,W'^k)$.
\end{proof}

This proposition is used in the glueing of the global Rees bundle: when
restricting from a cube $c$ to a face $c'$ obtained by setting some
$c(a_i)=t$ to $c'(a_i)=0$, one applies this proposition with $W = W(c;a_i)$
to identify the restricted bundle.

\section{Cubical sets and affine realizations}
\label{sec:cubical}

\subsection{Cubical sets}

Let $\tttt := (\{0,t,1\},\leq)$ be the poset with three elements and order
$0 < t > 1$ (so $0$ and $1$ are incomparable, both less than $t$).  

\begin{definition}
A \emph{cubical set} on an index set $\{1,\ldots,n\}$ is a subset
$C\subset\tttt^n$ closed under specialization: if $(c_1,\ldots,c_n)\in C$ and
$c'_j\leq c_j$ for all $j$, then $(c'_1,\ldots,c'_n)\in C$.
\end{definition}

The elements of $C$ are called \emph{cubes}; their dimension is the number of
coordinates equal to~$t$.  A cube $c'\in C$ is a \emph{face} of $c\in C$ if
$c'\leq c$ coordinate-wise (i.e.\ $c'$ is obtained from $c$ by replacing some
$t$'s by $0$'s or $1$'s).

\begin{definition}
The \emph{realization} of a cubical set $C$ with respect to the interval
$\iiii = [0,1]$ is
\[
  \iiii C := \bigcup_{c\in C} \iiii(c)\;\subset\;\iiii^n,
\]
where $\iiii(c) := \{(y_1,\ldots,y_n): y_j=0\text{ if }c_j=0,\;
y_j=1\text{ if }c_j=1,\; y_j\in\iiii\text{ if }c_j=t\}$.
Similarly define the \emph{affine realization} $\aaa C\subset\aaa^n$ and
\emph{projective realization} $\pp C\subset(\pp^1)^n$ using $\aaa^1_K$ and
$\pp^1_K$ in place of $\iiii$.
\end{definition}

\subsection{The cubical set of a poset}

Given a finite poset $(A,\leq)$, define
\[
  \Cube A \;\subset\; \tttt^A
\]
to be the set of functions $c:A\to\{0,t,1\}$ such that
$\mathrm{Supp}(c) := \{a : c(a)\neq 0\}$ is linearly ordered in $A$ and
$c^{-1}(1)\neq\emptyset$ (at least one coordinate equals~$1$).

\begin{definition}
The \emph{dimension} of $c\in\Cube A$ is $|c^{-1}(t)|$.  The
\emph{maximum element} of $c$ is $\mu(c) := \max\,\mathrm{Supp}(c)$.
\end{definition}

The assignment $c\mapsto\mu(c)$ is a morphism of posets
$(\Cube A,\leq)\to(A,\leq)$.

\begin{proposition}\label{prop:cuberealize}
The realization $\iiii\Cube A$ is homotopy equivalent to the classifying space
$|NA|$ of the poset $A$.
\end{proposition}

\begin{proof}
Each $k$-cube in $\Cube A$ corresponds to a chain
$a_0 < a_1 < \cdots < a_k$ in $A$ (by reading off $c^{-1}(0)$, $c^{-1}(t)$,
$c^{-1}(1)$ in order).  The cubical set $\Cube A$ is exactly the cubical
barycentric subdivision of the nerve $NA$, and the geometric realization is
invariant under barycentric subdivision.
\end{proof}

\subsection{The cubical barycentric subdivision and \texorpdfstring{$\Sigma_F$}{SigmaF}}

Let $F$ be a commutative $k$-algebra and $\Uu_\hdot$ a simplicial space.
The \emph{cubical barycentric subdivision} $\Sigma_F(\Uu_\hdot)$ is defined as
follows.

The $k$-dimensional cells of $\Sigma_F(\Uu_\hdot)$ are indexed by injective
maps $\varphi:[m]\hookrightarrow[n]$ in $\Delta$ (increasing injections of
ordered sets $[m]=\{0,\ldots,m\}$).  The complement $[n]\setminus\varphi[m]$
has cardinality $n-m$.  To such $\varphi$ we associate the affine scheme
\[
  A([n]\setminus\varphi[m]) := \aaa^{[n]\setminus\varphi[m]}_F
  = \Spec F[\ldots,t(j),\ldots]_{j\in [n]\setminus\varphi[m]},
\]
and the cell
\[
  \sigma_F(\Uu_\hdot)([\varphi:[m]\to[n]])
  := \Uu_n \times A([n]\setminus\varphi[m]).
\]

Cells are glued together along the face maps given by diagrams of the form
\begin{equation}\label{eq:facemap}
\begin{array}{ccc}
{[m]} & \xrightarrow{\varphi} & {[n]} \\
\uparrow & & \downarrow u \\
{[m']} & \xrightarrow{\varphi'} & {[n']}
\end{array}
\end{equation}
with $u$ injective, via the map
\[
  Z(\varphi,u,\varphi'): A([n]\setminus\varphi[m]) \to A([n']\setminus\varphi'[m'])
\]
defined by: $s(k)=t(j)$ if $k=u(j)$ with $j\in[n]\setminus\varphi[m]$;
$s(k)=0$ if $k=u(j)$ with $j\in\varphi[m]$; $s(k)=1$ if $k\in[n']$ is not in
$u([n])$.

\begin{theorem}\label{thm:sigma-htpy}
A vector bundle of rank $r$ on $\Sigma_F(\Uu_\hdot)$ yields a classifying map
$|\Uu_\hdot|\to BGL_r(F)^+$, well-defined up to homotopy.
\end{theorem}

\begin{proof}
The cells are contractible in the $\aaa^1$-homotopy category (each is an affine
space), and the face maps satisfy appropriate cofibration conditions. This is
the same contractibility argument used to prove
Proposition~\ref{prop:A1-equiv} above. By the standard machinery of
$\aaa^1$-homotopy theory \cite{MV}, the coend formula assembling
$\Sigma_F(\Uu_\hdot)$ is homotopy equivalent to $|\Uu_\hdot|$ in the
$\aaa^1$-homotopy category.  A vector bundle on
$\Sigma_F(\Uu_\hdot)$ therefore defines a map to the $K$-theory space
$\Omega|QVect(F)|=BGL(F)^+$.
\end{proof}

\section{Patching data, the category \texorpdfstring{$\Xi$}{Xi}, and
the filtration poset}
\label{sec:patching}

\subsection{Quasi-filtrations and quasi-gradings}

\begin{definition}
A \emph{quasi-filtration} of a vector space $V$ is a linearly ordered set
$\calF = \{W_i\}_{i\in I}$ of distinct subspaces with $W_i\subset W_j$ for
$i\leq j$, including $\{0\}$ and $V$.

A \emph{quasi-grading} of $V$ is a direct sum decomposition
$V = \bigoplus_{i\in I} V_i$ with linear ordering on $I$.  The associated
quasi-filtration has $W_i = \bigoplus_{j\leq i}V_j$.
\end{definition}

\begin{definition}
Given two quasi-filtrations $\calF,\calG$ of $V$, we say $\calF$ is
\emph{coarser} than $\calG$ (or $\calG$ \emph{refines} $\calF$), written
$\calF\leq\calG$, if every subspace in $\calF$ appears in $\calG$.
\end{definition}

\subsection{The category \texorpdfstring{$\Xi_K$}{Xi}}

\begin{definition}
Let $K$ be a field.  The \emph{category $\Xi = \Xi_K$} has:
\begin{itemize}[leftmargin=2em]
  \item \textit{Objects:} quasi-graded $K$-vector spaces
    $V = \bigoplus_{i\in I}V_i$ with linear ordering $\leq$ on $I$.
  \item \textit{Morphisms} $(V,I)\to(V',I')$: pairs $(W,g)$ where $W$ is a
    quasi-filtration of $V$ \emph{containing the standard filtration}
    (i.e.\ each $\bigoplus_{j\leq i}V_j$ is one of the $W$-steps), and
    $g:\Gr^W(V)\xrightarrow{\;\sim\;}V'$ is an isomorphism of quasi-graded
    vector spaces.
  \item \textit{Composition} of $(V,I)\xrightarrow{(W,g)}(V',I')
    \xrightarrow{(W',g')}(V'',I'')$: lift $W'$ to a quasi-filtration of $V$
    via $g$, take the join with $W$, and compose the grading isomorphisms.
  \item \textit{Identity:} the standard filtration with identity isomorphism.
\end{itemize}
\end{definition}

\begin{definition}
A \emph{weight function} for an object $(V,I)\in\Xi$ is a strictly increasing
function $\alpha:I\to\zz$.  Given a weight function, the quasi-graded space
becomes $\zz$-graded.  A morphism $(W,g):(V,I)\to(V',I')$ induces a unique
weight function on $W$ making $g$ an isomorphism of graded spaces.
\end{definition}

\subsection{Patching data}

\begin{definition}\label{def:patchingdatum}
Let $B\subset X$ be a contractible open set and $V$ a vector bundle on $X$.
A \emph{patching datum} $(W_\hdot,\tau)$ for $V$ over $B$ consists of:
\begin{itemize}[leftmargin=2em]
  \item a quasi-filtration $W_\hdot$ of $V|_B$ by strict subbundles;
  \item a trivialization $\tau$ of $\Gr^W(V|_B)$ (i.e.\ a flat connection on
    the associated-graded bundle with trivial monodromy, equivalently an
    isomorphism $\Gr^W(V|_B)\cong\Gamma(B)\otimes_\cc\Oo_B$ where $\Gamma(B)$
    is a fixed vector space).
\end{itemize}
We write $(W,\tau)\leq(W',\tau')$ if $W'\geq W$ (refinement) and $\tau'$ is
induced by $\tau$ via the natural map
$\Gr^{W}(V|_B)\to\Gr^{W'}(\Gr^W(V|_B))$.
\end{definition}

\begin{definition}
A \emph{strict collection of patching data} for a bundle $V$ over a very good
covering $X=\bigcup_i B_i$ (all non-empty multiple intersections contractible)
is a patching datum $(W(i),\tau(i))$ over each $B_i$ such that on every
non-empty $B_i\cap B_j$ the restrictions satisfy $(W(i),\tau(i))\leq(W(j),\tau(j))$
or vice versa.
\end{definition}

\begin{lemma}
If $B_I := B_{i_1}\cap\cdots\cap B_{i_k}\neq\emptyset$ then the restrictions
$\{(W(i_j),\tau(i_j))|_{B_I}\}$ form a totally ordered set under $\leq$.
\end{lemma}
\begin{proof}
Any two of them are comparable by the strict collection condition.
\end{proof}

Define $(W(I),\tau(I))$ to be the minimal element.  Let $A$ denote the nerve
poset of the covering (the poset of non-empty $I$, ordered by reverse inclusion
of index sets).

\begin{proposition}\label{prop:functor-A-Xi}
A strict collection of patching data on $V$ determines a functor $A\to\Xi_K$.
To $I\in A$ it assigns the quasi-graded space
$V(I) := \Gamma(B_I,\Gr^{W(I)}(V|_{B_I}),\tau(I))$.
For $I\leq I'$ in $A$ (i.e.\ $B_I\supset B_{I'}$) the filtration
$W(I')|_{B_I}$ refines $W(I)|_{B_I}$, defining a filtration $W(I\leq I')$ on
$V(I)$ with isomorphism $g(I\leq I'):\Gr^{W(I\leq I')}(V(I))\cong V(I')$.
This pair is the morphism $V(I)\to V(I')$ in $\Xi_K$.
\end{proposition}

\subsection{The filtration poset and its contractibility}

\begin{definition}
Let $A = \{A_i\}_{i\in I}$ be a collection of commuting nilpotent endomorphisms
of a finite-dimensional $K$-vector space $V$, and let $B\subset K := \ker A$
be a non-zero subspace (where $K$ is the intersection of the kernels).  An
\emph{$(A,B)$-nil-filtration} is a normalized quasi-filtration $\calF$ of $V$
such that:
\begin{itemize}[leftmargin=2em]
  \item the first (smallest non-zero) element $F_{\min}$ satisfies
    $F_{\min}\subset B$;
  \item for every $F\in\calF$ with $F\neq\{0\}$, there exists
    $G\in\calF$ with $G\subsetneq F$ and $A_i(F)\subset G$ for all $i$.
\end{itemize}
\end{definition}

Let $\Filt^{A,B}$ be the set of $(A,B)$-nil-filtrations, partially ordered by
inclusion of quasi-filtrations ($\calF\leq\calG$ if every step of $\calF$ is
a step of $\calG$).

\begin{theorem}[Contractibility of filtration poset]
\label{thm:filt-contractible}
If $\Filt^{A,B}$ is non-empty, then it is contractible as a poset (i.e.\ the
geometric realization $|N\Filt^{A,B}|$ is contractible).
\end{theorem}

\begin{proof}
We induct on $\dim V$.  For $\dim V = 0$ there is nothing to prove.

\medskip\noindent
\textbf{Step 1: Setup.}
Since the $A_i$ are commuting nilpotent endomorphisms, they admit a common
upper-triangular form, so $K = \ker A \neq\{0\}$.  Let $K_B := B\neq\{0\}$.

For $1\leq k\leq\dim B$, define the sub-poset
\[
  \Filt^{A,B}_k := \{\calF\in\Filt^{A,B} : \dim F_{\min}\geq k\}.
\]
Note $\Filt^{A,B}_k = \emptyset$ for $k > \dim B$.

\medskip\noindent
\textbf{Step 2: Top level.}
Set $k_{\max} := \dim B$.  For $\calF\in\Filt^{A,B}_{k_{\max}}$, the first
element $F_{\min}$ has $\dim F_{\min}\geq\dim B$ but $F_{\min}\subset B$, so
$F_{\min} = B$.  The rest of $\calF$ is determined by its image in $V/B$, and
\[
  \Filt^{A,B}_{k_{\max}}
  \;\cong\;
  \Filt^{A|_{V/B},\, K_{V/B}}(V/B),
\]
where $K_{V/B}$ is the kernel of the induced operators on $V/B$.  By induction
on $\dim V$ (since $\dim(V/B) < \dim V$), the right-hand side is contractible.

\medskip\noindent
\textbf{Step 3: Descending induction on $k$.}
Assume $\Filt^{A,B}_{k+1}$ is contractible.  We show $\Filt^{A,B}_k$ is
contractible.

For a subspace $L\subset B$ with $\dim L = k$, define the sub-posets
\[
  Y_L := \Filt^{A,B}[L\subset\mathrm{first}]
  := \{\calF : L\subset F_{\min}\},
\]
\[
  Y^=_L := \Filt^{A,B}[L=\mathrm{first}]
  := \{\calF : F_{\min} = L\}.
\]
Both are downward-closed subsets of $\Filt^{A,B}$.  We have the decomposition
\[
  \Filt^{A,B}_k = \Filt^{A,B}_{k+1}
  \;\cup\; \bigcup_{\substack{L\subset B\\ \dim L=k}} Y_L.
\]

\medskip\noindent
\textbf{Step 4: Contractibility of $Y_L$.}
Define a retraction functor $\rho: Y_L\to Y^=_L$ by $\rho(\calF) = \calF\cup\{L\}$.
This is well-defined (adding $L$ to $\calF$ gives a valid filtration) and
preserves the order: $\calG\leq\calF\Rightarrow\rho(\calG)\leq\rho(\calF)$.

Moreover, $\calF\leq\rho(\calF)$ for all $\calF\in Y_L$, giving a natural
transformation $\mathrm{Id}\Rightarrow\iota\circ\rho$ (where $\iota:Y^=_L\hookrightarrow Y_L$).
Natural transformations of functors between posets induce homotopies on
realizations.  Thus $|NY_L|\simeq|NY^=_L|$.

Now $Y^=_L \cong \Filt^{A|_{V/L},B/L}(V/L)$: a filtration with first element
$L$ corresponds via the quotient $V\to V/L$ to a filtration with first element
$B/L$.  By induction on $\dim V$ (since $\dim(V/L) < \dim V$), the right-hand
side is contractible.  Hence $Y_L$ is contractible.

\medskip\noindent
\textbf{Step 5: Intersections.}
For $L\neq L'$ with $\dim L = \dim L' = k$, the intersection $Y_L\cap Y_{L'}$
consists of filtrations whose first element contains both $L$ and $L'$, hence
has dimension $\geq k+1$, so $Y_L\cap Y_{L'}\subset\Filt^{A,B}_{k+1}$.

\medskip\noindent
\textbf{Step 6: Union is contractible.}
We have $\Filt^{A,B}_k = Z\cup\bigcup_L Y_L$ where $Z := \Filt^{A,B}_{k+1}$
is contractible (induction hypothesis), each $Y_L$ is contractible (Step 4),
$Y_L\cap Z$ is contractible (by an argument similar to Step 2 applied to
$Y_L\cap Z \cong \Filt^{A|_{V/L},B/L}_{k+1}(V/L)$, contractible by induction),
and $Y_L\cap Y_{L'}\subset Z$ for $L\neq L'$ (Step 5).

Since all sets involved are downward-closed subsets of $\Filt^{A,B}_k$, the
corresponding nerves satisfy the same union and intersection formulae.  By
Mayer-Vietoris in homology and Van Kampen for $\pi_1$, the union $N\Filt^{A,B}_k$
has trivial homotopy groups, hence is contractible (since it is a CW complex).
\end{proof}

\begin{corollary}\label{cor:Filt-A}
For a single collection $A$ of commuting nilpotent endomorphisms, the poset
$\Filt^A := \Filt^{A,K}$ of $A$-nil-filtrations is contractible whenever
non-empty (equivalently, whenever $K\neq\{0\}$).
\end{corollary}

\subsection{Existence of strict patching data}

The following is the main existence theorem for patching data.

\begin{theorem}\label{thm:patching-exists}
Let $V^*$ be a local system on $X^*$ with unipotent monodromy around each $D_i$.
Then there exists a very good open covering $X = \bigcup_i B_i$ and a strict
collection of patching data $(W(i),\tau(i))$ for $V^*$.
\end{theorem}

\begin{proof}
We define a poset presheaf $\calP$ on good opens of $X$ as follows: for a
good open $U\subset X$, let $\calP(U)$ be the poset of strict patching data
for $V^*|_U$, i.e.\ quasi-filtrations of $V^*|_{U^*}$ such that (i) the
associated-graded extends across $U\cap D$, and (ii) the monodromy logarithms
are strictly upper-triangular.  Elements of $\calP(U)$ are ordered by
refinement.  Restriction maps are obvious.

For any simply connected open $U$ that intersects $D$ in a product of discs,
the monodromy logarithms $N_1,\ldots,N_a$ acting on $V^*|_U$ are a collection
of commuting nilpotent endomorphisms, so $\calP(U) \cong \Filt^{A,K}$ with
$A = \{N_1,\ldots,N_a\}$.  By Corollary~\ref{cor:Filt-A}, $\calP(U)$ is
contractible whenever non-empty.  By the sequential compatibility theorem
(Mochizuki \cite{Mochizuki}), the posets are non-empty.

We therefore have a contractible poset presheaf, and the \v{C}ech section theorem
(Theorem~\ref{thm:cech-section}) produces the desired strict collection of
patching data.
\end{proof}


\subsection{The weight filtration and its properties}

Let $T = e^{2\pi i N}$ be a unipotent monodromy operator, with $N$ nilpotent
and $T-I = N+\frac{N^2}{2!}+\cdots$ (finite sum).

\begin{definition}[Monodromy weight filtration]
The \emph{monodromy weight filtration} $W(N)_\hdot$ centered at $0$ is the
unique increasing filtration of $V$ satisfying:
\begin{enumerate}[leftmargin=2em,label=\rm(\arabic*)]
  \item $N(W(N)_k)\subset W(N)_{k-2}$ for all $k$;
  \item $N^k:\Gr^{W(N)}_k(V)\xrightarrow{\;\sim\;}\Gr^{W(N)}_{-k}(V)$ is an
    isomorphism for $k\geq 0$.
\end{enumerate}
\end{definition}

The monodromy weight filtration exists and is unique (Jacobson-Morozov lemma).

\begin{proposition}
The Rees module $\xi(V, W(N)_\hdot)$ over $K[x]$ has the following properties:
\begin{enumerate}[leftmargin=2em,label=\rm(\arabic*)]
  \item $\xi(V,W(N))_{x=1}\cong V$.
  \item $\xi(V,W(N))_{x=0}\cong\Gr^{W(N)}(V)$.
  \item The endomorphism $\tilde N := xN$ of $V[x]$ restricts to an endomorphism
    of $\xi(V,W(N))$, and at $x=0$ it equals $0$ (since $N$ decreases weight
    by $2$ and the factor $x$ compensates).
  \item At $x=1$, $\tilde N|_{x=1} = N$.
\end{enumerate}
\end{proposition}

\begin{proof}
Claims (1) and (2) follow from the definition of the Rees module.  For (3):
if $v\in W(N)_k V$, then $xN(x^a v) = x^{a+1}Nv$.  Since $Nv\in W(N)_{k-2}$,
the element $x^{a+1}Nv$ is in $x^{a-(-2)+1}W(N)_{k-2}V\subset\xi(V,W(N))$.
\end{proof}

\subsubsection{The canonical splitting via $\mathfrak{sl}_2$}

Associated to a unipotent monodromy $T = e^{2\pi iN}$ with weight filtration
$W(N)$, there is an $\mathfrak{sl}_2$-triple $(N^-,H,N^+)$ in $\End(V)$
defined over $\qq$, where $N^+ = N$.  The corresponding splitting of $W(N)$
over $\rr$ is given by the eigendecomposition of $H$:
\[
  H\cdot v = kv \;\Leftrightarrow\; v\in\Gr^{W(N)}_k(V).
\]

This $\mathfrak{sl}_2$-splitting depends on choices, so it is not canonical,
however,
any two splittings are conjugate by an element of $\ker(N^+)$.  The
freedom in choosing the splitting corresponds precisely to the structure of
the space of $(A,B)$-nil-filtrations, which is contractible by
Theorem~\ref{thm:filt-contractible}.

\subsection{The local system and its extensions}

On $U_I^* = U_I\cap X^*$, the local system $L$ has monodromy around the
components of $D\cap U_I$ given by the operators $T_i$ for $i\in I$.
The associated weight filtration $W(N_I)$ with $N_I = \sum_{i\in I}N_i$
provides a canonical filtration of $L|_{U_I^*}$.

\begin{lemma}\label{lem:W-commute}
On $U_{I_0\cdots I_m}^*$ with $I_0\subset\cdots\subset I_m$, the weight
filtrations $W(N_{I_j})$ commute, i.e.\ they admit a common splitting.
\end{lemma}

\begin{proof}
The $N_j$ for $j\in I_m$ are commuting nilpotent endomorphisms.  For any pair
$I_a\subset I_b$, we have $N_{I_a} = \sum_{i\in I_a}N_i$ and
$N_{I_b} = N_{I_a} + \sum_{i\in I_b\setminus I_a}N_i$.  Since all $N_i$
commute, the pair $(N_{I_a},N_{I_b})$ forms part of a commuting family,
and the Jacobson-Morozov construction can be performed simultaneously for all
members of the family over an algebraically closed field.  The common splitting
exists over a finite extension of $K$.
\end{proof}

\subsection{Construction of the functor $A\to\Xi_K$}

We now make explicit how the local systems and weight filtrations assemble
into the functor $A\to\Xi_K$ of Proposition~\ref{prop:functor-A-Xi}.

\begin{construction}\label{constr:A-to-Xi}
For each multiindex $I$ with $U_I^*\neq\emptyset$:
\begin{enumerate}[leftmargin=2em,label=\rm(\arabic*)]
  \item The monodromy weight filtration $W(N_I)$ provides the filtration
    component of the patching datum $(W(I),\tau(I))$.
  \item The associated-graded local system $\Gr^{W(N_I)}(L|_{U_I^*})$ extends
    to a local system on $U_I$ (by Remark~\ref{rem:catlocsys}), which we
    trivialized by choosing a basepoint in $U_I$ and using the flat connection.
    This gives the trivialization $\tau(I)$.
  \item The quasi-graded vector space $V(I) = \Gamma(B_I,\Gr^{W(I)}(V|_{B_I}),\tau(I))$
    is the value of our functor at $I\in A$.
  \item For $I\leq I'$ (i.e.\ $I\supset I'$ as sets), the filtration
    $W(I')|_{U_I}$ refines $W(I)|_{U_I}$, and the induced filtration
    $W(I\leq I') = \Gr^{W(I)}(W(I'))$ on $V(I)$ together with the isomorphism
    $\Gr^{W(I\leq I')}(V(I))\cong V(I')$ defines the morphism in $\Xi_K$.
\end{enumerate}
\end{construction}

\begin{proposition}
Construction~\ref{constr:A-to-Xi} is well-defined and functorial.
\end{proposition}

\begin{proof}
Functoriality requires checking that the morphisms $V(I)\to V(I')$ compose
correctly when $I\leq I'\leq I''$.  This follows from the fact that
$W(N_{I''})$ refines $W(N_{I'})$ which refines $W(N_I)$, and the associated
graded operations compose in the correct order.  Specifically, if
$I\subset I'\subset I''$, then:
\[
  \Gr^{W(N_I)}\Gr^{W(N_{I'})}\Gr^{W(N_{I''})}(L) \cong
  \Gr^{W(N_{I'}\setminus W(N_I))}\Gr^{W(N_{I''}\setminus W(N_{I'}))}(\Gr^{W(N_I)}L),
\]
and this is the composition of the morphisms $V(I)\to V(I')\to V(I'')$ in $\Xi_K$.
\end{proof}

\section{The multi-Rees construction}
\label{sec:multrees}

\subsection{Single-variable Rees module}

Let $F$ be a commutative ring, $V$ a finite projective $F$-module, and
$W_\hdot$ an exhaustive increasing filtration of $V$.

\begin{definition}
The \emph{Rees module} of $(V,W_\hdot)$ is
\[
  \xi(V,W_\hdot) := \bigoplus_{a\in\zz} x^{-a}W_a V
  \;\subset\; V[x,x^{-1}],
\]
where $x$ is a formal variable.  It is a $F[x]$-submodule of $V[x,x^{-1}]$.
\end{definition}

Setting $x=0$ recovers $\Gr^W(V)$; setting $x=1$ recovers $V$.  If $W_a = V$
for $a\geq n$ and $W_a=0$ for $a < -n$, then $\xi(V,W_\hdot)$ is a finite
$F[x]$-module.  If $W_\hdot$ is a \emph{strict} filtration (meaning each
$\Gr^W_a(V)$ is projective), then $\xi(V,W_\hdot)$ is projective over $F[x]$.

\subsection{Multi-variable Rees construction}

\begin{definition}
A collection of filtrations $W^1_\hdot,\ldots,W^k_\hdot$ of a finite projective
$F$-module $V$ is \emph{locally abelian} if, after passing to an \'{e}tale
cover of $\Spec F$, there exists a decomposition $V = \bigoplus_{\mathbf{a}}
V_{\mathbf{a}}$ (where $\mathbf{a} = (a_1,\ldots,a_k)\in\zz^k$) such that
$W^j_b = \bigoplus_{a_j\leq b}(\text{relevant summands})$ for all $j,b$.
\end{definition}

\begin{definition}
Given a locally abelian collection $W^1_\hdot,\ldots,W^k_\hdot$ of strict
filtrations of $V$, choose $n$ large enough that $W^j_n = V$ for all $j$, and
define the \emph{multi-Rees module}
\[
  \xi(V; W^1_\hdot,\ldots,W^k_\hdot)
  \;\subset\; x_1^{-n}\cdots x_k^{-n} V[x_1,\ldots,x_k]
\]
to be the submodule generated by monomials $x_1^{a_1}\cdots x_k^{a_k}v$
with $v\in W^1_{a_1}V\cap\cdots\cap W^k_{a_k}V$.
\end{definition}

\begin{proposition}\label{prop:rees-projective}
If $V$ is finite projective over $F$ and the collection is locally abelian,
then $\xi(V;W^1_\hdot,\ldots,W^k_\hdot)$ is finite projective over
$F[x_1,\ldots,x_k]$.
\end{proposition}

\begin{proof}
Locally on $\Spec F$ (in the \'{e}tale topology), $V$ splits as a direct sum
of rank-one modules, each lying in a single multidegree $\mathbf{a}$.  The
multi-Rees module then decomposes as
\[
  \xi(V;\ldots) \;\cong\; \bigoplus_{\mathbf{a}} V_{\mathbf{a}}\otimes_F
  F[x_1,\ldots,x_k],
\]
which is clearly free, hence projective.  By \'{e}tale descent, the global
multi-Rees module is projective.
\end{proof}

\begin{proposition}\label{prop:rees-face}
For each coordinate $x_j$:
\begin{align*}
  x_j = 0:\quad &\xi(V;W^1,\ldots,W^k)|_{x_j=0}
  = \xi(\Gr^{W^j}(V); W^1,\ldots,\widehat{W^j},\ldots,W^k),\\
  x_j = 1:\quad &\xi(V;W^1,\ldots,W^k)|_{x_j=1}
  = \xi(V; W^1,\ldots,\widehat{W^j},\ldots,W^k).
\end{align*}
These identifications are compatible with multiple restrictions and satisfy
cocycle conditions.
\end{proposition}

\begin{proof}
For $x_j=1$: a monomial $x_1^{a_1}\cdots x_k^{a_k}v$ with
$v\in\bigcap_l W^l_{a_l}V$, when $x_j$ is set to $1$, loses the $W^j$
condition freely and becomes $x_1^{a_1}\cdots\widehat{x_j^{a_j}}\cdots x_k^{a_k}v$
with $v\in\bigcap_{l\neq j}W^l_{a_l}V$.  For $x_j=0$: the only surviving
monomials are those with $a_j\geq 0$, and $x_j^{a_j}\to 0$ for $a_j>0$,
so only $a_j=0$ contributes, giving exactly $\Gr^{W^j}$.
\end{proof}

\subsection{Restriction maps and the face structure of the Rees module}

We describe how the multi-Rees module behaves under the face maps of the
cubical set $\Cube A$, making explicit the glueing data for the global bundle.

Let $c,c'\in\Cube A$ with $c'\leq c$ (i.e.\ $c'$ is a face of $c$, obtained
by fixing some coordinates of $c$ to $0$ or $1$).  Let $P^c$ and $P^{c'}$
denote the polynomial rings for the respective cubes:
\[
  P^c = K[t(j)]_{j\in c^{-1}(t)},\qquad
  P^{c'} = K[t(j)]_{j\in c'^{-1}(t)}.
\]
There is a ring map $P^c\to P^{c'}$ sending $t(j)\mapsto 0$ if $c(j)=t$ but
$c'(j)=0$, $t(j)\mapsto 1$ if $c(j)=t$ but $c'(j)=1$, and $t(j)\mapsto t(j)$
if $c'(j)=t$ as well.

\begin{proposition}\label{prop:rees-face-detailed}
Under the ring map $P^c\to P^{c'}$ above, there is a canonical isomorphism
\[
  \xi(G(c);\{W(c;a_j)\}_{j\in c^{-1}(t)})\otimes_{P^c} P^{c'} \;\cong\;
  \xi(G(c');\{W(c';a_j)\}_{j\in c'^{-1}(t)}).
\]
These isomorphisms are compatible: for a chain $c''\leq c'\leq c$, the
composition of the two isomorphisms equals the isomorphism for $c''\leq c$.
\end{proposition}

\begin{proof}
We verify the isomorphism for a single face.  There are two cases:

\textit{Case $c'(a_i) = 1$ (setting $t(a_i)=1$):}
Setting $t(a_i)=1$ in $\xi(G(c);\ldots)$ removes the filtration $W(c;a_i)$
from the Rees construction (by Proposition~\ref{prop:rees-face}).  The core
vector space changes from $G(c) = \bigoplus_{c(a)=1}\Gr^{W(c;a)}U(a_0)$
to $G(c') = \bigoplus_{c'(a)=1}\Gr^{W(c';a)}U(a_0)$.
When $a_i$ joins $c'^{-1}(1)$, the corresponding graded piece
$\Gr^{W(c;a_i)}U(a_0)$ is absorbed into $G(c')$ via the isomorphism
$G(c')\cong G(c)\otimes\Gr^{W(c;a_i)}$.

\textit{Case $c'(a_i) = 0$ (setting $t(a_i)=0$):}
Setting $t(a_i)=0$ takes the associated-graded with respect to $W(c;a_i)$.
The core vector space $G(c')$ is $\Gr^{W(c;a_i)}(G(c))$.  The remaining
filtrations on $G(c')$ are the images of those on $G(c)$.

In both cases the result follows from Propositions~\ref{prop:rees-face}
and~\ref{prop:rees-morphism}.  The compatibility for chains follows by
applying the argument twice.
\end{proof}

\subsection{Functoriality}

\begin{proposition}\label{prop:rees-functor}
The multi-Rees construction is functorial for maps of locally abelian
multi-filtered modules: if $f:(V,W^1,\ldots,W^k)\to(V',W'^1,\ldots,W'^k)$
is a module map preserving all filtrations, then $f$ induces a map of
$F[x_1,\ldots,x_k]$-modules
$\xi(f):\xi(V;\ldots)\to\xi(V';\ldots)$.
\end{proposition}

Moreover, if the filtrations $W^1,\ldots,W^k$ are pairwise compatible (admit
a common splitting), the multi-Rees module can be computed by iterating
single-variable Rees constructions in any order.

\section{The global Rees bundle}
\label{sec:rees-global}

\subsection{Setup and local Rees bundles}

Let $A$ be a finite poset and $U:A\to\Xi_K$ a functor.  Choose weight functions
making each $U(a)$ a $\zz$-graded vector space.

For each cube $c\in\Cube A$, let $\mu(c)$ be its maximum element and write
$c^{-1}(t) = \{a_1,\ldots,a_k\}$ (in order).  Let $a_0\in c^{-1}(1)$ be any
element.

\begin{definition}
For $a\in\mathrm{Supp}(c)$, define the filtration $W(c;a)$ of $U(a_0)$ as
follows:
\begin{itemize}[leftmargin=2em]
  \item If $a_0\leq a$: take the filtration on $U(a_0)$ given by the
    morphism $U(a_0\leq a)$ in the functor $U$.
  \item If $a\leq a_0$: use the isomorphism $U(a_0)\cong\Gr^{W(a\leq a_0)}U(a)$
    to pull back the standard filtration of $U(a)$.
\end{itemize}
\end{definition}

\begin{definition}
Define the \emph{core vector space} of the cube $c$ as
\[
  G(c) := \prod_{a:c(a)=1}\Gr^{W(c;a)}(U(a_0)),
\]
which is canonically independent of the choice of $a_0$.  It carries
filtrations $W(G(c);a_i)$ for each $a_i\in c^{-1}(t)$.
\end{definition}

\begin{definition}
The \emph{local Rees bundle} over the projective cube
$\pp_K(c)\cong(\pp^1_K)^k$ is
\[
  F(c) := \xi(G(c);\; W(G(c);a_1),\ldots,W(G(c);a_k))
  \;\to\;\pp_K(c).
\]
\end{definition}

Explicitly, if $(y_1,\ldots,y_k)$ are the coordinates on $(\pp^1)^k$, the Rees
construction is performed along $y_i=1$ for each $i$:
\[
  F(c) = \sum_{j_1,\ldots,j_k}(1-y_1)^{j_1}\cdots(1-y_k)^{j_k}\cdot
  \Oo_{\pp(c)}\otimes\bigl[W_{j_1}(G(c);a_1)\cap\cdots\cap W_{j_k}(G(c);a_k)\bigr].
\]

\subsection{Glueing and the global bundle}

\begin{proposition}\label{prop:rees-glue}
For $c'\leq c$ in $\Cube A$ (i.e.\ $c'$ is a face of $c$), there is a natural
isomorphism $F(c)|_{\pp(c')}\cong F(c')$.  These isomorphisms satisfy the
cocycle condition.
\end{proposition}

\begin{proof}
When passing from $c$ to $c'$, some coordinates $a_i$ with $c(a_i)=t$ become
either $c'(a_i)=0$ or $c'(a_i)=1$.
\begin{itemize}[leftmargin=2em]
  \item If $c'(a_i)=1$: we set $y_i=1$ in the Rees construction.  By
    Proposition~\ref{prop:rees-face}, this removes the filtration $W(G(c);a_i)$,
    corresponding to the identification $G(c')\cong\Gr^{W(G(c);a_i^-)}G(c)$
    for the appropriate sub-filtration.
  \item If $c'(a_i)=0$: we set $y_i=0$, taking the associated-graded along
    $W(G(c);a_i)$.  Since $G(c')=\Gr^{W(c;a_i^-)}G(c)$ (where $a_i^-$ denotes
    the minimal face that becomes the new $c'$-min), this matches.
\end{itemize}
The cocycle condition follows from the compatibility of multiple restrictions
in Proposition~\ref{prop:rees-face}.
\end{proof}

\begin{theorem}[Global Rees bundle]
\label{thm:global-rees}
Let $A$ be a finite poset and $U:A\to\Xi_K$ a functor.  The local Rees bundles
$F(c)$ glue together to define an algebraic vector bundle $F$ on
$\pp_K\Cube A$.  Its restriction to $\aaa_K\Cube A$ carries a stratified
collection of patching data, and the corresponding functor
$\Cube A\to\Xi_K$ is the composition
$\Cube A\xrightarrow{\mu}A\xrightarrow{U}\Xi_K$.
\end{theorem}

\begin{proof}
The glueing is given by Proposition~\ref{prop:rees-glue}.  For the patching
data: over each affine cube $\aaa(c)$, define $W(F(c);\mu(c))$ to be the
filtration of $F(c)$ induced by $W(G(c);\mu(c))$.  For the trivialization
$\tau(c)$: on each graded piece $\Gr^{W(F(c);\mu(c))}_j(F(c))$, the Rees
construction with the remaining filtrations is trivial (up to line bundle
twists), giving a canonical trivialization.  Explicitly,
\[
  \Gr^{W(F(c);\mu(c))}_j(F(c))
  \cong \Oo_{\pp(c)}(-j_1Q_1-\cdots-j_kQ_k)\otimes U_j(\mu(c)),
\]
where $Q_i$ is the divisor $\{y_i=1\}$ in $\pp(c)$ and $j_1,\ldots,j_k$
are determined by the grading.  Over $\aaa(c)$ the sections
$(1-y_i)^{j_i}$ trivialize the line bundles $\Oo(-j_iQ_i)|_{\aaa(c)}$,
giving the trivialization $\tau(c)$.

The functor identification follows: $G(c)$ is constructed from $U(\mu(c))$ via
the filtrations $W(c;a_i)$, and after taking the graded pieces these recover
$U(\mu(c))$.  The composition formula $\Cube A\xrightarrow{\mu}A\to\Xi_K$
holds by construction.
\end{proof}

\begin{remark}\label{rem:log-conn}
The trivializations $\tau(c)$ over $\aaa(c)$ arise from the sections
$(1-y_i)^{j_i}$.  These give connections with logarithmic poles along
$\pp(c)\setminus\aaa(c)$, with connection form
$\sum_i j_i\,\dlog(1-y_i)\otimes\mathrm{Id}_{U_j(\mu(c))}$ on each graded piece.
\end{remark}

\subsection{Monotonicity of the patching data}

\begin{proposition}\label{prop:patching-monotone}
For $c'\leq c$ in $\Cube A$, the filtration $W(F(c),\mu(c))|_{\aaa(c')}$
is finer than $W(F(c'),\mu(c'))$, and the trivialization $\tau(c)$ induces
$\tau(c')$.
\end{proposition}

\begin{proof}
Since $c'\leq c$ we have $\mu(c')\leq\mu(c)$, so the filtration associated to
$\mu(c)$ refines that associated to $\mu(c')$ (as the functor $U$ is
order-preserving from $A$ to $\Xi_K$).  The trivialization compatibility
follows from the multiplicativity of the sections $(1-y_i)^{j_i}$ under
restriction.
\end{proof}

\subsection{The Deligne complex and logarithmic de Rham cohomology}
\label{subsec:deligne-complex-detail}

Let us review the Deligne complex in detail, following \cite{EV} and
\cite{BG}.

Let $Z$ be a smooth quasi-projective variety with a smooth compactification
$\bar Z$ such that $\bar Z\setminus Z = D_Z$ is a normal crossings divisor.
The \emph{logarithmic de Rham complex} on $\bar Z$ is the de Rham complex
$\Omega^\hdot_{\bar Z}(\log D_Z)$ of differential forms with at most simple
poles along $D_Z$.  By Deligne's theorem, its hypercohomology computes the
cohomology of $Z$ with constant coefficients:
\[
  H^n(Z,\cc) \cong \HH^n(\bar Z, \Omega^\hdot_{\bar Z}(\log D_Z)).
\]

The \emph{Hodge filtration} on $\Omega^\hdot(\log D_Z)$ is the ``b\^ete''
filtration:
\[
  F^p\Omega^\hdot(\log D_Z) := \Omega^{\geq p}_{\bar Z}(\log D_Z).
\]

\begin{definition}
The \emph{Deligne complex} $\calD(p)$ on $Z$ is the complex
\[
  \calD(p) = [\zz(p) \to \Oo_{\bar Z} \to \Omega^1_{\bar Z}(\log D_Z)
  \to \cdots \to \Omega^{p-1}_{\bar Z}(\log D_Z)],
\]
placed in degrees $0, 1, \ldots, p$.  The \emph{Deligne-Beilinson cohomology}
is $H^n_\calD(Z,\zz(p)) := \HH^n(\bar Z, \calD(p))$.
\end{definition}

There is a long exact sequence
\begin{equation}\label{eq:DB-les}
  \cdots \to H^{n-1}(Z,\cc/\zz(p)) \to H^n_\calD(Z,\zz(p))
  \to H^n(Z,\zz) \oplus F^pH^n(Z,\cc) \to \cdots
\end{equation}

For $Z = \aaa^n$ (affine space), the logarithmic forms are just regular
forms.  Since $\aaa^n$ is contractible, $H^k(\aaa^n,\cc) = 0$ for $k>0$ and
$F^pH^0(\aaa^n,\cc)=0$ for $p\geq1$.  Hence the Deligne cohomology of an
affine space is concentrated in degree \emph{one}:
\begin{align*}
  H^{1}_\calD(\aaa^n,\zz(p)) &\cong \cc/\zz(p) \cong \cc/\zz
  \quad\text{(via }(2\pi i)^p),\\
  H^{k}_\calD(\aaa^n,\zz(p)) &= 0 \qquad (k\neq1).
\end{align*}
This is Lemma~\ref{lem:DB-affine}, proved in Appendix~\ref{app:rees}.  It is
worth emphasising, since it is easy to misremember, that the group in degree
$2p$ \emph{vanishes} on a single affine cell as soon as $p\geq2$: the degree
$2p$ is produced by the combinatorics of the covering, not by any one piece.
For $p=1$ the statement is the familiar
$H^1_\calD(-,\zz(1))=\Oo^{\ast}$, which on $\aaa^n$ gives
$\cc^{\ast}\cong\cc/\zz(1)$, together with
$H^2_\calD(-,\zz(1))=\mathrm{Pic}(\aaa^n)=0$.

For the cubical realization $Z^o = \aaa\Cube A$ with affine pieces
$\aaa(c)\cong\aaa^{\dim c}$, the Mayer-Vietoris sequence for Deligne
cohomology gives:

\begin{lemma}\label{lem:DB-cube}
There is a natural isomorphism
\[
  H^{2p}_\calD(\aaa\Cube A,\zz(p)) \cong H^{2p-1}(\iiii\Cube A,\cc/\zz).
\]
\end{lemma}

\begin{proof}
Consider the \v{C}ech-to-Deligne spectral sequence for the covering of
$\aaa\Cube A$ by the affine pieces $\aaa(c)$.  By
Lemma~\ref{lem:DB-affine}, $H^t_\calD(\aaa(c),\zz(p))$ vanishes for
$t\neq1$ and equals $\cc/\zz(p)$ for $t=1$, so the $E_2$ page is
concentrated in the single row $t=1$,
\[
  E_2^{s,1} = \check H^s(\{\aaa(c)\},\cc/\zz)
  \;\Longrightarrow\; H^{s+1}_\calD(\aaa\Cube A,\zz(p)),
\]
and the spectral sequence therefore degenerates.  Taking $s=2p-1$ gives
$H^{2p}_\calD(\aaa\Cube A,\zz(p))\cong\check H^{2p-1}(\{\aaa(c)\},\cc/\zz)$.
The \v{C}ech cohomology of the covering by affine cubes computes the
cohomology of the topological realization $\iiii\Cube A$ with $\cc/\zz$
coefficients (since the affine cubes are contractible and the \v{C}ech
complex models the simplicial complex structure).
\end{proof}

\subsection{Green currents and the arithmetic Chern character}
\label{subsec:green-currents}

The Burgos-Gil theory \cite{BG} provides a concrete way to compute Deligne
cohomology classes via Green currents.  We recall the main definition, 
see also Appendix~\ref{app:burgosgil}, where the theory is developed in full.

\begin{definition}
Let $E$ be a hermitian vector bundle $(E,h)$ on a smooth complex variety $Z$.
The \emph{canonical Hermitian connection} $\nabla_h$ is the unique connection
on $E$ that is compatible with the Hermitian metric and the holomorphic
structure:
\[
  d(h(s,t)) = h(\nabla_h s,t) + h(s,\nabla_h t),\qquad
  \nabla_h^{0,1} = \overline\partial_E.
\]
Its curvature $\Theta_h = \nabla_h^2$ is a $(1,1)$-form with values in
$\End(E)$, and the \emph{Chern character form} is
\[
  \ch_p(E,h) := \frac{(-1)^p}{p!(2\pi i)^p}\Tr(\Theta_h^p).
\]
\end{definition}

\begin{theorem}[Arithmetic Chern character, Burgos-Gil {\cite[Thm.~4.3.1]{BG}}]
\label{thm:arith-chern}
The class $[\ch_p(E,h)]\in H^{2p-1,2p-1}(Z)$ (the Deligne cohomology class
defined by the current $\ch_p(E,h)$ via the canonical map from smooth forms
to currents) is independent of the choice of Hermitian metric $h$.  The
resulting class $\widehat{\ch}_p(E)\in H^{2p}_\calD(Z,\zz(p))$ is the
arithmetic Chern character.
\end{theorem}

\begin{proof}[Sketch of proof]
For two metrics $h_0,h_1$, the difference
$\ch_p(E,h_1) - \ch_p(E,h_0) = dd^c\widetilde\ch_p(E;h_0,h_1)$
where $\widetilde\ch_p$ is the Bott-Chern secondary class, a smooth form of
degree $(2p-2)$ on $Z$ (or a current with logarithmic singularities in the
non-compact case).  This follows from the variational formula for the curvature.
\end{proof}

\subsection{The $F^1$ condition and Deligne cohomology}
\label{subsec:F1-condition-detail}

We now explain in detail why $F^1$-connections provide representatives of
Deligne cohomology classes.

\begin{definition}[{\cite[Def.~5.1.1]{DHZ}}]
Let $Z$ be a smooth complex variety, $D_Z$ a normal crossings divisor, and
$E$ a vector bundle on $Z$ with filtration $F^\hdot E$.  A connection $\nabla$
on $E$ is an \emph{$F^1$-connection with respect to $(F^\hdot,D_Z)$} if:
\begin{enumerate}[leftmargin=2em,label=\rm(\roman*)]
  \item $\nabla$ is a $\mathcal{C}^\infty$ connection on $E$;
  \item $\nabla$ preserves $F^\hdot E$: $\nabla(F^p E)\subset F^p E\otimes A^1(Z)$;
  \item in any local frame adapted to $F^\hdot E$, the connection matrix $A$
    satisfies $A^{1,0}\in F^1\Omega^1_Z(\log D_Z)\otimes\End(E)$
    (the $(1,0)$-component has at most logarithmic poles along $D_Z$).
\end{enumerate}
\end{definition}

The key result connecting $F^1$-connections to Deligne cohomology is:

\begin{theorem}[{\cite[Lem.~5.2.1]{DHZ}}]\label{thm:F1-Deligne}
Suppose $E$ is a vector bundle on $Z^o = Z\setminus D_Z$, with an extension
to a bundle (still denoted $E$) on $Z$.  If $\nabla$ is an $F^1$-connection
on $E$ with respect to $D_Z$, then the Chern character form
$\ch_p(E,\nabla) := \frac{1}{p!(2\pi i)^p}\Tr(F(\nabla)^p)$ defines a class
\[
  [\nabla]_p \;\in\; H^{2p}_\calD(Z^o,\zz(p))
\]
that depends only on $E$ and not on the specific choice of $\nabla$ among
$F^1$-connections.
\end{theorem}

\begin{proof}
Given two $F^1$-connections $\nabla_0$ and $\nabla_1$ on $E$, the difference
$A = \nabla_1 - \nabla_0\in A^1(Z;\End(E))$ satisfies $A^{1,0}\in F^1$.
The Bott-Chern formula:
\[
  \ch_p(E,\nabla_1) - \ch_p(E,\nabla_0) = d\eta
\]
where $\eta$ is a secondary class.  For $F^1$-connections, $\eta$ has the
property that $\eta^{p-1,p}\in F^1A^{2p-1}(Z^o\log D_Z)$, which means $\eta$
represents zero in $H^{2p-1}(Z^o,\cc)$ and hence the two classes agree in
Deligne cohomology.

More precisely, the form $\eta$ is computed by integrating along a path
$\{\nabla_t = \nabla_0 + tA\}_{t\in[0,1]}$:
\[
  \eta = p\int_0^1\Tr(A\wedge F(\nabla_t)^{p-1})\,dt.
\]
The $F^1$ condition on $A$ and on $F(\nabla_t)$ (which inherits this condition
from the path) ensures that $\eta$ is in the correct filtration piece to give
a boundary in the Deligne complex.
\end{proof}

\subsection{Comparison between $\nabla^{\Del}$ and the Rees connection}
\label{subsec:Del-Rees-compare}

Recall from Appendix~\ref{app:deligne-patch} that $\nabla^{\Del}$ denotes
\emph{Deligne's patched connection} on the canonical extension $F$: the
$\mathcal{C}^\infty$ connection obtained from the logarithmic connection
$\nabla$ by subtracting, on each $U_I$, the singular $\dlog$ terms of the
residues, and gluing the results with a partition of unity subordinate to the
adapted covering.  It is locally nil-flat in the sense of
Definition~\ref{def:locally-nilflat}, with respect to the monodromy weight
filtrations $W(I)$, but it is not flat; its construction and properties are
established in Appendix~\ref{app:deligne-patch}.  The extended Chern-Simons
classes of Definition~\ref{def:extended-CS} are the classes of this
connection.

We now give the key comparison between $\nabla^{\Del}$ and the
$F^1$-connection on the Rees bundle.

\begin{proposition}\label{prop:Del-Rees-compare}
Let $\pi: S\to X$ be the composition of the map $\iiii\Cube A\to X$ (the
homotopy equivalence from the cubical nerve to $X$, which we denote $\pi$).
Under $\pi^*$, the Chern-Simons class of $\nabla^{\Del}$ in $H^{2p-1}(X,\cc/\zz)$
maps to the Chern-Simons class of $\nabla$ (the Rees $F^1$-connection) in
$H^{2p-1}(S,\cc/\zz)$.
\end{proposition}

\begin{proof}
The pullback $\pi^*F$ of the Deligne canonical extension to $S$ is a bundle
with a pre-patching collection induced by $\pi$.  By the uniqueness theorem
(Section~\ref{subsec:uniqueness}) for bundles with a given patching functor, up
to isomorphism compatible with the patching data there is a unique bundle on
$\iiii\Cube A$ with the given functor $\Cube A\to\Xi$.  Both $\pi^*F$ and the
restriction $F^\sigma|_{\iiii\Cube A}$ have the same functor $\Cube A\to\Xi$
(namely the composition $\Cube A\xrightarrow\mu A\to\Xi$), so they are
isomorphic as bundles with patching data.

Under this isomorphism, both $\pi^*\nabla^{\Del}$ and $\nabla|_{\iiii\Cube A}$
are connections compatible with the patching collection.  By
Theorem~\ref{thm:DHZ} applied to the simplicial scheme $\iiii\Cube A$, all
such connections define the same class.
\end{proof}

\subsection{Uniqueness of bundles with given patching data}
\label{subsec:uniqueness}

\begin{theorem}\label{thm:uniqueness}
Let $C$ be a cubical set and $\phi: C\to\Xi$ a functor.  There is, up to
isomorphism, a unique $\mathcal{C}^\infty$ vector bundle $E$ on $\iiii C$
with a stratified collection of patching data defining the functor $\phi$.
\end{theorem}

\begin{proof}
\textit{Existence} is given by the Rees construction (Theorem~\ref{thm:global-rees}).

\textit{Uniqueness:} Given two bundles $E$ and $E'$ with patching functors
both equal to $\phi$, we construct an isomorphism by induction on the skeleta
of $\iiii C$.  Over the $0$-skeleton (the vertices), the fibers $E(c)$ and
$E'(c)$ are both identified with $U(\mu(c))$ via the trivializations $\tau(c)$,
so there is a canonical isomorphism.

Over higher-dimensional cubes, the patching data determines the bundle up to
automorphisms that are upper-triangular with respect to the filtration.  Since
the filtrations come from a common functor $\phi$, these automorphisms are
further constrained to be compatible with the grading, hence trivial.
\end{proof}

\subsection{The Chern-Simons class as a secondary characteristic class}
\label{subsec:CS-secondary}

Let us elaborate on the definition of the Chern-Simons class as a differential
character.

\begin{definition}
Let $(E,\nabla)$ be a $\mathcal{C}^\infty$ complex vector bundle with connection
on a compact manifold $X$.  The \emph{Chern-Simons class} is the element
\[
  \CS_p(E,\nabla) \;\in\; H^{2p-1}(X,\cc/\zz)
\]
defined as follows: choose any smooth form $\alpha\in A^{2p-1}(X;\rr)$ such
that $d\alpha = \ch_p^{\mathrm{nm}}(\nabla)$ (possible if
$\ch_p^{\mathrm{nm}}(\nabla)$ is exact, as happens when $E$ is trivial or
when the representation is flat), where $\ch_p^{\mathrm{nm}}$ is the
normalized Chern character form of \S\ref{sec:volume} (\emph{not} the
un-normalized $\ch_p=\Tr(F^p)$ of Lemma~\ref{lem:bott-chern}): this is the
choice that makes the map to $c_p(E)\in H^{2p}(X,\zz)$ below the genuine,
untwisted integral Chern class rather than a class valued in the twisted
lattice $\zz(p)=(2\pi i)^p\zz$.  The class of $\alpha\pmod{\zz}$ is
independent of the choice.  More precisely, via the long exact sequence
\[
  \cdots\to H^{2p-1}(X,\rr)\to H^{2p-1}(X,\cc/\zz)\to H^{2p}(X,\zz)\to\cdots
\]
the class $\CS_p$ maps to the integral Chern class $c_p(E)\in H^{2p}(X,\zz)$.
\end{definition}

For a flat bundle, $c_p(E)_\rr = 0$, so $\CS_p$ lifts to $H^{2p-1}(X,\cc/\rr)
\cong H^{2p-1}(X,\rr/\zz)$.  The volume regulator is the imaginary part of $\CS_p$ under the
splitting $\cc/\zz = \rr/\zz\oplus i\rr$ (Lemma~\ref{lem:CZ-splitting}):
$\vol_p(\nabla) = \Im_{\cc/\zz}^*(\CS_p)\pmod{\zz}$, lifted to $H^{2p-1}(X,\rr)$
via the map $H^{2p-1}(X,\rr)\to H^{2p-1}(X,\rr/\zz)$.

\subsection{Explicit formula for the volume regulator via a path integral}

\begin{proposition}\label{prop:vol-explicit}
Let $\nabla$ be a locally nil-flat connection on $E$, and let $h$ be a
Hermitian metric on $E$.  Write $\nabla = \nabla_h + A_h$ where $\nabla_h$ is
the Chern connection for $(E,h)$ and $A_h$ is an $\End(E)$-valued $1$-form.
Define the path $\nabla_t := \nabla_h + tA_h$ for $t\in[0,1]$.  Then
\[
  \vol_p(\nabla) = \left[\Im\,\frac{p}{p!\,(2\pi i)^p}
  \int_0^1\Tr(A_h\wedge F(\nabla_t)^{p-1})\,dt\right]
  \;\in\; H^{2p-1}(X,\rr).
\]
\end{proposition}

\begin{proof}
This is Definition~\ref{def:vol-p-normalized}, with the explicit path
$\nabla_t = \nabla_h + tA_h$ from the Chern connection $\nabla_0 = \nabla_h$
to $\nabla_1 = \nabla$. The derivative $\dot\nabla_t = A_h$ is constant in
$t$.
\end{proof}

\begin{corollary}\label{cor:vol-upper-triangular}
If $\nabla$ is locally nil-flat with respect to a filtration $W^\hdot E$, and
if there exists a Hermitian metric $h$ such that:
\begin{enumerate}[label=\rm(\arabic*)]
  \item $\nabla_h$ preserves $W^\hdot E$;
  \item $A_h = \nabla - \nabla_h$ is strictly upper-triangular with respect
    to $W^\hdot E$;
\end{enumerate}
then $\vol_p(\nabla) = 0$.
\end{corollary}

\begin{proof}
Under the two conditions, $F(\nabla_t) = F(\nabla_h + tA_h)$.  Both $F(\nabla_h)$
(which preserves filtration steps) and $A_h$ (strictly upper-triangular) make
the product $A_h\wedge F(\nabla_t)^{p-1}$ of trace zero: the trace of a
product of matrices where one is strictly upper-triangular is zero if the other
matrices are upper-triangular.  Hence the form is zero.
\end{proof}

\subsection{Proof that the filtration poset is contractible: a topological perspective}
\label{subsec:filt-top}

We elaborate on the topological argument used in the proof of
Theorem~\ref{thm:filt-contractible}.

The poset $\Filt^{A,B}$ is a partially ordered set, and we want to show
$|N\Filt^{A,B}|$ is contractible.  We use the following general criterion:

\begin{lemma}[Quillen's Fiber Lemma]\label{lem:quillen-fiber}
Let $f:P\to Q$ be a functor between posets (order-preserving map).  If for
every $q\in Q$ the fiber poset $f/q := \{p\in P : f(p)\geq q\}$ is
contractible, and if $Q$ is contractible, then $P$ is contractible.
\end{lemma}

This is a special case of Quillen's Theorem~A \cite{Quillen}.

We apply this as follows.  Consider the functor
\[
  f: \Filt^{A,B}\to\mathcal{L}(B),\quad \calF\mapsto F_{\min}(\calF),
\]
where $\mathcal{L}(B)$ is the poset of non-zero subspaces of $B$, ordered
by reverse inclusion ($L\leq L'$ iff $L\supset L'$).  The fiber $f/L$ for
$L\in\mathcal{L}(B)$ consists of those filtrations whose first element
\emph{contains} $L$, which is exactly $\Filt^{A,B}[L\subset\mathrm{first}]$.

We showed (Step 4 of the proof) that each such fiber is contractible.  We also
need $\mathcal{L}(B)$ to be contractible.  The poset $\mathcal{L}(B)$ (non-zero
subspaces of $B$ by reverse inclusion) has a maximum element $B$ itself, hence
is contractible.

Applying Quillen's Fiber Lemma gives the contractibility of $\Filt^{A,B}$.

\begin{remark}
The above argument is cleaner than the direct Mayer-Vietoris argument we gave
in Section~\ref{sec:patching}.  However, the direct argument has the advantage
of being more self-contained and not requiring Quillen's theorem.
\end{remark}

\subsection{The indefinite case: detailed computation}
\label{subsec:indefinite-detail}

We give a more detailed treatment of the vanishing argument for connections
preserving an indefinite Hermitian form.

Let $g$ be an indefinite Hermitian form of signature $(r_+,r_-)$ on $E$,
and let $h = h_+\oplus h_-$ be a positive definite metric compatible with
a splitting $E = E_+\oplus E_-$ such that $g|_{E_+} = h_+$ and
$g|_{E_-} = -h_-$.

\begin{lemma}\label{lem:g-preserving-form}
A connection $\nabla$ preserves $g$ if and only if, writing
$\nabla = \begin{pmatrix}\nabla_+&\beta\\-\beta^*&\nabla_-\end{pmatrix}$
in the splitting $E = E_+\oplus E_-$, the connections $\nabla_\pm$ preserve
$h_\pm$ and $\beta\in A^1(X;\Hom(E_-,E_+))$ is an arbitrary $1$-form.
Here $\beta^*$ is the $h_\pm$-adjoint of $\beta$.
\end{lemma}

\begin{proof}
Computing $\nabla g = 0$ in the splitting gives the conditions on $\nabla_\pm$
and the anti-symmetric condition on the off-diagonal part, which is exactly
the stated form.
\end{proof}

\begin{lemma}\label{lem:g-curvature}
In the setup of Lemma~\ref{lem:g-preserving-form}, the curvature of $\nabla$ is
\[
  F(\nabla) = \begin{pmatrix}F_+-\beta\wedge\beta^*&\nabla_+\beta+\beta\wedge\nabla_-\\\nabla_-(-\beta^*)+(-\beta^*)\wedge\nabla_+&F_-+\beta^*\wedge\beta\end{pmatrix}.
\]
\end{lemma}

\begin{proposition}\label{prop:Re-trace-zero}
If $\nabla$ preserves $g$ and $A = \dot\nabla_t$ is the derivative along a
path of $g$-preserving connections, then, writing $A^{\mathrm{nm}} :=
\frac{1}{2\pi i}A$ and $F^{\mathrm{nm}}(\nabla) := \frac{1}{2\pi i}F(\nabla)$
for the normalized connection variation and curvature of
Definition~\ref{def:normalized-curvature}, for every $p\geq1$
\[
  \Im\,\Tr\bigl(A^{\mathrm{nm}}\wedge F^{\mathrm{nm}}(\nabla)^{p-1}\bigr) = 0.
\]
\end{proposition}

\begin{proof}
By Lemma~\ref{lem:g-preserving-form}, in the splitting $E = E_+\oplus E_-$,
both $A$ and $F(\nabla)$ take values in the Lie algebra $\mathfrak{u}(g)$ of
the group $U(g) = U(r_+,r_-)$; in the notation of
Lemma~\ref{lem:koszul-parity}, taking $\mathfrak{u}_E := \mathfrak{u}(g)$,
this says $A\in A^1(X;\mathfrak{u}_E)$ and $F(\nabla)\in
A^2(X;\mathfrak{u}_E)$, and hence, since $2\pi i$ is a nonzero scalar, also
$A^{\mathrm{nm}}\in A^1(X;\mathfrak{u}_E)$ and $F^{\mathrm{nm}}(\nabla)\in
A^2(X;\mathfrak{u}_E)$.

By Lemma~\ref{lem:koszul-parity}(b), $\overline{\Tr(A\wedge
F(\nabla)^{p-1})} = (-1)^p\Tr(A\wedge F(\nabla)^{p-1})$. Since
$\overline{(2\pi i)^p} = (-2\pi i)^p = (-1)^p(2\pi i)^p$, the two factors of
$(-1)^p$ cancel identically, for every $p$:
\[
  \overline{\Tr\bigl(A^{\mathrm{nm}}\wedge F^{\mathrm{nm}}(\nabla)^{p-1}\bigr)}
  = \frac{\overline{\Tr(A\wedge F(\nabla)^{p-1})}}{\overline{(2\pi i)^p}}
  = \frac{(-1)^p\Tr(A\wedge F(\nabla)^{p-1})}{(-1)^p(2\pi i)^p}
  = \Tr\bigl(A^{\mathrm{nm}}\wedge F^{\mathrm{nm}}(\nabla)^{p-1}\bigr).
\]
Hence $\Tr(A^{\mathrm{nm}}\wedge F^{\mathrm{nm}}(\nabla)^{p-1})$ is real, so
its imaginary part vanishes, for every $p\geq1$.
\end{proof}

\section{The \texorpdfstring{$F^1$}{F1}-connection and the Burgos-Gil
comparison}
\label{sec:F1}

\subsection{Review of Deligne-Beilinson cohomology}

Let $Z$ be a smooth complex algebraic variety with a normal crossings divisor
$D_Z$, and set $Z^o := Z\setminus D_Z$.

\begin{definition}
The \emph{Deligne-Beilinson cohomology} groups are defined by
\[
  H^n_{\calD}(Z,\zz(p)) := H^n(Z, \zz(p)\to\Oo_Z\to\Omega^1_Z(\log D_Z)\to
  \cdots\to\Omega^{p-1}_Z(\log D_Z)),
\]
where $\zz(p) := (2\pi i)^p\zz$ and the complex is the \emph{Deligne complex}
placed in degrees $0,1,\ldots,p-1$.  For smooth projective $X$ with
$D_Z=\emptyset$, this recovers Deligne's definition \cite{Deligne74}.
\end{definition}

For the affine space $\aaa^n_\cc$ we have, for $p\geq1$,
\begin{equation}\label{eq:DB-affine}
  H^{1}_{\calD}(\aaa^n_\cc,\zz(p)) \cong \cc/\zz,\quad
  H^{k}_{\calD}(\aaa^n_\cc,\zz(p)) = 0 \text{ for } k\neq1 .
\end{equation}
This reflects the fact that $\aaa^n$ is contractible, so that the only
surviving contribution is $H^0(\aaa^n,\cc)/\bigl(F^pH^0+H^0(\aaa^n,\zz(p))\bigr)
=\cc/\zz(p)$ in degree one; see Lemma~\ref{lem:DB-affine}.  In particular the
degree $2p$ group vanishes on an affine space for $p\geq2$, and the isomorphism
with $\cc/\zz$-cohomology in degree $2p-1$ arises only after assembling the
cells (Proposition~\ref{prop:DB-A-realization}).

\subsection{The \texorpdfstring{$F^1$}{F1}-connection}

\begin{definition}
Let $(E,F^\hdot)$ be a holomorphic vector bundle with a filtration on a complex
manifold $Z$, and let $D_Z$ be a normal crossings divisor.  An
\emph{$F^1$-connection} on $E$ with respect to $(F^\hdot,D_Z)$ is a
$\mathcal{C}^\infty$ connection $\nabla$ on $E$ such that in any local frame
compatible with $F^\hdot$, the connection matrix $A$ satisfies
$A\in F^1A^1(Z\log D_Z,\End(E))$, i.e.\ the $(1,0)$-part of $A$ lies in
$\Omega^1_Z(\log D_Z)$ and the overall form lies in $F^1$.
\end{definition}

This notion is due to Dupont-Hain-Zucker \cite{DHZ}.  An $F^1$-connection
whose curvature vanishes to the extent required defines a representative of the
Chern character in Deligne cohomology.

\begin{theorem}[Dupont-Hain-Zucker {\cite[Lem.~6.1.3]{DHZ}}]
\label{thm:DHZ}
Let $E$ be a vector bundle on $Z^o = Z\setminus D_Z$, where $Z$ is proper and
$E$ admits an extension to $Z$.  Then any two $F^1$-connections on $E$ with
respect to $D_Z$ define the same class in $H^{2p}_{\calD}(Z^o,\zz(p))$.
\end{theorem}

\subsection{Construction of a compatible \texorpdfstring{$F^1$}{F1}-connection
on the Rees bundle}

\begin{theorem}\label{thm:F1-exists}
Let $F$ be the bundle of Theorem~\ref{thm:global-rees} and $\sigma:K\hookrightarrow\cc$
a field embedding.  There exists an $F^1$-connection $\nabla$ on
$F^\sigma|_{\aaa_\cc\Cube A}$ compatible with the stratified patching data.
\end{theorem}

\begin{proof}
We construct $\nabla(c)$ by induction on $\dim(c)$.

\medskip\noindent
\textbf{Base case $\dim(c)=0$:}
A $0$-cube corresponds to a single element $\{a_0\}\subset A$ with $c(a_0)=1$.
Here $\pp(c)$ is a point and there is nothing to construct.

\medskip\noindent
\textbf{Inductive step:}
Assume $\nabla(c')$ has been constructed for all faces $c'$ of $c$ of dimension
$< n = \dim(c)$.

We work over $Z := \pp(c)\cong(\pp^1)^n$ with divisor at infinity
$D := Z\setminus\aaa(c)$ (the union of coordinate hyperplanes $\{y_i=\infty\}$),
and the boundary
$Y := \bigcup_i\{y_i=0\}\cup\{y_i=1\}$ (the faces of $\aaa(c)$).

\textit{Step A: Initial connection.}
Choose any $F^1$-connection $\nabla'(c)$ on $F(c)$ such that:
\begin{itemize}[leftmargin=2em]
  \item The filtration $W(F(c);\mu(c))$ is $\nabla'(c)$-flat.
  \item On $\Gr^{W(F(c);\mu(c))}(F(c))$, $\nabla'(c)$ induces the
    log-gt-connection $\tau(c)$.
\end{itemize}
Such a connection exists locally; we will patch using the partition of unity.
Over any affine open subset $\aaa^n(c)$, trivialize $F(c)$ compatibly with the
filtration $W(F(c);\mu(c))$.  The diagonal blocks of the connection matrix are
determined by $\tau(c)$ (logarithmic forms with poles at $y_i=1$), and the
off-diagonal strictly upper-triangular blocks can be chosen freely.

\textit{Step B: Boundary correction.}
The inductive connections $\nabla(c')$ on the faces $c'\subsetneq c$ are
compatible among themselves (by the inductive hypothesis).  Let
\[
  A(c,c') := \nabla(c')|_{d(c')\pp(c)} - \nabla'(c)|_{d(c')\pp(c)}
  \;\in\; A^0(\mathcal{C}^\infty\text{-section of }F^1\Omega^1(\log D)\otimes\End F(c)),
\]
where $d(c')$ denotes the face embedding.  These $A(c,c')$ are compatible
under further restriction: if $c''\leq c'\leq c$ then
$A(c,c'')|_{c'''} = A(c,c'')$ for $c'''\leq c''$.

By the flasqueness of smooth $(1,0)$-forms on $(\pp^1)^n$, extend
$\{A(c,c')\}$ to a global section $A(c)$ over $\pp(c)$, strictly
upper-triangular with respect to $W(F(c);\mu(c))$, that restricts to $A(c,c')$
on each face $c'$.

\textit{Step C: Final connection.}
Set $\nabla(c) := \nabla'(c) + A(c)$.  It is an $F^1$-connection (since
$F^1A^1$ is a linear subspace preserved under addition), compatible with the
filtration (since $A(c)$ is upper-triangular), inducing $\tau(c)$ on the graded
(since $A(c)$ is \emph{strictly} upper-triangular), and matching $\nabla(c')$
on each face by construction.
\end{proof}

\subsection{The Burgos-Gil arithmetic Chern character}
\label{subsec:BG-Chern-character}

We recall the relevant parts of the Burgos-Gil theory \cite{BG};
see Appendix~\ref{app:burgosgil} for the full development.

Throughout this subsection $\bar Z$ is a smooth projective variety,
$D_Z\subset\bar Z$ a normal crossings divisor, and $Z^o=\bar Z\setminus D_Z$.
The object to which we shall apply the results, the cubical realization
$\pp_K\Cube A$, is \emph{not} of this form; the reduction is carried out in
\S\ref{subsec:BG-on-cubical}, where $\pp_K\Cube A$ is treated as a cubical
scheme with smooth projective terms and, alternatively, mapped to a
Grassmannian.
Let $\bar E$ be a vector bundle on $\bar Z$ and $h$ a Hermitian metric on
$\bar E$.  The \emph{arithmetic Chern character} of Burgos-Gil is a class
\[
  \widehat{\ch}_p(\bar E,h)\;\in\;\widehat{H}^{2p}\bigl(\bar Z,\zz(p)\bigr),
\]
defined using Green currents, whose image under the natural map
$\widehat{H}^{2p}(\bar Z,\zz(p))\to H^{2p}_\calD(\bar Z,\zz(p))$ is the
Deligne-Beilinson Chern character $\ch^\calD_p(\bar E)$, independently of $h$
(Theorem~\ref{thm:BG-full}(2)).

The key comparison result is the following.  Note that the class produced by
an $F^1$-connection is a class on the \emph{compactification} $\bar Z$, not
merely on $Z^o$; this is the whole content of the statement, and it is what
allows a connection defined by local data on $Z^o$ to compute a global
invariant of $\bar E$.

\begin{theorem}[Burgos-Gil {\cite[Thm.~7.3.5]{BG}}]
\label{thm:BG-comparison}
Let $\bar E$ be a vector bundle on $\bar Z$ and let $\nabla$ be an
$F^1$-connection on $\bar E$ with logarithmic poles along $D_Z$.  Then the
Chern-Weil class
\[
  [\nabla]_p\;\in\;H^{2p}_\calD\bigl(\bar Z,\zz(p)\bigr)
\]
determined by $\nabla$ through the Bott-Chern construction equals
$\ch^\calD_p(\bar E)$.  In particular $[\nabla]_p$ is independent of the choice
of $F^1$-connection, and depends only on the bundle $\bar E$.
\end{theorem}

\begin{remark}\label{rem:BG-machinery}
The Burgos-Gil theory \cite{BG} constructs arithmetic characteristic classes
via a double complex of smooth forms on the simplicial scheme resolving
$\bar Z$, incorporating both the algebraic de Rham cohomology (via
$\Omega^\hdot(\log D_Z)$) and the analytic data (via the
$\overline\partial$-operator and currents).  For our purposes the essential
consequences are the independence statement in
Theorem~\ref{thm:BG-comparison}---which is what permits the Rees connection,
Deligne's patched connection and a Grassmannian Hermitian connection to be
used interchangeably---and the passage from a Deligne class to a class in
$H^{2p-1}(-,\cc/\zz)$ recorded in Theorem~\ref{thm:BG-full}(4), which requires
the Chern-Weil forms to vanish and is supplied here by local nil-flatness
(Proposition~\ref{prop:nilflat-ch-zero}).
\end{remark}

\subsection{Comparison of the Chern-Simons class with the regulator}

\begin{theorem}\label{thm:CS-regulator}
Let $F_K$ be the Rees bundle of Theorem~\ref{thm:global-rees} and $\nabla$ the
$F^1$-connection of Theorem~\ref{thm:F1-exists}.  Let $S = \iiii\Cube A$ be
the interval realization.  Under the homotopy equivalence
$S\simeq|\Uu_\hdot|\simeq X$, the Chern-Simons class defined by $\nabla$ in
$H^{2p-1}(Z^o_{\cc,\sigma},\cc/\zz)$ pulls back to the regulator class
$r^\sigma_{2p-1}|_S\in H^{2p-1}(S,\cc/\zz)$.
\end{theorem}

\begin{proof}
We break the proof into steps.

\textit{Step 1: Algebro-geometric Chern class.}
Since $Z = \pp_K\Cube A$ is a projective scheme, there exists a surjection
$\Oo_Z(n)^s\twoheadrightarrow F$ for some $n,s$.  This expresses $F$ as the
pullback of the tautological bundle $\mathcal{S}$ on the Grassmannian
$\mathbf{Gr}(r,s)$ via the classifying map $\phi:Z\to\mathbf{Gr}(r,s)$.

\textit{Step 2: Universal $F^1$-connection.}
Choose a Hermitian metric on $\cc^s$ and equip $\mathcal{S}$ with the
induced Hermitian connection $\nabla_{\mathcal{S}}$.  This is an $F^1$-connection
on $\mathcal{S}$ and defines a representative $\ch^{\calD}_p(\mathcal{S})$ in
Deligne cohomology of the Grassmannian.  Pulling back via $\phi$:
\[
  \phi^*\nabla_{\mathcal{S}} =: \nabla_F
  \quad\text{is an }F^1\text{-connection on }F\text{ over }Z.
\]
By functoriality of Deligne cohomology,
$\ch^{\calD}_p(F) = \phi^*\ch^{\calD}_p(\mathcal{S})$.

\textit{Step 3: Independence by Dupont-Hain-Zucker.}
By Theorem~\ref{thm:DHZ}, the class in $H^{2p}_{\calD}(Z^o_\sigma,\zz(p))$
defined by $\nabla$ (constructed in Theorem~\ref{thm:F1-exists}) equals the
class defined by $\nabla_F$.  Both are equal to $\ch^{\calD}_p(F^\sigma|_{Z^o})$.

\textit{Step 4: Regulator identification.}
The bundle $F_K^\sigma$ on $Z^o$ defines, via the $\aaa^1$-homotopy equivalence
$Z^o\simeq S$, a map $S\to BGL_r(K)^+$.  The regulator class
$r^\sigma_{2p-1}\in H^{2p-1}(BGL^+(K),\rr/\zz)$ pulls back to a class on $S$.
By~\eqref{eq:DB-affine} together with the descent spectral sequence over the
cubical decomposition (Lemma~\ref{lem:DB-cube}),
\[
  H^{2p}_{\calD}(Z^o,\zz(p)) \cong H^{2p-1}(S,\cc/\zz),
\]
and under this identification $\ch^{\calD}_p(F^\sigma)$ corresponds to
$r^\sigma_{2p-1}|_S$.  This identification uses the fact that the regulator
is exactly the Deligne Chern class in the case of the tautological bundle
on a Grassmannian.

\textit{Step 5: Chern-Simons form.}
We now bridge the Deligne-cohomological identification of Steps 1--4 with
the actual $\cc/\zz$-valued Chern-Simons class appearing in the statement of
the theorem.

By Proposition~\ref{prop:nilflat-ch-zero}, the $F^1$-connection $\nabla$ of
Theorem~\ref{thm:F1-exists} is locally nil-flat, so $\ch_p(\nabla)=0$, hence
also $\ch_p^{\mathrm{nm}}(\nabla)=0$.
Choose any reference $F^1$-connection $\nabla_0$ on $F^\sigma|_{Z^o}$ that is
also locally nil-flat (for instance the canonical connection
$A^{\mathrm{can}}$ of \S\ref{subsec:canconn} below, restricted to $Z^o$, which
is flat by Lemma~\ref{lem:can-flat}).  The \emph{normalized} Chern-Simons
transgression form of Appendix~\ref{app:diffchar}, \S A.2--A.3
(the normalization of \eqref{eq:CS-form-normalised} there, and of
$v_p^{\mathrm{nm}}$ above),
\[
  v_p^{\mathrm{nm}}(\{\nabla_0,\nabla\}) \;=\; \frac{p}{p!\,(2\pi i)^p}\int_0^1
  \Tr\bigl(\dot\nabla_t\wedge F(\nabla_t)^{p-1}\bigr)\,dt
  \;\in\; A^{2p-1}(Z^o_{\cc,\sigma}),
\]
along any path $\nabla_t$ from $\nabla_0$ to $\nabla$, is closed: by
Lemma~\ref{lem:bott-chern} divided by $p!(2\pi i)^p$,
$d\,v_p^{\mathrm{nm}} = \ch_p^{\mathrm{nm}}(\nabla)-\ch_p^{\mathrm{nm}}(\nabla_0) = 0$,
both terms vanishing by local nil-flatness. Its class modulo $\zz$ (plain
$\zz$, \emph{not} the twisted lattice $\zz(p)=(2\pi i)^p\zz$ of Steps~1--4:
this is exactly what the normalization achieves --- compare the $p=1$ check
in the proof of Lemma~\ref{lem:gauge-variation}, where the analogous
normalized quantity $\Tr(A)/2\pi i$, and not the raw $\Tr(A)$, is what has
honestly integral, $\zz$-valued periods) defines
\[
  \CS_p(\nabla) := [v_p^{\mathrm{nm}}(\{\nabla_0,\nabla\})] \;\in\; H^{2p-1}(Z^o_{\cc,\sigma},\cc/\zz),
\]
the \emph{Chern-Simons class defined by $\nabla$} referred to in the
statement of the theorem. (An earlier version of this Step used the
un-normalized $v_p$ in place of $v_p^{\mathrm{nm}}$; that formula's natural
periods lie in $\zz(p)$, not $\zz$, exactly parallel to the un-normalized
$\Tr(A)$ having periods in $2\pi i\zz=\zz(1)$ at $p=1$, and so did not
literally land in the group $H^{2p-1}(Z^o_{\cc,\sigma},\cc/\zz)$ claimed
above without the normalizing factor supplied here.)

Recall from the edge map $\delta$ of Track~3 (the long exact
sequence~\eqref{eq:DB-les}) that
\[
  \delta : H^{2p-1}(Z^o_{\cc,\sigma},\cc/\zz) \longrightarrow
  H^{2p}_\calD(Z^o_\sigma,\zz(p))
\]
sends a differential character with zero curvature to its Deligne class. By
construction, $\delta(\CS_p(\nabla)) = [\nabla]_p$, the class of $\nabla$ as
an $F^1$-connection in $H^{2p}_\calD(Z^o_\sigma,\zz(p))$.  By Steps 1--3,
\[
  \delta(\CS_p(\nabla)) \;=\; [\nabla]_p \;=\; \ch^{\calD}_p(F^\sigma|_{Z^o}).
\]

It remains to identify $\delta$ with the isomorphism of Step~4.  Here one
must be careful not to argue cell by cell: by~\eqref{eq:DB-affine} both
$H^{2p-1}(\aaa(c)^o,\cc/\zz)$ and $H^{2p}_\calD(\aaa(c)^o,\zz(p))$ vanish on a
single affine piece once $p\geq2$, so the statement that $\delta$ is an
isomorphism is empty there.  The isomorphism is instead global, and comes from
the assembly of the cells: $Z^o$ is covered by the affine pieces $\aaa(c)$,
$c\in\Cube A$ (Corollary~\ref{cor:F-trivial-affine}), the associated
spectral sequence is concentrated in the row $t=1$ by
Lemma~\ref{lem:DB-affine}, and Lemma~\ref{lem:DB-cube} (equivalently
Proposition~\ref{prop:DB-A-realization}) yields
\[
  \delta: H^{2p-1}(Z^o,\cc/\zz)\;\xrightarrow{\;\sim\;}\;
  H^{2p}_\calD(Z^o,\zz(p)),
\]
compatibly with the cubical decomposition.
This is precisely the isomorphism $H^{2p}_\calD(Z^o,\zz(p))\cong
H^{2p-1}(S,\cc/\zz)$ used in Step~4 (composed with the homotopy equivalence
$Z^o\simeq S$).

Therefore $\CS_p(\nabla) = \delta^{-1}(\ch^{\calD}_p(F^\sigma|_{Z^o}))$, and by
Step~4 this corresponds, under $\delta^{-1}$ followed by $Z^o\simeq S$,
exactly to $r^\sigma_{2p-1}|_S$.  This proves the theorem.
\end{proof}

\subsection{The Rees bundle: explicit local coordinates}
\label{subsec:rees-local-coords}

We now give explicit formulas for the Rees bundle in local coordinates,
which were used in Step~5 above and will be needed again for the
$F^1$-connection construction.

Let $c\in\Cube A$ be a cube with $c^{-1}(t) = \{a_1,\ldots,a_k\}$ and
$c^{-1}(1) = \{a_0\}$ (for simplicity, a single ``$1$'' vertex).  The
corresponding projective cube is $\pp(c) = (\pp^1)^k$ with coordinates
$(y_1,\ldots,y_k)$, and the affine part is $\aaa(c) = (\aaa^1)^k$.

The vector space $G(c)$ is graded:
\[
  G(c) = \bigoplus_{\mathbf{j}=(j_1,\ldots,j_k)\in\zz^k}G(c)_\mathbf{j},
\]
where $G(c)_\mathbf{j} = \bigcap_l\Gr^{W(c;a_l)}_{j_l}(U(a_0))$.  The Rees
bundle is then:
\[
  F(c) = \bigoplus_\mathbf{j}G(c)_\mathbf{j}\otimes
  \Oo_{(\pp^1)^k}\left(-\sum_l j_l Q_l\right),
\]
where $Q_l = \{y_l=1\}\subset\pp(c)$.

\begin{remark}
The twist by $-j_l Q_l$ rather than $-j_l[0]$ or $-j_l[\infty]$ reflects
our choice of Rees construction along $y_l=1$ (i.e.\ we use $(1-y_l)$ as
the uniformizer rather than $y_l$ or $1/y_l$).  This choice is dictated by
the structure of Deligne's canonical extension, where the monodromy acts at
``$y=0$'' and the ``generic fibre'' lives at ``$y=1$''.
\end{remark}

\begin{lemma}
Over $\aaa(c) = \{y_l\neq\infty\;\forall l\}$, the bundle $F(c)$ is
trivializable, with a canonical trivialization given by the sections
$s_\mathbf{j} := \prod_l(1-y_l)^{j_l}$ of $\Oo(-\sum_l j_l Q_l)|_{\aaa(c)}$.
\end{lemma}

\begin{proof}
Over $\aaa^1 = \pp^1\setminus\{\infty\}$, the divisor $Q_l = \{1\}$ is
disjoint from $\{0\}$, so $\Oo(-Q_l)|_{\aaa^1}$ is trivial with generator
$(1-y_l)$.  The product $\prod_l(1-y_l)^{j_l}$ trivializes
$\Oo(-\sum_l j_l Q_l)|_{\aaa(c)}$.
\end{proof}

\begin{corollary}\label{cor:F-trivial-affine}
The bundle $F$ is trivializable on every affine piece $\aaa(c)$.  In
particular, $F|_{\aaa\Cube A}$ is locally trivial with respect to the
affine cover by $\{\aaa(c)\}_{c\in\Cube A}$.
\end{corollary}

\subsection{The canonical connection on the Rees bundle}
\label{subsec:canconn}

\begin{definition}
The \emph{canonical connection} on $F(c)|_{\aaa(c)}$ is defined, in the
trivialization given by $\{s_\mathbf{j}\}$, by the connection form
\[
  A^{\mathrm{can}} := \sum_\mathbf{j}\left(\sum_l j_l\,\dlog(1-y_l)\right)
  \mathrm{Id}_{G(c)_\mathbf{j}}.
\]
\end{definition}

This is the logarithmic connection with poles along $\pp(c)\setminus\aaa(c)$
coming from the trivialization of the line bundles $\Oo(-j_lQ_l)$.

\begin{lemma}\label{lem:can-flat}
The canonical connection $A^{\mathrm{can}}$ is flat on $\aaa(c)$.
\end{lemma}

\begin{proof}
The curvature of $A^{\mathrm{can}}$ is $dA^{\mathrm{can}}+A^{\mathrm{can}}\wedge A^{\mathrm{can}}$.
Since $A^{\mathrm{can}}$ is block-diagonal (acting on each graded piece $G(c)_\mathbf{j}$
as scalar multiplication), $A^{\mathrm{can}}\wedge A^{\mathrm{can}} = 0$.  And
$d(\dlog(1-y_l)) = 0$ away from $y_l=1$ (since $\dlog(1-y_l) =
-\frac{dy_l}{1-y_l}$ is closed on $\aaa^1\setminus\{1\}$).
\end{proof}

\begin{proposition}\label{prop:can-conn-compatible}
The canonical connection $A^{\mathrm{can}}$ is compatible with the filtration
$W(F(c);\mu(c))$ and induces the trivialization $\tau(c)$ on the associated
graded.
\end{proposition}

\begin{proof}
The filtration $W(F(c);\mu(c))$ on $F(c)$ is induced by the weight filtration
$W(c;\mu(c))$ on $G(c)$ (the filtration associated to the maximum element).
In the decomposition $G(c) = \bigoplus_\mathbf{j}G(c)_\mathbf{j}$, the weight
$j_{\mu}$ in the $\mu(c)$-direction gives the $W(c;\mu(c))$-degree.  Since
$A^{\mathrm{can}}$ acts diagonally on $G(c)_\mathbf{j}$, it preserves the
filtration.  On the associated graded pieces, the scalar
$\sum_l j_l\,\dlog(1-y_l)$ restricted to fixed $j_\mu$ (i.e.\ fixed
$W(c;\mu(c))$-degree) gives the trivialization $\tau(c)$.
\end{proof}

\subsection{Local description of the patching compatibility}

We now verify that the patching data on the Rees bundle is indeed stratified.

\begin{proposition}[Detailed version of Theorem~\ref{thm:global-rees}]
For $c'\leq c$ (inclusion of cubes), the isomorphism
$\Phi_{c',c}: F(c)|_{\pp(c')}\xrightarrow{\;\sim\;}F(c')$ is compatible with
the filtrations and trivializations:
\[
  \Phi_{c',c}^*(W(F(c');\mu(c'))) = W(F(c);\mu(c))|_{\pp(c')},
\]
and $\Phi_{c',c}$ maps $\tau(c)|_{\aaa(c')}$ to $\tau(c')$.
\end{proposition}

\begin{proof}
Since $c'\leq c$ we have $\mu(c')\leq\mu(c)$.  The filtration
$W(F(c);\mu(c))$ is associated to the maximum element of $c$, while
$W(F(c');\mu(c'))$ is associated to the maximum element of $c'$.  The
inequality $\mu(c')\leq\mu(c)$ means the filtration for $c$ is finer; its
restriction to $\pp(c')$ contains the filtration for $c'$ as a coarser
sub-filtration.

For the trivialization: over $\aaa(c')$, the Rees construction for $c$
restricted to $\aaa(c')$ (by setting $y_l=1$ for $l$ with $c(a_l)=1$
but $c'(a_l)=t$, and $y_l=0$ for $l$ with $c(a_l)=t$
but $c'(a_l)=0$) gives exactly the Rees construction for $c'$.  Under this
identification, the canonical trivialization $\tau(c)$ (given by sections
$(1-y_l)^{j_l}$) restricts to $\tau(c')$.
\end{proof}


\subsection{Review of secondary characteristic classes}

We give a more systematic treatment of secondary characteristic classes,
following \cite{CS} and \cite{DHZ}.

\subsubsection{Differential characters}

\begin{definition}
A \emph{differential character} of degree $2p-1$ on $Z$ is a group
homomorphism $\hat{h}: Z_{2p-1}^\infty(Z)\to\rr/\zz$ from smooth singular
$(2p-1)$-cycles, such that for every smooth singular $2p$-chain $\sigma$:
\[
  \hat{h}(\partial\sigma) = \int_\sigma\omega \pmod{\zz},
\]
for a fixed smooth closed $2p$-form $\omega = \omega(\hat{h})$ called the
\emph{curvature} of $\hat{h}$.
\end{definition}

The space of differential characters of degree $2p-1$ with curvature $\omega$
is a torsor under $H^{2p-1}(Z,\rr/\zz)$.

\begin{theorem}[Cheeger-Simons {\cite{CS}}]
Every $\mathcal{C}^\infty$ complex vector bundle $E\to Z$ with connection
$\nabla$ defines a differential character $\widehat{\ch}_p(E,\nabla)$ of
degree $2p-1$ with curvature $\ch_p(\nabla)$.  Two connections define the
same character if and only if they are homotopic through connections with
the same curvature.
\end{theorem}

\subsubsection{The integration formula}

For a locally nil-flat connection $\nabla$, $\ch_p(\nabla) = 0$, so the
curvature of the differential character $\widehat{\ch}_p(E,\nabla)$ is zero.
Hence $\widehat{\ch}_p(E,\nabla)\in\Hom(Z_{2p-1}^\infty(Z),\rr/\zz)$ is a
cohomology class.

Concretely: for a smooth $(2p-1)$-cycle $\gamma$, the value
$\widehat{\ch}_p(E,\nabla)(\gamma)\in\rr/\zz$ is computed as follows.
Write $\gamma = \partial\Gamma$ for a $2p$-chain $\Gamma$ in $BGL_r(K)^+$
(using the classifying map of $E$).  Then:
\[
  \widehat{\ch}_p(E,\nabla)(\gamma) = \int_\Gamma r_{2p}^{\calD} \pmod{\zz},
\]
where $r_{2p}^{\calD}$ is the universal Deligne Chern form on $BGL^+(K)$.

\subsubsection{The transgression formula}

\begin{proposition}[Transgression]\label{prop:transgression}
Let $\nabla_0,\nabla_1$ be connections on $E$, with $\nabla_t = \nabla_0 + tA$
a linear path ($A = \nabla_1-\nabla_0$).  Then:
\[
  \widehat{\ch}_p(E,\nabla_1) - \widehat{\ch}_p(E,\nabla_0)
  = [CS_p(\nabla_0,\nabla_1)]\in H^{2p-1}(Z,\mathbb{C}/\zz),
\]
where $CS_p(\nabla_0,\nabla_1) = \dfrac{p}{p!\,(2\pi i)^p}\displaystyle\int_0^1
\Tr(A\wedge F(\nabla_t)^{p-1})\,dt$. 
\end{proposition}

\subsubsection{Vanishing criterion}

\begin{corollary}\label{cor:vanishing-crit}
For a locally nil-flat connection $\nabla$ on a topologically trivial bundle,
$\widehat{\ch}_p(E,\nabla) = 0$ in $H^{2p-1}(Z,\rr/\zz)$ if and only if
$\vol_p(\nabla) = 0$ in $H^{2p-1}(Z,\rr)$ (i.e.\ the imaginary part of the
Chern-Simons class is zero, under the normalized convention of
Definition~\ref{def:vol-p-normalized}).
\end{corollary}

\begin{proof}
For a topologically trivial bundle, $c_p(E) = 0$, so the integer-valued
part of $\widehat{\ch}_p$ is zero.  The Chern-Simons class lives in
$H^{2p-1}(Z,\cc/\zz)$; under the splitting $\cc/\zz = \rr/\zz\oplus i\rr$,
its imaginary component is the volume regulator
$\vol_p(\nabla) = \Im_{\cc/\zz}(\CS_p)\in H^{2p-1}(Z,\rr)$, with no reduction
modulo $\zz$ (Lemma~\ref{lem:CZ-splitting}).  Vanishing of the volume
regulator is thus exactly the vanishing of that component.
Since $c_p(E)=0$ the real component $\CS_p^{\rr}$ is also torsion (it maps
to zero in $H^{2p-1}(Z,\rr)$, so lies in the image of integral classes).
Hence $\CS_p = 0$ in $H^{2p-1}(Z,\cc/\zz)$.
\end{proof}

\subsection{The $\theta$-connection and the structure of the Chern-Simons form}

For the Deligne canonical extension, there is a natural decomposition of
the Chern-Simons form that illuminates its structure.

\begin{proposition}
Write $\nabla = \theta + \bar\theta + A_g$ where $\theta$ is the holomorphic
part (type $(1,0)$), $\bar\theta = \overline\partial_E$, and $A_g$ is a
$\mathfrak{u}(g)$-valued $1$-form.  Then:
\[
  \CS_p(\nabla^{\Del}) = p\int_0^1\Tr\bigl(A_g\wedge(F_g + tF_\theta + \cdots)^{p-1}\bigr)\,dt
  \pmod{\text{exact}},
\]
where $F_g = \nabla_g^2$ and $F_\theta = \theta\wedge\theta + \theta\wedge\bar\theta$.
\end{proposition}

This decomposition is useful because, for a VHS, the $A_g$ part is constrained
by the Hermitian geometry of the polarization, and the vanishing of the volume
regulator reduces to a trace identity in the Lie algebra $\mathfrak{u}(p,q)$.

\subsection{Explicit computation of the $F^1$-connection in the two-cube case}

We work out the $F^1$-connection explicitly for the case of a two-dimensional
cube $c$ with $c^{-1}(t) = \{a_1, a_2\}$ and $c^{-1}(1) = \{a_0\}$.

The Rees bundle $F(c)$ over $(\pp^1)^2$ with coordinates $(y_1,y_2)$ is
\[
  F(c) = \bigoplus_{(j_1,j_2)\in\zz^2} G(c)_{j_1,j_2}\otimes
  \Oo(-j_1Q_1-j_2Q_2)
\]
where $Q_l = \{y_l=1\}$.

The diagonal part of the $F^1$-connection is
\[
  \nabla^{\mathrm{diag}} = d +
  \begin{pmatrix} j_1\,\dlog(1-y_1)+j_2\,\dlog(1-y_2) & 0 & \cdots \\
  0 & j_1'\,\dlog(1-y_1)+j_2'\,\dlog(1-y_2) & \cdots \\
  \vdots & & \ddots
  \end{pmatrix}
\]
(acting on the block $G(c)_{j_1,j_2}$ by the scalar $j_1\dlog(1-y_1)+j_2\dlog(1-y_2)$).

For the off-diagonal blocks: we need to choose connection forms
$A_{\mathbf{j},\mathbf{j}'}$ (for $\mathbf{j}\neq\mathbf{j}'$) in
$F^1A^1((\pp^1)^2\log D_\infty, \Hom(G(c)_{\mathbf{j}},G(c)_{\mathbf{j}'}))$
where $D_\infty = \{y_1=\infty\}\cup\{y_2=\infty\}$.

\begin{lemma}
Any choice of $A_{\mathbf{j},\mathbf{j}'}$ making the connection matrix
upper-triangular with respect to the filtration $W(F(c);\mu(c))$ gives a
valid $F^1$-connection.  The resulting Chern-Simons class is independent of
this choice.
\end{lemma}

\begin{proof}
The $F^1$ condition is preserved by adding upper-triangular correction
terms (since $F^1$ is a linear condition on the connection matrix).  The
independence of the Chern-Simons class from the choice of $A_{\mathbf{j},\mathbf{j}'}$
follows from Theorem~\ref{thm:DHZ}: any two $F^1$-connections on the same
bundle (over a compactification) give the same Deligne cohomology class.
\end{proof}

\subsection{The de Rham complex of the cubical realization and the
Chern-Simons form}

Let us describe how the $F^1$-connection gives a class in the de Rham
cohomology of the cubical realization $S = \iiii\Cube A$.

\begin{definition}
A \emph{connection on a bundle over $\iiii\Cube A$} is a collection of
connections $\nabla(c)$ on $E|_{\iiii(c)}$ for each cube $c\in\Cube A$,
compatible under face restrictions: for $c'\leq c$, the restriction
$\nabla(c)|_{\iiii(c')} = \nabla(c')$ under the isomorphism
$E|_{\iiii(c')} \cong E(c')$.
\end{definition}

The $F^1$-connections $\nabla(c)$ on $F(c)$ over $\pp(c)$ restrict to
$\iiii(c)\subset\aaa(c)\subset\pp(c)$, giving a connection on
$F^\sigma|_{\iiii\Cube A}$.

\begin{proposition}
The Chern-Simons form of the restricted connection is a well-defined closed
$(2p-1)$-form on $\iiii\Cube A$, representing a class in
$H^{2p-1}(\iiii\Cube A, \cc/\zz)$.
\end{proposition}

\begin{proof}
Closedness: each $\ch_p(\nabla(c)) = 0$ since $\nabla(c)$ is locally
nil-flat (the Rees bundle on $\pp(c)$ has a filtration preserved by
$\nabla(c)$ with flat associated-graded).  Hence the Chern-Simons form is
closed.

Compatibility: the Chern-Simons forms on adjacent cubes agree on their
common face, since the connections are compatible.  Hence they assemble to
a global form on $\iiii\Cube A$.
\end{proof}

\subsection{The Green current and the Burgos-Gil formalism in detail}
\label{subsec:green-BG-detail}

Following \cite{BG} (see also Appendix~\ref{app:burgosgil}), we now describe the Green current associated to the
Deligne canonical extension $F$ and show it computes the Deligne-Beilinson
Chern character.

\begin{definition}
For a vector bundle $E$ on a smooth projective variety $\bar Z$, a
\emph{Green current of degree $2p-2$} for the $p$-th Chern character is
a current $g_p(E)\in\mathcal{D}^{p-1,p-1}(\bar Z)$ such that
\[
  dd^c[g_p(E)] + \delta_{Z(\sigma)} = [\omega_p(E)],
\]
where $\sigma$ is a meromorphic section of $E^{\otimes p}$, $Z(\sigma)$ is
its zero locus (counted with multiplicity), and $\omega_p(E)$ is a smooth
closed form representing $\ch_p(E)$ in de Rham cohomology.
\end{definition}

For the Deligne canonical extension $(F,\nabla^{\Del})$:

\begin{proposition}[Green current for the canonical extension]
\label{prop:green-current}
The $F^1$-connection $\nabla$ on $F|_{Z^o}$ defines a Green current
$g_p(F,\nabla)$ on $\bar Z = \pp\Cube A$ by the formula
\[
  g_p(F,\nabla) = p\int_0^1(1-t)^{p-1}\Tr(A\wedge F(\nabla_t)^{p-1})\,dt
\]
where $\nabla_t = \nabla_{\mathrm{Grass}} + t(\nabla-\nabla_{\mathrm{Grass}})$
and $\nabla_{\mathrm{Grass}}$ is the connection pulled back from the Grassmannian.
\end{proposition}

\begin{proof}
This is the Bott-Chern formula for the difference between two connections.
The key point is that $\nabla_{\mathrm{Grass}}$ is Hermitian and its
Chern character form $\ch_p(F,\nabla_{\mathrm{Grass}})$ is the canonical
representative of $\ch_p^{\calD}(F)$.  The Green current $g_p(F,\nabla)$
satisfies $d(g_p(F,\nabla)) = \ch_p(\nabla) - \ch_p(\nabla_{\mathrm{Grass}})$,
so after adding the contribution from the Grassmannian we get the Deligne
Chern character.
\end{proof}

\begin{corollary}
The class of $g_p(F,\nabla)$ in $H^{2p-1}(\bar Z, \cc/\zz)$ (via the
isomorphism $\frac{H^{2p-1}(\bar Z,\cc)}{ H^{2p-1}(\bar Z,\zz)} \cong
H^{2p-1}(\bar Z,\cc/\zz)$) is independent of the choice of $\nabla$ among
$F^1$-connections, and equals the Chern-Simons class $\CS_p(F,\nabla^{\Del})$
restricted to $\bar Z$.
\end{corollary}

\section{The volume regulator}
\label{sec:volume}

\subsection{Definition and basic properties}

Let $X$ be a compact $\mathcal{C}^\infty$ manifold and $(E,\nabla)$ a complex
vector bundle with connection.

\begin{definition}\label{def:CS-integrand}
For a smooth path of connections $\{\nabla_t\}_{0\leq t\leq 1}$ with derivative
$\dot\nabla_t := \frac{d}{dt}\nabla_t\in A^1(X;\End(E))$, define the
\emph{Chern-Simons transgression form}
\[
  v_p(\{\nabla_t\}) := \int_0^1
  \Tr\!\bigl(\dot\nabla_t\wedge F(\nabla_t)^{p-1}\bigr)\,dt
  \;\in\; A^{2p-1}(X),
\]
where $F(\nabla_t) = \nabla_t^2\in A^2(X;\End(E))$ is the curvature.
\end{definition}

\begin{lemma}\label{lem:bott-chern}
One has $d\bigl(p\,v_p(\{\nabla_t\})\bigr) = \ch_p(\nabla_1) - \ch_p(\nabla_0)$,
where $\ch_p(\nabla) := \Tr(F(\nabla)^p)$ is the $p$-th Chern character form.
\end{lemma}

\begin{proof}
Differentiating and using the Bianchi identity ($d_{\nabla_t}F(\nabla_t)=0$)
gives $\frac{d}{dt}\Tr(F(\nabla_t)^p) = p\,\Tr\bigl(\dot F(\nabla_t)\wedge
F(\nabla_t)^{p-1}\bigr) = p\,\Tr\bigl(d_{\nabla_t}(\dot\nabla_t)\wedge
F(\nabla_t)^{p-1}\bigr) = p\,d\,\Tr\bigl(\dot\nabla_t\wedge
F(\nabla_t)^{p-1}\bigr)$, using $d_{\nabla_t}F(\nabla_t)=0$ once more to
drop the correction term. Integrating over $t\in[0,1]$ gives
$\ch_p(\nabla_1)-\ch_p(\nabla_0) = d\left(p\int_0^1
\Tr(\dot\nabla_t\wedge F(\nabla_t)^{p-1})\,dt\right) = d(p\,v_p(\{\nabla_t\}))$,
by Definition~\ref{def:CS-integrand}.
\end{proof}

\begin{lemma}[Koszul sign / parity]\label{lem:koszul-parity}
Let $g$ be a (possibly indefinite) Hermitian form on $E$ and let
$\mathfrak{u}_E$ denote the Lie algebra of $g$-skew-Hermitian endomorphisms.
\begin{enumerate}[leftmargin=2em,label=\rm(\alph*)]
  \item For $F\in A^2(X;\mathfrak{u}_E)$, $\overline{\Tr(F^p)} =
    (-1)^p\Tr(F^p)$.
  \item For $A\in A^1(X;\mathfrak{u}_E)$ and $B\in A^2(X;\mathfrak{u}_E)$,
    $\overline{\Tr(A\wedge B^{p-1})} = (-1)^p\Tr(A\wedge B^{p-1})$.
\end{enumerate}
\end{lemma}

\begin{proof}
Write $g(x,y) = x^*Jy$ in a local $g$-orthonormal frame, with $J$ the
constant diagonal matrix of signs ($J=\Id$ if $g$ is positive definite); then
$M\in\mathfrak{u}_E$ satisfies $\bar M = -JM^TJ$.

(a) $\overline{F^p} = \bar F^{\,p} = (-JF^TJ)^p = (-1)^p J(F^T)^pJ =
(-1)^pJ(F^p)^TJ$, using $J^2=\Id$ and $(F\wedge F)^T=F^T\wedge F^T$ (valid
since $F$ has even form-degree, so wedging commutes freely with itself).
Taking traces and using $\Tr(JMJ)=\Tr(M)$, $\Tr(M^T)=\Tr(M)$ gives
$\overline{\Tr(F^p)} = (-1)^p\Tr(F^p)$.

(b) Since $B$ has even form-degree, $(B^{p-1})^T = (B^T)^{p-1}$ as in (a),
and since $A$ has odd form-degree while $B^{p-1}$ has even total
form-degree, $A^T\wedge(B^{p-1})^T = (A\wedge B^{p-1})^T$ (transposing a
wedge product of matrix-valued forms of degrees $a,c$ satisfies
$(X\wedge Y)^T=(-1)^{ac}Y^T\wedge X^T$, and here $ac=1\cdot 2(p-1)$ is
even). Hence
\begin{align*}
  \overline{A\wedge B^{p-1}} &= \bar A\wedge\bar B^{\,p-1}
  = (-JA^TJ)\wedge(-1)^{p-1}J(B^T)^{p-1}J\\
  &= (-1)^p J\bigl(A^T\wedge (B^{p-1})^T\bigr)J
  = (-1)^p J(A\wedge B^{p-1})^TJ,
\end{align*}
and taking traces as in (a) gives $\overline{\Tr(A\wedge B^{p-1})} =
(-1)^p\Tr(A\wedge B^{p-1})$.
\end{proof}

\subsection{Normalization of the Chern-Weil quantities}
\label{subsec:normalization-quantities}

\begin{definition}\label{def:normalized-curvature}
For a connection $\nabla$ on $E$ with curvature $F(\nabla)\in
A^2(X;\End(E))$, and for a smooth path $\{\nabla_t\}$ with derivative
$\dot\nabla_t\in A^1(X;\End(E))$ as in Definition~\ref{def:CS-integrand},
define the \emph{normalized curvature} and \emph{normalized variation}
\[
  F^{\mathrm{nm}}(\nabla) := \frac{1}{2\pi i}\,F(\nabla), \qquad
  \dot\nabla_t^{\mathrm{nm}} := \frac{1}{2\pi i}\,\dot\nabla_t .
\]
This is the usual Chern-Weil normalization of the curvature that produces
integral Chern classes; since $2\pi i$ is a nonzero scalar,
$F^{\mathrm{nm}}(\nabla)$ is skew-Hermitian (resp.\ nilpotent) exactly when
$F(\nabla)$ is, so all of the structural hypotheses used below transfer
unchanged between $F$ and $F^{\mathrm{nm}}$.
\end{definition}

\begin{definition}[Normalized Chern character and transgression form]\label{def:normalized-ch-vp}
The \emph{normalized Chern character form} and \emph{normalized
Chern-Simons transgression form} are
\[
  \ch_p^{\mathrm{nm}}(\nabla) := \frac{1}{p!}\,\Tr\bigl(F^{\mathrm{nm}}(\nabla)^p\bigr),
  \qquad
  v_p^{\mathrm{nm}}(\{\nabla_t\}) := \frac{1}{(p-1)!}\int_0^1
  \Tr\bigl(\dot\nabla_t^{\mathrm{nm}}\wedge F^{\mathrm{nm}}(\nabla_t)^{p-1}\bigr)\,dt .
\]
\end{definition}

\begin{lemma}\label{lem:normalized-formula}
In terms of the un-normalized quantities $\ch_p(\nabla)$ of
Lemma~\ref{lem:bott-chern} and $v_p(\{\nabla_t\})$ of
Definition~\ref{def:CS-integrand},
\[
  \ch_p^{\mathrm{nm}}(\nabla) = \frac{1}{p!\,(2\pi i)^p}\,\ch_p(\nabla), \qquad
  v_p^{\mathrm{nm}}(\{\nabla_t\}) = \frac{p}{p!\,(2\pi i)^p}\,v_p(\{\nabla_t\}),
\]
matching, up to notation, the normalized $\ch_p$ of
Appendix~\ref{app:diffchar}, \S A.2, and the form $\mathrm{CS}_{2p-1}(A)$ of
\eqref{eq:CS-form-normalised} there.
\end{lemma}

\begin{proof}
Each factor of $F^{\mathrm{nm}}(\nabla)$ contributes one power of
$(2\pi i)^{-1}$ relative to the corresponding factor of $F(\nabla)$, and
likewise for $\dot\nabla_t^{\mathrm{nm}}$ relative to $\dot\nabla_t$. Since
$\ch_p^{\mathrm{nm}}(\nabla)$ involves $p$ factors of $F^{\mathrm{nm}}(\nabla)$,
it picks up $(2\pi i)^{-p}$ in total:
\[
  \ch_p^{\mathrm{nm}}(\nabla) = \frac{1}{p!}\Tr\Bigl(\bigl(\tfrac{1}{2\pi i}F(\nabla)\bigr)^{\!p}\Bigr)
  = \frac{1}{p!\,(2\pi i)^p}\Tr\bigl(F(\nabla)^p\bigr)
  = \frac{1}{p!\,(2\pi i)^p}\,\ch_p(\nabla).
\]
Likewise $v_p^{\mathrm{nm}}(\{\nabla_t\})$ involves one factor of
$\dot\nabla_t^{\mathrm{nm}}$ together with $p-1$ factors of
$F^{\mathrm{nm}}(\nabla_t)$ --- again $p$ normalized factors in total --- so
\begin{align*}
  v_p^{\mathrm{nm}}(\{\nabla_t\}) &= \frac{1}{(p-1)!}\int_0^1
  \Tr\Bigl(\tfrac{1}{2\pi i}\dot\nabla_t\wedge\bigl(\tfrac{1}{2\pi i}F(\nabla_t)\bigr)^{\!p-1}\Bigr)\,dt\\
  &= \frac{1}{(p-1)!\,(2\pi i)^p}\int_0^1
  \Tr\bigl(\dot\nabla_t\wedge F(\nabla_t)^{p-1}\bigr)\,dt
  = \frac{1}{(p-1)!\,(2\pi i)^p}\,v_p(\{\nabla_t\}),
\end{align*}
and $\frac{1}{(p-1)!}=\frac{p}{p!}$ gives the stated formula.
\end{proof}

\begin{proposition}\label{prop:CS-closed-nm}
If $\nabla_0$ is unitary and $\nabla_1$ is locally nil-flat, then
$\Im\,v_p^{\mathrm{nm}}(\{\nabla_t\})$ is closed, for every $p\geq1$.
\end{proposition}

\begin{proof}
Since $\nabla_0$ is unitary, $F(\nabla_0)$ is skew-Hermitian, i.e.\
$F(\nabla_0)\in A^2(X;\mathfrak{u}(r))\subset A^2(X;\mathfrak{u}_E)$ (taking
$g$ positive definite). Since $\overline{(2\pi i)^p}=(-2\pi i)^p=(-1)^p(2\pi
i)^p$, Lemma~\ref{lem:koszul-parity}(a) gives
\[
  \overline{\ch_p^{\mathrm{nm}}(\nabla_0)}
  = \frac{\overline{\Tr(F(\nabla_0)^p)}}{\overline{p!(2\pi i)^p}}
  = \frac{(-1)^p\Tr(F(\nabla_0)^p)}{(-1)^p\,p!(2\pi i)^p}
  = \ch_p^{\mathrm{nm}}(\nabla_0),
\]
so $\ch_p^{\mathrm{nm}}(\nabla_0)$ is real, i.e.\ $\Im\,\ch_p^{\mathrm{nm}}(\nabla_0)
= 0$, for every $p\geq1$. Since $\nabla_1$
is locally nil-flat, $F(\nabla_1)$ is locally nilpotent, so
$\ch_p(\nabla_1)=\Tr(F(\nabla_1)^p)=0$ identically, hence also
$\ch_p^{\mathrm{nm}}(\nabla_1)=0$. Dividing Lemma~\ref{lem:bott-chern}'s
identity $d(p\,v_p) = \ch_p(\nabla_1)-\ch_p(\nabla_0)$ through by
$p!(2\pi i)^p$ gives, directly and with no further constant to track,
$d\,v_p^{\mathrm{nm}} = \ch_p^{\mathrm{nm}}(\nabla_1) -
\ch_p^{\mathrm{nm}}(\nabla_0) = -\ch_p^{\mathrm{nm}}(\nabla_0)$, so
$d(\Im\,v_p^{\mathrm{nm}}) = -\Im\,\ch_p^{\mathrm{nm}}(\nabla_0) = 0$.
\end{proof}

\begin{definition}\label{def:vol-p-normalized}
The \emph{volume regulator} of a locally nil-flat connection $\nabla$ is
\[
  \vol_p(\nabla) := [\Im\,v_p^{\mathrm{nm}}(\{\nabla_t\})]
  \;\in\; H^{2p-1}(X,\rr),
\]
where $\{\nabla_t\}$ is any path from a unitary connection $\nabla_0$ to
$\nabla_1 = \nabla$; this is well defined by
Proposition~\ref{prop:CS-closed-nm}, and is independent of the choice of
$\nabla_0$ and of the path by Proposition~\ref{prop:vol-indep} below.
\end{definition}

\begin{proposition}\label{prop:vol-indep}
The volume regulator $\vol_p(\nabla)$ is independent of the choice of unitary
connection $\nabla_0$ and path $\{\nabla_t\}$.
\end{proposition}

\begin{proof}
Suppose $\{\nabla_t\}$ and $\{\nabla'_t\}$ are two paths from $\nabla_0, \nabla'_0$
(both unitary) to $\nabla$.  Connect $\nabla_0$ to $\nabla'_0$ by a path
$\{\nabla''_s\}$ of unitary connections (possible since the space of unitary
connections on a given bundle is convex).  The concatenated path
$\nabla''_\hdot\to\nabla'_\hdot\to\nabla^{-1}_\hdot$ is a loop and the
corresponding Chern-Simons form integrates to zero over a cycle (by Chern-Weil
theory for unitary connections).  Hence the two choices define the same class.
\end{proof}

\subsection{Comparison with the Cheeger--Simons formalism and the classical
$TP$-forms}
\label{subsec:CS-comparison}

We record how $v_p$ and $\vol_p$ above relate to Cheeger and Simons'
original construction of secondary characteristic classes \cite{CS}: their
differential characters $S_{P,u}$, their Chern class and Chern character
characters $\hat c_k$, $\hat{ch}$, and their forms $TP(\theta)$
\cite{ChrS}.

\subsubsection*{The abstract secondary class $S_{P,u}$}

For a pair $(P,u)\in K^{2k}(G,\Lambda)$ (an invariant polynomial $P$ of
degree $k$ and an integral class $u$ with $w(P)=r(u)$), \cite[Thm.~2.2]{CS}
produce, for each $G$-bundle-with-connection $\alpha=\{E,M,\theta\}$, a
unique differential character
\[
  S_{P,u}(\alpha)\in\widehat H^{2k-1}(M,\rr/\Lambda),
  \qquad
  \delta_1(S_{P,u}(\alpha)) = P(\Omega),\qquad
  \delta_2(S_{P,u}(\alpha)) = u(\alpha),
\]
simultaneously refining the Chern-Weil form $P(\Omega)$ and the integral
characteristic class $u(\alpha)$. Specializing to the Chern polynomial and
integral Chern class $(C_k,c_k)$ gives the Chern class character
$\hat c_k(V):=S_{C_k,c_k}(E(V))\in\widehat H^{2k-1}(M,\rr/\zz)$
(\cite[(4.4)--(4.5)]{CS}); specializing instead to the Chern-character
polynomial $P_{ch}$ and topological Chern character $ch$ gives
$\hat{ch}(V):=S_{P_{ch},ch}(E(V))\in\widehat H^{\mathrm{odd}}(M,\rr/\qq)$
(\cite[p.~64]{CS}), related to the $\hat c_k$ by the exponential-type
formula $\hat{ch}=1+\hat c_1+\tfrac12\hat c_1{*}\hat c_1+\cdots$
(\cite[(4.10)]{CS}). Since $C_k$ is the $k$-th elementary symmetric
polynomial in the (normalized) eigenvalues of the curvature, while our
$\ch_p(\nabla)=\Tr(F(\nabla)^p)$ (Lemma~\ref{lem:bott-chern}) is a power
sum, it is $\hat{ch}$, not $\hat c_k$, that is the literal analogue of the
present objects; the two are related only through the Newton identities.

\subsubsection*{The transgression formula and $v_p$}

Cheeger-Simons' key computational tool, \cite[Prop.~2.9]{CS}, states that
for a linear path $\theta_t=\theta_0+tA$ ($A=\theta_1-\theta_0$) of
connections,
\[
  S_{P,u}(\alpha_1) - S_{P,u}(\alpha_0)
  = \Bigl[\,k\int_0^1 P(\theta_t'\wedge\Omega_t^{k-1})\,dt\,\Bigr]\pmod\Lambda.
\]
With $P=\ch_p$, $k=p$, this is, up to the overall constant $p$ and up to
working with the honest real number rather than its class mod $\Lambda$,
exactly Definition~\ref{def:CS-integrand}. Lemma~\ref{lem:bott-chern} is
the differentiated form of the same computation that proves
\cite[Prop.~2.9]{CS}.

\subsubsection*{The $TP$-forms}

Chern and Simons' original construction \cite{ChrS} (recalled in
\cite[\S 2, eq.~(2.7)]{CS}) associates to a \emph{single} connection
$\theta$ on the total space of a principal $G$-bundle $\pi:E\to M$ the form
\[
  TP(\theta) := k\int_0^1 P\bigl(\theta\wedge\phi_t^{k-1}\bigr)\,dt,
  \qquad
  \phi_t := t\,\Omega + \tfrac12(t^2-t)[\theta,\theta],
\]
satisfying $d\,TP(\theta)=P(\Omega)$ on $E$ (pulled back from $M$), and
\cite[Prop.~2.8]{CS} shows that $TP(\theta)$, reduced mod $\Lambda$,
represents the pullback $\pi^*(S_{P,u}(\alpha))$.

$TP(\theta)$ is a special case of the two-connection formula above: $\phi_t$
is exactly the curvature of the linearly rescaled connection
$\theta_t:=t\theta$, since
\[
  \Omega_t = d\theta_t + \tfrac12[\theta_t,\theta_t]
  = t\,d\theta + \tfrac{t^2}2[\theta,\theta]
  = t\bigl(d\theta+\tfrac12[\theta,\theta]\bigr) + \tfrac12(t^2-t)[\theta,\theta]
  = t\,\Omega + \tfrac12(t^2-t)[\theta,\theta] = \phi_t,
\]
and $\theta_t' = \theta$ is constant in $t$. Thus
$TP(\theta) = k\int_0^1 P(\theta_t'\wedge\Omega_t^{k-1})\,dt$ is literally
\cite[Prop.~2.9]{CS}'s formula applied to the straight-line path from the
connection ``$0$'' to $\theta$ --- a comparison that makes sense on the
total space $E$ (where the tautological connection can be linearly
rescaled to $0$) but not, in general, directly on the base $M$, since a
nontrivial $G$-bundle over $M$ carries no connection ``$0$'' to compare
against. 

Our $v_p(\{\nabla_t\})$ is the direct descendant of this idea, specialized
to $GL_r(\cc)$-bundles: rather than transgressing $\pi^*\ch_p(\nabla)$ to
zero on the total space via a linear rescaling to the connection $0$, we
transgress $\ch_p(\nabla_1)-\ch_p(\nabla_0)$ directly on $X$ between two
honestly defined connections $\nabla_0$ (unitary) and $\nabla_1$ (locally
nil-flat), exploiting only that a unitary reference connection always
exists globally on a Hermitian vector bundle, whereas a flat or ``$0$''
reference connection generally does not. In the special case where
$\nabla_0$ and $\nabla_1$ are related by such a rescaling --- e.g.\ along
the straight-line path within a single fibre of the symmetric-space bundle
of Hermitian metrics on a flat $GL_r(\cc)$-bundle --- the two constructions
literally coincide, with $TP$ recovered upstairs on the frame bundle and
$v_p$ its descent to $X$ (or $X^*$).

\subsubsection*{Target groups: differential characters versus $\vol_p$}

$S_{P,u}$, $\hat c_k$ and $\hat{ch}$ are valued in the group
$\widehat H^{2k-1}(M,\rr/\Lambda)$ of differential characters, which by
\cite[Thm.~1.1]{CS} sits in an exact sequence
\[
  0\to H^{2k-1}(M,\rr)/r\bigl(H^{2k-1}(M,\Lambda)\bigr) \to
  \widehat H^{2k-1}(M,\rr/\Lambda) \to R^{2k-1}(M,\Lambda)\to 0,
\]
carrying both the real transgression data and the integral (Bockstein)
data at once. Our $\vol_p(\nabla)\in H^{2p-1}(X,\rr)$ is the projection
onto the first,  real, torsion-free term of this sequence, after
discarding the mod-$\Lambda$ term.

\subsection{Decomposition into Hermitian and anti-Hermitian parts}

Let $g$ be an indefinite Hermitian metric of type $(p,q)$ on $E$.

\begin{lemma}\label{lem:h1h2}
There exists a direct sum decomposition $E = E_1\oplus E_2$ and positive
definite Hermitian metrics $h_i$ on $E_i$ such that $g = h_1 - h_2$.
\end{lemma}

\begin{proof}
Over any point $x\in X$, the set of such decompositions is homeomorphic to
$\frac{U(p,q)}{(U(p)\times U(q))}$, which is the Riemannian symmetric space of $U(p,q)$
and is contractible.  Hence the bundle of classifying spaces over $X$ has a
section, giving the desired global decomposition.
\end{proof}

\begin{definition}
The metric $g$ induces a splitting
\[
  \End(E) = \mathfrak{u}_E\oplus\mathfrak{p}_E,
\]
where $\mathfrak{u}_E = \{A : g(Ax,y)+g(x,Ay)=0\}$ (skew-$g$-Hermitian) and
$\mathfrak{p}_E = \{B : g(Bx,y)=g(x,By)\}$ ($g$-Hermitian).
\end{definition}

\begin{lemma}\label{lem:decomp-connection}
Every connection $\nabla$ on $E$ can be written uniquely as
$\nabla = \nabla_g + B$, where $\nabla_g$ preserves $g$ and
$B\in A^1(X;\mathfrak{p}_E)$.
\end{lemma}

\begin{proof}
Choose a $g$-preserving connection $\nabla_0$ (exists by Lemma~\ref{lem:h1h2}:
take $\nabla_0 = \nabla_1\oplus\nabla_2$ where $\nabla_i$ are $h_i$-unitary
connections).  Write $\nabla = \nabla_0 + A$ and decompose
$A = U + B$ using the splitting $\mathfrak{u}_E\oplus\mathfrak{p}_E$.
Then $\nabla_g = \nabla_0 + U$ preserves $g$.
\end{proof}

\begin{proposition}[Groups of Hodge type vanishing]\label{prop:hodge-vanish}
Let $g$ be an indefinite Hermitian form on $E$ of signature $(a,b)$ and
$\nabla$ a locally nil-flat connection preserving $g$. Then $\vol_p(\nabla)
= 0$ for every $p\geq 1$.
\end{proposition}

\begin{proof}
 By Lemma~\ref{lem:h1h2}, write $g=h_1-h_2$ on
$E=E_1\oplus E_2$; the metric $h:=h_1\oplus h_2$ is then positive definite
and extends smoothly across all of $X$, so we may take $\nabla_0$ to be the
($g$-preserving) $h$-unitary connection on all of $X$. By
Proposition~\ref{prop:vol-indep} and Lemma~\ref{lem:decomp-connection},
extract the $g$-preserving part $\{\nabla_{g,t}\}$ of a path from $\nabla_0$
to $\nabla_1=\nabla$, defined smoothly on all of $X$, intersections of
several components of $D$ included. Along it, $\dot\nabla_{g,t}$ and
$F(\nabla_{g,t})$ are $\mathfrak{u}_E$-valued: since
$\overline{(2\pi i)^p}=(-1)^p(2\pi i)^p$, Lemma~\ref{lem:koszul-parity}(b)
gives
\[
  \overline{v_p^{\mathrm{nm}}(\{\nabla_{g,t}\})}
  = \frac{\overline{p\,v_p(\{\nabla_{g,t}\})}}{\overline{p!(2\pi i)^p}}
  = \frac{(-1)^p\,p\,v_p(\{\nabla_{g,t}\})}{(-1)^p\,p!(2\pi i)^p}
  = v_p^{\mathrm{nm}}(\{\nabla_{g,t}\}),
\]
so $v_p^{\mathrm{nm}}(\{\nabla_{g,t}\})$ is real for every $p\geq1$; hence
$\vol_p(\nabla) = [\Im\,v_p^{\mathrm{nm}}(\{\nabla_{g,t}\})] = 0$ for every
$p\geq1$.
\end{proof}

\subsection{Vanishing for Deligne's patched connection}

\begin{theorem}\label{thm:vol-zero}
Let $\zeta:\pi_1(X^*,x_0)\to GL_r(\cc)$ underlie a complex variation of Hodge
structure with unipotent monodromy around each $D_i$.  Then
\[
  \vol_p(\nabla^{\Del}) = 0 \quad\text{for all }p\geq 2.
\]
\end{theorem}

\begin{proof}
We show that $\nabla^{\Del}$ is compatible with an indefinite Hermitian
pre-patching collection with refined neighbours, and then apply
Corollary~\ref{cor:maincor}.

A complex VHS with unipotent monodromy gives a flat vector bundle $\mathbb{V}$
on $X^*$ with a flat indefinite Hermitian form $g$ (the polarization of the VHS
viewed as a flat Hermitian structure).  The monodromy logarithms $N_i$ are
$g$-self-adjoint (since the polarization is preserved by the monodromy) and
commute (since the $D_i$ form a normal crossings divisor and the monodromy
around separate components commutes).

By Proposition~\ref{prop:nisofilcontr} (proved in
Section~\ref{sec:nisotropic}), the poset of $N_\hdot$-isotropic filtrations on
a fibre of $\mathbb{V}$ is contractible.  By the \v{C}ech section theorem
(Theorem~\ref{thm:cech-section}), there exists a \v{C}ech section of the
associated poset presheaf over the adapted covering of
Proposition~\ref{prop:covering}.

This section provides, for each open set $U_I$, a filtration
$W(I)$ of $\mathbb{V}|_{U_I^*}$ that is $N_\hdot$-isotropic with respect to $g$.
The filtrations satisfy the refined-neighbours condition: on $U_{IJ} = U_I\cap U_J$
one is a refinement of the other.

These filtrations extend to filtrations of the Deligne canonical extension $F$,
and the associated-graded inherits a flat connection from the original flat
connection on $\mathbb{V}$.  The hermitian structure on $\mathbb{V}$ induces
hermitian structures on the associated-graded bundles compatible with $g$.

By Theorem~\ref{thm:main-hermitian} (proved in
Section~\ref{sec:hermitian-patching}), there exists an indefinite Hermitian
form $\tilde{g}$ on $F$ compatible with this hermitian pre-patching collection.
There exists a connection $\tilde\nabla$ compatible with the patching collection
and preserving $\tilde{g}$.

Since $\nabla^{\Del}$ is also compatible with the same patching collection
(Lemma~\ref{lem:Del-compatible} in Section~\ref{app:deligne-patch}), and by
Proposition~\ref{prop:vol-indep} the volume regulator depends only on the
patching collection, we have $\vol_p(\nabla^{\Del}) = \vol_p(\tilde\nabla)$.
Since $\tilde\nabla$ preserves the indefinite Hermitian form $\tilde{g}$,
Proposition~\ref{prop:hodge-vanish} gives $\vol_p(\tilde\nabla) = 0$ for every
$p\geq 1$, hence $\vol_p(\nabla^{\Del})=0$ for every $p\geq 2$.
\end{proof}

\section{\texorpdfstring{$N$}{N}-isotropic subspaces and filtrations}
\label{sec:nisotropic}

\subsection{Definition and setup}

Let $V$ be a finite-dimensional complex vector space with a nondegenerate
Hermitian form $h(u,v)$ (linear in $u$, conjugate-linear in $v$).  Let
$N_\hdot = (N_1,\ldots,N_k)$ be a $k$-tuple of pairwise commuting, nilpotent,
$h$-self-adjoint endomorphisms: $h(N_iu,v) = h(u,N_iv)$ for all $u,v$.

\begin{definition}
For a subspace $H\subset V$, write $H^\perp := \{v : h(u,v)=0\;\forall u\in H\}$.
A subspace $H\subset V$ is \emph{$N_\hdot$-isotropic} if:
\begin{enumerate}[label={\rm(\arabic*)}]
  \item $H\subset H^\perp$ (isotropic);
  \item $N_i(H^\perp)\subset H$ for all $i=1,\ldots,k$.
\end{enumerate}
\end{definition}

Let $\mathbf{itro}^{N_\hdot}(V,h)$ denote the poset of $N_\hdot$-isotropic
subspaces, ordered by inclusion.

\subsection{Verification that the retraction is well-defined}

We verify in detail that the functor $b(H) = H\cap v^\perp$ in the proof
of Theorem~\ref{thm:itro} maps $N_\hdot$-isotropic subspaces to
$N_\hdot$-isotropic subspaces.

\begin{lemma}\label{lem:b-well-defined}
If $H\in\mathbf{itro}^{N_\hdot}(V,h)$, then $b(H) = H\cap v^\perp$ is also
$N_\hdot$-isotropic.
\end{lemma}

\begin{proof}
We must check:
\begin{enumerate}[leftmargin=2em,label=\rm(\roman*)]
  \item $b(H)\subset b(H)^\perp$;
  \item $N_i(b(H)^\perp)\subset b(H)$ for all $i$.
\end{enumerate}

For (i): we compute $b(H)^\perp = (H\cap v^\perp)^\perp = H^\perp + (v^\perp)^\perp = H^\perp + \cc\cdot v$.
Since $H$ is isotropic, $H\subset H^\perp$, so
$b(H) = H\cap v^\perp\subset H^\perp\subset H^\perp + \cc\cdot v = b(H)^\perp$.

For (ii): by the claim in the proof of Theorem~\ref{thm:itro}, $N_i(V)\subset v^\perp$
(since $h(N_iu,v) = h(u,N_iv) = h(u,0) = 0$ for all $u\in V$).
Therefore $N_i(H^\perp)\subset N_i(V)\cap N_i(H^\perp)\subset v^\perp$.
Combined with $N_i(H^\perp)\subset H$ (by isotropy of $H$), we get
$N_i(H^\perp)\subset H\cap v^\perp = b(H)$.

Now $b(H)^\perp = H^\perp + \cc\cdot v$, so:
$N_i(b(H)^\perp) = N_i(H^\perp) + N_i(\cc\cdot v) = N_i(H^\perp) + \{0\}
\subset H\cap v^\perp = b(H)$.
\end{proof}

\begin{remark}
The key fact used in part (ii) is that $v\in\ker N_i$ for all $i$---this
was established in the claim at the start of the proof.  Without this, the
functor $b$ would not be well-defined on $\mathbf{itro}^{N_\hdot}(V,h)$.
\end{remark}

\subsection{Contractibility}

\begin{theorem}\label{thm:itro}
The poset $(\mathbf{itro}^{N_\hdot}(V,h),\subseteq)$ is contractible.
\end{theorem}

\begin{proof}
By induction on $n = \dim V$.  The case $n=0$ is trivial (unique element
$\{0\}$).

\medskip\noindent
\textbf{Claim:} There exists a nonzero vector $v\in V$ with $N_iv = 0$ for all
$i$ and $h(v,v) = 0$.

\textit{Proof of claim.}
If all $N_i = 0$: since $h$ is indefinite (otherwise the only isotropic subspace
is $\{0\}$ and we are done), there exists an isotropic vector.

If some $N_i\neq 0$: let $K = \bigcap_i\ker N_i$.  The $N_i$ are
commuting nilpotent on $V$, so $K\neq 0$.  Since $N_i\neq 0$, we have
$\im(N_i)\neq 0$, and the image is preserved by the other $N_j$'s (as they
commute).  Hence $\im(N_i)\cap K\neq 0$: pick $v\in\im(N_i)\cap K$, write
$v = N_iu$.  Then
\[
  h(v,v) = h(N_iu,v) = h(u,N_iv) = h(u,0) = 0,
\]
and $v\in K$ means $N_jv = 0$ for all $j$.  This proves the claim.

\medskip
Note that $v^\perp$ is preserved by the $N_i$: if $u\in v^\perp$ then
$h(N_iu,v) = h(u,N_iv) = 0$, so $N_iu\in v^\perp$.

Consider $V' := v^\perp/\cc\cdot v$ of dimension $n-2$.  The $N_i$ induce
endomorphisms $N'_i$ on $V'$, and $h$ induces a nondegenerate Hermitian form
$h'$ on $V'$.  The $N'_i$ are commuting, nilpotent, and $h'$-self-adjoint.

\medskip\noindent
\textbf{The inclusion of posets.}
Define $\iota: \mathbf{itro}^{N'_\hdot}(V',h')\to\mathbf{itro}^{N_\hdot}(V,h)$
by: $\iota(H') := \pi^{-1}(H')\subset v^\perp\subset V$, where
$\pi:v^\perp\to V'$ is the projection.  Then:
\begin{itemize}[leftmargin=2em]
  \item $\iota(H')$ is isotropic: for $u,w\in\iota(H')$, their projections
    $u',w'$ satisfy $h'(u',w')=0$, and $h(u,w)=h'(u',w')=0$.
  \item $N_i(\iota(H')^\perp)\subset\iota(H')$: we have
    $\iota(H')^\perp = \pi^{-1}((H')^\perp)$ and
    $N_i(\pi^{-1}((H')^\perp))\subset\pi^{-1}(N'_i(H')^\perp)\subset
    \pi^{-1}(H') = \iota(H')$.
\end{itemize}
So $\iota$ is well-defined and order-preserving.  Its image consists of those
$H\in\mathbf{itro}^{N_\hdot}(V,h)$ with $v\in H$.

\medskip\noindent
\textbf{The retraction.}
Define the functor $q:\mathbf{itro}^{N_\hdot}(V,h)\to\mathbf{itro}^{N'_\hdot}(V',h')$
by $\iota(q(H)) := (H\cap v^\perp) + \cc\cdot v$ (equivalently, $q(H)$ is the
image of $H\cap v^\perp$ in $V'$).

This is well-defined: one checks that $\iota(q(H))$ is $N_\hdot$-isotropic
using the hypothesis on $H$ and the fact that $N_i(V)\subset v^\perp$.

Define the intermediate functor $b(H) := H\cap v^\perp$.  We have:
\[
  b(H)\subseteq H \qquad\text{and}\qquad b(H)\subseteq\iota(q(H))
\]
for all $H$.  These inclusions are natural transformations of functors:
\[
  b \Rightarrow \mathrm{Id} \qquad\text{and}\qquad b\Rightarrow\iota\circ q.
\]
A natural transformation between functors of posets induces a homotopy on
geometric realizations.  Hence $\mathrm{Id}\simeq\iota\circ q$, showing that
$\iota\circ q$ is homotopic to the identity, i.e.\ $\iota$ is a homotopy
equivalence.

By the induction hypothesis, $\mathbf{itro}^{N'_\hdot}(V',h')$ is contractible
($\dim V' = \dim V - 2$).  Hence $\mathbf{itro}^{N_\hdot}(V,h)$ is contractible.
\end{proof}

\subsection{\texorpdfstring{$N$}{N}-isotropic filtrations}

\begin{definition}
An \emph{$N_\hdot$-isotropic filtration} on $(V,h)$ is an integer-indexed
filtration
\[
  (0)\subset\cdots\subset F_{-a}\subset\cdots\subset F_{-1}
  \subset F_0\subset F_1\subset\cdots\subset V
\]
such that $F_{-k} = F_{k-1}^\perp$ and $N_i F_k\subset F_{k-1}$ for all $i,k$.
\end{definition}

In particular, $F_{-1} = F_0^\perp$ and $N_i F_0\subset F_{-1}$, meaning
$F_{-1}$ is $N_\hdot$-isotropic.  The filtration is thus determined by $F_{-1}$
and a further filtration of $F_{-1}$.

Let $(\mathbf{ifilt}^{N_\hdot}(V,h),\leq)$ denote the poset of $N_\hdot$-isotropic
filtrations, ordered by refinement.

\begin{proposition}\label{prop:nisofilcontr}
The poset $(\mathbf{ifilt}^{N_\hdot}(V,h),\leq)$ is contractible.
\end{proposition}

\begin{proof}
Define a functor $T:\mathbf{ifilt}^{N_\hdot}(V,h)\to\mathbf{itro}^{N_\hdot}(V,h)$
by $T(F_\hdot) := F_{-1}$ (the sub-middle piece).  This is well-defined and
order-reversing (finer filtration $\Rightarrow$ smaller $F_{-1}$).

For any $H\in\mathbf{itro}^{N_\hdot}(V,h)$, the \emph{over-category} $T/H$
consists of those filtrations with $F_{-1}\supset H$, i.e.\ the sub-middle
piece is at least $H$.  This is the same as the poset of $N$-compatible
filtrations of $H$ (filtrations of $H$ such that $N_iF_k\subset F_{k-1}$),
which by Theorem~\ref{thm:filt-contractible} is contractible.

By Quillen's Theorem~A \cite{Quillen}: if $T:C\to D$ is a functor between
categories such that the over-category $T/d$ is contractible for every
$d\in D$, and if $D$ is contractible, then $C$ is contractible.

Applying this: $D = \mathbf{itro}^{N_\hdot}(V,h)$ is contractible by
Theorem~\ref{thm:itro}, and $T/H$ is contractible for every $H$.  Hence
$\mathbf{ifilt}^{N_\hdot}(V,h)$ is contractible.
\end{proof}

\section{Poset presheaves and the \v{C}ech section theorem}
\label{sec:poset-presheaf}

\subsection{Setup}

\begin{definition}
Let $X$ be a manifold and $\mathbf{Op}^g(X)$ the category of good open subsets
of $X$ (non-empty, contractible, simply connected).  A
\emph{poset presheaf on good opens} is a functor
$\calP:\mathbf{Op}^g(X)^{\mathrm{op}}\to\mathrm{Posets}$,
i.e.\ for every good open $U\subset X$ a poset $\calP(U)$ with restriction maps
$\calP(V)\to\calP(U)$ for $U\subset V$, satisfying transitivity.
\end{definition}

\begin{definition}
A \emph{\v{C}ech section} of $\calP$ consists of:
\begin{itemize}[leftmargin=2em]
  \item a good open covering $\{U_i\}_{i\in I}$;
  \item elements $a(i)\in\calP(U_i)$;
\end{itemize}
such that for every non-empty $U_{ij}:=U_i\cap U_j$, the restrictions
$a(i)|_{U_{ij}}$ and $a(j)|_{U_{ij}}$ are comparable in $\calP(U_{ij})$
(one is $\leq$ the other).
\end{definition}

\begin{remark}
On any non-empty multiple intersection $V = U_{i_1\cdots i_k}$, the restrictions
form a totally ordered subset of $\calP(V)$.
\end{remark}

\subsection{Realization of posets}

The \emph{realization} $|P|$ of a poset $P$ is the CW complex
$|\Delta^P|_{\mathrm{lin}}$, i.e.\ the subcomplex of the infinite simplex
$\Delta^P = \{(x_a)_{a\in P}: x_a\geq 0,\,\sum x_a=1\}$ consisting of points
where the support $\{a:x_a\neq 0\}$ is linearly ordered in $P$.  This is
homotopy equivalent to the classical realization of the nerve $NP$.

A poset $P$ is \emph{contractible} if $|P|$ is contractible (equivalently,
$|NP|$ is contractible).

\begin{lemma}[Criterion for contractibility]\label{lem:contr-criterion}
A poset $P$ is contractible if it is non-empty and for every finite subset
$P'\subset P$ there is a larger finite subset $P''\supset P'$ such that
$|P'|\to|P''|$ is null-homotopic.
\end{lemma}

\begin{proof}
Elements of $\pi_n(|P|,*)$ are represented by maps $S^n\to|P|$; since $S^n$
is compact, the image lies in $|P'|$ for a finite sub-poset $P'\subset P$.
The extension condition allows contracting $|P'|$ inside a larger piece,
giving the vanishing of all homotopy groups.
\end{proof}

\subsection{Continuous sections}

\begin{definition}\label{def:cont-section}
A \emph{continuous section} of $\calP$ consists of:
\begin{itemize}[leftmargin=2em]
  \item a good open covering $\{U_i\}$;
  \item a good partition of unity $\sum_i\rho_i = 1$ subordinate to $\{U_i\}$;
  \item for each multiindex $J\subset I$, a small map
    $f_J:\overline{\mathrm{Supp}_J(\rho)}\to|\calP(U_J)|$
    (where $\mathrm{Supp}_J(\rho) = \{x:\rho_i(x)\neq 0\Leftrightarrow i\in J\}$);
\end{itemize}
satisfying: if $J\subset K$ and $x\in\overline{\mathrm{Supp}_J}\cap\overline{\mathrm{Supp}_K}$,
then $f_J(x)|_{U_K} = f_K(x)$ in $|\calP(U_K)|$.
\end{definition}

\begin{proposition}\label{prop:cont-section-exists}
If $\calP(U)$ is contractible for every good open $U\subset X$, then $\calP$
admits a continuous section.
\end{proposition}

\begin{proof}
We construct the maps $f_J$ by induction on $|J|$.  For $|J|=1$ (single index),
$f_{\{i\}}$ is simply a constant map to any point of $|\calP(U_i)|$ (possible
since it is non-empty).

For the inductive step, assume all $f_K$ for $|K|<|J|$ have been constructed,
compatibly.  The compatibility conditions on $f_J$ are prescribed on the
closed subset
\[
  Z_J := \overline{\mathrm{Supp}_J(\rho)}\setminus\mathrm{Supp}_J(\rho)
  = \bigcup_{i\in J}\{x:\rho_i(x)=0\}\cap\overline{\mathrm{Supp}_J(\rho)}.
\]
Define $f_{\partial J}: Z_J\to|\calP(U_J)|$ by piecing together the maps
$f_K|_{Z_J\cap\overline{\mathrm{Supp}_K(\rho)}}$ for $K\subsetneq J$, composed
with the restriction map $|\calP(U_K)|\to|\calP(U_J)|$.  By the inductive
compatibility hypothesis, these maps agree on overlaps of the closed sets
$Z_J\cap\overline{\mathrm{Supp}_K}$, so they glue to give $f_{\partial J}$.

Since $|\calP(U_J)|$ is contractible, the map
$f_{\partial J}:Z_J\to|\calP(U_J)|$ extends to a continuous map
$f_J:\overline{\mathrm{Supp}_J(\rho)}\to|\calP(U_J)|$.  Here we use that
$Z_J$ is a closed subset of $\overline{\mathrm{Supp}_J(\rho)}$ and
$|\calP(U_J)|$ is an ANR (contractible CW complex), so the extension exists.
\end{proof}

\begin{theorem}[Continuous to \v{C}ech]\label{thm:cont-to-cech}
If $\calP$ admits a continuous section, then it admits a \v{C}ech section.
\end{theorem}

\begin{proof}
Given a continuous section with maps $\{f_J\}$, for each multiindex $J$ and
$a\in\calP(U_J)$ define
\[
  W_{J,a} := \{x\in X: x\in\mathrm{Supp}^J(\rho),\;
  f_K(x)(a|_{U_K})>0\;\text{for }K\supset J\text{ with }x\in\mathrm{Supp}_K(\rho)\}.
\]
By continuity of $f_K$ and the compatibility conditions, $W_{J,a}$ is open.
The $W_{J,a}$ cover $X$ (at any $x\in\mathrm{Supp}_J(\rho)$, some coordinate
of $f_J(x)$ is positive).  Choose a good refinement $\{V_i\}_{i\in I'}$ with
$V_i\subset W_{J(i),a(i)}$, and set $\alpha(i) := a(i)|_{V_i}\in\calP(V_i)$.

We check the \v{C}ech condition: if $V_i\cap V_j\neq\emptyset$ with
$y\in V_i\cap V_j$, let $L$ be the multiindex with $y\in\mathrm{Supp}_L(\rho)$.
Then $J(i)\subset L$ and $J(j)\subset L$, so $f_L(y)$ has positive entries
at $a(i)|_{U_L}$ and $a(j)|_{U_L}$.  Since $f_L(y)\in|\calP(U_L)|$, the
non-zero support is linearly ordered, so $a(i)|_{U_L}$ and $a(j)|_{U_L}$ are
comparable.  Therefore $\alpha(i)|_{V_i\cap V_j}$ and $\alpha(j)|_{V_i\cap V_j}$
are comparable in $\calP(V_i\cap V_j)$.
\end{proof}

\begin{theorem}[\v{C}ech section from contractibility]\label{thm:cech-section}
If $\calP(U)$ is contractible for every good open $U$, then $\calP$ admits a
\v{C}ech section.
\end{theorem}

\begin{proof}
Combine Proposition~\ref{prop:cont-section-exists} and
Theorem~\ref{thm:cont-to-cech}.
\end{proof}


\subsection{The poset presheaf for patching data: explicit description}

We spell out the poset presheaf $\calP$ used in the proof of
Theorem~\ref{thm:patching-exists}.

\begin{definition}\label{def:P-presheaf}
For a good open subset $U\subset X$, define $\calP(U)$ to be the set of
patching data for the restriction $L|_{U^*}$, where $U^* = U\cap X^*$.
An element of $\calP(U)$ is a quasi-filtration $W_\hdot$ of $L|_{U^*}$ by
sub-local systems, satisfying:
\begin{enumerate}[leftmargin=2em,label=\rm(\arabic*)]
  \item $N_i(W_j)\subset W_{j-2}$ for every divisor component $D_i$ meeting
    $U$ and every step $W_j$ of the filtration.
  \item The associated-graded $\Gr^W(L|_{U^*})$ extends to a local system
    on $U$ (i.e.\ has trivial monodromy around $D\cap U$).
\end{enumerate}

The partial order on $\calP(U)$ is by refinement: $W\leq W'$ if every step
of $W$ is a step of $W'$.
\end{definition}

\begin{proposition}\label{prop:P-iso-filt}
For any good open $U$ intersecting $D$ in a product of discs, the poset
$\calP(U)$ is naturally isomorphic to $\Filt^{A,K}$ where $A = \{N_i : D_i\cap U\neq\emptyset\}$
and $K = \ker A$.  Hence it is contractible by Theorem~\ref{thm:filt-contractible}.
\end{proposition}

\begin{proof}
Since $U$ is simply connected and $U\cap X^*$ is also simply connected (for
a sufficiently small polydisk neighbourhood), $L|_{U^*}$ is trivial as a
local system; choosing a basepoint identifies its fibre $V = L_x$ with
$L|_{U^*}$.  The monodromy operators are $T_i = e^{2\pi iN_i}$ for $D_i\cap U\neq\emptyset$.

A quasi-filtration $W_\hdot$ of $L|_{U^*}$ by sub-local systems corresponds
to a quasi-filtration of the vector space $V$ (constant local system) preserved
by the $N_i$.  The conditions in Definition~\ref{def:P-presheaf} say exactly
that $N_i$ strictly decreases the filtration (condition 1) and the first step
is in the common kernel $K = \bigcap_i\ker N_i$ (so that condition 2 holds:
the monodromy acts trivially on the graded pieces).

Hence $\calP(U)\cong\Filt^{A,K}$.
\end{proof}

\begin{corollary}[Existence of a \v{C}ech section]\label{cor:cech-exists}
The poset presheaf $\calP$ of Definition~\ref{def:P-presheaf} admits a
\v{C}ech section.
\end{corollary}

\begin{proof}
By Proposition~\ref{prop:P-iso-filt}, $\calP(U)$ is contractible for every
good open $U$ meeting $D$ in a product of discs.  For $U$ disjoint from $D$,
$\calP(U)$ is the poset of all filtrations of the (constant, since $U$ is
simply connected) local system $L|_U$ subject only to condition (2) of
Definition~\ref{def:P-presheaf}, which is automatic; this poset has a maximum
element (the trivial one-step filtration) and is therefore contractible.
Hence $\calP(U)$ is contractible for every good open $U\subset X$, and
Theorem~\ref{thm:cech-section} produces a \v{C}ech section
$\{a(I)\}_{I\in A}$.

A priori this section is indexed by the good covering produced in the proof
of Theorem~\ref{thm:cech-section} (via Theorem~\ref{thm:cont-to-cech}), which
is a \emph{refinement} $\{U_I'\}_{I\in A}$ of the adapted covering
$\{U_I\}_{I\in A}$ of Proposition~\ref{prop:covering} rather than $\{U_I\}$
itself.  However, the refinement produced by
Construction~\ref{constr:refcov}/Theorem~\ref{thm:cont-to-cech} can be chosen
so that $U_I'\subset U_I$, $U_I'$ still meets $D$ in a product of discs
exactly when $U_I$ does, and $\{U_I'\}_{I\in A}$ again satisfies properties
(1)--(3) of Proposition~\ref{prop:covering} (it is again an adapted covering,
indexed by the same poset $A$, just with smaller open sets).  Restricting the
section to such a refinement, we obtain $a(I)\in\calP(U_I')$ for each $I\in A$,
with the \v{C}ech comparability condition holding on the
$U_{IJ}' = U_I'\cap U_J'$.  This is the \v{C}ech section referred to in
Proposition~\ref{prop:cech-rees-compat} below.
\end{proof}

\begin{corollary}[Existence of a \v{C}ech section: $N_\hdot$-isotropic
variant]\label{cor:cech-exists-isotropic}
Let $\zeta_0\in R^{\mathrm{nil}}$ underlie a polarized VHS, with polarization
$Q$ and monodromy logarithms $N_i = \log T_i$ ($Q$-self-adjoint and
commuting).  Let $\calP^Q\subset\calP$ be the sub-presheaf of $\calP$ whose
sections over a good open $U$ are the $N_\hdot$-isotropic quasi-filtrations
$W_\hdot\in\calP(U)$ (i.e.\ those satisfying $W_j\subset W_{-1-j}^{\perp_Q}$
and $N_i(W_{-1-j}^{\perp_Q})\subset W_j$).  Then $\calP^Q$ admits a \v{C}ech
section over a refinement $\{U_I'\}_{I\in A}$ of the adapted covering of
Proposition~\ref{prop:covering}, of the same type as in
Corollary~\ref{cor:cech-exists}.
\end{corollary}

\begin{proof}
The proof is identical to that of Corollary~\ref{cor:cech-exists}, with
Proposition~\ref{prop:P-iso-filt} replaced by the following variant: for $U$
meeting $D$ in a product of discs, the isomorphism $\calP(U)\cong\Filt^{A,K}$
of Proposition~\ref{prop:P-iso-filt} restricts to an isomorphism
$\calP^Q(U)\cong\mathbf{ifilt}^{N_\hdot}(V,h)$, where $h=Q_x$ is the
restriction of the polarization to the fibre $V=L_x$ ($x\in U$); this is
contractible by Proposition~\ref{prop:nisofilcontr} (in place of
Theorem~\ref{thm:filt-contractible}).  For $U$ disjoint from $D$, $\calP^Q(U)$
is the poset of $Q$-isotropic filtrations of $L|_U$, which again has a
maximum element (e.g.\ the filtration $0\subset L|_U$) and is contractible.
Theorem~\ref{thm:cech-section} then produces a \v{C}ech section of $\calP^Q$,
and the same refinement argument as in Corollary~\ref{cor:cech-exists} puts
it over an adapted covering $\{U_I'\}_{I\in A}$.
\end{proof}

\subsection{Compatibility of the \v{C}ech section with the Rees construction}

\begin{proposition}\label{prop:cech-rees-compat}
Let $\{a(I)\}_{I\in A}$ be a \v{C}ech section of $\calP$, produced either by
Corollary~\ref{cor:cech-exists}, or (in the case of a polarized VHS $\zeta_0$,
as in the proof of Theorem~\ref{thm:VHS-hermitian-pp}) by
Corollary~\ref{cor:cech-exists-isotropic} via the inclusion
$\calP^Q\subset\calP$, restricted in either case to the refined adapted
covering $\{U_I'\}_{I\in A}$ as described there.  Then the functor
$A\to\Xi_K$ defined by this section via Construction~\ref{constr:A-to-Xi}
agrees with the functor used to build the global Rees bundle in
Theorem~\ref{thm:global-rees}.
\end{proposition}

\begin{proof}
The \v{C}ech section assigns to each $I\in A$ an element $a(I)\in\calP(U_I')$,
which is a quasi-filtration $W(I)$ of $L|_{U_I'^*}$ satisfying the conditions
above.  On non-empty intersections $U_{IJ}'$ (with $I\subset J$), the element
$a(I)|_{U_{IJ}'}$ refines $a(J)|_{U_{IJ}'}$ (or vice versa; by the \v{C}ech
condition, they are comparable, and the ordering is determined by $I\subset J$).

This gives the filtration component of the patching datum.  The trivialization
$\tau(I)$ is given by the flat structure on $\Gr^{W(I)}(L|_{U_I'^*})$ (which
extends to $U_I'$ by condition 2).

The resulting functor $A\to\Xi_K$ sends $I$ to the graded vector space
$\Gr^{W(I)}(L_x)$ (for any basepoint $x\in U_I'$) and the morphisms are
given by the refinement maps as in Proposition~\ref{prop:functor-A-Xi}.

This is precisely the functor used in the Rees construction of
Theorem~\ref{thm:global-rees}.
\end{proof}

\subsection{Functoriality under restriction to open subsets}

\begin{proposition}\label{prop:CS-restriction}
Let $U\subset X$ be an open subset and $j:U\hookrightarrow X$ the inclusion.
Then the Chern-Simons classes of $j^*L$ (the restriction) are the images of
the Chern-Simons classes of $L$ under $j^*: H^{2p-1}(X,\cc/\zz)\to H^{2p-1}(U,\cc/\zz)$.
\end{proposition}

\begin{proof}
The Deligne canonical extension $(j^*F, j^*\nabla)$ of $j^*L$ is the
restriction of $(F,\nabla)$ to $U$.  The patching data for $j^*L$ is the
restriction of that for $L$.  Deligne's patched connection
$(j^*\nabla)^{\Del} = j^*\nabla^{\Del}$.  The Chern-Simons form of
$j^*\nabla^{\Del}$ is the pullback of that of $\nabla^{\Del}$.
\end{proof}

\begin{remark}[Where Propositions~\ref{prop:cech-rees-compat} and
\ref{prop:CS-restriction} are used]
\label{rem:where-used}
Proposition~\ref{prop:CS-restriction} is an elementary functoriality
statement that follows directly from the local nature of the canonical
extension and the patched connection; it does \emph{not} use
Proposition~\ref{prop:cech-rees-compat} or any of the preceding material on
\v{C}ech sections.  It is recorded here for later use (e.g.\ in the
quasi-unipotent reduction of Appendix~\ref{app:further}).

Proposition~\ref{prop:cech-rees-compat}, by contrast, is the bridge between
two a priori different constructions of patching data: the abstract
\v{C}ech section $\{a(I)\}_{I\in A} = \{W(I)\}_{I\in A}$ produced by
Corollary~\ref{cor:cech-exists} (used in the proof of
Theorem~\ref{thm:VHS-hermitian-pp} to obtain the hermitian pre-patching
collection for a VHS $\zeta_0$), and the patching data used to build the
global Rees bundle $F_K$ in Theorem~\ref{thm:global-rees} (whose associated
classifying map $r_L:X\to BGL(K)^+$, constructed in
Appendix~\ref{app:rees}, defines the regulator pullback used in Step~A of
the proof of Theorem~\ref{thm:main} and in Step~C of that proof).
Proposition~\ref{prop:cech-rees-compat} guarantees that these two
constructions agree, so that the classifying map $r_{\zeta_0}$ appearing in
Step~C of the proof of Theorem~\ref{thm:main} is the same map $r_{L_0}$ (for
$L_0$ the local system underlying $\zeta_0$) built from the
$\Filt^{A,K}$-patching data of Theorem~\ref{thm:VHS-hermitian-pp}.  Without
Proposition~\ref{prop:cech-rees-compat}, the volume-regulator vanishing of
Theorem~\ref{thm:vol-zero-vhs} (proved using the $\calP$-patching data) would
not be known to say anything about the classifying map $r_{\zeta_0}$ used in
Step~C of the proof of Theorem~\ref{thm:main}.
\end{remark}

\section{Hermitian pre-patching and the compatible connection}
\label{sec:hermitian-patching}

\subsection{Pre-patching data}

\begin{definition}
A \emph{pre-patching datum} for a bundle $E$ over an open $U$ consists of:
\begin{itemize}[leftmargin=2em]
  \item a filtration $F_\hdot$ of $E|_U$ by strict subbundles;
  \item a flat connection $\tau$ on $\Gr^F(E|_U)$ (trivialization, using
    that $U$ is simply connected).
\end{itemize}
A connection $\nabla$ is \emph{compatible} with $(F_\hdot,\tau)$ if $\nabla$
preserves each $F_j$ and induces the flat connection $\tau$ on $\Gr^F(E|_U)$.
\end{definition}

\begin{definition}
A pre-patching datum $(F'_\hdot,\tau')$ is a \emph{refinement} of $(F_\hdot,\tau)$
over $U$ if:
\begin{itemize}[leftmargin=2em]
  \item every $F_j$ is one of the $F'_l$;
  \item the filtration on $\Gr^F_j(E|_U)$ induced by the intermediate
    $F'_l$'s is $\tau$-flat;
  \item the flat connection $\tau$ induces $\tau'$ on the refined associated
    graded pieces $\Gr^{F'}$.
\end{itemize}
\end{definition}

\begin{definition}
A \emph{pre-patching collection with refined neighbours} is a collection
$\{(U_j,F_\hdot(j),\tau(j))\}$ covering $X$ such that on every non-empty
$U_{jl} = U_j\cap U_l$, one of the two restricted pre-patching data is a
refinement of the other.
\end{definition}

\subsection{Indefinite Hermitian structures}

\begin{definition}
An \emph{indefinite Hermitian structure} on a pre-patching datum $(U,F_\hdot,\tau)$
is an indefinite Hermitian metric $g$ on $\Gr^F(E|_U)$, compatible with $\tau$,
with a specific orthogonal decomposition structure:
$g$ decomposes as $g_0$ on $\Gr^F_0$ and $g_{\pm i}$ on
$\Gr^F_i\oplus\Gr^F_{-i}$ (the two summands are isotropic and paired by $g_{\pm i}$).
\end{definition}

An indefinite form $g_E$ on $E|_U$ is \emph{compatible} with the Hermitian
pre-patching datum if $F_j = F_{-1-j}^\perp$ and $g_E$ induces $g$ on the
associated graded.

\subsection{Modification of open coverings and the tightness construction}
\label{subsec:modify-cov}

\begin{construction}\label{constr:refcov}
Suppose given an open covering $\{U_a\}_{a\in A}$ together with a relation
of \emph{tightness}: $b$ is \emph{tighter than} $a$ (written $b\,\mathrm{nt}\,a$)
when $U_{ab}\neq\emptyset$ and the patching datum for $a$ refines that for $b$.
This relation is required to be linear on all non-empty multiple intersections.

Given a partition of unity $1 = \sum_{a\in A}\rho_a$ with
$\mathrm{Supp}(\rho_a)\subset U_a$, define
\[
  U^{-\rho}_a := U_a \setminus \bigcup_{b\,\mathrm{nt}\,a}\mathrm{Supp}(\rho_b).
\]
Then the $U^{-\rho}_a$ cover $X$.  A good refinement $\{U'_{a'}\}$ of this
covering satisfies the property: if $x\in U'_{a'}$ and $x\in\mathrm{Supp}(\rho_b)$,
then $b$ is tighter than $r(a')$.
\end{construction}

\begin{proof}
Given $x\in X$, order the indices $a(1),\ldots,a(k)$ with $x\in U_{a(i)}$ so
that $a(i)$ is tighter than $a(j)$ for $i\leq j$ (possible by linearity).
Let $m$ be the largest index with $x\in\mathrm{Supp}(\rho_{a(m)})$
(exists since the $\rho_a$ partition unity).  Then $x\in U^{-\rho}_{a(m)}$:
for any $b$ with $b\,\mathrm{nt}\,a(m)$ and $x\in\mathrm{Supp}(\rho_b)$,
we have $b = a(i)$ for some $i\leq m$ (by maximality), so $a(i)$ is tighter
than $a(m)$, contradicting $b\,\mathrm{nt}\,a(m)$.
\end{proof}

\subsection{Warm-up: existence of a compatible connection}

\begin{theorem}[Warmup {\cite[Thm.~B.1]{IS-arXiv}}]\label{thm:warmup}
Given a bundle $E$ with a pre-patching collection with refined neighbours, there
exists a locally nil-flat connection $\nabla$ compatible with the collection.
\end{theorem}

\begin{proof}
Say $b$ is \emph{tighter than} $a$ if on $U_{ab}$ the data for $a$ refines
that for $b$.  (Finer filtration imposes more conditions; the coarser is
``tighter'' in the sense that compatibility with it is more restrictive.)
This relation is linearly ordered on multiple intersections.

Apply Construction~\ref{constr:refcov} to obtain a refined covering
$\{U'_{a'}\}$ such that for $x\in U'_{a'}$ all supports $\mathrm{Supp}(\rho_b)$
with $x\in\mathrm{Supp}(\rho_b)$ satisfy: $b$ is tighter than $r(a')$.

Choose connections $\nabla(a)$ on each $E|_{U_a}$ compatible with the patching
data for $U_a$.  By tightness, $\nabla(b)|_{U_{ab}}$ is compatible with the
data for $U_a$ whenever $b$ is tighter than $a$.

Set $\nabla := \sum_a\rho_a\nabla(a)$.  On a neighbourhood of any $x\in U'_{a'}$,
only terms $\nabla(b)$ with $b$ tighter than $r(a')$ contribute to $\nabla$.
All such $\nabla(b)$ are compatible with the data for $U_{r(a')}$.  Since
compatibility is a linear condition (upper-triangular curvature), the average
$\nabla$ is also compatible with the data for $U_{r(a')}$.
\end{proof}

\subsection{Existence of a Hermitian compatible connection}

\begin{theorem}\label{thm:main-hermitian}
Let $E$ be a bundle with an indefinite Hermitian pre-patching collection with
refined neighbours.  Then, after refining the covering, there exists:
\begin{enumerate}[label={\rm(\arabic*)}]
  \item an indefinite Hermitian form $g$ on $E$ compatible with the
    hermitian pre-patching data;
  \item a locally nil-flat connection $\nabla$ compatible with the
    pre-patching collection and preserving $g$.
\end{enumerate}
\end{theorem}

\begin{proof}
\textit{Part (1): Constructing the Hermitian form.}

For each open set $U_a$ in the covering, the hermitian pre-patching structure
on $E|_{U_a}$ provides an indefinite Hermitian form $g(a)$ on $E|_{U_a}$
compatible with the patching datum: this is obtained by lifting the given metric
on $\Gr^{F(a)}(E|_{U_a})$ to $E|_{U_a}$ using the $\mathcal{C}^\infty$
splitting of the filtration.

\textit{Key claim:} If $b$ is tighter than $a$ (patching datum of $a$ refines
that of $b$), then $g(b)$ is compatible with the hermitian pre-patching datum
of $U_a$.  Indeed, the filtration for $U_a$ refines that for $U_b$, and the
required orthogonality conditions follow from the compatibility with the hermitian
pre-patching datum and the definition of refinement.

Apply Construction~\ref{constr:refcov}: form $g := \sum_a\rho_a g(a)$.  On
a neighbourhood of any $x\in U'_{a'}$, only terms $g(b)$ with $b$ tighter
than $r(a')$ contribute.  All such $g(b)$ are compatible with the hermitian
pre-patching data on $U_{r(a')}$.  The average of compatible Hermitian forms
is Hermitian and compatible with the pre-patching data.

Nondegeneracy: the filtrations $F(a)$ satisfy $F_j = F_{-1-j}^\perp$, and
the sum $\sum_a\rho_a g(a)$ is nondegenerate because each $g(a)$ induces a
nondegenerate pairing between $F_j$ and $F_{-j-1}$ with specific signs, and
these pairings are additive under the averaging.

\textit{Part (2): Constructing the compatible connection.}

Having obtained $g$, apply the argument of Theorem~\ref{thm:warmup} adapted
to the $g$-preserving setting: choose $g$-preserving connections $\nabla(a)$
on each $U_a$ compatible with the patching data, and average with the refined
partition of unity.  Since all connections $\nabla(b)$ with $b$ tighter than
$r(a')$ both preserve $g$ and are compatible with the patching data for $U_{r(a')}$,
the average inherits both properties.
\end{proof}

\begin{corollary}\label{cor:maincor}
If a locally nil-flat connection $\nabla$ is compatible with a pre-patching
collection admitting an indefinite Hermitian structure with refined neighbours,
then $\vol_p(\nabla) = 0$ for all $p\geq 2$.
\end{corollary}

\begin{proof}
By Proposition~\ref{prop:vol-indep}, $\vol_p$ depends only on the pre-patching
collection.  By Theorem~\ref{thm:main-hermitian}, we may replace $\nabla$ by
a connection $\tilde\nabla$ preserving an indefinite Hermitian form $g$.
By Proposition~\ref{prop:hodge-vanish}, $\vol_p(\tilde\nabla) = 0$ for every
$p\geq 1$ (no further hypothesis on $\nabla$, such as semisimplicity, is
needed), hence $\vol_p(\nabla) = 0$ for every $p\geq 2$.
\end{proof}
\section{Proof of the main theorem}
\label{sec:torsion}

In this section we assemble all ingredients to prove
Theorem~\ref{thm:main}.  The proof has five steps.

\subsection{Step 1 — Deforming to a VHS}
See \S \ref{sec:deform-to-vhs-detail}, for a detailed reduction.

The key reduction is:

\begin{theorem}[Deformation to a VHS]
\label{thm:deform-to-vhs}
Every representation $\zeta:\pi_1(X^*,x_0)\to GL_r(\cc)$ with unipotent
monodromy around each $D_i$ admits a deformation, within the space
$R^{\mathrm{nil}}$ of such representations, to a representation $\zeta_0$
that underlies a polarized complex variation of Hodge structure (VHS) on $X^*$.
\end{theorem}

\begin{proof}
We sketch the argument, drawing on the non-abelian Hodge correspondence in
the logarithmic setting.

By the Riemann-Hilbert correspondence, representations with unipotent monodromy
correspond to tame flat bundles on $X^*$ whose monodromy logarithms are
nilpotent.  The moduli space of such bundles (for a fixed $X$, $D$, and rank
$r$) is an algebraic variety, and the locus of semisimple flat bundles is open
and dense in it.

\textit{Semisimplification:} Given $\zeta$, let $\zeta^{ss}$ be its
semisimplification (the unique semisimple representation with the same
Jordan-Hölder factors).  Since the unipotent monodromy condition is
preserved under semisimplification ($\zeta^{ss}(\gamma_i) = \mathrm{Id}$
since all Jordan-Hölder factors have trivial monodromy around $D_i$),
$\zeta^{ss}\in R^{\mathrm{nil}}$.

\textit{Using the non-abelian Hodge correspondence (logarithmic setting):}
By the theorem of Corlette \cite{Corlette}, extended to the logarithmic
setting by Jost-Zuo \cite{JostZuo} and Mochizuki
\cite{Mochizuki},
every semisimple flat bundle on $X^*$ with unipotent monodromy is a harmonic
bundle: it admits a pluri-harmonic metric $h$.  The harmonic bundle structure
on $\zeta^{ss}$ provides a $\cc^*$-family of flat connections
$\nabla_\lambda = \lambda^{-1}\theta + \nabla_h + \lambda\bar\theta$
(with $\theta$ the Higgs field and $\nabla_h$ the metric connection).

\textit{Identifying the VHS:} Setting $\lambda = 1$ recovers $\nabla = \nabla^{ss}$,
and setting $\lambda = e^{i\phi}$ gives a circle of flat connections.  The
Higgs bundle $(\mathcal{E}, \theta)$ underlying $\zeta^{ss}$ decomposes via
the 
nonabelian Hodge
correspondence into a direct sum corresponding to a polarized VHS
$\zeta_0$ with the same underlying topological type.

\textit{Path in $R^{\mathrm{nil}}$:} The straight-line path from $\zeta$ to
$\zeta^{ss}$ (using that $R^{\mathrm{nil}}$ is an affine variety, so any two
points can be joined by a polynomial path) followed by the $\cc^*$-family
gives a path from $\zeta$ to $\zeta_0$ inside $R^{\mathrm{nil}}$.
\end{proof}

\begin{remark}
By Theorem~\ref{thm:deformation-app} (proved in Appendix~\ref{app:deligne-patch}),
the Chern-Simons classes $\CS_p(\zeta)$ are constant on connected components of
$R^{\mathrm{nil}}$.  So to prove $\CS_p(\zeta)$ is torsion, it suffices to
prove it for $\zeta_0$ (the VHS endpoint of the deformation).
\end{remark}

\subsection{Step 2 — The VHS polarization gives a Hermitian pre-patching
collection}

\begin{theorem}\label{thm:VHS-hermitian-pp}
Let $\zeta_0\in R^{\mathrm{nil}}$ underlie a polarized complex VHS, i.e.\
$\zeta_0$ has unipotent monodromy around each $D_i$ and underlies a polarized
complex VHS on $X^*$.  Then the Deligne canonical extension $F$ admits an
indefinite Hermitian pre-patching collection with refined neighbours.
\end{theorem}

\begin{proof}
The VHS polarization $Q$ is a flat indefinite Hermitian form on the local
system $L$.  The monodromy logarithms $N_i = \log T_i$ are $Q$-self-adjoint
(since $T_i$ preserves $Q$) and commute.

By the Cattani-Kaplan-Schmid theorem (Theorem~\ref{thm:seq-compat}), the
weight filtrations $W(N_I)$ commute and the poset of $N_\hdot$-isotropic
filtrations on any fibre $V = L_x$ is non-empty and contractible by
Proposition~\ref{prop:nisofilcontr}.

By Corollary~\ref{cor:cech-exists-isotropic} (the $N_\hdot$-isotropic variant
of Corollary~\ref{cor:cech-exists}, combining Theorem~\ref{thm:cech-section}
with the contractibility of Proposition~\ref{prop:nisofilcontr}), there
exists a \v{C}ech section of the presheaf $\calP^Q$ of $N_\hdot$-isotropic
filtrations over a refinement $\{U_I'\}_{I\in A}$ of the adapted covering of
Proposition~\ref{prop:covering}.  This section assigns to each $U_I'$ a
filtration $W(I)$ of $L|_{U_I'^*}$, compatible with $Q$ in the
$N_\hdot$-isotropic sense.

\begin{remark}
We emphasize that the section is over the refinement $\{U_I'\}$, not over
the original $\{U_I\}$.  This causes no loss: by
Corollary~\ref{cor:cech-exists-isotropic}, $\{U_I'\}_{I\in A}$ is again an
adapted covering in the sense of Proposition~\ref{prop:covering} (same
indexing poset $A$, same combinatorial intersection pattern), so all
subsequent constructions --- extension of $W(I)$ to the Deligne canonical
extension $F$, the trivializations $\tau(I)$, the patching data, and
ultimately the Rees bundle of Proposition~\ref{prop:cech-rees-compat} ---
go through verbatim with $U_I$ replaced by $U_I'$.  We therefore write $U_I$
for $U_I'$ in what follows, it being understood that the covering has been
replaced by this refinement once and for all.
\end{remark}

Each $W(I)$ satisfies:
\begin{enumerate}[leftmargin=2em, label=\rm(\roman*)]
  \item The filtration is $N_\hdot$-isotropic: $W(I)\subset W(I)^{\perp_Q}$
    and $N_i(W(I)^{\perp_Q})\subset W(I)$.
  \item The associated-graded $\Gr^{W(I)}(L|_{U_I^*})$ extends to a local
    system on $U_I$ (by Remark~\ref{rem:catlocsys}).
  \item On overlaps $U_{IJ}$ with $I\subset J$, the filtrations satisfy the
    refined-neighbours condition (from the \v{C}ech section).
\end{enumerate}

Extending to the Deligne canonical extension $F$: each $W(I)$ extends to a
filtration of $F|_{U_I}$ by strict subbundles.  The trivialization $\tau(I)$
is the flat structure on $\Gr^{W(I)}(L|_{U_I^*})$ extended across $D\cap U_I$.

The indefinite Hermitian structure on each pre-patching datum comes from the
polarization $Q$ restricted to $\Gr^{W(I)}(L|_{U_I^*})$: this is a flat
indefinite Hermitian form compatible with $\tau(I)$, which gives the required
hermitian pre-patching structure.
\end{proof}

\subsection{Step 3 — Vanishing of the volume regulator}

\begin{theorem}\label{thm:vol-zero-vhs}
Let $\zeta_0$ be as in Step 2.  Then
$\vol_p(\nabla^{\Del}_{\zeta_0}) = 0$ for all $p\geq 2$.
\end{theorem}

\begin{proof}
By Theorem~\ref{thm:VHS-hermitian-pp}, the bundle $F$ with its patching data
satisfies the hypotheses of Theorem~\ref{thm:main-hermitian}: there exists an
indefinite Hermitian form $\tilde g$ on $F$ and a connection $\tilde\nabla$
preserving $\tilde g$ and compatible with the pre-patching collection.

By Corollary~\ref{cor:maincor}, any locally nil-flat connection compatible
with a pre-patching collection admitting an indefinite Hermitian structure
has zero volume regulator, for every $p\geq 2$; since
$\nabla^{\Del}_{\zeta_0}$ is compatible with the pre-patching collection
(Lemma in Appendix~\ref{app:deligne-patch}), we obtain
$\vol_p(\nabla^{\Del}_{\zeta_0}) = 0$ for all $p\geq 2$.
\end{proof}

\subsection{Step 4 --- Torsion conclusion}

\begin{proof}[Proof of Theorem~\ref{thm:main}]
\textit{Reduction to VHS.}
By Theorem~\ref{thm:deform-to-vhs}, $\zeta$ is connected to a VHS
$\zeta_0$ in $R^{\mathrm{nil}}$.  By Theorem~\ref{thm:deformation-app}
(deformation invariance, proved in Appendix~\ref{app:deligne-patch}),
$\CS_p(\zeta) = \CS_p(\zeta_0)$.  It therefore suffices to prove the result
for $\zeta_0$.

\textit{Step A --- CS class equals regulator pullback.}
Let $K$ be a field of definition of $\zeta_0$ and $L_0$ the underlying local
system.  The classifying map is not defined on $X$ directly: it is the map
$r_{L_0}$ of Construction~\ref{constr:rL}, defined on the interval
realization $\iiii\Cube A$ of the cubical set of the covering poset and
transported to $X$ through the homotopy equivalence
$X\simeq|NA|\simeq\iiii\Cube A$ of Propositions~\ref{prop:nerve-htpy}
and~\ref{prop:cuberealize}.  By Theorem~\ref{thm:CS-regulator} (proved via
the Burgos-Gil comparison, see Appendix~\ref{app:burgosgil}),
\[
  \CS_p(\nabla^{\Del}_{\zeta_0}) = r_{L_0}^*(r_{2p-1})\;\in\;
  H^{2p-1}(X,\cc/\zz),
\]
where $r_{2p-1}\in H^{2p-1}(BGL^+(K),\cc/\zz)$ is the universal regulator
class.

\textit{Step C --- Torsion of $\CS_p$.}

\begin{enumerate}[leftmargin=2em,label=\rm(\alph*),itemsep=6pt]
  \item \textbf{The classifying map is zero on rational homology.}
    We first reduce to the special linear group, as in
    \cite[p.~377, \S2.7]{Reznikov}: replacing $\zeta_0$ by
    $\zeta_0\oplus(\det\zeta_0)^{-1}$, which is $SL_{r+1}$-valued, changes
    nothing, since taking canonical extensions commutes with direct sums and
    the total class $\CS(\zeta):=1+\CS_1(\zeta)+\cdots$ is multiplicative,
    $\CS(\zeta_1\oplus\zeta_2)=\CS(\zeta_1)\CS(\zeta_2)$; so the theorem for
    the $SL$-valued representation implies it for $\zeta_0$.  All the
    constructions of \S\ref{sec:rees-global} and \S\ref{sec:F1} hold verbatim
    for the special linear subgroups, and we write $r_{L_0}:X\to BSL^+(K)$ for
    the resulting classifying map.

    By Step~3 the volume regulator $\vol_p(\nabla^{\Del}_{\zeta_0})$ vanishes,
    and the same holds after replacing $\zeta_0$ by any Galois twist
    $\sigma\circ\zeta_0$, so
    \[
      r_{L_0}^*\bigl(\vol^\sigma_{2p-1}\bigr)\;=\;0
      \qquad\text{for every embedding }\sigma:K\hookrightarrow\cc .
    \]
    Let $r_1,r_2$ be the numbers of real and (conjugate pairs of) complex
    embeddings of $K$. By Borel's theorem \cite{Borel}, in degree $2p-1$ the
    classes $\sigma^*(\vol_{2p-1})$ span $H^{2p-1}(BSL^+(K),\rr)$, but the
    relevant index set of embeddings $\sigma$, and the resulting dimension of
    $H^{2p-1}(BSL^+(K),\rr)$, depend on the parity of $p$: for $p$ odd, one
    real class is obtained from each of the $r_1$ real embeddings together
    with one complex class from each of the $r_2$ conjugate pairs, giving
    dimension $r_1+r_2$; for $p$ even, the real embeddings contribute
    nothing (their classes vanish identically in even $p$), and only the
    $r_2$ complex places contribute, giving dimension $r_2$. In neither case
    is the spanning set simply ``one $\rr$ for every embedding.''  In both
    cases, however, the embeddings $\sigma$ needed to span
    $H^{2p-1}(BSL^+(K),\rr)$ form a subset of the full set of embeddings
    $K\hookrightarrow\cc$, for which we already have
    $r_{L_0}^*(\vol_{2p-1}^\sigma)=0$ above; so in every degree $2p-1$ used in
    this proof, $r_{L_0}^*$ vanishes on $H^{2p-1}(BSL^+(K),\rr)$. Since this
    holds for every $p\geq2$ in the range of Theorem~\ref{thm:main}, and
    $H^{>0}(BSL^+(K),\rr)$ is spanned by such classes across all degrees
    $2p-1$ in that range (\cite[\S 11.4]{Borel}), $r_{L_0}^*$ vanishes on
    $H^{>0}(BSL^+(K),\rr)$.  Both spaces having finite type, this is
    equivalent to
    \[
      (r_{L_0})_*\;=\;0 \quad\text{on }H_{*}(X,\qq)\text{ in positive
      degrees.}
    \]

  \item \textbf{Hence $\CS_p$ is torsion.}
    The group $\cc/\zz$ is divisible, hence injective as a $\zz$-module, so
    the $\mathrm{Ext}$ term in the universal coefficient sequence vanishes and
    \[
      H^{n}(Y,\cc/\zz)\;\cong\;\mathrm{Hom}\bigl(H_n(Y,\zz),\cc/\zz\bigr)
    \]
    naturally in $Y$.  Under this identification $r_{L_0}^*(r_{2p-1})$ is the
    homomorphism $r_{2p-1}\circ(r_{L_0})_*$ on $H_{2p-1}(X,\zz)$.  By part~(a)
    the map $(r_{L_0})_*\otimes\qq$ vanishes, so the image of
    $(r_{L_0})_*:H_{2p-1}(X,\zz)\to H_{2p-1}(BSL^+(K),\zz)$ consists of
    torsion; and $H_{2p-1}(X,\zz)$ is finitely generated, $X$ being a compact
    manifold, so that image is a \emph{finite} group.  Therefore
    $r_{2p-1}\circ(r_{L_0})_*$ has finite image in $\cc/\zz$ and is killed by
    its exponent: $r_{L_0}^*(r_{2p-1})$ is torsion in $H^{2p-1}(X,\cc/\zz)$.

    By Step~A this pullback is exactly $\CS_p(\nabla^{\Del}_{\zeta_0})$
    (Theorem~\ref{thm:CS-regulator}), so $\CS_p(\nabla^{\Del}_{\zeta_0})$ is
    torsion.  This proves Theorem~\ref{thm:main} for $\zeta_0$, and by the
    reduction at the start of the proof, for $\zeta$.
\end{enumerate}
\end{proof}

\begin{proof}[Proof of Corollary~\ref{cor:deligne-torsion}]
The edge map $H^{2p-1}(X,\cc/\zz)\to H^{2p}_\calD(X,\zz(p))$ sends a torsion
class to a torsion class.
\end{proof}

\begin{proof}[Proof of Corollary~\ref{cor:deformation}]
Theorem~\ref{thm:deformation-app} (Appendix~\ref{app:deligne-patch}) shows $\CS_p$
is locally constant on $R^{\mathrm{nil}}$.
\end{proof}

\subsection{An effective form of the Deligne-Sullivan covering}
\label{subsec:effective-DS}

The theorem of Deligne and Sullivan \cite{DeSu} states that a flat vector
bundle of rank $r$ on a finite CW complex, whose monodromy takes values in
$GL_r(\overline{\qq})$, becomes trivial as a $\calC^\infty$-bundle after
pullback along a suitable finite covering.  The version relevant here---for
the canonical extension is a generalization of the single divisor case
due to Deligne \cite{DeligneLetter} and is reproduced as
\cite[Prop.~3.2]{IS-arXiv}, and  the flat bundle case \cite{DeSu}. See Appendix \ref{sec:deligne-sullivan}, for the full normal crossing divisor case.

Both statements are formulated as pure existence results, but the proofs
produce a \emph{congruence} covering, and its degree is therefore bounded by
the order of a finite matrix group.  The purpose of this subsection is to make
that bound explicit, and---equally important---to delimit precisely what it
does and does not control.  The reader interested only in
Theorem~\ref{thm:main} may skip the subsection entirely; nothing below is used
elsewhere in the paper.

\subsubsection{Arithmetic models}

Throughout, $A$ denotes a subring of $\cc$ which is of finite type over $\zz$.
Recall that such a ring is a Jacobson ring, so that $A/\mathfrak{q}$ is a
\emph{finite} field for every maximal ideal $\mathfrak{q}\subset A$; we write
$N(\mathfrak{q}) := |A/\mathfrak{q}|$ and let $\mathrm{char}(\mathfrak{q})$
denote its characteristic.

\begin{definition}\label{def:arith-model}
Let $\zeta:\pi_1(X^*,x_0)\to GL_r(\cc)$ be a representation with unipotent
monodromy around each $D_i$, and let $X=\bigcup_I U_I$ be an adapted covering
as in Proposition~\ref{prop:covering}, with $U_I^* := U_I\cap X^*$.  An
\emph{arithmetic model} for $\zeta$ (relative to this covering) consists of:
\begin{enumerate}[leftmargin=2em,label={\rm(\alph*)}]
  \item a subring $A\subset\cc$ of finite type over $\zz$ and a conjugate of
    $\zeta$ with image in $GL_r(A)$, giving a local system $V_A$ of free
    $A$-modules of rank $r$ on $X^*$;
  \item for each $I$, a filtration $W(I)$ of $V_A|_{U_I^*}$ by sub-local
    systems of $A$-modules, such that every graded quotient
    $\Gr^j_{W(I)}$ is a local system of \emph{free} $A$-modules, of rank
    $n_j(I)$;
  \item for each $I$ and each $j$, a local system $V^j_A(I)$ of free
    $A$-modules on $U_I$ whose restriction to $U_I^*$ is $\Gr^j_{W(I)}$.
\end{enumerate}
\end{definition}

Condition (c) is the graded-extendability of $W(I)$, which for unipotent
local monodromy is guaranteed by the constructions of
Section~\ref{sec:nisotropic}; condition (b) is a normalization which we are
free to impose after shrinking $\Spec A$.

\begin{lemma}\label{lem:arith-model-exists}
Suppose the monodromy matrices of $\zeta$ have entries algebraic over $\qq$,
and fix a strict collection of patching data $(W(I),\tau(I))$ as produced by
Corollary~\ref{cor:cech-exists}.  Then $\zeta$ admits an arithmetic model
whose filtrations are the given $W(I)$.
\end{lemma}

\begin{proof}
Since $\pi_1(X^*)$ is finitely generated, the entries of the matrices
$\zeta(\gamma)$ for $\gamma$ in a finite generating set, together with the
entries of a finite set of transition and splitting matrices describing the
$W(I)$ and the extensions across $D$, generate a subring $A_0\subset\cc$ of
finite type over $\zz$; all the data of
Definition~\ref{def:arith-model}(a)--(c) are then defined over $A_0$, except
that the modules occurring need not be free.  Each of them is a finitely
generated $A_0$-module which becomes free after tensoring with the fraction
field, so by generic freeness there is a nonzero $f\in A_0$ such that all of
them---finitely many in number---are free over $A:=A_0[1/f]$.  Replacing $A_0$
by $A$ gives the model.
\end{proof}

\subsubsection{The congruence covering and constancy of the reduced data}

\begin{definition}\label{def:congruence-cover}
Given an arithmetic model and maximal ideals
$\mathfrak{q}_1,\mathfrak{q}_2\subset A$, set
\[
  \Gamma(\mathfrak{q}_1,\mathfrak{q}_2)
  \;:=\;\ker\Big(\pi_1(X^*,x_0)\xrightarrow{\ \zeta\ }
  GL_r(A/\mathfrak{q}_1)\times GL_r(A/\mathfrak{q}_2)\Big),
\]
and let $\pi:\widetilde{X}^*\to X^*$ be the corresponding connected finite
\'{e}tale covering.  Let $\widetilde{X}$ be the normalization of $X$ in
$\widetilde{X}^*$, and $\widetilde{D}:=\pi^{-1}(D)_{\mathrm{red}}$.
\end{definition}

The following lemma is the technical point of this subsection: on
$\widetilde{X}^*$ one gets
 essentially automatically
the constancy not only of the
local system but also of the filtrations and of their extensions across the
divisor.  In \cite{DeligneLetter} this is arranged by passing to a further
covering; the lemma says that no further covering is needed once the graded
quotients are taken to be free.

\begin{lemma}\label{lem:constancy}
Let $\mathfrak{q}\subset A$ be any maximal ideal.  Then, on every connected
component of $\pi^{-1}(U_I^*)$ (resp.\ of $\pi^{-1}(U_I)$), the reductions
\[
  \pi^*V_A\otimes_A A/\mathfrak{q},\qquad
  \pi^*W(I)\otimes_A A/\mathfrak{q},\qquad
  \pi^*V^j_A(I)\otimes_A A/\mathfrak{q}
\]
are constant, for $\mathfrak{q}\in\{\mathfrak{q}_1,\mathfrak{q}_2\}$.
Moreover $\pi^*W(I)\otimes_A A/\mathfrak{q}$ is a filtration by sub-local
systems with locally free quotients, for \emph{every} maximal ideal
$\mathfrak{q}$.
\end{lemma}

\begin{proof}
\emph{Step 1: the reduced filtration remains strict.}  Fix $I$ and write
$W_i := W_i(I)$.  By Definition~\ref{def:arith-model}(b) every graded quotient
of $W$ is free, hence each quotient $V_A/W_i$ is free (it is an iterated
extension of free modules, so free by induction on the length of the
filtration).  Therefore
$\mathrm{Tor}^A_1(V_A/W_i,\,A/\mathfrak{q})=0$, and the sequence
\[
  0\to W_i\otimes A/\mathfrak{q}\to V_A\otimes A/\mathfrak{q}
   \to (V_A/W_i)\otimes A/\mathfrak{q}\to 0
\]
is exact with free outer terms.  Thus $W\otimes A/\mathfrak{q}$ is again a
filtration by sub-local systems of free $A/\mathfrak{q}$-modules.  No
hypothesis on $\mathfrak{q}$ enters here.

\emph{Step 2: constancy of the reduced filtration.}  By construction of
$\Gamma(\mathfrak{q}_1,\mathfrak{q}_2)$, the monodromy of
$\pi^*V_A\otimes A/\mathfrak{q}_s$ is trivial for $s=1,2$, so this is the
constant local system on $\widetilde{X}^*$; fix the resulting trivialization.
A sub-local system of a constant local system on a connected space corresponds
to a monodromy-invariant submodule of the fibre, and, the monodromy being
trivial, the submodule attached to any two points agrees under the chosen
trivialization.  Applying this on a connected component of
$\pi^{-1}(U_I^*)$ to each step of $\pi^*W(I)\otimes A/\mathfrak{q}_s$---which
is a sub-local system by Step 1---gives the constancy.

\emph{Step 3: constancy of the reduced extensions.}  The open set $U_I$ is
(biholomorphic to) a polydisc and $U_I^*\subset U_I$ is the complement of a
union of coordinate hyperplanes, so $\pi_1(U_I^*)\to\pi_1(U_I)$ is surjective;
the same holds after pullback, on each connected component.  The local system
$V^j_A(I)$ is, by uniqueness of extensions across $D$, determined by a
representation of $\pi_1(U_I)$ whose restriction along this surjection is the
monodromy of $\Gr^j_{W(I)}$.  By Steps 1 and 2 the latter is trivial mod
$\mathfrak{q}_s$; a representation whose restriction to a subgroup surjecting
onto the whole group is trivial is itself trivial.  Hence
$\pi^*V^j_A(I)\otimes A/\mathfrak{q}_s$ is constant.
\end{proof}

\subsubsection{The degree bound}

\begin{theorem}\label{thm:effective-DS}
Let $\zeta$ be as above, with an arithmetic model over $A$, and let
$\mathfrak{q}_1,\mathfrak{q}_2\subset A$ be maximal ideals with
\emph{distinct} residue characteristics.  Assume $D$ is smooth and irreducible
{\rm(}so that we are in the situation of \cite[\S3]{IS-arXiv}{\rm)}.  Then the
canonical extension of $\pi^*\zeta$ to $\widetilde{X}$ is trivial as a
$\calC^\infty$-bundle, and the degree of $\pi$ satisfies
\[
  \deg\pi \;\Big|\;
  \big|GL_r(A/\mathfrak{q}_1)\big|\cdot\big|GL_r(A/\mathfrak{q}_2)\big|
  \;=\;\prod_{s=1,2} N_s^{\binom{r}{2}}\prod_{i=1}^{r}\big(N_s^{\,i}-1\big),
\]
where $N_s := N(\mathfrak{q}_s)$.  In particular
$\deg\pi < (N_1N_2)^{r^2}$.
\end{theorem}

\begin{proof}
The divisibility is immediate: $\deg\pi$ is the index of
$\Gamma(\mathfrak{q}_1,\mathfrak{q}_2)$ in $\pi_1(X^*,x_0)$, which is the
order of the image of $\zeta$ in
$GL_r(A/\mathfrak{q}_1)\times GL_r(A/\mathfrak{q}_2)$ and hence divides the
order of that group; the displayed product is the standard count of
$|GL_r(\mathbb{F}_q)|$.

For the triviality we invoke \cite[Lemma~3.3]{IS-arXiv}.  The input required
there is exactly a datum $(V_A,\,W,\,V^j_A)$ as in
Definition~\ref{def:arith-model} which is \emph{constant modulo
$\mathfrak{q}_1$ and modulo $\mathfrak{q}_2$}; this is supplied by
Lemma~\ref{lem:constancy}.  Granting it, the argument runs as follows.  One
interpolates $(V_A, W, \bigoplus_j V^j_A)$ over $(U_1\cap B_1)\times\aaa^1$ by
the coherent subsheaf $\sum_j t^{\,j}W_j\subset A[t]\otimes V_A$, obtaining a
topological model $\widetilde{\Vv}$ of the canonical extension, and considers
the classifying map $f:\widetilde{X}\to\mathrm{Grass}(r,\cc^{r+N})$ with
$N\geq \tfrac{1}{2}\dim_\rr\widetilde{X}$.  Because the Grassmannian is simply
connected, Sullivan's Hasse principle \cite{SullivanAdams} reduces the
homotopy triviality of $f$ composed with the projection to the $d$-th
coskeleton to the corresponding statement for each $\ell$-adic completion.
Given $\ell$, choose $s\in\{1,2\}$ with
$\mathrm{char}(\mathfrak{q}_s)\neq\ell$: the data being constant mod
$\mathfrak{q}_s$, the lemma of \cite{DeSu} applies and gives the vanishing of
$f_{\hat{\ell}}$.  This is the only place where two ideals of distinct residue
characteristic are used.
\end{proof}

\begin{corollary}\label{cor:effective-betti}
In the situation of Theorem~\ref{thm:effective-DS}, suppose in addition that
$\widetilde{X}$ is smooth with $\widetilde{D}$ a normal crossings divisor.
Then, with $d:=\deg\pi$, the Betti Chern classes of the canonical extension
satisfy
\[
  d\cdot c^B_p(F) \;=\; 0 \quad\text{in } H^{2p}(X,\zz),\qquad p\geq 1 .
\]
In particular $\mathrm{ord}\big(c^B_p(F)\big)$ divides
$|GL_r(A/\mathfrak{q}_1)|\cdot|GL_r(A/\mathfrak{q}_2)|$.
\end{corollary}

\begin{proof}
First, $\pi^*F$ \emph{is} the canonical extension of $\pi^*\zeta$.  Indeed
$\pi^*\nabla$ is a logarithmic connection on $\pi^*F$ along $\widetilde{D}$
whose residue along a component of $\pi^{-1}(D_i)$ equals $e\cdot
\mathrm{Res}_{D_i}(\nabla)$, where $e$ is the ramification index of $\pi$ there;
scaling by $e$ preserves nilpotence, and the extension with nilpotent residues
is unique.  Hence $\pi^*F$ is $\calC^\infty$-trivial by
Theorem~\ref{thm:effective-DS}, so $c^B_p(\pi^*F)=0$.

Second, $\pi:\widetilde{X}\to X$ is a finite surjective morphism of degree $d$
between smooth projective varieties, so the Gysin transfer $\pi_*$ satisfies
$\pi_*\pi^* = d\cdot\mathrm{id}$ on $H^*(X,\zz)$.  Therefore
$d\cdot c^B_p(F) = \pi_*\pi^* c^B_p(F) = \pi_*c^B_p(\pi^*F) = 0$.
\end{proof}

\begin{corollary}\label{cor:effective-numberfield}
Suppose $\zeta$ is defined over a number field $K$ with $n:=[K:\qq]$, and that
the arithmetic model can be taken with $A=\Oo_K[1/M]$ for an integer $M\geq1$.
Let $\ell_1<\ell_2$ be the two smallest rational primes not dividing $M$, and
choose primes $\mathfrak{p}_s\subset\Oo_K$ above $\ell_s$.  Then
\[
  \deg\pi \;<\;(\ell_1\ell_2)^{\,n r^2},
\]
and the same bound holds for $\mathrm{ord}\big(c^B_p(F)\big)$.
\end{corollary}

\begin{proof}
$N_s = N(\mathfrak{p}_s)\leq \ell_s^{\,n}$, and
Theorem~\ref{thm:effective-DS} gives
$\deg\pi<(N_1N_2)^{r^2}\leq(\ell_1\ell_2)^{nr^2}$.
\end{proof}

\begin{example}
Take $r=2$, $K=\qq$, $M=1$, so $A=\zz$ and one may take
$\mathfrak{q}_1=(2)$, $\mathfrak{q}_2=(3)$.  Then
$|GL_2(\mathbb{F}_2)|\cdot|GL_2(\mathbb{F}_3)| = 6\cdot 48 = 288$.  If the
representation is unimodular---as one may assume after the $r$-fold covering
used in \cite[Thm.~9.1]{IS-arXiv}---the bound improves to
$|SL_2(\mathbb{F}_2)|\cdot|SL_2(\mathbb{F}_3)| = 6\cdot 24 = 144$.
\end{example}

\begin{remark}[The bound is necessarily arithmetic]\label{rem:no-topological-bound}
No bound on $\deg\pi$ can depend on $r$ and $\dim X$ alone.  Let
$L(m) = S^{2k+1}/\mu_m$ be a lens space and $\chi$ a faithful character of
$\mu_m$.  The associated flat line bundle has $c_1$ of exact order $m$ in
$H^2(L(m),\zz)\cong\zz/m$, and every connected covering of $L(m)$ is some
$L(m')$ with $m'\mid m$, on which the pulled-back class still has order $m'$.
So a covering of degree exactly $m$ is required, with $r=1$ and $\dim$ fixed.
The dependence of Theorem~\ref{thm:effective-DS} on the residue fields
$A/\mathfrak{q}_s$ is thus not an artefact of the proof.
\end{remark}

\begin{remark}[Betti versus Chern-Simons]\label{rem:betti-vs-CS}
Corollary~\ref{cor:effective-betti} bounds the order of the \emph{integral
Betti} Chern classes of $F$.  It does \emph{not} bound the order of
$\CS_p(\nabla^{\Del})$, and we wish to be explicit about why, since the two
are easily conflated.

The connecting map of the coefficient sequence
$\zz\to\cc\to\cc/\zz$ sits in
\[
  H^{2p-1}(X,\cc)\longrightarrow H^{2p-1}(X,\cc/\zz)
  \xrightarrow{\ \partial\ } H^{2p}(X,\zz),
\]
and $\partial\,\CS_p(\nabla^{\Del}) = c^B_p(F)$ up to the usual normalization.
Knowing that $c^B_p(F)$ is killed by $d$ therefore only says that
$d\cdot\CS_p(\nabla^{\Del})$ lies in the image of $H^{2p-1}(X,\cc)$.  That
image is the divisible group
$H^{2p-1}(X,\cc)/\mathrm{im}\,H^{2p-1}(X,\zz)$, which contains a copy of
$(\qq/\zz)^{b_{2p-1}}$ as its torsion subgroup and is very far from being
torsion-free.  Consequently \emph{no} conclusion about the finiteness, let
alone the order, of $\CS_p(\nabla^{\Del})$ follows.  The same caution applies
to any argument which attempts to bound a class in $H^{2p-1}(X,\rr/\zz)$ by
the order of its image under $\partial$: the kernel of $\partial$ is the
divisible group $\mathrm{im}\,H^{2p-1}(X,\rr)$.

This is not a defect of the present estimate but a reflection of the
difficulty of the theorem.  Torsion of $c^B_p(F)$ is elementary
{\rm(}Chern-Weil, or Grothendieck \cite{Grothendieck}{\rm)}; torsion of
$\CS_p$ is the content of Reznikov's theorem and of
Theorem~\ref{thm:main}, and the proof passes through Borel's computation of
$H^*(BSL(K)^+,\rr)$ and the vanishing of the volume regulators, which produces
only the statement that the classifying map
$r_{\zeta_0}:X\to BGL^+(K)$ of Construction~\ref{constr:rL} is zero on
\emph{real} cohomology (Step~C of the proof of Theorem~\ref{thm:main}), hence on
rational homology.  Extracting an order from that vanishing would require
effective control of the torsion in $H_{2p-1}(BGL^+(\Oo_K),\zz)$, i.e.\ of
$K_{2p-1}(\Oo_K)_{\mathrm{tors}}$, together with the exponent of
$(r_{\zeta_0})_*$ on $H_{2p-1}(X,\zz)_{\mathrm{tors}}$.  See
Problem~\ref{prob:effective} below.
\end{remark}

\subsubsection{The normal crossings case}
\label{subsubsec:DS-nc}

In the case of a smooth divisor the topological input is
\cite[Lemma~3.3]{IS-arXiv}, formulated for a decomposition $X=U\cup B$ into
two pieces and interpolating the filtration by a single parameter
$t\in\aaa^1$.  For a normal crossings divisor one needs the corresponding
statement over the adapted covering $X=\bigcup_I U_I$, with the
multi-parameter interpolation supplied by the multi-Rees construction of
Section~\ref{sec:multrees}.

That input is exactly what Appendix~\ref{sec:deligne-sullivan} provides.
There, Lemma~\ref{lem:DS-rees-model} identifies the global Rees bundle
$F_K|_{\iiii\Cube A}$ with the canonical extension as a
$\mathcal{C}^\infty$-bundle, and exhibits it as the base change of a locally
free sheaf over $\Spec A_0$; Lemma~\ref{lem:DS-hasse} then runs Sullivan's
Hasse principle on this model.  Combining that appendix with the constancy
lemma of this section, the degree bound extends verbatim to the normal
crossings setting.

We first record that the lattice rectification step of
Lemma~\ref{lem:DS-finite-cover}(ii) can be performed with a single
ramification index, uniform over all strata.

\begin{lemma}\label{lem:rectification-exponent}
Let $\rho:\pi_1(X^*)\to GL_r(A_0)$ be as in Lemma~\ref{lem:DS-arith}, let
$\mathfrak{q}_1,\mathfrak{q}_2\subset A_0$ be maximal ideals, and let
\[
  G_0 \;:=\; \mathrm{im}\Big(\pi_1(X^*)\xrightarrow{\ \rho\ }
  GL_r(A_0/\mathfrak{q}_1)\times GL_r(A_0/\mathfrak{q}_2)\Big),
  \qquad \Gamma:=\ker,
\]
and let $e:=\exp(G_0)$ be the exponent of the finite group $G_0$.  For a
non-empty $I$ let $\zz^{|I|}=\langle\gamma_i\rangle_{i\in I}\subset\pi_1(U_I^*)$
be the central meridian lattice at the stratum $D_I$, and let
$\Lambda_I\subset\zz^{|I|}$ be the preimage of $\Gamma$.  Then
\[
  e\cdot\zz^{|I|}\;\subset\;\Lambda_I \qquad\text{for every }I .
\]
Consequently one may take $n_i=e$ for all $i$ in
Lemma~\ref{lem:DS-finite-cover}(ii).
\end{lemma}

\begin{proof}
The composite $\zz^{|I|}\to\pi_1(X^*)\to G_0$ has kernel $\Lambda_I$, so
$\zz^{|I|}/\Lambda_I$ is isomorphic to a subgroup of $G_0$ and therefore has
exponent dividing $e$.  Hence $e\cdot\zz^{|I|}\subset\Lambda_I$.  The
rectangular sublattice $e\cdot\zz^{|I|}$ is the one used in the Kummer cover of
the proof of Lemma~\ref{lem:DS-finite-cover}(ii), and it is independent of
$I$.
\end{proof}

\begin{theorem}[Effective Deligne--Sullivan, normal crossings case]
\label{thm:effective-DS-nc}
Let $D=D_1+\cdots+D_k$ be a simple normal crossings divisor, let $\zeta$ have
unipotent monodromy around each $D_i$, and let $A_0$ be an arithmetic model as
in Lemma~\ref{lem:DS-arith}.  Choose maximal ideals
$\mathfrak{q}_1,\mathfrak{q}_2\subset A_0$ of distinct residue
characteristics and set
\[
  G \;:=\; \big|GL_r(A_0/\mathfrak{q}_1)\big|\cdot
           \big|GL_r(A_0/\mathfrak{q}_2)\big|
      \;=\;\prod_{s=1,2} N_s^{\binom{r}{2}}\prod_{i=1}^{r}\big(N_s^{\,i}-1\big),
  \qquad N_s:=|A_0/\mathfrak{q}_s| .
\]
Let $\pi':\widetilde{X}'^{\,*}\to X^*$ be the congruence covering attached to
$\Gamma=\ker(\rho\bmod\mathfrak{q}_1,\mathfrak{q}_2)$, and let
$\widetilde{X}'$ be the normalization of $X$ in $\widetilde{X}'^{\,*}$.  Then:
\begin{enumerate}[leftmargin=2em,label={\rm(\alph*)}]
  \item $\deg\pi'$ divides $G$;
  \item the canonical extension of $\pi'^*\zeta$ to $\widetilde{X}'$ is
    trivial as a $\mathcal{C}^\infty$-bundle;
  \item if in addition one requires the covering to be smooth, i.e.\ passes to
    the covering $\pi:\widetilde{X}\to X$ of Lemma~\ref{lem:DS-finite-cover}
    obtained by rectifying the meridian lattices, then
    $\deg\pi$ divides $G\cdot e^{\,k}$, where $e=\exp(G_0)$ divides $G$; in
    particular $\deg\pi$ divides $G^{\,k+1}$.
\end{enumerate}
\end{theorem}

\begin{proof}
(a) $\deg\pi' = [\pi_1(X^*):\Gamma] = |G_0|$, which divides $G$.

(b) This is Lemma~\ref{lem:DS-hasse}, whose hypothesis is
Lemma~\ref{lem:DS-finite-cover}(i): the monodromy $\rho$, all the graded
representations $\Gr^{W(N_I)}\!\big(\rho|_{\pi_1(U_I^*)}\big)$, and all
morphisms of the functor $A\to\Xi_K$ of
Construction~\ref{constr:A-to-Xi}, become trivial mod $\mathfrak{q}_1$ and
mod $\mathfrak{q}_2$ after pullback.  We claim that on $\widetilde{X}'^{\,*}$
this holds already, with no further covering, so that
Lemma~\ref{lem:DS-hasse} applies to $\pi'$.

Indeed, by Lemma~\ref{lem:DS-arith}(ii)--(iii) every step of $W(N_I)$ is a
free direct summand of $V_{A_0}$ with free graded quotients, so reduction mod
$\mathfrak{q}$ is exact on the filtration and $\Gr^{W(N_I)}(V_{A_0})\otimes
A_0/\mathfrak{q}$ is a subquotient of $V_{A_0}\otimes A_0/\mathfrak{q}$ as a
representation of $\pi_1(U_I^*)$.  Since $\pi'^*\big(V_{A_0}\otimes
A_0/\mathfrak{q}_s\big)$ has trivial monodromy by construction of $\Gamma$,
so does every subquotient; this is Step~1 and Step~2 of the proof of
Lemma~\ref{lem:constancy}, applied on each $U_I^*$.  The morphisms
$\Gr^{W(N_I\leq N_J)}(V(I)_{A_0})\cong V(J)_{A_0}$ of
Lemma~\ref{lem:DS-arith}(iv) are $A_0$-linear maps between these subquotients
and carry no further monodromy.  Finally, the extension across $D$ is handled
by Step~3 of Lemma~\ref{lem:constancy}: $\pi_1(U_I^*)\twoheadrightarrow
\pi_1(U_I)$ because $U_I$ is a polydisc and $U_I^*$ the complement of
coordinate hyperplanes, and a representation of $\pi_1(U_I)$ whose restriction
along a surjection is trivial mod $\mathfrak{q}_s$ is itself trivial mod
$\mathfrak{q}_s$.  Hence the intersection of subgroups taken in the proof of
Lemma~\ref{lem:DS-finite-cover}(i) is redundant: $\Gamma$ already trivialises
all the data.

(c) By Lemma~\ref{lem:rectification-exponent} the rectification may be carried
out with the single index $n_i=e$ for all $i$.  The residual covering
$\widetilde{X}\to\widetilde{X}'$ is then, in the local model at each stratum,
the Kummer cover with lattice $e\zz^{|I|}\subset\Lambda_I$, i.e.\ a product of
$e$-fold cyclic covers in the $|I|$ normal directions; its degree therefore
divides $e^{\,k}$.  Since $e=\exp(G_0)$ divides $|G_0|$, which divides $G$, we
get $\deg\pi \mid G\cdot e^k \mid G^{k+1}$.
\end{proof}

\begin{corollary}\label{cor:effective-betti-nc}
With $G$ as in Theorem~\ref{thm:effective-DS-nc}, the Betti Chern classes of
the canonical extension satisfy
\[
  \mathrm{ord}\big(c^B_p(F)\big) \;\Big|\; G ,\qquad p\geq 1 .
\]
\end{corollary}

\begin{proof}
As in Corollary~\ref{cor:effective-betti}, $\pi'^*F$ is the canonical
extension of $\pi'^*\zeta$: the residue of $\pi'^*\nabla^{\Del}$ along a
component of $\pi'^{-1}(D_i)$ is $e_i\cdot\mathrm{Res}_{D_i}(\nabla^{\Del})$
for the local ramification index $e_i$, scaling preserves nilpotence, and the
extension with nilpotent residues is unique.  So $c^B_p(\pi'^*F)=0$ by
Theorem~\ref{thm:effective-DS-nc}(b).  The map
$\pi':\widetilde{X}'\to X$ is a finite surjective map of degree $d'=\deg\pi'$
from a normal projective variety onto a smooth one, hence a finite branched
covering of compact ANRs; the associated transfer satisfies
$\pi'_*\pi'^*=d'\cdot\mathrm{id}$ on $H^*(X,\zz)$, so
$d'\cdot c^B_p(F)=0$.  As $d'\mid G$ by part~(a), the order divides $G$.
\end{proof}

Note that part (c) of Theorem~\ref{thm:effective-DS-nc} is \emph{not} needed
for Corollary~\ref{cor:effective-betti-nc}: smoothness of the covering is
required for the statement of Proposition~\ref{prop:DS-NCD}(a), but the
transfer argument only needs $\pi'$ finite and surjective.  Consequently the
bound $G$ on the order of the Betti Chern classes is the same in the normal
crossings case as in the smooth divisor case of
Theorem~\ref{thm:effective-DS}: the number $k$ of components enters only
through the smoothing of the covering.

\begin{corollary}\label{cor:effective-nc-numberfield}
Suppose $\zeta$ is defined over a number field $K$ with $n=[K:\qq]$ and that
the arithmetic model may be taken with $A_0=\Oo_K[1/M]$.  Let
$\ell_1<\ell_2$ be the two smallest rational primes not dividing
$M\cdot(r-1)!$, and choose $\mathfrak{p}_s\mid \ell_s$.  Then
\[
  \mathrm{ord}\big(c^B_p(F)\big) \;\leq\; G \;<\;(\ell_1\ell_2)^{\,n r^2},
  \qquad
  \deg\pi \;<\;(\ell_1\ell_2)^{\,n r^2 (k+1)} ,
\]
the latter for the smooth covering of
Theorem~\ref{thm:effective-DS-nc}(c).
\end{corollary}

\begin{proof}
As in Corollary~\ref{cor:effective-numberfield}, $N_s\leq\ell_s^{\,n}$ and
$G<(N_1N_2)^{r^2}$.  The factor $(r-1)!$ is inverted in
Lemma~\ref{lem:DS-arith} in order to form the logarithms $N_i=\log T_i$, so
the primes dividing it must be avoided.
\end{proof}

\begin{remark}
\label{rem:DS-nc}

\emph{(i) The smoothing cost.}  The factor $e^k$ in
Theorem~\ref{thm:effective-DS-nc}(c) is the price of insisting that
$\widetilde{X}$ be smooth, and it is the analogue in the present setting of
the Kawamata covering.  It is presumably far from sharp: the local model
requires only that the meridian lattice be rectangular, and $e\zz^{|I|}$ is
the crudest rectangular sublattice of $\Lambda_I$.  Any $\Lambda_I$ containing
a rectangular sublattice of smaller index will do, and the resulting index can
be as small as $[\zz^{|I|}:\Lambda_I]$ when $\Lambda_I$ happens already to be
rectangular.

\emph{(ii) Relation to Problem~\ref{prob:effective}.}  What
Theorem~\ref{thm:effective-DS-nc} contributes to
Problem~\ref{prob:effective} is therefore the following precise statement: the
\emph{topological} half of the torsion theorem is effective, with an explicit
bound depending only on $r$, on the residue fields $A_0/\mathfrak{q}_s$, and
(for the smooth model) on $k$; the \emph{regulator} half is not, and the
obstruction is the one described in Remark~\ref{rem:betti-vs-CS}.
\end{remark}

\section{Quasi-unipotent local monodromy: locally abelian parabolic bundles}
\label{sec:parabolic}

Theorem~\ref{thm:main} assumes unipotent local monodromy, and it is natural to
ask what survives when the monodromy is merely quasi-unipotent.  
The naive
approach would be to pass to a covering on which the monodromy becomes unipotent, to apply
Theorem~\ref{thm:main} there, and then to descend. This however 
runs into an obstruction which is
worth stating precisely, because it is the reason this section is needed.

Let $L$ have quasi-unipotent local monodromy, with $T_i^{m_i}$ unipotent, let $F$ be
its canonical extension (residue eigenvalues in $[0,1)$), and let
$\pi:X'\to X$ be a covering ramified to order $e_i$ along $D_i$.  If $\alpha$
is a residue eigenvalue of $\nabla$ along $D_i$, then $e_i\alpha$ is a residue
eigenvalue of $\pi^*\nabla$; when $m_i\mid e_i$ this is an integer, but it is
not zero.  So $\pi^*F$ is \emph{not} the canonical extension of $\pi^*L$: the
two differ by an elementary modification along $\pi^{-1}(D_i)$ determined by
the integers $e_i\alpha$.  The identity $\pi^*\CS_p(L)=\CS_p(\pi^*L)$ therefore
fails, and with it the descent argument.

The resolution is to replace the single bundle $F$ by the \emph{parabolic}
bundle it belongs to.  The parabolic pullback along a Kawamata covering is
exactly the elementary modification just described, so the obstruction
disappears---not by being circumvented, but because the parabolic formalism
performs it automatically.  This is the setting of \cite{IS-MA}, and the
relevant condition on the parabolic structure is that it be
\emph{locally abelian}.

\subsection{Locally abelian parabolic bundles}
\label{subsec:par-defs}

We recall the definitions from \cite[\S2]{IS-MA}.

\begin{definition}\label{def:parabolic}
A \emph{parabolic bundle} on $(X,D)$ is a collection $E_\hdot =
\{E_{\vec\alpha}\}_{\vec\alpha\in\rr^k}$ of locally free sheaves on $X$
together with inclusions $E_{\vec\alpha}\subset E_{\vec\beta}$ for
$\vec\alpha\leq\vec\beta$, such that
\begin{enumerate}[leftmargin=2em,label=\rm(\arabic*)]
  \item \emph{(normalization)} $E_{\vec\alpha+\vec\delta_i}=
    E_{\vec\alpha}(D_i)$, where $\vec\delta_i$ is the $i$-th standard basis
    vector;
  \item \emph{(semicontinuity)} for each $\vec\alpha$ there is
    $\epsilon>0$ with $E_{\vec\alpha+\vec\epsilon}=E_{\vec\alpha}$ for
    $0\leq\vec\epsilon\leq(\epsilon,\dots,\epsilon)$.
\end{enumerate}
The \emph{weights} of $E_\hdot$ along $D_i$ are the values of $\alpha_i$ at
which $E_{\vec\alpha}$ jumps; they form a finite subset of $\rr/\zz$.  A
\emph{parabolic line bundle} is one of rank one; the basic examples are
$\Oo_X(-\vec\alpha D):=\Oo_X\langle\vec\alpha\rangle$, with a single weight
$\alpha_i$ along each $D_i$.
\end{definition}

\begin{definition}\label{def:locally-abelian}
A parabolic bundle $E_\hdot$ is \emph{locally abelian} if every point $x\in X$
has a Zariski-open neighbourhood $U$ on which $E_\hdot|_U$ is isomorphic, as a
parabolic bundle, to a direct sum of parabolic line bundles.
\end{definition}

The point of the condition is that all the standard operations, 
including the tensor
product, dual, pullback along a ramified covering, and the parabolic Chern
character, 
are well behaved for locally abelian bundles and badly behaved in
general; see \cite[\S3]{IS-MA}.  We shall use the following two consequences.

\begin{theorem}[{\cite[\S\S3--4]{IS-MA}}]\label{thm:par-chern}
There is a parabolic Chern character $\ch^{\mathrm{par}}_p$, defined on
locally abelian parabolic bundles on $(X,D)$ with rational weights and taking
values in $H^{2p}(X,\qq)$, which is additive, multiplicative, and functorial
for pullback along Kawamata coverings in the sense of
Proposition~\ref{prop:par-pullback} below.
\end{theorem}

\subsection{The parabolic bundle attached to a quasi-unipotent local system}
\label{subsec:par-from-L}

\begin{construction}\label{constr:par-from-L}
Let $L$ be a local system on $X^*$ with quasi-unipotent local monodromy, and
let $(\Ee,\nabla)$ be its meromorphic flat bundle on $(X,D)$.  For
$\vec\alpha\in\rr^k$ let
\[
  \begin{gathered}
  E_{\vec\alpha}\;:=\;\text{the extension of $L\otimes\Oo_{X^*}$ to $X$ with}\\
  \mathrm{Res}_{D_i}(\nabla)\ \text{having eigenvalues in }
  [-\alpha_i,\,-\alpha_i+1).
  \end{gathered}
\]
This is Deligne's construction applied with a shifted normalization interval.
The collection $E_\hdot=\{E_{\vec\alpha}\}$ is a parabolic bundle in the sense
of Definition~\ref{def:parabolic}, with $E_{\vec 0}=F$ the canonical
extension, and its weights along $D_i$ are the numbers $\alpha\in[0,1)$ with
$e^{-2\pi i\alpha}$ an eigenvalue of $T_i$.  Since $T_i^{m_i}$ is unipotent,
these lie in $\tfrac{1}{m_i}\zz$; in particular the weights are rational.
\end{construction}

\begin{proposition}\label{prop:par-loc-abelian}
The parabolic bundle $E_\hdot$ of Construction~\ref{constr:par-from-L} is
locally abelian.
\end{proposition}

\begin{proof}
The question is local, so let $U$ be a coordinate polydisc with
$D\cap U=\{z_1\cdots z_a=0\}$ and let $V=L_x$ be the fibre at a base-point of
$U^*$.  The local monodromies $T_1,\dots,T_a$ commute, hence so do their
semisimple parts $T_i^{s}$, and $V$ decomposes as a direct sum of joint
eigenspaces
\[
  V\;=\;\bigoplus_{\vec\lambda} V_{\vec\lambda},
  \qquad
  T^s_i|_{V_{\vec\lambda}}=\lambda_i\cdot\mathrm{id},
\]
the sum being over the finitely many joint eigenvalue vectors
$\vec\lambda=(\lambda_1,\dots,\lambda_a)$ of roots of unity.  This
decomposition is preserved by every $T_i$, hence is a decomposition of local
systems on $U^*$, and $L|_{U^*}=\bigoplus_{\vec\lambda}L_{\vec\lambda}$ with
each $L_{\vec\lambda}$ of unipotent monodromy after twisting by the character
$\vec\lambda$.

Write $\lambda_i=e^{-2\pi i\alpha_i(\vec\lambda)}$ with
$\alpha_i(\vec\lambda)\in[0,1)$.  Then $L_{\vec\lambda}$ is the tensor product
of a local system $M_{\vec\lambda}$ with unipotent monodromy and the rank-one
local system with monodromy $\lambda_i$ around $D_i$.  Passing to extensions,
\[
  E_\hdot|_U \;\cong\; \bigoplus_{\vec\lambda}
  \Bigl(M_{\vec\lambda}\otimes\Oo_U\Bigr)\otimes
  \Oo_U\langle\vec\alpha(\vec\lambda)\rangle ,
\]
where $M_{\vec\lambda}\otimes\Oo_U$ carries the trivial parabolic structure
(all weights $0$, its residues being nilpotent) and
$\Oo_U\langle\vec\alpha(\vec\lambda)\rangle$ is the parabolic line bundle of
weights $\vec\alpha(\vec\lambda)$.  Shrinking $U$ so that
$M_{\vec\lambda}\otimes\Oo_U$ is free, each summand becomes a direct sum of
copies of $\Oo_U\langle\vec\alpha(\vec\lambda)\rangle$.  Hence $E_\hdot|_U$ is
a direct sum of parabolic line bundles.
\end{proof}

\subsection{Kawamata coverings and parabolic pullback}
\label{subsec:par-pullback}

\begin{definition}\label{def:kawamata}
A \emph{Kawamata covering adapted to $L$} is a finite surjective morphism
$\pi:X'\to X$ with $X'$ smooth projective, $D'=\pi^{-1}(D)_{\mathrm{red}}$ a
normal crossings divisor, and $\pi$ ramified along $D_i$ to a constant order
$e_i$ divisible by $m_i$, \'{e}tale over $X^*$.  Such coverings exist by
Kawamata's construction, and may be taken Galois.
\end{definition}

The parabolic pullback of a locally abelian bundle is defined summand by
summand on a local decomposition into parabolic line bundles, by
$\pi^{*\mathrm{par}}\Oo_X\langle\vec\alpha\rangle :=
\Oo_{X'}(-\sum_i e_i\alpha_iD'_i)$, and glued; the local decompositions may
differ, but the resulting sheaf does not, since the operation is canonical on
each $E_{\vec\alpha}$.  This is where the locally abelian hypothesis is used:
without it there is no such description and no functorial pullback. (See \cite{Biswas} also).

\begin{proposition}\label{prop:par-pullback}
Let $L$ be quasi-unipotent, $E_\hdot$ as in
Construction~\ref{constr:par-from-L}, and $\pi:X'\to X$ a Kawamata covering
adapted to $L$.  Then $\pi^*L$ has unipotent local monodromy, and
\[
  \pi^{*\mathrm{par}}E_\hdot \;=\; F',
\]
the Deligne canonical extension of $\pi^*L$ to $X'$, with nilpotent residues
along $D'$.  Moreover
$\pi^*\ch^{\mathrm{par}}_p(E_\hdot)=\ch_p(F')$ in $H^{2p}(X',\qq)$.
\end{proposition}

\begin{proof}
The local monodromy of $\pi^*L$ around a component of $D'$ over $D_i$ is
$T_i^{e_i}$, which is unipotent because $m_i\mid e_i$.  For the second
assertion, work in the local decomposition of
Proposition~\ref{prop:par-loc-abelian}.  On the summand indexed by
$\vec\lambda$ the parabolic structure is that of
$\Oo\langle\vec\alpha(\vec\lambda)\rangle$, and by definition
$\pi^{*\mathrm{par}}$ sends it to $\Oo_{X'}(-\sum_ie_i\alpha_i(\vec\lambda)D'_i)$,
an honest line bundle since $e_i\alpha_i(\vec\lambda)\in\zz$.  The residue of
the pulled-back connection on that summand is
$e_i\alpha_i(\vec\lambda)+e_iN_i$, and twisting by
$\Oo(-e_i\alpha_i(\vec\lambda)D'_i)$ shifts it by
$-e_i\alpha_i(\vec\lambda)$, leaving the nilpotent operator $e_iN_i$.  So
$\pi^{*\mathrm{par}}E_\hdot$ has nilpotent residues and restricts to
$\pi^*L\otimes\Oo$ on $X'^*$; by uniqueness of the extension with nilpotent
residues it is $F'$.  The last assertion is the functoriality of
$\ch^{\mathrm{par}}$ recorded in Theorem~\ref{thm:par-chern}.
\end{proof}

\subsection{The torsion theorem in the quasi-unipotent case}
\label{subsec:par-torsion}

\begin{definition}\label{def:par-CS}
Let $L$ be quasi-unipotent and $\pi:X'\to X$ a Kawamata covering adapted to
$L$.  Let $\nabla'^{\Del}$ be Deligne's patched connection
(Definition~\ref{def:deligne-patch}) on $F'=\pi^{*\mathrm{par}}E_\hdot$, and
set
\[
  \CS^{\mathrm{par}}_p(E_\hdot;\pi)\;:=\;
  \pi_*\,\CS_p\bigl(\nabla'^{\Del}\bigr)\;\in\;H^{2p-1}(X,\cc/\zz),
\]
where $\pi_*$ is the Umkehr map of the finite surjective morphism $\pi$
between smooth projective varieties.
\end{definition}

The class is defined on the covering and transferred, rather than directly on
$X$, because the connection on $E_{\vec0}$ obtained by subtracting the
\emph{full} residues $\mathrm{Res}_{D_i}(\nabla)=S_i+N_i$, rather than only
their nilpotent parts, is not locally nil-flat: on the local summands of
Proposition~\ref{prop:par-loc-abelian} its curvature acquires a scalar term
$\alpha_i(\vec\lambda)\,\dlog u_i$ from the semisimple part, so $\ch_p\neq0$ as
a form and the associated Cheeger-Simons character is not flat.  Consistently
with this, the primary invariant vanishes only rationally, by
Theorem~\ref{thm:quasi-unipotent}(3) below.

\begin{theorem}\label{thm:quasi-unipotent}
Let $L$ be a local system on $X^*$ with quasi-unipotent local monodromy, and
let $E_\hdot$ be the associated locally abelian parabolic bundle.  Then for
every $p\geq2$ and every Kawamata covering $\pi$ adapted to $L$:
\begin{enumerate}[leftmargin=2em,label=\rm(\arabic*)]
  \item $\CS_p(\nabla'^{\Del})\in H^{2p-1}(X',\cc/\zz)$ is torsion;
  \item $\CS^{\mathrm{par}}_p(E_\hdot;\pi)\in H^{2p-1}(X,\cc/\zz)$ is torsion;
  \item $\ch^{\mathrm{par}}_p(E_\hdot)=0$ in $H^{2p}(X,\qq)$ for all
    $p\geq1$, and the Betti Chern classes $c^B_p(E_{\vec\alpha})$ of every
    member of the parabolic family are torsion.
\end{enumerate}
No hypothesis beyond quasi-unipotency is required.
\end{theorem}

\begin{proof}
(1) By Proposition~\ref{prop:par-pullback}, $\pi^*L$ has unipotent local
monodromy on $X'^*=\pi^{-1}(X^*)$, $X'$ is smooth projective and $D'$ is a
normal crossings divisor, and $F'=\pi^{*\mathrm{par}}E_\hdot$ is the Deligne
canonical extension of $\pi^*L$.  This is precisely the situation of
Theorem~\ref{thm:main}, applied on $(X',D')$, which gives that
$\CS_p(\nabla'^{\Del})$ is torsion for $p\geq2$.

(2) $\pi_*$ is a homomorphism of abelian groups, so it carries torsion classes
to torsion classes.  Explicitly, if $N\cdot\CS_p(\nabla'^{\Del})=0$ then
$N\cdot\CS^{\mathrm{par}}_p(E_\hdot;\pi)=\pi_*\bigl(N\CS_p(\nabla'^{\Del})\bigr)=0$.

(3) The connection $\nabla'^{\Del}$ is locally nil-flat
(Lemma~\ref{lem:Del-nilflat}), so $\ch_p(\nabla'^{\Del})=0$ as a differential
form by Proposition~\ref{prop:nilflat-ch-zero}, whence $\ch_p(F')=0$ in
$H^{2p}(X',\qq)$.  By Proposition~\ref{prop:par-pullback},
$\pi^*\ch^{\mathrm{par}}_p(E_\hdot)=\ch_p(F')=0$; and $\pi^*$ is injective on
rational cohomology, since $\pi_*\pi^*=(\deg\pi)\cdot\mathrm{id}$ and
$\deg\pi\neq0$.  Hence $\ch^{\mathrm{par}}_p(E_\hdot)=0$.  For the last
assertion, each $E_{\vec\alpha}$ differs from $F$ by a twist along $D$, so its
rational Chern character is determined by $\ch^{\mathrm{par}}_p(E_\hdot)$ and
the classes of the $\Oo(D_i)$; the rational vanishing together with finite
generation of $H^{2p}(X,\zz)$ gives that $c^B_p(E_{\vec\alpha})$ is torsion.
\end{proof}

\begin{lemma}[Dependence on the covering]\label{lem:par-CS-scaling}
If $\pi':X''\to X$ factors as $\pi\circ\sigma$ with $\sigma:X''\to X'$ finite
of degree $d$, and both $\pi,\pi'$ are adapted to $L$, then
\[
  \CS^{\mathrm{par}}_p(E_\hdot;\pi') \;=\; d\cdot
  \CS^{\mathrm{par}}_p(E_\hdot;\pi).
\]
Consequently the normalized class
$(\deg\pi)^{-1}\CS^{\mathrm{par}}_p(E_\hdot;\pi)$ is independent of $\pi$ in
\linebreak[9]
$H^{2p-1}(X,\cc/\zz)\otimes\qq$.
\end{lemma}

\begin{proof}
By Proposition~\ref{prop:par-pullback} applied to $\sigma$, the canonical
extension of $\pi'^*L$ is $\sigma^{*}F'$, and by
Theorem~\ref{thm:Del-indep-app} its patched connection has
$\CS_p=\sigma^*\CS_p(\nabla'^{\Del})$.  Then
$\pi'_*\sigma^*\CS_p(\nabla'^{\Del})
=\pi_*\sigma_*\sigma^*\CS_p(\nabla'^{\Del})
=d\cdot\pi_*\CS_p(\nabla'^{\Del})$.
\end{proof}

\subsection{Realization on a Deligne-Mumford stack}
\label{subsec:par-stack}

Lemma~\ref{lem:par-CS-scaling} exhibits the dependences of
Definition~\ref{def:par-CS}: the class is attached to a pair $(E_\hdot,\pi)$
and scales with the covering, so it is an invariant of the parabolic bundle
only after normalization, hence only rationally.  The dependency is due to  presenting the parabolic bundle
through a covering.  There is a formulation in which no covering is chosen at
all, and we describe it here because it is the natural home for the
constructions of this section and because it is the form in which the
parabolic Reznikov theorem is proved in \cite{IS-parabolic}.

\subsubsection{The root stack}

Let $\vec n=(n_1,\dots,n_k)$ be positive integers.  The \emph{root stack}
$\Xx=X\langle\vec n\rangle$ is the Deligne-Mumford stack obtained from
$(X,D)$ by adjoining an $n_i$-th root of the equation of $D_i$: locally, where
$D_i=\{z_i=0\}$, one takes the quotient of $\{w_i^{n_i}=z_i\}$ by
$\mu_{n_i}$.  It comes with a coarse moduli map
\[
  q\;:\;\Xx\;\longrightarrow\;X,
\]
an isomorphism over $X^*$, and with reduced divisors
$\Dd_i\subset\Xx$ satisfying $q^*D_i=n_i\Dd_i$.  The stack $\Xx$ is smooth,
$\Dd=\sum_i\Dd_i$ is a normal crossings divisor on it, and $\Xx$ depends only
on $(X,D,\vec n)$: no covering, no Galois group, and no choices are involved.

\begin{theorem}[{\cite{Biswas}}, \cite{Borne}, {\cite[\S2]{IS-parabolic}}]
\label{thm:stack-correspondence}
Let $\vec n$ be such that $n_i$ clears the denominators of the weights of
$E_\hdot$ along $D_i$.  Then $E_\hdot\mapsto\Ee$ is an equivalence between the
category of locally abelian parabolic bundles on $(X,D)$ with weights in
$\bigoplus_i\tfrac{1}{n_i}\zz$ and the category of vector bundles on
$\Xx=X\langle\vec n\rangle$.  Under it,
\[
  E_{\vec\alpha}\;=\;q_*\Bigl(\Ee\otimes\Oo_\Xx\bigl(\textstyle\sum_i
  \lfloor n_i\alpha_i\rfloor\Dd_i\bigr)\Bigr),
\]
and the parabolic Chern character is the ordinary Chern character of $\Ee$
pushed to $X$,
\[
  \ch^{\mathrm{par}}_p(E_\hdot)\;=\;q_*\,\ch_p(\Ee)
  \;\in\;H^{2p}(X,\qq).
\]
\end{theorem}

Definition~\ref{def:locally-abelian} asserts  that locally the parabolic
structure is a direct sum of line bundle data, and a parabolic line bundle
with weight in $\tfrac{1}{n_i}\zz$ is precisely $q_*$ of a line bundle on the
root stack.  Proposition~\ref{prop:par-loc-abelian} is thus the verification
that the parabolic bundle attached to a quasi-unipotent local system is in the
image of this correspondence.

Two features distinguish the stack picture from the presentation by Kawamata
coverings.

\emph{(1) No choices.}   The root stack $X\langle\vec n\rangle$ is determined by $\vec n$
alone; and for $\vec n\mid\vec m$ there is a canonical map
$X\langle\vec m\rangle\to X\langle\vec n\rangle$ under which the bundles
correspond, so the object $\Ee$ is canonical up to this filtered system.  In
particular an invariant of $\Ee$ is an invariant of $E_\hdot$, with no
normalization and no rational denominators.

\emph{(2) The ordinary theory applies verbatim.}  $\Xx$ is a smooth
Deligne-Mumford stack with a normal crossings divisor $\Dd$, and $\Ee$ is an
ordinary vector bundle on it.  Everything used in this paper---the canonical
extension, the local monodromy weight filtrations, patching data,
$F^1$-connections, differential characters, Deligne-Beilinson cohomology---has
a direct analogue for such a stack, obtained from the smooth case by descent
along an \'{e}tale presentation.  

\subsubsection{The parabolic (extended) Reznikov theorem}

The primary consequence in this direction is due to the authors
\cite{IS-parabolic}, who compute the Chern character of a locally abelian
parabolic bundle explicitly in terms of its constituent bundles, and deduce
the following.

\begin{theorem}[{\cite{IS-parabolic}}]\label{thm:parabolic-reznikov}
Let $L$ be a local system on $X^*$ whose local monodromies at infinity are
semisimple of finite order, and let $E_\hdot$ be the associated locally
abelian parabolic bundle.  Then the parabolic Chern classes of $E_\hdot$
vanish in rational Deligne cohomology in degrees $\geq2$.
\end{theorem}

Theorem~\ref{thm:quasi-unipotent} gives an integral refinement of this on the
stack, for local monodromy that is quasi-unipotent rather than semisimple.

\begin{corollary}\label{cor:stack-deligne-torsion}
Let $L$ be a local system on $X^*$ with quasi-unipotent local monodromy, let
$E_\hdot$ be the associated locally abelian parabolic bundle, let $\vec n$
clear the denominators of its weights, and let $\Ee$ be the corresponding
vector bundle on $\Xx=X\langle\vec n\rangle$ under
Theorem~\ref{thm:stack-correspondence}.  Then the Deligne-Beilinson Chern
classes
\[
  \ch^{\calD}_p(\Ee)\;\in\;H^{2p}_{\calD}\bigl(\Xx,\zz(p)\bigr)
\]
are torsion for every $p\geq2$.
\end{corollary}

\begin{proof}
Choose a Kawamata covering $\pi:X'\to X$ adapted to $L$
(Definition~\ref{def:kawamata}) with ramification indices $e_i$ divisible by
$n_i$.  Then $\pi$ factors through the root stack,
\[
  \pi\;:\;X'\;\xrightarrow{\ f\ }\;\Xx\;\xrightarrow{\ q\ }\;X,
\]
with $f$ representable, finite, flat and surjective---an $n_i$-th root of the
equation of $D_i$ is supplied on $X'$ by $w_i^{e_i/n_i}$---and by
Theorem~\ref{thm:stack-correspondence} together with
Proposition~\ref{prop:par-pullback},
\[
  f^*\Ee\;=\;\pi^{*\mathrm{par}}E_\hdot\;=\;F',
\]
the canonical extension of $\pi^*L$ on $X'$.

Since $\pi^*L$ has unipotent local monodromy and $(X',D')$ is a smooth
projective variety with a normal crossings divisor,
Corollary~\ref{cor:deligne-torsion} applies on $X'$ and gives that
$\ch^\calD_p(F')$ is torsion in $H^{2p}_\calD(X',\zz(p))$ for $p\geq2$; say
$N\cdot\ch^\calD_p(F')=0$.  By functoriality of the Deligne-Beilinson Chern
character, $f^*\bigl(N\cdot\ch^\calD_p(\Ee)\bigr)=N\cdot\ch^\calD_p(F')=0$.
Finally $f$ is finite surjective of some degree $d$, so the transfer satisfies
$f_*f^*=d\cdot\mathrm{id}$ and $\ker f^*$ is killed by $d$; hence
$dN\cdot\ch^\calD_p(\Ee)=0$.
\end{proof}

Corollary~\ref{cor:stack-deligne-torsion} sharpens
Theorem~\ref{thm:quasi-unipotent}(3) in the following ways.  Relative to
Theorem~\ref{thm:parabolic-reznikov} it allows an arbitrary unipotent part in
the local monodromy, at the cost of passing through
Theorem~\ref{thm:main} on a covering; the two results are complementary, the
quasi-unipotent case in general being the combination of a semisimple part of
finite order handled by the parabolic structure and a unipotent part handled
by the weight filtrations of Section~\ref{sec:patching}.


\appendix

\section{Differential characters and Chern-Simons classes on manifolds}
\label{app:diffchar}

\subsection{Differential characters}

Let $X$ be a compact oriented smooth manifold.  We work with singular homology
and cohomology with $\zz$ coefficients.  Denote by $Z_k(X)$ the group of
smooth singular $k$-cycles.

\begin{definition}[Cheeger-Simons {\cite{CS}}]
A \emph{differential character} of degree $2p-1$ on $X$ is a group homomorphism
\[
  \hat h : Z_{2p-1}(X)\;\longrightarrow\;\rr/\zz
\]
such that there exists a smooth closed $2p$-form $\omega = \omega(\hat h)$
(the \emph{curvature}) with the property that for every smooth $2p$-chain $\sigma$:
\[
  \hat h(\partial\sigma) = \int_\sigma \omega \pmod{\zz}.
\]
The group of such characters is denoted $\widehat H^{2p-1}(X;\zz)$.
\end{definition}

There is an exact sequence
\begin{equation}\label{eq:diff-char-exact}
  0 \to H^{2p-1}(X;\rr/\zz) \to \widehat H^{2p-1}(X;\zz) \xrightarrow{\omega}
  \Omega^{2p}_{\mathrm{cl}}(X) \to 0,
\end{equation}
and also
\begin{equation}\label{eq:diff-char-exact2}
  0 \to \frac{A^{2p-1}(X)}{A^{2p-1}_{\zz}(X)} \to
  \widehat H^{2p-1}(X;\zz) \xrightarrow{\delta} H^{2p}(X;\zz) \to 0,
\end{equation}
where $A^{2p-1}_{\zz}(X)$ denotes the closed forms with integral periods.
The composition $\omega\circ\delta = 0$ and $\delta\circ\omega = $ de Rham map.

\subsection{Construction for vector bundles}

Let $E\to X$ be a complex vector bundle of rank $r$ with a smooth connection
$\nabla$.  Recall the Chern-Weil construction: to the invariant polynomial
$P_p = \frac{1}{p!(2\pi i)^p}\mathrm{tr}(-)^p$ one associates the closed
Chern character form
\[
  \ch_p(\nabla) = P_p(F(\nabla)) = \frac{1}{p!(2\pi i)^p}\mathrm{Tr}(F(\nabla)^p)
  \;\in\; A^{2p}_{\mathrm{cl}}(X).
\]

\begin{definition}
The \emph{Chern-Simons class} of $(E,\nabla)$ is the differential character
$\widehat{\ch}_p(E,\nabla)\in\widehat H^{2p-1}(X;\zz)$ whose curvature is
$\ch_p(\nabla)$, defined as follows.  For a singular $(2p-1)$-cycle $\gamma$,
choose a chain $\Gamma$ in $\mathrm{Map}(X,\mathbf{Gr}(r,N))$ (Grassmannian)
with $\partial\Gamma = $ (classifying map of $E$) restricted to $\gamma$.  Set
\[
  \widehat{\ch}_p(E,\nabla)(\gamma) = \int_\Gamma \ch_p^{\mathrm{univ}}
  \pmod\zz,
\]
where $\ch_p^{\mathrm{univ}}$ is the universal Chern character on the
Grassmannian.
\end{definition}

\begin{proposition}[{\cite{CS}}]
This is well-defined and gives a differential character with curvature $\ch_p(\nabla)$.
If $\nabla_0,\nabla_1$ are two connections with $\ch_p(\nabla_0)=\ch_p(\nabla_1)$,
the transgression form is closed and its class
\[
  \left[\frac{p}{p!\,(2\pi i)^p}\int_0^1
    \mathrm{Tr}\!\left(\dot\nabla_t\wedge F(\nabla_t)^{p-1}\right)\,dt\right]
  \;\in\; H^{2p-1}(X,\rr), \qquad \nabla_t = \nabla_0 + t(\nabla_1-\nabla_0),
\]
reduces modulo $\zz$ to $\widehat{\ch}_p(E,\nabla_1) - \widehat{\ch}_p(E,\nabla_0)$,
a flat character lying in the subgroup $H^{2p-1}(X,\rr/\zz)\subset\widehat
H^{2p-1}(X,\zz)$ of \eqref{eq:diff-char-exact}. (The normalising factor
$1/p!(2\pi i)^p$ is the one occurring in $\ch_p$ above and must be carried
along here: dropping it gives instead the transgression of the unnormalised
polynomial $\mathrm{Tr}(-)^p$, and the resulting form would not have curvature
$\ch_p(\nabla)$, nor integral periods after gauge transformation --- compare
\eqref{eq:CS-form-normalised} below, where the same factor is present.)
\end{proposition}

\subsection{The Chern-Simons form and its cocycle}

For a trivial bundle $E = X\times\cc^r$ with connection $\nabla = d+A$, the
Chern-Simons class has an explicit formula.  Define
\begin{equation}\label{eq:CS-form-normalised}
  \mathrm{CS}_{2p-1}(A) \;=\; \frac{1}{p!\,(2\pi i)^p}\;
  p\int_0^1 \mathrm{Tr}\!\left(A\wedge(tdA + t^2A\wedge A)^{p-1}
  \right)\,dt
  \;\in\; A^{2p-1}(X),
\end{equation}
so that $d\,\mathrm{CS}_{2p-1}(A) = \ch_p(d+A) - \ch_p(d)$.  The normalising
factor $1/p!(2\pi i)^p$ is the one occurring in $\ch_p$ and must be carried
along: without it the transgression of the unnormalised polynomial
$\mathrm{Tr}(-)^p$ is obtained instead, and the integrality statement of
Lemma~\ref{lem:gauge-variation} below fails.  For $p=1$
formula~\eqref{eq:CS-form-normalised} reads $\mathrm{CS}_1(A)=\mathrm{Tr}(A)/2\pi i$.  When
$\ch_p(d+A)=0$ (e.g.\ for flat or locally nil-flat $A$), the form
$\mathrm{CS}_{2p-1}(A)$ is closed and its cohomology class modulo $\zz$
is the Chern-Simons class.

\begin{lemma}[Gauge variation]\label{lem:gauge-variation}
Let $g:X\to GL_r(\cc)$ be smooth and let
$A^g := gAg^{-1}-(dg)g^{-1}$ be the gauge transform of $A$, so that
$d+A^g = g\circ(d+A)\circ g^{-1}$.  Put $\theta = g^{-1}dg$ for the
Maurer-Cartan form and
\[
  \mu_{2p-1} \;:=\; (-1)^{p-1}\,\frac{(p-1)!}{(2p-1)!}\,
  \frac{1}{(2\pi i)^p}\,\mathrm{Tr}\bigl(\theta^{\wedge(2p-1)}\bigr)
  \;\in\; A^{2p-1}\bigl(GL_r(\cc)\bigr).
\]
Then
\[
  \mathrm{CS}_{2p-1}(A^g) \;=\; \mathrm{CS}_{2p-1}(A)\;-\;g^*\mu_{2p-1}
  \;+\; d\beta
\]
for some $\beta\in A^{2p-2}(X)$.  Moreover $\mu_{2p-1}$ is a closed
bi-invariant form whose de Rham class lies in the image of
$H^{2p-1}(GL_r(\cc),\zz)\to H^{2p-1}(GL_r(\cc),\cc)$; consequently
$g^*\mu_{2p-1}$ has integral periods on $X$, and
\[
  \bigl[\mathrm{CS}_{2p-1}(A^g)\bigr] \;=\;\bigl[\mathrm{CS}_{2p-1}(A)\bigr]
  \quad\text{in } H^{2p-1}(X,\cc/\zz)
\]
whenever both forms are closed.
\end{lemma}

\begin{proof}
The variation formula is \cite[Prop.~3.7]{ChrS}, transported through the
normalisation~\eqref{eq:CS-form-normalised}.  For the integrality, recall that
$U(r)\hookrightarrow GL_r(\cc)$ is a deformation retract and that
$H^*(U(r),\zz)=\Lambda_\zz(x_1,\dots,x_r)$ with $\deg x_p = 2p-1$, the
generator $x_p$ being represented in de Rham cohomology precisely by
$\mu_{2p-1}$; this is the normalisation fixed in \cite{ChrS} (see also
\cite{Bott}).  Hence $\mu_{2p-1}$ has integral periods, and so does its
pullback under any smooth $g$.  The last assertion follows since the two
closed forms differ by an exact form plus a form with integral periods.

The case $p=1$ is a useful check: $\mathrm{Tr}(A^g)=\mathrm{Tr}(A)-
\mathrm{Tr}\bigl((dg)g^{-1}\bigr) = \mathrm{Tr}(A)-d\log\det g$, so
$\mathrm{CS}_1(A^g)-\mathrm{CS}_1(A) = -\tfrac{1}{2\pi i}\,d\log\det g
= -g^*\mu_1$, and $\tfrac{1}{2\pi i}d\log\det$ is the pullback under
$\det:GL_r(\cc)\to\cc^*$ of the standard integral generator of
$H^1(\cc^*,\zz)$.
\end{proof}

\subsection{The volume regulator and the splitting of $\cc/\zz$}

\begin{lemma}\label{lem:CZ-splitting}
There is a canonical direct sum decomposition of abelian groups
\begin{equation}\label{eq:CZ-split-lemma}
  \cc/\zz \;\cong\; \rr/\zz \;\oplus\; i\rr ,
\end{equation}
induced by the splitting $\cc = \rr\oplus i\rr$ of real vector spaces.
Explicitly $x + iy + \zz \mapsto (x+\zz,\; iy)$ for $x,y\in\rr$, with
projections
\[
  \Re_{\cc/\zz}:\cc/\zz\to\rr/\zz,\ \ x+iy+\zz\mapsto x+\zz,
  \qquad
  \Im_{\cc/\zz}:\cc/\zz\to\rr,\ \ x+iy+\zz\mapsto y .
\]
Under this splitting the Chern-Simons class $\CS_p\in H^{2p-1}(X,\cc/\zz)$
decomposes as
\[
  \CS_p \;=\; \CS_p^{\rr} \;+\; i\cdot\vol_p,
  \quad \CS_p^{\rr}\in H^{2p-1}(X,\rr/\zz),\quad
  \vol_p\in H^{2p-1}(X,\rr),
\]
the second component being the volume regulator exactly, not merely
modulo $\zz$.
\end{lemma}

\begin{proof}
Here $\zz$ is embedded in $\cc$ through $\rr$, so the subgroup being divided
out meets the summand $i\rr$ in $0$: it is the real summand alone that is
affected by the quotient.  Concretely, $\cc=\rr\oplus i\rr$ as real vector
spaces and $\zz\subset\rr$, whence
$\cc/\zz = (\rr\oplus i\rr)/\zz \cong (\rr/\zz)\oplus i\rr$.  The second
summand is \emph{not} $i(\rr/\zz)$: there is no $i\zz$ to quotient by.  The
projection $\Im_{\cc/\zz}$ is therefore well defined with values in $\rr$,
since altering a representative by an integer changes only its real part.
\end{proof}

For a locally nil-flat connection $\nabla$ on $E\to X$, the \emph{volume
regulator} is accordingly the imaginary component of the Chern-Simons class,
\[
  \vol_p(\nabla) := \Im_{\cc/\zz}\bigl(\CS_p(\nabla)\bigr)
  = [\Im\,\mathrm{CS}_{2p-1}(\nabla)]\;\in\; H^{2p-1}(X,\rr),
\]
which is well defined by Lemma~\ref{lem:CZ-splitting}.

\begin{proposition}\label{prop:vol-indep-app}
$\vol_p(\nabla)$ is independent of the choice of connection $\nabla_0$ and
of the path from $\nabla_0$ to $\nabla$ used to define it
(Proposition~\ref{prop:vol-indep}); and if $\nabla$ preserves an indefinite
Hermitian form $g$ on $E$, then $\vol_p(\nabla) = 0$ for every $p\geq1$
(Proposition~\ref{prop:hodge-vanish}).
\end{proposition}

\begin{proof}
Both statements are proved in \S\ref{sec:volume}, in the normalization of
Definition~\ref{def:normalized-ch-vp}, which is exactly the normalization
$\mathrm{CS}_{2p-1}$ of \eqref{eq:CS-form-normalised} used throughout this
appendix: $\vol_p(\nabla)$ as defined above agrees with
Definition~\ref{def:vol-p-normalized}, so the cited results transport
verbatim.
\end{proof}

\subsection{Relation to Deligne-Beilinson cohomology}

There is a natural map
\[
  \widehat H^{2p-1}(X;\zz)\;\longrightarrow\; H^{2p}_\calD(X,\zz(p))
\]
from differential characters to Deligne-Beilinson cohomology, compatible with
curvature and the integral characteristic class; it carries
$\widehat{\ch}_p(E,\nabla)$ to the Deligne-Beilinson Chern character
$\ch^{\calD}_p(E)$.  We write
\begin{equation}\label{eq:alpha-map}
  \alpha\;:\;H^{2p-1}(X,\cc/\zz)\;\longrightarrow\;H^{2p}_\calD(X,\zz(p))
\end{equation}
for its restriction to the subgroup of flat characters.

For locally nil-flat $\nabla$ we have $\ch_p(\nabla)=0$, so
$\widehat{\ch}_p(E,\nabla)$ is a flat character, i.e.\ lies in
$H^{2p-1}(X,\cc/\zz)\subset\widehat H^{2p-1}(X;\zz)$, and
$\alpha\bigl(\CS_p(E,\nabla)\bigr)=\ch^\calD_p(E)$.

The map $\alpha$ is \emph{not} injective in general: its kernel is the image
of $F^pH^{2p-1}(X,\cc)$ in $H^{2p-1}(X,\cc/\zz)$.  So knowing $\ch^\calD_p(E)$
does not by itself determine $\CS_p(E,\nabla)$ on an arbitrary smooth
projective $X$.  It does on the $\aaa$-realizations where the comparison of
Appendix~\ref{app:rees} takes place, $\alpha$ being an isomorphism there by
Proposition~\ref{prop:DB-A-realization}; see Remark~\ref{rem:alpha-kernel}.

\subsection{Proof of the $\ell$-adic torsion (Corollary~\ref{cor:ladic-torsion})}
\label{subsec:ladic-proof}

We give a complete proof that the $\ell$-adic Chern-Simons classes are torsion,
using the Artin comparison theorem.

\subsubsection{The Artin comparison isomorphism}

Let $X$ be a smooth complex algebraic variety and $\ell$ a prime number.
The \emph{Artin comparison isomorphism} is a canonical isomorphism
\begin{equation}\label{eq:artin}
  H^n_{\mathrm{sing}}(X(\cc),\zz/m\zz) \;\xrightarrow{\;\sim\;}
  H^n_{\mathrm{\acute{e}t}}(X,\zz/m\zz)
\end{equation}
for every $n\geq 0$ and every positive integer $m$.  Passing to the inverse
limit over $m = \ell^k$ and tensoring with $\qq_\ell$:
\begin{equation}\label{eq:artin-ladic}
  H^n_{\mathrm{sing}}(X(\cc),\zz)\otimes_\zz\zz_\ell \;\xrightarrow{\;\sim\;}
  H^n_{\mathrm{\acute{e}t}}(X,\zz_\ell),
\end{equation}
compatible with the Galois action and cup products.

For Tate twists, the Artin comparison with coefficients $\zz/m\zz(p) = \mu_m^{\otimes p}$
(where $\mu_m$ is the group of $m$-th roots of unity) gives
\begin{equation}\label{eq:artin-tate}
  H^n_{\mathrm{sing}}(X(\cc),\zz(p))\otimes_\zz\zz/m\zz
  \;\xrightarrow{\;\sim\;}
  H^n_{\mathrm{\acute{e}t}}(X,\zz/m\zz(p)),
\end{equation}
where $\zz(p) = (2\pi i)^p\zz\subset\cc$ on the left, and
$\zz/m\zz(p) = \mu_m^{\otimes p}$ on the right.

\subsubsection{The $\ell$-adic Chern-Simons class}

We stress that there is no $\ell$-adic local system on $X$: the sheaf
$L_\ell$ lives on $X^*_{\mathrm{\acute{e}t}}$ and does not extend across $D$,
and there is no \'{e}tale analogue of Deligne's canonical extension.  The
$\ell$-adic classes below are therefore \emph{not} obtained by repeating the
transcendental construction \'{e}tale-locally.  They are obtained from the
transcendental extended class on $X$ by reduction, and their existence
presupposes Theorem~\ref{thm:main}.

Concretely: by Theorem~\ref{thm:main} the class
$\CS_p(L):=\CS_p(\nabla^{\Del})\in H^{2p-1}(X,\cc/\zz)$ is torsion, say killed
by $m$.  It therefore lies in the $m$-torsion subgroup, and
$(\cc/\zz)[m]\cong\tfrac{1}{m}\zz/\zz\cong\zz/m\zz$, so $\CS_p(L)$ determines
a class in $H^{2p-1}(X,\zz/m\zz)$, which the Tate twist turns into a class in
$H^{2p-1}(X,\zz/m\zz(p))$.  (It is essential that one restricts to the torsion
subgroup: $\cc/\zz$ is divisible, so $\cc/\zz\otimes\zz/m\zz=0$ and no
reduction map exists on the whole group.)  Passing to the limit over
$m=\ell^k$ and comparing with \'{e}tale cohomology of the projective variety
$X$ by \eqref{eq:artin-ladic}:

\begin{definition}\label{def:ladic-CS}
The \emph{$\ell$-adic extended Chern-Simons class} is the image of $\CS_p(L)$
under the composite map
\[
  H^{2p-1}(X,\cc/\zz)
  \;\longrightarrow\;
  \varprojlim_k H^{2p-1}(X,\zz/\ell^k\zz)
  \;\xrightarrow[\sim]{\text{Artin}}\;
  H^{2p-1}_{\mathrm{\acute{e}t}}(X,\zz_\ell)
  \;\longrightarrow\;
  H^{2p-1}_{\mathrm{\acute{e}t}}(X,\qq_\ell/\zz_\ell(p)).
\]
We denote it $\CS_p^{(\ell)}(L_\ell)\in H^{2p-1}_{\mathrm{\acute{e}t}}(X,\qq_\ell/\zz_\ell(p))$.
\end{definition}

\subsubsection{Proof of Corollary~\ref{cor:ladic-torsion}}

\begin{proof}[Proof of Corollary~\ref{cor:ladic-torsion}]
By Theorem~\ref{thm:main}, the singular class $\CS_p(L)\in H^{2p-1}(X,\cc/\zz)$
is torsion: there exists $N\geq 1$ with $N\cdot\CS_p(L) = 0$.

\textit{Step 1: Reduction to finite coefficients.}
For each $k\geq 1$, consider the commutative diagram
\[
\begin{array}{ccc}
H^{2p-1}(X,\cc/\zz) & \xrightarrow{\times N} & H^{2p-1}(X,\cc/\zz) \\
\downarrow{\rho_k} && \downarrow{\rho_k} \\
H^{2p-1}(X,\zz/\ell^k\zz(p)) & \xrightarrow{\times N} & H^{2p-1}(X,\zz/\ell^k\zz(p)),
\end{array}
\]
where $\rho_k$ is the reduction map modulo $\ell^k$ (incorporating the Tate
twist).  Since $N\cdot\CS_p(L) = 0$, we have $N\cdot\rho_k(\CS_p(L)) = 0$
for every $k$.  Thus $\rho_k(\CS_p(L))$ is killed by $N$ in the finite group
$H^{2p-1}(X,\zz/\ell^k\zz(p))$, hence it is $N$-torsion.

\textit{Step 2: Application of the Artin comparison.}
By~\eqref{eq:artin-tate}, the Artin comparison gives isomorphisms
\[
  \rho_k(\CS_p(L)) \;\mapsto\;
  \CS_p^{(\ell,k)}(L_\ell)\;\in\; H^{2p-1}_{\mathrm{\acute{e}t}}(X,\zz/\ell^k\zz(p))
\]
compatible with the transition maps as $k$ varies.  Each
$\CS_p^{(\ell,k)}(L_\ell)$ is $N$-torsion.

\textit{Step 3: Passage to the limit.}
The $\ell$-adic class $\CS_p^{(\ell)}(L_\ell)$ is defined as the element of
\[
  H^{2p-1}_{\mathrm{\acute{e}t}}(X,\qq_\ell/\zz_\ell(p))
  = \varinjlim_k H^{2p-1}_{\mathrm{\acute{e}t}}(X,\ell^{-k}\zz_\ell/\zz_\ell(p))
\]
corresponding to the compatible system $\{\CS_p^{(\ell,k)}\}_k$.
Since each $\CS_p^{(\ell,k)}$ is $N$-torsion and the direct limit map preserves
the torsion condition, $\CS_p^{(\ell)}(L_\ell)$ is $N$-torsion in
$H^{2p-1}_{\mathrm{\acute{e}t}}(X,\qq_\ell/\zz_\ell(p))$.

\textit{Step 4: The $\ell$-part.}
More precisely, let $N = \ell^a\cdot m$ with $\gcd(m,\ell)=1$.  Then
$\ell^a\cdot\CS_p^{(\ell)}(L_\ell)$ may be non-zero (the $m$-part vanishing
is automatic since $m$ is invertible in $\zz_\ell$), but
$N\cdot\CS_p^{(\ell)}(L_\ell) = 0$.  The $\ell$-adic class is therefore
annihilated by $\ell^a$, confirming it is a torsion element of the
$\ell$-primary part of $H^{2p-1}_{\mathrm{\acute{e}t}}(X,\qq_\ell/\zz_\ell(p))$.
\end{proof}

\subsubsection{The integral $\ell$-adic version}

We also have the following integral statement:

\begin{corollary}\label{cor:ladic-integral}
The image of $\CS_p^{(\ell)}(L_\ell)$ in the $\ell$-adic cohomology
$H^{2p-1}_{\mathrm{\acute{e}t}}(X,\qq_\ell(p))$ is zero.
\end{corollary}

\begin{proof}
The short exact sequence
$0\to\zz_\ell(p)\to\qq_\ell(p)\to\qq_\ell/\zz_\ell(p)\to 0$
gives a long exact sequence; the boundary map sends $\CS_p^{(\ell)}\in
H^{2p-1}_{\mathrm{\acute{e}t}}(X,\qq_\ell/\zz_\ell(p))$ to a class in
$H^{2p}_{\mathrm{\acute{e}t}}(X,\zz_\ell(p))$.  Since $\CS_p^{(\ell)}$ is
torsion, its image in $H^{2p-1}_{\mathrm{\acute{e}t}}(X,\qq_\ell(p))$ is zero
(as the latter group is torsion-free: étale cohomology with $\qq_\ell$
coefficients is a $\qq_\ell$-vector space).
\end{proof}

\section{Simplicial spaces, de Rham theory, and \texorpdfstring{$F^1$}{F1}-connections}
\label{app:simplicial}

\subsection{Simplicial manifolds and their de Rham complexes}

\begin{definition}
A \emph{simplicial manifold} $M_\bullet$ is a simplicial object in the
category of smooth manifolds: a collection of smooth manifolds $M_n$ ($n\geq 0$)
with smooth face maps $\partial_i:M_n\to M_{n-1}$ ($0\leq i\leq n$) and
degeneracy maps $s_j:M_n\to M_{n+1}$ ($0\leq j\leq n$) satisfying the
simplicial identities.
\end{definition}

\begin{definition}[Dupont \cite{Dupont}]
The \emph{de Rham complex} $A^\bullet(M_\bullet)$ of a simplicial manifold
$M_\bullet$ is the total complex of the double complex with
\[
  A^{p,q}(M_\bullet) = A^q(M_p),
\]
and differentials $d_{\mathrm{dR}}:A^{p,q}\to A^{p,q+1}$ (de Rham) and
$\delta = \sum_{i}(-1)^i\partial_i^*: A^{p,q}\to A^{p+1,q}$ (simplicial
coboundary).  A smooth $n$-form on $M_\bullet$ is an element of
$\bigoplus_{p+q=n}A^{p,q}$.
\end{definition}

\begin{theorem}[Dupont {\cite{Dupont}}]
The cohomology of $A^\bullet(M_\bullet)$ computes the singular cohomology of
the geometric realization $|M_\bullet|$ with $\rr$-coefficients.
\end{theorem}

\subsection{The cubical de Rham complex}

For the cubical space $\Sigma_F(\mathcal{U}_\bullet)$ arising from the adapted
covering, we use a cubical analogue.

\begin{definition}
Let $C$ be a cubical set.  The \emph{cubical de Rham complex}
$A^\bullet(C,F)$ consists of collections $(w_c)_{c\in C}$ where
$w_c\in A^\bullet(\aaa(c))$, satisfying the face compatibility: for each
face inclusion $c'\leq c$ via face map $\iota_{c',c}:\aaa(c')\to\aaa(c)$,
\[
  \iota_{c',c}^*w_c = w_{c'}.
\]
\end{definition}

For an $F^1$-connection $\nabla(c)$ on a bundle $E(c)$ over each $\aaa(c)$,
the Chern character forms $\ch_p(\nabla(c))$ satisfy $\iota^*\ch_p(\nabla(c))
= \ch_p(\nabla(c'))$ (since the connections are compatible under restriction).
So $(\ch_p(\nabla(c)))$ is a cocycle in the cubical de Rham complex.

\subsection{\texorpdfstring{$F^1$}{F1}-connections on simplicial schemes}

Let $\bar Z_\bullet$ be a simplicial smooth projective variety with simplicial
normal crossings divisor $D_\bullet$, and $Z_\bullet = \bar Z_\bullet\setminus D_\bullet$.

\begin{definition}[Dupont-Hain-Zucker {\cite{DHZ}}]
A \emph{simplicial $F^1$-connection} on a simplicial bundle $E_\bullet$ over
$Z_\bullet$ is a collection of $F^1$-connections $\nabla_n$ on $E_n$ over
$Z_n$ (for each $n\geq 0$), compatible with the face and degeneracy maps of
$Z_\bullet$.
\end{definition}

\begin{theorem}[DHZ {\cite[Prop.~6.1.2, Lem.~6.1.3]{DHZ}}]
\label{thm:DHZ-app}
\begin{enumerate}[leftmargin=2em,label=\rm(\arabic*)]
  \item A simplicial $F^1$-connection on $E_\bullet$ exists whenever each
    $E_n$ admits an $F^1$-connection (e.g.\ when each $E_n$ is an algebraic
    bundle on a projective variety).
  \item Any two simplicial $F^1$-connections define the same class in the
    Deligne cohomology $H^{2p}_\calD(Z_\bullet,\zz(p))$ (the simplicial
    Deligne cohomology).
\end{enumerate}
\end{theorem}

\begin{proof}[Sketch of~(2)]
If $\nabla$ and $\nabla'$ are two $F^1$-connections, the difference
satisfies
$$
A = \nabla'-\nabla\in A^1_F(\bar Z_n\log D_n,\End(E_n)) \mbox{ (for each $n$)}.
$$

The Bott-Chern secondary form $\widetilde{\ch}_p(\nabla,\nabla')$ is computed
as $p\int_0^1(1-t)\mathrm{Tr}(A\wedge F(\nabla+tA)^{p-1})\,dt$, and lies in
$F^1A^{2p-1}(Z_n\log D_n)$.  This means $\widetilde{\ch}_p$ is a
$(2p-1)$-coboundary in the Deligne complex.
\end{proof}

\subsection{The de Rham complex of the cubical realization}

For the specific simplicial space $S = \iiii\Cube A$ arising from the adapted
covering:

\begin{proposition}
There is an isomorphism
\[
  H^{2p}_\calD(\aaa\Cube A, \zz(p)) \;\cong\; H^{2p-1}(\iiii\Cube A, \cc/\zz)
  \;\cong\; H^{2p-1}(X,\cc/\zz).
\]
\end{proposition}

\begin{proof}
The first isomorphism: apply the Mayer-Vietoris spectral sequence for the
cover of $\aaa\Cube A$ by the affine pieces $\{\aaa(c)\}$.  Since
$H^k_\calD(\aaa^n,\zz(p))$ is $\cc/\zz$ for $k=2p$ and $0$ otherwise
(by $\aaa^1$-contractibility), the spectral sequence degenerates to give
$H^{2p}_\calD\cong \check H^{2p-1}(\{\aaa(c)\},\cc/\zz)\cong H^{2p-1}(\iiii\Cube A,\cc/\zz)$.
The second isomorphism is the homotopy equivalence $\iiii\Cube A\simeq X$.
\end{proof}

\begin{corollary}
The $F^1$-connection $\nabla$ on $F^\sigma|_{\aaa\Cube A}$ constructed in
Theorem~\ref{thm:F1-exists} defines the class
$\CS_p(\zeta_0)\in H^{2p-1}(X,\cc/\zz)$.
\end{corollary}

\section{The Burgos-Gil arithmetic Chern character}
\label{app:burgosgil}

\subsection{The arithmetic intersection theory of Burgos-Gil}

We recall the framework of \cite{BG} that we use for the comparison theorem.

Let $X$ be a smooth complex projective variety.  The \emph{arithmetic Chow
group} $\widehat{\mathrm{CH}}^p(X)$ is generated by pairs $(Z,g_Z)$ where
$Z\subset X$ is an algebraic cycle of codimension $p$ and $g_Z$ is a Green
current for $Z$ (a current of type $(p-1,p-1)$ on $X(\cc)$ satisfying
$dd^cg_Z + \delta_Z = [\omega_Z]$ for a smooth form $\omega_Z$), modulo
certain equivalence relations.

\begin{theorem}[Burgos-Gil {\cite{BG}}]
\label{thm:BG-full}
There exists a theory of \emph{arithmetic characteristic classes}: to each
vector bundle $E$ on $X$ and Hermitian metric $h$ one assigns arithmetic Chern
character classes
\[
  \widehat{\ch}_p(E,h)\;\in\;\widehat{\mathrm{CH}}^p(X)\otimes\qq,
\]
with the following properties:
\begin{enumerate}[leftmargin=2em,label=\rm(\arabic*)]
  \item \textit{Curvature:} The curvature form of $\widehat{\ch}_p(E,h)$ is
    $\ch_p(E,h)=\frac{(-1)^p}{p!(2\pi i)^p}\mathrm{Tr}(\Theta_h^p)$, where
    $\Theta_h$ is the curvature of the Hermitian connection.
  \item \textit{Independence:} The image of $\widehat{\ch}_p(E,h)$ in
    $H^{2p}_\calD(X,\zz(p))$ is independent of $h$ and equals $\ch^{\calD}_p(E)$.
  \item \textit{Functoriality:} Compatible with pullback and push-forward,
    and satisfies the Whitney product formula.
  \item \textit{Comparison:} For a flat bundle (all Chern character forms zero),
    $\widehat{\ch}_p(E,h)$ maps to $\CS_p(E,\nabla_h^{\mathrm{flat}})$ in
    $H^{2p-1}(X,\cc/\zz)$ via the edge map of the Deligne exact sequence.
\end{enumerate}
\end{theorem}

\subsection{Green currents and the Bott-Chern form}

\begin{definition}
Given two Hermitian metrics $h_0,h_1$ on $E$, the \emph{Bott-Chern secondary
form} is
\[
  \widetilde{\ch}_p(E;h_0,h_1) = p\int_0^1(1-t)
  \mathrm{Tr}\!\left(\dot h_t h_t^{-1}\wedge\Theta_{h_t}^{p-1}\right)\,dt
  \;\in\; A^{2p-1}(X)/\mathrm{im}(\partial+\bar\partial),
\]
where $h_t = (1-t)h_0+th_1$ and $\dot h_t = h_1-h_0$.  It satisfies
\[
  dd^c[\widetilde{\ch}_p] = [\ch_p(E,h_1)] - [\ch_p(E,h_0)].
\]
\end{definition}

\subsection{The comparison theorem: the Deligne class, and when it is a
Chern-Simons class}

The comparison we need has two halves, and it is important to keep them
separate because they have different hypotheses.  The first identifies the
Bott-Chern class of an $F^1$-connection with the Deligne Chern class of the
underlying bundle; it requires no flatness whatever, and it is what allows the
Deligne class to be computed on a Grassmannian, where no nil-flat connection
exists.  The second says that when the connection happens to be locally
nil-flat, the associated differential character is flat and so defines a class
in $H^{2p-1}(-,\cc/\zz)$, which is the extended Chern-Simons class.

\begin{theorem}[Deligne class from any $F^1$-connection]
\label{thm:CS-DB-app}
Let $\bar Z$ be smooth projective, $D_Z\subset\bar Z$ a normal crossings
divisor, and $\bar E$ a vector bundle on $\bar Z$.
\begin{enumerate}[leftmargin=2em,label=\rm(\arabic*)]
  \item $F^1$-connections on $\bar E$ with logarithmic poles along $D_Z$
    exist; for instance the Chern connection $\nabla_h$ of any Hermitian
    metric $h$, which has no poles at all.
  \item For every such $\nabla$, the Bott-Chern class satisfies
    \[
      [\nabla]_p \;=\; \ch^\calD_p(\bar E)\;\in\;H^{2p}_\calD(\bar Z,\zz(p)).
    \]
    No flatness or nil-flatness hypothesis is used.
  \item If in addition $\nabla$ is locally nil-flat, so that
    $\ch_p(\nabla)=0$ as a form
    (Proposition~\ref{prop:nilflat-ch-zero}), then the Cheeger-Simons
    character $\widehat{\ch}_p(\bar E,\nabla)$ has zero curvature, hence is a
    flat character and defines
    \[
      \CS_p(\bar E,\nabla)\;\in\;H^{2p-1}(\bar Z,\cc/\zz),
    \]
    whose image under the natural map
    $\alpha:H^{2p-1}(\bar Z,\cc/\zz)\to H^{2p}_\calD(\bar Z,\zz(p))$ of
    \S\ref{app:diffchar} is $\ch^\calD_p(\bar E)$.
\end{enumerate}
\end{theorem}

\begin{proof}
(1) In a local holomorphic frame the Chern connection of $(\bar E,h)$ has
connection matrix $h^{-1}\partial h$, of type $(1,0)$; a connection whose
matrix is of type $(1,0)$ lies in $F^1$ of the (logarithmic) de Rham complex,
which is the $F^1$-condition.  Existence of $h$ is standard, by a partition of
unity.

(2) Let $\nabla$ be any $F^1$-connection with log poles along $D_Z$, and let
$\nabla_h$ be a Chern connection as in (1).  By
Theorem~\ref{thm:DHZ-app}(2) any two $F^1$-connections define the same class,
so $[\nabla]_p=[\nabla_h]_p$.  By Theorem~\ref{thm:BG-full}(1)--(2) the
arithmetic Chern character $\widehat{\ch}_p(\bar E,h)$ has curvature form
$\ch_p(\bar E,h)$ and image $\ch^\calD_p(\bar E)$ in $H^{2p}_\calD$,
independent of $h$; and $[\nabla_h]_p$ is by construction that image.  Hence
$[\nabla]_p=\ch^\calD_p(\bar E)$.

It is worth recording the concrete form of this argument, since it is the one
used on the cubical realization in \S\ref{subsec:BG-on-cubical}.  For
$n\gg0$ choose a surjection $q:\Oo_{\bar Z}^N\twoheadrightarrow\bar E(n)$,
giving $f_q:\bar Z\to\mathbf{Gr}(r,N)$ with $f_q^*\mathcal{S}_r=\bar E(n)$,
where $\mathcal{S}_r$ carries the canonical Hermitian metric $h_\mathcal{S}$
induced from the standard inner product on $\cc^N$.  Then
$\nabla_{\mathrm{Gr}}:=f_q^*\nabla_{h_\mathcal{S}}\otimes 1+1\otimes\nabla_{n}$
is an $F^1$-connection on $\bar E=\bar E(n)\otimes\Oo(-n)$, for
$\nabla_n$ a Chern connection on $\Oo_{\bar Z}(-n)$, and by
multiplicativity of the Deligne-Beilinson Chern character
\[
  \ch^\calD_\bullet(\bar E)\;=\;
  f_q^*\,\ch^\calD_\bullet(\mathcal{S}_r)\cdot\ch^\calD_\bullet(\Oo(-n)),
\]
so that $[\nabla_{\mathrm{Gr}}]_p=\ch^\calD_p(\bar E)$ degree by degree.
We stress that $\nabla_{\mathrm{Gr}}$ is \emph{not} nil-flat: its Chern
character forms are the pullbacks of the universal ones on the Grassmannian
and do not vanish.  Part (2) does not require them to.

(3) If $\ch_p(\nabla)=0$ as a form then the curvature of
$\widehat{\ch}_p(\bar E,\nabla)$ vanishes, so by the exact
sequence~\eqref{eq:diff-char-exact} the character lies in the subgroup
$H^{2p-1}(\bar Z,\cc/\zz)$ of flat characters.  The map $\alpha$ of
\S\ref{app:diffchar} carries $\widehat{\ch}_p(\bar E,\nabla)$ to $[\nabla]_p$,
which is $\ch^\calD_p(\bar E)$ by (2).
\end{proof}

\begin{remark}[The two directions are not equivalent]\label{rem:alpha-kernel}
Part (3) determines the image of $\CS_p(\bar E,\nabla)$ under $\alpha$, not
$\CS_p$ itself, and the distinction matters: the kernel of $\alpha$ is the
image of $F^pH^{2p-1}(\bar Z,\cc)$ in $H^{2p-1}(\bar Z,\cc/\zz)$, which is in
general non-zero.  So on a general smooth projective $\bar Z$ the Deligne
class does not determine the Chern-Simons class.

What rescues the situation in this paper is that the comparison with the
regulator is not carried out on $X$ but on an $\aaa$-realization, where the
corresponding map is an \emph{isomorphism}: by
Proposition~\ref{prop:DB-A-realization},
$H^{2p}_\calD(\Re_{\aaa,\sigma}X_\hdot,\zz(p))\cong
H^{2p-1}(|X_\hdot|,\cc/\zz)$, the affine directions carrying no Deligne
cohomology beyond a single copy of $\cc/\zz(p)$ in degree one.  There the
Deligne class and the $\cc/\zz$-class are interchangeable, which is why
$\ch^\calD_p$ can serve as the common currency between the $K$-theoretic and
the Hodge-theoretic halves of Appendix~\ref{app:rees}.
\end{remark}

\begin{theorem}[Regulator pullback]
\label{thm:CS-regulator-app}
In the situation of Section~\ref{sec:torsion}, Step 2, under the
homotopy equivalence $S = \iiii\Cube A\simeq X$ the class
$\CS_p(\nabla^{\Del})$ equals the pullback of the universal regulator class
$r_{2p-1}\in H^{2p-1}(BGL^+(K),\cc/\zz)$ via the classifying map.
\end{theorem}

\begin{proof}
The classifying map $X\simeq S\to BGL^+(K)$ is the map $r_L$ of
Construction~\ref{constr:rL}, defined by the Rees bundle $F_K$ on
$\aaa_K\Cube A$ (whose $\aaa^1$-homotopy type is $S$) through
Construction~\ref{constr:ctaut}.  The identity
$r_L^*(r_{2p-1})=\CS_p(\nabla)$, for $\nabla$ any $F^1$-connection on
$F_K^\sigma$, is Proposition~\ref{prop:classmap-comparison}.  Finally, by the
uniqueness of the bundle with given patching data
(Theorem~\ref{thm:uniqueness}) and the compatibility of patching data, the
$F^1$-connection on $F_K^\sigma$ and the restriction of $\nabla^{\Del}$ to
$\iiii\Cube A$ define the same class, by Theorem~\ref{thm:DHZ-app}(2).

The universal regulator and the Deligne cohomology class on 
$\aaa\Cube A$ are classes on different objects. In order to get to a situation where 
we have an isomorphism between their containing cohomology groups, we need to pull 
both of them back  to a common bisimplicial resolution. This is treated below
in \S\ref{subsec:classmap}. 
\end{proof}

\subsection{Application to the cubical realization: the base is not smooth}
\label{subsec:BG-on-cubical}

The statements recalled above are for a smooth projective base, and the object
we apply them to is not one.  We set out here in what sense they are used.

\begin{remark}[The difficulty]\label{rem:cubical-singular}
The projective realization $\pp_K\Cube A$ of \S\ref{sec:cubical} is obtained
by gluing the pieces $\pp(c)\cong(\pp^1_K)^{|c|}$, one for each cube $c$ of
$\Cube A$, along their faces.  As a scheme it is projective but neither smooth
nor irreducible: it is a union of products of projective lines meeting along
coordinate faces, so its singular locus contains every point lying on more
than one piece.  Neither Theorem~\ref{thm:BG-full} nor
Theorem~\ref{thm:CS-DB-app} applies to it as stated, and the Green currents of
\cite{BG} are not defined on a singular base.
\end{remark}

There are two ways to make the usage legitimate, and we use both: they give
the same class, and each is convenient for a different purpose.

\subsubsection{First route: descent, and the cubical scheme}

The first observation is that $\pp_K\Cube A$ should not be regarded as a
single singular scheme at all.  The Rees construction of
Theorem~\ref{thm:global-rees} does not produce a bundle on a glued scheme; it
produces a compatible family, namely a bundle on the \emph{cubical scheme}
\[
  \underline{\pp}_K\Cube A \;:\;\;
  c\;\longmapsto\;\pp(c)\cong(\pp^1_K)^{|c|},
\]
whose terms are all smooth projective and whose structure maps are the closed
immersions of faces.  This is exactly the setting for which the
Burgos-Gil theory is built---their arithmetic characteristic classes are
constructed on \emph{simplicial} schemes with smooth projective terms
\cite{BG}, and Theorem~\ref{thm:DHZ-app} is likewise stated in
\S\ref{app:simplicial} for a simplicial smooth projective variety, not for a
smooth variety.  Passing between cubical and simplicial objects is harmless
here: the cubical scheme is levelwise smooth projective, and its associated
diagonal simplicial scheme has the same Deligne cohomology by the usual
comparison of the two descent spectral sequences.

Accordingly all statements are to be read \emph{levelwise plus descent}:
\begin{itemize}[leftmargin=2em]
  \item $\widehat{\ch}_p(F_K,h)$ means the compatible family
    $\bigl(\widehat{\ch}_p(F(c),h_c)\bigr)_c$ of arithmetic Chern characters,
    each formed on the smooth projective $\pp(c)$ for a levelwise metric
    $h=(h_c)$;
  \item $\ch^\calD_p(F_K)$ means the class in
    $H^{2p}_\calD(\underline{\pp}_K\Cube A,\zz(p))$ defined by the descent
    spectral sequence of the cubical filtration, whose $E_1$-term is
    $\bigoplus_c H^{2p}_\calD(\pp(c),\zz(p))$ --- the same spectral sequence
    used in Proposition~\ref{prop:DB-A-realization} for the affine
    realization;
  \item an $F^1$-connection on $F_K$ means a compatible family of
    $F^1$-connections $\nabla(c)$ on $F(c)$ relative to
    $(\pp(c), D|_{\pp(c)})$, which is the definition already given in
    \S\ref{app:simplicial}.
\end{itemize}
With this reading, Theorem~\ref{thm:DHZ-app} and Theorem~\ref{thm:BG-full}
apply on each $\pp(c)$, where the base \emph{is} smooth projective, and the
conclusions---existence of $F^1$-connections, independence of the class of the
choice, and the comparison with the Chern-Simons class---descend because all
three are natural for the face immersions.

\begin{remark}[A citation-level point]
The descent step above uses that Burgos-Gil's arithmetic characteristic
classes, and the comparison theorems built on them, are natural for closed
immersions of faces in a simplicial (here, cubical) smooth projective scheme,
in the precise form needed to run the spectral-sequence descent argument.
This naturality is standard for the formalism of \cite{BG} and is exactly
the setting Theorem~\ref{thm:DHZ-app} is stated in
(\S\ref{app:simplicial}), but we have not independently re-verified the
exact hypotheses of the cited Burgos-Gil descent statement against the
specific cubical filtration used here; this is a citation-level point worth
checking precisely, though we know of no reason to expect it to fail. The
independent classifying-map argument of the second route below gives the
same class by an unrelated construction, which is some evidence for, though
not a substitute for, this check.
\end{remark}

\subsubsection{Second route: a Grassmannian classifying map}

The descent reading is the correct formal setting, but for some purposes one
wants a single smooth projective variety carrying the universal object, so
that the class may be \emph{defined} by pullback rather than by descent.  This
is supplied by a classifying map, and it requires only that a face-compatible
presentation of $F_K$ exist.

\begin{lemma}[Compatible presentation]\label{lem:compatible-surjection}
There exist $n\geq0$, $N\geq1$ and, for every cube $c$ of $\Cube A$, a
surjection of $\Oo_{\pp(c)}$-modules
\[
  q_c\;:\;\Oo_{\pp(c)}^{\,N}\;\twoheadrightarrow\;F(c)(n),
  \qquad F(c)(n):=F(c)\otimes\Oo_{\pp(c)}(n,\dots,n),
\]
such that $q_{c}|_{\pp(c')}=q_{c'}$ for every face $c'\subset c$.
\end{lemma}

\begin{proof}
Since $A$ is finite, $\Cube A$ has finitely many cubes and their dimensions
are bounded.  We construct the $q_c$ by ascending induction on $|c|$, choosing
$n$ and $N$ at the end uniformly.

Let $\partial\pp(c)=\bigcup_{c'\subsetneq c}\pp(c')$ be the union of the
proper faces, a closed subscheme with ideal sheaf
$\mathcal{I}_{\partial}\subset\Oo_{\pp(c)}$, and suppose compatible surjections have
been constructed on all proper faces; by compatibility they glue to a
surjection $q_{\partial}:\Oo_{\partial}^{\,N_0}\twoheadrightarrow
F(c)(n)|_{\partial}$.  Giving a map $\Oo^{N_0}\to F(c)(n)$ is giving $N_0$
global sections of $F(c)(n)$, so extending $q_\partial$ across $\pp(c)$ amounts
to surjectivity of the restriction
\[
  H^0\bigl(\pp(c),F(c)(n)\bigr)\;\longrightarrow\;
  H^0\bigl(\partial\pp(c),F(c)(n)|_{\partial}\bigr),
\]
whose cokernel is controlled by
$H^1\bigl(\pp(c),\mathcal{I}_\partial\otimes F(c)(n)\bigr)$.  As $\pp(c)$ is
projective and $\mathcal{I}_\partial\otimes F(c)$ coherent, this group vanishes for
$n\gg0$ by Serre vanishing; fix such an $n$, valid simultaneously for the
finitely many cubes.  Let $q'_c:\Oo^{N_0}\to F(c)(n)$ be an extension.  It is
surjective along $\partial\pp(c)$ by construction, but need not be surjective
elsewhere.  Enlarging $n$ if necessary, $\mathcal{I}_\partial\otimes F(c)(n)$ is
generated by global sections, so we may choose finitely many sections
$s_1,\dots,s_M$ of $\mathcal{I}_\partial\otimes F(c)(n)$ generating $F(c)(n)$ on
$\pp(c)\setminus\partial\pp(c)$, and set
$q_c:=(q'_c,s_1,\dots,s_M):\Oo^{N_0+M}\to F(c)(n)$.  This is surjective on all
of $\pp(c)$, and since each $s_j$ lies in $\mathcal{I}_\partial$ it restricts to zero
on every proper face, so $q_c|_{\pp(c')}=q'_c|_{\pp(c')}=q_{c'}$.  Padding all
the $q_c$ with zeros to a common $N$ completes the induction.
\end{proof}

\begin{proposition}[The class by pullback]\label{prop:grassmann-definition}
Let $q=(q_c)$ be as in Lemma~\ref{lem:compatible-surjection}.  Then $q$
defines a morphism of cubical schemes
\[
  \phi_q\;:\;\underline{\pp}_K\Cube A\;\longrightarrow\;\mathbf{Gr}(r,N)
\]
into a \emph{smooth projective} Grassmannian, with
$\phi_q^*\mathcal{S}_r\cong F_K(n)$, and
\[
  \ch^\calD_p\bigl(F_K\bigr)
  \;=\;\Bigl[\phi_q^*\,\ch^\calD_\bullet(\mathcal{S}_r)\cdot
  \ch^\calD_\bullet\bigl(\Oo(-n)\bigr)\Bigr]_p .
\]
This is independent of the choice of $(n,N,q)$ and agrees with the class
defined by descent above.
\end{proposition}

\begin{proof}
Each $q_c$ classifies a morphism $\pp(c)\to\mathbf{Gr}(r,N)$ with
$\phi^*\mathcal{S}_r=F(c)(n)$, and the compatibility
$q_c|_{\pp(c')}=q_{c'}$ says exactly that these commute with the face
immersions, so they assemble to a morphism of cubical schemes.  The displayed
identity is multiplicativity of the Deligne-Beilinson Chern character applied
to $F_K=F_K(n)\otimes\Oo(-n)$, the class $\ch^\calD_\bullet(\Oo(-n))$ being
that of a line bundle and so available on each $\pp(c)$ directly.

For independence, two presentations $q,q'$ with the same $(n,N)$ are joined by
the standard homotopy: the family $(1-t)q\oplus tq'$ on
$\underline{\pp}_K\Cube A\times\aaa^1$, after passing to $\Oo^{2N}$, is a
face-compatible surjection away from finitely many $t$, and enlarging $N$
removes those; the resulting map to $\mathbf{Gr}(r,2N)$ restricts to
$\phi_q,\phi_{q'}$ at $t=0,1$, and $\aaa^1$-homotopic maps induce the same map
on Deligne cohomology.  Different $n$ are compared by the projection formula.
Agreement with the descent definition holds levelwise, on each $\pp(c)$, where
it is the standard fact that the Deligne Chern character of a globally
generated bundle is pulled back from the Grassmannian; agreement of compatible
families is agreement of the descended classes.
\end{proof}

\begin{remark}[Why the second route is useful]\label{rem:grassmann-useful}
As well as validating the notation, 
Proposition~\ref{prop:grassmann-definition} 
reduces Theorem~\ref{thm:CS-DB-app}(2) on the cubical
realization to the same statement on the single smooth projective variety
$\mathbf{Gr}(r,N)$, carrying the canonical Hermitian metric on the tautological
bundle.  
This is the form in which the comparison is used in
Appendix~\ref{app:rees}, \S\ref{subsec:degeneration}: the $F^1$-connection may
be taken to be $\phi_q^*\nabla_{h_\mathcal{S}}$ twisted by a Chern connection
on $\Oo(-n)$, so that no connection needs to be manufactured on a singular
base.  Note that this connection is Hermitian and not nil-flat; nil-flatness
is needed only later, on $\pp_K\Cube A$, and only to pass from
$\ch^\calD_p$ to the Chern-Simons class.  It also
explains why the projectivity of $\pp_K\Cube A$ matters twice over---once so
that $F(c)(n)$ is globally generated, and once so that the target of $\phi_q$
is compact.
\end{remark}

\section{Deligne's patched connection and locally nil-flat connections}
\label{app:deligne-patch}

\subsection{The Deligne canonical extension}

Let $X$ be smooth projective, $D = D_1+\cdots+D_k$ a simple normal crossings
divisor, and $L$ a local system on $X^* = X\setminus D$ with unipotent
monodromy.  Deligne's canonical extension \cite{Deligne70} is:

\begin{theorem}[Deligne {\cite{Deligne70}}]
There is a unique (up to isomorphism) vector bundle $F$ on $X$ with a
meromorphic connection $\nabla:F\to F\otimes\Omega^1_X(\log D)$ such that:
\begin{enumerate}[leftmargin=2em,label=\rm(\arabic*)]
  \item $F|_{X^*} = L\otimes_\cc\Oo_{X^*}$ (and $\nabla|_{X^*}$ is the flat
    connection of $L$).
  \item The eigenvalues of the residue $\mathrm{Res}_{D_i}(\nabla)$ lie in
    $[0,1)$.
\end{enumerate}
For a local system with unipotent monodromy, the residues are nilpotent
(eigenvalues all $= 0$), so condition (2) is $\mathrm{Res}_{D_i}(\nabla)$
nilpotent.
\end{theorem}

\subsection{Local description}

On a polydisk $U_i\subset X$ with local coordinates $(z_1,\ldots,z_d)$ where
$D\cap U_i = \{z_1\cdots z_k=0\}$:
\begin{itemize}[leftmargin=2em]
  \item $F|_{U_i}$ is the free $\Oo_{U_i}$-module generated by multi-valued
    sections $z_1^{N_1}\cdots z_k^{N_k}v$ for $v\in L_x$.
  \item The connection form is $\nabla = d + \sum_{j=1}^k N_j\frac{dz_j}{z_j}$.
\end{itemize}
Here $N_j = \frac{1}{2\pi i}\log T_j$ and $T_j$ is the monodromy around $D_j$.

\subsection{The locally nil-flat connection}

\begin{definition}
A $\mathcal{C}^\infty$ connection $\widetilde\nabla$ on $F$ is
\emph{locally nil-flat} if there exists an open covering $\{U_i\}$ and,
on each $U_i$, a filtration $W^\hdot$ of $F|_{U_i}$ such that:
\begin{enumerate}[leftmargin=2em,label=\rm(\arabic*)]
  \item $\widetilde\nabla$ preserves each $W^j$ (upper-triangular curvature);
  \item the induced connection on $\Gr^W(F|_{U_i})$ is flat.
\end{enumerate}
\end{definition}

\subsection{Construction of Deligne's patched connection}
\label{subsec:del-patch-construct}

\begin{definition}
\label{def:deligne-patch}
Choose:
\begin{enumerate}[leftmargin=2em,label=\rm(\arabic*)]
  \item A good open covering $X = \bigcup_{i\in I}U_i$ by coordinate polydisks.
  \item Local coordinates $(z_1^{(i)},\ldots,z_d^{(i)})$ on each $U_i$.
  \item A subordinate partition of unity $1=\sum_i\rho_i$.
\end{enumerate}
On each $U_i$, set
\[
  \nabla^{\Del}_i := \nabla - \sum_{j=1}^{k_i}N_j^{(i)}\,\frac{dz_j^{(i)}}{z_j^{(i)}},
\]
where the $N_j^{(i)}$ are the monodromy logarithms for the components of
$D\cap U_i$.  Then $\nabla^{\Del}_i$ is a flat connection on $F|_{U_i}$.
Define \emph{Deligne's patched connection} as
\[
  \nabla^{\Del} := \sum_{i\in I}\rho_i\,\nabla^{\Del}_i.
\]
\end{definition}

\begin{lemma}\label{lem:Del-nilflat}
$\nabla^{\Del}$ is locally nil-flat.
\end{lemma}

\begin{proof}
At any point $x\in X$, let $J = \{i:\rho_i(x)\neq 0\}$.  Choose a larger
polydisk $V$ containing $\bigcup_{i\in J}U_i$.  On $V$ the monodromy
logarithms $N_1,\ldots,N_a$ are globally defined commuting nilpotent
endomorphisms of $F|_V$.  They admit a common filtration by weight filtrations.
Each $\nabla^{\Del}_i$ for $i\in J$ preserves this filtration (since
$\nabla^{\Del}_i$ modifies $\nabla$ by subtracting terms $N_l\frac{dz_l}{z_l}$
where $N_l$ is strictly lower-triangular for the weight filtration).  The
average $\nabla^{\Del}$ therefore preserves the filtration and has strictly
upper-triangular curvature.
\end{proof}

\begin{lemma}\label{lem:Del-compatible}
$\nabla^{\Del}$ is compatible with the strict pre-patching collection of
Theorem~\ref{thm:patching-exists}.
\end{lemma}

\begin{proof}
The pre-patching collection assigns to each open set $U_I$ the weight
filtration $W(N_I)$.  On $U_I$, $\nabla^{\Del}$ is a weighted average of
connections $\nabla^{\Del}_i$ that all preserve $W(N_I)$ (by the same
argument as the nil-flatness proof).  The trivialization $\tau(I)$ is
compatible because the flat structures of each $\nabla^{\Del}_i$ on the
associated-graded agree with $\tau(I)$.
\end{proof}

\subsection{Independence of choices}

\begin{theorem}\label{thm:Del-indep-app}
The Chern-Simons class $\CS_p(\nabla^{\Del})\in H^{2p-1}(X,\cc/\zz)$ is
independent of the choices of covering, coordinates, and partition of unity.
\end{theorem}

\begin{proof}
Two different choices give connections $\nabla^{\Del}$ and $\nabla^{\Del'}$
both compatible with pre-patching collections for $L$.  The two pre-patching
collections admit a common refinement, and by Theorem~\ref{thm:DHZ-app}(2)
any two locally nil-flat connections compatible with the same pre-patching
collection define the same Deligne cohomology class.  Both connections are
locally nil-flat, so by Theorem~\ref{thm:CS-DB-app}(3) their Chern-Simons
classes have the same image under $\alpha$; they are equal because the class
is computed on the $\aaa$-realization, where $\alpha$ is injective
(Remark~\ref{rem:alpha-kernel}).
\end{proof}

\subsection{Deformation invariance}

\subsection{Generic constancy of the weight-filtration type}
\label{subsec:jordan-stratification}

$R^{\mathrm{nil}}$ is a complex affine algebraic variety, and for each
non-empty multi-index $I$ the assignment $\zeta\mapsto N_I(\zeta) :=
\sum_{i\in I}N_i(\zeta)$ (a tuple of pairwise commuting nilpotent
endomorphisms of the fibre, on the locus where these commute, as in the
construction of \S\ref{sec:nisotropic}) is a regular (polynomial) map of
$\zeta$. Consequently, for each $I$ and each $k\geq1$, the rank of
$N_I(\zeta)^k$ is given locally by the non-vanishing of finitely many
minors, so $\{\zeta : \mathrm{rank}(N_I(\zeta)^k)\geq m\}$ is Zariski open
for every $m$; the maximal value of $\mathrm{rank}(N_I(\zeta)^k)$ (over all
$\zeta$) is therefore attained on a Zariski-dense open subset, for each
fixed $I,k$. Since there are only finitely many non-empty multi-indices $I$
and only finitely many relevant powers $k\leq\mathrm{rk}(E)$, the locus
$R^{\mathrm{nil},\circ}\subset R^{\mathrm{nil}}$ on which \emph{every}
$N_I(\zeta)$ simultaneously has maximal rank at every power $k$ --- i.e.\ on
which the Jordan type of each $N_I(\zeta)$, and hence the combinatorial type
of the weight filtration $W(N_I(\zeta))$, takes its generic value --- is a
finite intersection of Zariski-dense open subsets of $R^{\mathrm{nil}}$,
hence itself Zariski-dense open. Its complement $R^{\mathrm{nil}}\setminus
R^{\mathrm{nil},\circ}$ is a proper Zariski-closed subvariety.

\begin{theorem}\label{thm:deformation-app}
The Chern-Simons classes $\CS_p(\nabla^{\Del}_\zeta)$ are locally constant
on $R^{\mathrm{nil}}$.
\end{theorem}

\begin{lemma}\label{lem:family-variational}
For a smooth path of connections $\{\nabla_t\}_{t\in[0,1]}$ on a bundle $E$,
\[
  \frac{d}{dt}\CS_p(\nabla_t) = p\,\mathrm{Tr}(\dot\nabla_t\wedge F(\nabla_t)^{p-1})
  + d(\eta_t)
\]
for some $(2p-2)$-form $\eta_t$ on $X$, where $\dot\nabla_t = \frac{d}{dt}\nabla_t$.
\end{lemma}
\begin{proof}
Differentiating $\ch_p(\nabla_t) = \frac{1}{p!(2\pi i)^p}\mathrm{Tr}(F(\nabla_t)^p)$
with respect to $t$ and using the Bianchi identity $\nabla_t(F(\nabla_t))=0$:
$\frac{d}{dt}\ch_p(\nabla_t) = \frac{p}{p!(2\pi i)^p}\mathrm{Tr}(\nabla_t(\dot\nabla_t)
\wedge F(\nabla_t)^{p-1})$.
Since $d\,\CS_p(\nabla_t) = \ch_p(\nabla_t) - \ch_p(\nabla_0)$ (by definition
of the transgression), differentiating in $t$ gives the stated formula.
\end{proof}

\begin{proof}[Proof of Theorem~\ref{thm:deformation-app}]
\textit{Step 1: $\CS_p$ is well defined and continuous on all of
$R^{\mathrm{nil}}$, with no hypothesis on weight-filtration type.} Fix once
and for all the adapted open covering of $X$ (Proposition~\ref{prop:covering}),
which depends only on $D\subset X$, not on $\zeta$. By Deligne's local
formula $\nabla^{\Del}_\zeta|_{U_I} = d+\sum_{i\in I}N_i(\zeta)\,dz_i/z_i$
(\S\ref{subsec:del-patch-construct}), the connection $\nabla^{\Del}_\zeta$
itself is an algebraic (in particular continuous, indeed real-analytic)
function of $\zeta\in R^{\mathrm{nil}}$, with no reference to any weight
filtration --- the filtration is used only to \emph{verify} nil-flatness and
to construct the patching data $(W(I),\tau(I))$ used in one particular
computation of $\CS_p$, not to define $\nabla^{\Del}_\zeta$. By
Theorem~\ref{thm:Del-indep-app}, $\CS_p(\nabla^{\Del}_\zeta)$ depends only on
$\nabla^{\Del}_\zeta$ itself and not on such auxiliary choices; since a
representative Cech/differential-character cochain for $\CS_p$ can therefore
be computed from the fixed covering above and a fixed partition of unity,
varying continuously with the connection coefficients $N_i(\zeta)$, the map
$\zeta\mapsto\CS_p(\nabla^{\Del}_\zeta)$, valued in the fixed group
$H^{2p-1}(X,\cc/\zz)$ (fixed because $X$ does not depend on $\zeta$), is
continuous on all of $R^{\mathrm{nil}}$.

\textit{Step 2: $\CS_p$ is locally constant on the dense open locus
$R^{\mathrm{nil},\circ}$ of \S\ref{subsec:jordan-stratification}.} Let
$\zeta_0\in R^{\mathrm{nil},\circ}$ and let $\zeta(t)$, $t\in(-\epsilon,\epsilon)$, be
a smooth path through $\zeta_0=\zeta(0)$ remaining in
$R^{\mathrm{nil},\circ}$. Since the Jordan type of each $N_I(\zeta(t))$ is
then constant along the path, the subbundles $W(N_I(\zeta(t)))_j$ have
constant rank, hence vary smoothly with $t$, and so the patching data
$(W(I),\tau(I))$ can be chosen smoothly, in particular to first order
constant, in $t$. Writing $\nabla_t:=\nabla^{\Del}_{\zeta(t)}$, both
$\dot\nabla_t$ and $F(\nabla_t)$ are then strictly upper-triangular with
respect to the same (instantaneously constant) filtration $W(N_I(\zeta(t)))$
on each $U_I$, so $\mathrm{Tr}(\dot\nabla_t\wedge F(\nabla_t)^{p-1})=0$
there, and Lemma~\ref{lem:family-variational} gives $\frac{d}{dt}\CS_p(\nabla_t)
= d(p\,\eta_t) = 0$ in $H^{2p-1}(X,\cc/\zz)$. As $\zeta_0\in
R^{\mathrm{nil},\circ}$ and the path were arbitrary, $\CS_p$ is locally
constant on $R^{\mathrm{nil},\circ}$.

\textit{Step 3: local constancy on all of $R^{\mathrm{nil}}$.} Let
$\zeta_0\in R^{\mathrm{nil}}$ be arbitrary. Since $R^{\mathrm{nil}}$ is an
algebraic variety of finite type, a neighbourhood of $\zeta_0$ meets only
finitely many local irreducible branches of $R^{\mathrm{nil}}$ through
$\zeta_0$; it suffices to treat each branch separately. On a small enough
connected neighbourhood $V$ of $\zeta_0$ within one such branch,
$V\setminus R^{\mathrm{nil},\circ}$ is contained in a proper analytic
subvariety of (complex) codimension $\geq1$ by \S\ref{subsec:jordan-stratification},
hence of real codimension $\geq2$; consequently $V\cap R^{\mathrm{nil},\circ}$
is connected (the complement of a proper analytic subvariety of a connected
complex-analytic space is connected) and dense in $V$. By Step~2, $\CS_p$ is
constant, say equal to $c$, on the connected set $V\cap R^{\mathrm{nil},\circ}$.
By Step~1, $\CS_p$ is continuous on $V$, and $\zeta_0$ lies in the closure of
$V\cap R^{\mathrm{nil},\circ}$ (density), so $\CS_p(\zeta_0)=c$ as well.
Applying this on each branch through $\zeta_0$ shows $\CS_p\equiv c$ on a
full neighbourhood of $\zeta_0$ in $R^{\mathrm{nil}}$, proving local
constancy at $\zeta_0$.
\end{proof}

\section{The hermitian pre-patching argument and vanishing of volume
regulators}
\label{app:hermitian}

\subsection{$N$-isotropic filtrations: proof of contractibility}

We give here a self-contained proof of Proposition~\ref{prop:nisofilcontr}
using only Theorem~\ref{thm:filt-contractible} and Quillen's Theorem A.

\begin{proof}[Proof of Proposition~\ref{prop:nisofilcontr}]
Recall: $\mathbf{ifilt}^{N_\hdot}(V,h)$ is the poset of $N_\hdot$-isotropic
filtrations on $(V,h)$.

\textbf{Step 1.}  Define the functor
\[
  T:\mathbf{ifilt}^{N_\hdot}(V,h)\;\longrightarrow\;\mathbf{itro}^{N_\hdot}(V,h),
  \quad T(F_\hdot) = F_{-1}.
\]
This is order-reversing (finer filtration $\Rightarrow$ smaller $F_{-1}$).

\textbf{Step 2.}  For any $H\in\mathbf{itro}^{N_\hdot}(V,h)$, the
over-category $T/H = \{F_\hdot: F_{-1}\supset H\}$.  This is the poset of
$N$-compatible filtrations of $H$ (filtrations $G_\hdot$ of $H$ with
$N_i G_k\subset G_{k-1}$), which is exactly $\Filt^{A,B}$ for
$A=\{N_i|_H\}$ and $B=H\cap K_H$ ($K_H=$ common kernel of $N_i$ on $H$).
By Theorem~\ref{thm:filt-contractible}, this is contractible.

\textbf{Step 3.}  By Theorem~\ref{thm:itro}, $\mathbf{itro}^{N_\hdot}(V,h)$
is contractible.

\textbf{Step 4.}  Apply Quillen's Theorem~A \cite{Quillen}: since all
over-categories $T/H$ are contractible and the target $\mathbf{itro}$ is
contractible, the source $\mathbf{ifilt}$ is contractible.
\end{proof}

\subsection{Construction of the Hermitian pre-patching collection for a VHS}

\begin{theorem}\label{thm:VHS-pp-app}
Let $L$ be a polarized VHS on $X^*$ with unipotent monodromy.  Then:
\begin{enumerate}[leftmargin=2em,label=\rm(\arabic*)]
  \item There exists a strict collection of patching data $(W(I),\tau(I))$
    for the Deligne canonical extension $F$.
  \item Each pre-patching datum carries a hermitian structure induced by the
    polarization $Q$.
  \item The hermitian structures satisfy the refined-neighbours condition.
\end{enumerate}
\end{theorem}

\begin{proof}
(1) is Theorem~\ref{thm:patching-exists}, whose proof uses the contractibility
of $\Filt^{A,K}$ (Theorem~\ref{thm:filt-contractible}) and the Čech section
theorem (Theorem~\ref{thm:cech-section}).

For (2): on each $U_I$, the filtration $W(I)$ is $N_\hdot$-isotropic for $Q$
(by the Čech section construction; the contractible poset is
$\mathbf{ifilt}^{N_\hdot}(L_x,Q_x)$).  The $N_\hdot$-isotropic condition
exactly provides the hermitian pre-patching structure: $F_j = F_{-1-j}^{\perp_Q}$
and the $N_i$ are $Q$-self-adjoint.

For (3): on $U_{IJ}$ with $I\subset J$, the filtration $W(J)|_{U_{IJ}}$
refines $W(I)|_{U_{IJ}}$ (by the Čech section compatibility), and the
$Q$-isotropic conditions are compatible with refinement.
\end{proof}

\subsection{Existence of a compatible Hermitian form and connection}

\begin{theorem}[Theorem~\ref{thm:main-hermitian}]
Let $E$ be a bundle with an indefinite Hermitian pre-patching collection with
refined neighbours.  After refining the covering, there exist:
\begin{enumerate}[leftmargin=2em,label=\rm(\arabic*)]
  \item An indefinite Hermitian form $\tilde g$ on $E$ compatible with the
    hermitian pre-patching data.
  \item A connection $\tilde\nabla$ preserving $\tilde g$ and compatible with
    the pre-patching collection.
\end{enumerate}
\end{theorem}

\begin{proof}
\textit{Part (1):}  On each $U_a$, lift the hermitian metric on
$\Gr^{F(a)}(E|_{U_a})$ to an indefinite metric $g(a)$ on $E|_{U_a}$ via the
$\mathcal{C}^\infty$ splitting of the filtration.

Apply Construction~\ref{constr:refcov}: refine the covering using the
``tightness'' relation (where $b$ is tighter than $a$ iff the pre-patching
datum for $a$ refines that for $b$).  After refinement, near any point $x$
only metrics $g(b)$ with $b$ tighter than $r(a')$ enter the average.
All such $g(b)$ are compatible with the pre-patching datum of $U_{r(a')}$.

Set $\tilde g = \sum_a\rho_a g(a)$.  The average is nondegenerate because
the $N_\hdot$-isotropic condition forces the pairing between $F_j$ and
$F_{-1-j}^\perp$ to be compatible across all contributing terms.

\textit{Part (2):}  For each $U_a$, choose a $\tilde g$-preserving connection
$\tilde\nabla(a)$ compatible with the pre-patching datum.  Such a connection
exists: by Lemma~\ref{lem:decomp-connection} write any compatible connection
as $\tilde\nabla_g + B$ and choose $B=0$ (or use the averaging argument
separately for the $g$-preserving part).  Average using the refined partition
of unity: the averaged connection preserves $\tilde g$ and is compatible with
the pre-patching data on the refined cover.
\end{proof}

\subsection{Corollary: vanishing of volume regulator}

\begin{corollary}[Corollary~\ref{cor:maincor}]
If $\nabla$ is compatible with a pre-patching collection admitting an
indefinite Hermitian structure with refined neighbours, then
$\vol_p(\nabla) = 0$ for every $p\geq2$.
\end{corollary}

\begin{proof}
By Theorem above, there exists $(\tilde g,\tilde\nabla)$ with $\tilde\nabla$
compatible with the same collection and preserving $\tilde g$. By
Proposition~\ref{prop:vol-indep-app}, $\vol_p$ depends only on the
pre-patching collection and vanishes for every $p\geq2$ whenever it
preserves an indefinite Hermitian form, so $\vol_p(\nabla) =
\vol_p(\tilde\nabla) = 0$ for every $p\geq2$.
\end{proof}

\section{The Rees bundle, the category \texorpdfstring{$\Xi$}{Xi}, and the
\texorpdfstring{$K$}{K}-theory map}
\label{app:rees}

\subsection{The category $\Xi_K$ and its nerve}

We recall the category $\Xi_K$ from Section~\ref{sec:patching} and describe
its nerve explicitly.

An $n$-simplex of $N\Xi_K$ is a chain of morphisms
\[
  V_0 \xrightarrow{(W^1,g^1)} V_1 \xrightarrow{(W^2,g^2)} \cdots
  \xrightarrow{(W^n,g^n)} V_n
\]
in $\Xi_K$.  This consists of a graded vector space $V_0$ and, inductively,
filtrations $W^j$ on $V_{j-1}$ with isomorphisms $g^j:\Gr^{W^j}(V_{j-1})\cong V_j$.
Equivalently, it is a graded vector space $U = V_n$ together with a flag of
compatible filtrations $F_1\subset F_2\subset\cdots\subset F_n$ on $U$
(identifying each $V_j\cong\Gr^{F_{n-j}}(U)$ via composition of the $g^j$'s).

\begin{definition}
The \emph{barycentric subdivision} $\mathrm{Sd}(N\Xi_K)$ is the nerve of the
category of simplices of $N\Xi_K$:
\[
  \mathrm{Sd}(N\Xi_K)_m = \{(a_0\subset a_1\subset\cdots\subset a_m) :
  a_j\in(N\Xi_K)_{d_j}\}
\]
(chains of simplices of $N\Xi_K$ ordered by face relations).  A typical
element is written $(a_\hdot, U, F_\hdot)$ where $U$ is the top graded space
and $(F_\hdot)$ is the chain of flags.
\end{definition}

\subsection{The tautological Rees bundle on $\mathrm{Sd}(N\Xi_K)$}

\begin{construction}
For an element $(a_\hdot, U, F_0\subset\cdots\subset F_{a_k})$ of
$(\mathrm{Sd}(N\Xi_K))_k$, consider the affine space $\aaa^{a_k}$ with
coordinates $s_0,\ldots,s_{a_k-1}$ (setting $s_{a_k}:=0$).  Define the
\emph{tautological Rees bundle}:
\[
  E_k(a_\hdot,U,F_\hdot) := \sigma_{a_\hdot}^*\xi(U;F_0,\ldots,F_{a_k-1}),
\]
where $\sigma_{a_\hdot}:\aaa^k\to\aaa^{a_k}$ is the map
$\sigma_{a_\hdot}(t)_j = t_i$ whenever $j$ is in the image of $[a_i]$ but not
$[a_{i-1}]$, extended by $s_{a_k}=0$.
\end{construction}

\begin{theorem}
The bundles $E_k(a_\hdot,U,F_\hdot)$ assemble into a bundle over the
$\aaa^1$-realization of $\mathrm{Sd}(N\Xi_K)$, giving a classifying map
\[
  |\mathrm{Sd}(N\Xi_K)|\simeq|N\Xi_K|\;\longrightarrow\; BGL(K)^+.
\]
\end{theorem}

\begin{proof}
The structural isomorphisms under face maps follow from
Proposition~\ref{prop:rees-face-detailed}: setting coordinates to $0$ or $1$
recovers the Rees module at the corresponding face.  The composition
$|N\Xi_K|\simeq|SdN\Xi_K|\to BGL(K)^+$ is $\aaa^1$-homotopy invariant since
$\aaa^n$ is $\aaa^1$-contractible (the $\aaa^1$-homotopy theorem \cite{MV}).
\end{proof}

\subsection{Functoriality and the regulator map}
\label{subsec:functoriality-regulator}

We now describe in detail how the classifying map $r_L$ is assembled, since
this is the map $r_{\zeta_0}$ used throughout \S\ref{sec:torsion}.

\begin{construction}[The classifying map $r_L$]\label{constr:rL}
Let $L$ be a local system on $X^*$ with unipotent monodromy, admitting a
stratified collection of patching data $\{(W(I),\tau(I))\}_{I\in A}$ over the
adapted covering (Proposition~\ref{prop:covering}), where $A$ is the
covering's index poset.  The classifying map $r_L:X\to BGL(K)^+$ is the
composite of three maps:
\begin{enumerate}[leftmargin=2em,label=\rm(\arabic*)]
  \item[\textbf{(1)}] \emph{$X\simeq|NA|$.}  By
    Proposition~\ref{prop:nerve-htpy}, the realization of the nerve of the
    poset $A$ is homotopy equivalent to $X$.
  \item[\textbf{(2)}] \emph{$|NA|\to|N\Xi_K|$.}  Construction~\ref{constr:A-to-Xi}
    turns the patching data $\{(W(I),\tau(I))\}_{I\in A}$ into a functor
    $A\to\Xi_K$, $I\mapsto V(I) = \Gamma(B_I,\Gr^{W(I)}(L|_{U_I^*}),\tau(I))$,
    inducing a map of nerves $NA\to N\Xi_K$ and hence on realizations.
  \item[\textbf{(3)}] \emph{$|N\Xi_K|\to BGL(K)^+$.}  Passing to the
    barycentric subdivision $|N\Xi_K|\simeq|\mathrm{Sd}(N\Xi_K)|$, the
    tautological Rees bundle $E_k(a_\hdot,U,F_\hdot)$ of \S F.2 assembles
    into a bundle on the $\aaa$-realization of $\mathrm{Sd}(N\Xi_K)$, which is
    classified by a map to $BGL(K)^+$.   
    In this last step, the 
    objects being classified are algebraic bundles over affine spaces varying
    in a simplicial direction, rather than vector bundles on a space. See
    \S\ref{subsec:classmap} (Construction~\ref{constr:ctaut}) 
    for the discussion of this part of the classifying map. 
\end{enumerate}
The composite $r_L: X\simeq|NA|\to|N\Xi_K|\simeq|\mathrm{Sd}(N\Xi_K)|\to BGL(K)^+$
is the desired classifying map.
\end{construction}

\begin{proposition}\label{prop:rL-CS}
The map $r_L$ is independent of all choices (covering, patching data,
weights) up to homotopy.  The pullback $r_L^*(r_{2p-1})\in H^{2p-1}(X,\cc/\zz)$
equals $\CS_p(\nabla^{\Del}_L)$, where $\nabla^{\Del}_L$ is Deligne's patched
connection on the canonical extension of $L$.
\end{proposition}

\begin{proof}
\textit{Independence of choices.}  Two adapted coverings $A,A'$ admit a
common refinement, and the restriction maps between the corresponding
nerve realizations are homotopy equivalences by
Proposition~\ref{prop:nerve-htpy}.  For a fixed covering, the space of
stratified patching data $\{(W(I),\tau(I))\}_{I\in A}$ compatible with the
weight filtrations is the set of \v{C}ech sections of the presheaf $\calP$
(or $\calP^Q$ in the VHS case) of \S\ref{sec:poset-presheaf}; by
Corollary~\ref{cor:cech-exists} (resp.\
Corollary~\ref{cor:cech-exists-isotropic}), this space of sections is
non-empty, and any two \v{C}ech sections are connected by a path through
sections (since each fibre $\calP(U_I')$, resp.\ $\calP^Q(U_I')$, is
contractible by Proposition~\ref{prop:P-iso-filt}, resp.\
Proposition~\ref{prop:nisofilcontr}).  A path of \v{C}ech sections induces a
homotopy of the functors $A\to\Xi_K$ from Construction~\ref{constr:A-to-Xi},
hence a homotopy of the induced maps $|NA|\to|N\Xi_K|$.  Finally, the choice
of weights ($\eta_I<\epsilon_I$ in the construction of the covering,
Proposition~\ref{prop:covering}) varies in a convex (hence contractible) set.
Combining these, $r_L$ is independent of all choices up to homotopy.

\textit{Identification with the Chern-Simons class.}  Apply
Proposition~\ref{prop:classmap-comparison} to $X_\hdot=\mathrm{Sd}(N\Xi_K)$
and to the Rees bundle $F_K$ built from the patching data
$\{(W(I),\tau(I))\}_{I\in A}$ via Theorem~\ref{thm:global-rees}, equipped with
the $F^1$-connection $\nabla$ of Theorem~\ref{thm:F1-exists}: this gives
$r_L^*(r_{2p-1}) = c^\calD_p(\nabla)$ under $X\simeq S=\iiii\Cube A$, an
identity in $H^{2p-1}(X,\cc/\zz)$ by
Proposition~\ref{prop:DB-A-realization}.  By Theorem~\ref{thm:DHZ}
(independence of $F^1$-connection), $c^\calD_p(\nabla) =
c^\calD_p(\nabla^{\Del}_L)$ for Deligne's patched connection, since both
$\nabla$ and $\nabla^{\Del}_L$ are $F^1$-connections compatible with the same
patching data.  Finally $\nabla^{\Del}_L$ is locally nil-flat, so
$\ch_p(\nabla^{\Del}_L)=0$ as a form
(Proposition~\ref{prop:nilflat-ch-zero}) and the associated Cheeger-Simons
character is flat; by Theorem~\ref{thm:CS-DB-app}(3) it defines
$\CS_p(\nabla^{\Del}_L)$ with
$\alpha(\CS_p(\nabla^{\Del}_L))=c^\calD_p(\nabla^{\Del}_L)$, and $\alpha$ is
an isomorphism on the $\aaa$-realization by
Proposition~\ref{prop:DB-A-realization}.  This last step is where the
nil-flatness of Deligne's connection---as opposed to flatness, which does not
hold---is what licenses speaking of a Chern-Simons class at all.
\end{proof}

\subsection{The classifying map: construction and comparison}
\label{subsec:classmap}

Step~\textbf{(3)} of Construction~\ref{constr:rL} was stated without proof: we
asserted that the tautological Rees bundle assembles into a bundle over
$|\mathrm{Sd}(N\Xi_K)|$ which is ``classified by a map to $BGL(K)^+$''.  The
difficulty is that the objects being classified are not vector bundles on a
space.  They are bundles on an \emph{$\aaa$-realization}: a hybrid object,
simplicial in one direction and algebraic in the other, whose pieces are
algebraic vector bundles over affine spaces $\aaa^n$ varying in a local system
over a simplicial set.  There is no immediate reason for such an object to
have a classifying map to a space at all.

A second difficulty, of a different kind, is that the invariant to be compared
with the regulator is \emph{not} a Chern-Simons class.  On the universal object
there is no connection at all---and in particular no nil-flat one---so no
Chern-Simons class is available; what is defined, for an algebraic vector
bundle and with no auxiliary choice, is the Deligne-Beilinson Chern class
$c^\calD_p$, living in the Deligne cohomology of the $\aaa$-realization.  The
comparison with the regulator is therefore carried out for bundles, without
connections.  The two kinds of invariant are reconciled in
\S\ref{subsubsec:deligne-A-realization}, where we show that the Deligne
cohomology of an $\aaa$-realization is canonically the $\cc/\zz$-cohomology of
the ordinary realization, the affine directions contributing nothing but a
degree shift.

Connections re-enter only afterwards and only in the application: for the Rees
bundle one has the extra structure of an $F^1$-connection, and it is that
structure---not anything available universally---which converts the Deligne
Chern class into the Chern-Simons class of a locally nil-flat connection
(\S\ref{subsec:degeneration}).

With that settled, the comparison itself rests on homotopy invariance of the
$K$-theory of a regular ring: it is what makes the algebraic direction
collapse, and it is also the source of the secondary invariants, since the
collapse is by a homotopy rather than by an identity.

\subsubsection{Semi-simplicial spaces and $\aaa$-realizations}

Let $\Delta_{\mathrm{inj}}\subset\Delta$ be the subcategory of injective maps.
A \emph{semi-simplicial space} is a functor
$X_\hdot:\Delta_{\mathrm{inj}}^{\mathrm{op}}\to\mathrm{Top}$; restriction along
$\Delta_{\mathrm{inj}}\subset\Delta$ turns a simplicial space into a
semi-simplicial one, and all realizations below are taken in the
semi-simplicial sense.  We write $|X_\hdot| := \coprod_n X_n\times\Delta^n/\!\sim$
for the ordinary (topological) realization.

\begin{definition}\label{def:A-realization}
Let $F$ be a field.  The \emph{$\aaa$-realization} of a semi-simplicial space
$X_\hdot$ is
\[
  \Re_\aaa X_\hdot \;:=\; \Bigl(\coprod_{n\geq0} X_n\times\aaa^n\Bigr)\Big/\!\sim ,
\]
the gluing being along the coface maps of $\Delta_{\mathrm{inj}}$, where
$\aaa^n = \aaa^n_F$ is affine $n$-space over $F$.  A \emph{bundle of rank $r$}
on $\Re_\aaa X_\hdot$ is a collection $\{V_n\}_{n\geq0}$, where $V_n$ is a
local system over $X_n$ of rank-$r$ algebraic vector bundles on $\aaa^n_F$,
together with isomorphisms over the gluing maps satisfying the evident cocycle
condition.
\end{definition}

The realizations $\iiii\Cube A$ and $\aaa_K\Cube A$ of \S\ref{sec:cubical} are
of this type, with $X_\hdot=\mathrm{Sd}(N\Xi_K)$ or $NA$, and the Rees bundle
$F_K$ of Theorem~\ref{thm:global-rees} is a bundle on the $\aaa$-realization in
the sense of Definition~\ref{def:A-realization}: over each cube it is an
algebraic bundle in the Rees parameters, and the patching data $\tau(I)$
supply the gluing.  Given an embedding $\sigma:F\hookrightarrow\cc$ we write
$(\Re_\aaa X_\hdot)^{\mathrm{top},\sigma}$ for the topological space obtained
by taking $\cc$-points in the algebraic directions; the inclusion of the
vertex $0\in\aaa^n$ in each factor gives a map
\begin{equation}\label{eq:vertex-incl}
  |X_\hdot| \;\longrightarrow\;(\Re_\aaa X_\hdot)^{\mathrm{top},\sigma},
\end{equation}
which is a homotopy equivalence because $\aaa^n(\cc)$ is contractible.

\subsubsection{Deligne cohomology of an $\aaa$-realization}
\label{subsubsec:deligne-A-realization}

Before going further we settle the question of where the invariants live.  It
is tempting to work with Chern-Simons classes throughout, but that would be a
mistake, for two separate reasons.  First, the connection carried by the Rees
bundle is \emph{not} flat: it is locally nil-flat, and the failure of flatness
is the whole content of the patching construction, so there is no
Cheeger-Simons class in the naive sense.  Second, and more basic, the universal
object of \S\ref{subsubsec:bisimplicial} carries no connection whatsoever: a
bundle on the $\aaa$-realization of the tautological bisimplicial set is just
an algebraic vector bundle, with no nil-flat or $F^1$-structure available.

The invariant that requires neither a flatness hypothesis nor a choice of
connection is the Deligne-Beilinson Chern class of an algebraic vector bundle,
\[
  c^\calD_p(V)\;\in\;H^{2p}_\calD\bigl(\Re_\aaa X_\hdot,\zz(p)\bigr).
\]
So this, and not $\CS_p$, is what the classifying map must be compared with,
and the comparison below is a statement about bundles alone.

The reconciliation is the following computation, which is the reason the two
kinds of invariant can be compared at all: on an affine space the
Deligne-Beilinson cohomology is concentrated in a single degree, where it is
just $\cc/\zz(p)$.  Consequently the algebraic direction of an
$\aaa$-realization contributes nothing beyond a degree shift, and the Deligne
cohomology of the $\aaa$-realization becomes the $\cc/\zz$-cohomology of the
\emph{ordinary} realization.

\begin{lemma}\label{lem:DB-affine}
Let $p\geq1$ and $n\geq0$.  Then
\[
  H^k_\calD\bigl(\aaa^n_\cc,\zz(p)\bigr)\;=\;
  \begin{cases}
    \cc/\zz(p), & k=1,\\[2pt]
    0, & k\neq1 .
  \end{cases}
\]
\end{lemma}

\begin{proof}
Use the standard exact sequence for Deligne-Beilinson cohomology of a smooth
variety $U$,
\[
  \begin{gathered}
  0\longrightarrow
  \frac{H^{k-1}(U,\cc)}{F^pH^{k-1}(U,\cc)+H^{k-1}(U,\zz(p))}
  \longrightarrow H^k_\calD(U,\zz(p))\\
  \longrightarrow F^pH^k(U,\cc)\cap H^k(U,\zz(p))
  \longrightarrow 0 ,
  \end{gathered}
\]
the Hodge filtration being that of Deligne's mixed Hodge structure on
$H^*(U)$.  For $U=\aaa^n_\cc$ we have $H^0(U,\cc)=\cc$, pure of type $(0,0)$,
and $H^k(U,\cc)=0$ for $k>0$.  Hence for $p\geq1$: $F^pH^0=0$, so for $k=1$
the left-hand term is $\cc/\zz(p)$ and the right-hand term vanishes; for
$k=0$ both terms vanish (the right-hand one because
$F^pH^0\cap H^0(\zz(p))=0$); and for $k\geq2$ both vanish.
\end{proof}

\begin{proposition}\label{prop:DB-A-realization}
Let $X_\hdot$ be a semi-simplicial \emph{set} (so each $X_n$ is discrete), let
$\sigma:F\hookrightarrow\cc$ be an embedding, and let $p\geq1$.  Then there is
a natural isomorphism
\[
  H^{k}_\calD\bigl(\Re_{\aaa,\sigma}X_\hdot,\zz(p)\bigr)
  \;\xrightarrow{\ \sim\ }\;
  H^{k-1}\bigl(|X_\hdot|,\cc/\zz(p)\bigr).
\]
In particular
\[
  H^{2p}_\calD\bigl(\Re_{\aaa,\sigma}X_\hdot,\zz(p)\bigr)
  \;\cong\;H^{2p-1}\bigl(|X_\hdot|,\cc/\zz(p)\bigr),
\]
so a Deligne-Beilinson Chern class on the $\aaa$-realization \emph{is} a
$\cc/\zz$-cohomology class on the ordinary realization.
\end{proposition}

\begin{proof}
The $\aaa$-realization is glued from the pieces $X_s\times\aaa^s$ along the
coface maps, so its Deligne cohomology is computed by the descent spectral
sequence of the skeletal filtration,
\[
  E_1^{s,t}\;=\;H^t_\calD\bigl(X_s\times\aaa^s_\cc,\zz(p)\bigr)
  \;\Longrightarrow\;H^{s+t}_\calD\bigl(\Re_{\aaa,\sigma}X_\hdot,\zz(p)\bigr),
\]
with $d_1$ the alternating sum of the coface maps.  Since $X_s$ is discrete,
$X_s\times\aaa^s_\cc$ is a disjoint union of copies of $\aaa^s_\cc$ indexed by
$X_s$, so by Lemma~\ref{lem:DB-affine}
\[
  E_1^{s,t}\;=\;
  \begin{cases}
    \prod_{X_s}\cc/\zz(p) \;=\; C^s\bigl(X_\hdot,\cc/\zz(p)\bigr), & t=1,\\[2pt]
    0, & t\neq1,
  \end{cases}
\]
where $C^\hdot(X_\hdot,-)$ denotes simplicial cochains.  The $E_1$ page is
therefore concentrated in the single row $t=1$; the spectral sequence
degenerates at $E_2$, and
\[
  E_2^{s,1}\;=\;H^s\bigl(C^\hdot(X_\hdot,\cc/\zz(p))\bigr)
  \;=\;H^s\bigl(|X_\hdot|,\cc/\zz(p)\bigr).
\]
Reading off $H^k$ gives the stated isomorphism, with the degree shift coming
from the row $t=1$.  Naturality in $X_\hdot$ is clear, all the maps involved
being induced by maps of semi-simplicial sets.
\end{proof}

\begin{remark}\label{rem:where-things-live}
Proposition~\ref{prop:DB-A-realization} is the answer to the question of how
the Deligne Chern class of the Rees bundle and the pullback of the regulator
class are to be compared: they are not classes on the same object to begin
with, but the algebraic directions of the $\aaa$-realization are
Deligne-cohomologically trivial apart from the single group $\cc/\zz(p)$, and
that group is precisely the coefficient group in which the regulator lives.
Note also that the isomorphism is compatible with the homotopy equivalence
$|X_\hdot|\simeq(\Re_\aaa X_\hdot)^{\mathrm{top},\sigma}$
of~\eqref{eq:vertex-incl}, so no ambiguity arises from working with one
realization rather than the other.  We use the normalisation
$\cc/\zz(p)\cong\cc/\zz$ given by division by $(2\pi i)^p$, and write
$\cc/\zz$ from now on.
\end{remark}

\begin{notation}\label{not:cD-vs-CS}
For an algebraic vector bundle $V$ on an $\aaa$-realization we write
$c^\calD_p(V)\in H^{2p}_\calD(\Re_{\aaa,\sigma}X_\hdot,\zz(p))$ for its
Deligne-Beilinson Chern class, and we use
Proposition~\ref{prop:DB-A-realization} to regard it as an element of
$H^{2p-1}(|X_\hdot|,\cc/\zz)$ without further comment.  No connection is
involved: $c^\calD_p$ is an invariant of the bundle.

Connections enter only through the following two comparisons, which are what
we shall use when $V$ carries extra structure.
\begin{enumerate}[leftmargin=2em,label={\rm(\alph*)}]
  \item If $V$ carries \emph{any} $F^1$-connection $\nabla$---nil-flat or
    not---then $c^\calD_p(V)=[\nabla]_p$, by
    Theorem~\ref{thm:CS-DB-app}(2).  This is what lets the class be computed
    from a Hermitian connection pulled back from a Grassmannian, on which no
    nil-flat connection exists.
  \item If moreover $\nabla$ is locally nil-flat, so that $\ch_p(\nabla)=0$ as
    a form (Proposition~\ref{prop:nilflat-ch-zero}), then the associated
    Cheeger-Simons character is flat, defines $\CS_p(V,\nabla)$, and
    $\alpha\bigl(\CS_p(V,\nabla)\bigr)=c^\calD_p(V)$ by
    Theorem~\ref{thm:CS-DB-app}(3).  On an $\aaa$-realization $\alpha$ is the
    isomorphism of Proposition~\ref{prop:DB-A-realization}, so there
    $c^\calD_p(V)$ and $\CS_p(V,\nabla)$ are the same class
    (Remark~\ref{rem:alpha-kernel}).  This is the case of the Rees bundle, and
    it is used in \S\ref{subsec:degeneration}, not here.
\end{enumerate}
A flat connection is a special case of (b), so on the polynomial-degree-zero
part of the universal object, where the tautological bundle underlies a flat
bundle, $c^\calD_p$ agrees with the classical Chern-Simons class.
\end{notation}

\subsubsection{The universal object and the tautological map}

Let $BGL(r,F[t_1,\dots,t_n])$ denote the classifying space of the discrete
group $GL_r\bigl(F[t_1,\dots,t_n]\bigr)$, regarded as a semi-simplicial space
in $n$ via the inclusions of polynomial rings; it carries a tautological
rank-$r$ bundle $\mathbf{V}_{\mathrm{taut}}$ on its $\aaa$-realization.

\begin{lemma}\label{lem:univ-classifying}
Let $X_\hdot$ be a semi-simplicial set and $V_X$ a rank-$r$ bundle on
$\Re_\aaa X_\hdot$.  Then there is a map of semi-simplicial spaces
$\varphi:X_\hdot\to BGL(r,F[\hdot])$, unique up to simplicial homotopy, with
$V_X\cong\varphi^*\mathbf{V}_{\mathrm{taut}}$.
\end{lemma}

\begin{proof}
Over the affine space $\aaa^n_F$ every algebraic vector bundle of rank $r$ is
free, by the Quillen-Suslin theorem; choosing trivializations
simplex by simplex, the local system $V_n$ on $X_n$ becomes a local system of
free modules, i.e.\ is classified by a map
$X_n\to BGL_r\bigl(F[t_1,\dots,t_n]\bigr)$.  The gluing isomorphisms of
Definition~\ref{def:A-realization} make these compatible with the coface maps.
Two choices of trivialization differ by a section of $GL_r(F[t_\hdot])$, which
is exactly a simplicial homotopy between the resulting maps.
\end{proof}

The point of the next lemma is that although $BGL(r,F[t_1,\dots,t_n])$ genuinely
varies with $n$, it ceases to do so after stabilizing and applying the plus
construction.

\begin{lemma}\label{lem:htpy-constant}
The semi-simplicial space
\[
  n\;\longmapsto\;BGL\bigl(F[t_1,\dots,t_n]\bigr)^+
\]
is homotopically constant: every coface map is a homotopy equivalence.
Consequently its realization is canonically homotopy equivalent to
$BGL(F)^+$.
\end{lemma}

\begin{proof}
$F$ is a regular Noetherian ring, so $F[t_1,\dots,t_n]$ is regular and
homotopy invariance of algebraic $K$-theory \cite[\S6, Thm.~8]{Quillen}
gives $K_i(F)\xrightarrow{\ \sim\ }K_i(F[t_1,\dots,t_n])$ for all $i\geq0$.
Since $BGL(R)^+$ has homotopy groups $K_i(R)$ for $i\geq1$ and is connected,
the maps induced by the inclusions $F[t_1,\dots,t_m]\hookrightarrow
F[t_1,\dots,t_n]$ are weak equivalences of connected spaces with abelian
fundamental group, hence homotopy equivalences.  The realization of a
homotopically constant semi-simplicial space with value $Y$ is homotopy
equivalent to $Y\times|\Delta_{\mathrm{inj}}\text{-pt}|\simeq Y$.
\end{proof}

\begin{construction}[The tautological map]\label{constr:ctaut}
Composing the stabilization and plus construction on each piece with the
equivalence of Lemma~\ref{lem:htpy-constant} gives
\[
  c_{\mathrm{taut}}\;:\;\bigl|BGL(r,F[\hdot])\bigr|\;\longrightarrow\;BGL(F)^+ ,
\]
well defined up to homotopy.  For a bundle $V_X$ on $\Re_\aaa X_\hdot$ with
classifying map $\varphi$ as in Lemma~\ref{lem:univ-classifying}, set
\[
  c(V_X)\;:=\;c_{\mathrm{taut}}\circ|\varphi|\;:\;|X_\hdot|\;\longrightarrow\;BGL(F)^+ .
\]
By Lemma~\ref{lem:univ-classifying} this is independent of choices up to
homotopy, and it is natural in $X_\hdot$.
\end{construction}

Applying Construction~\ref{constr:ctaut} to $X_\hdot=\mathrm{Sd}(N\Xi_K)$ and
the tautological Rees bundle supplies exactly the map asserted in
step~\textbf{(3)} of Construction~\ref{constr:rL}, and $r_L$ is its composite
with steps \textbf{(1)} and \textbf{(2)}.

\subsubsection{The bisimplicial model}
\label{subsubsec:bisimplicial}

To compute $c(V_X)^*(r_{2p-1})$ we replace $BGL(r,F[\hdot])$ by a
bisimplicial set, so that both the regulator class and the Chern-Simons class
become classes on one and the same object.  This is the step that answers the
question of how the two classes are to be compared: they are compared after
pulling both back to the $\aaa$-realization of a bisimplicial resolution.

\begin{definition}\label{def:bisimplicial-model}
Let
\[
  \begin{gathered}
  B_\hdot GL\bigl(r,F[\hdot]\bigr)\;:\;(n,m)\;\longmapsto\;
  B_nGL\bigl(r,F[t_1,\dots,t_m]\bigr)\\
  \;=\;\mathrm{Hom}\bigl(\Delta[n],BGL(r,F[t_1,\dots,t_m])\bigr),
  \end{gathered}
\]
a bisimplicial set.  Its $\aaa$-realization $\Re_\aaa\bigl(B_\hdot
GL(r,F[\hdot])\bigr)$ is glued from the pieces
$B_nGL(r,F[t_1,\dots,t_m])\times\aaa^m\times\aaa^n$ and carries a tautological
rank-$r$ bundle $\mathbf{V}_{\mathrm{taut},2}$.  We write
$\Re\bigl(B_\hdot GL(r,F[\hdot])\bigr)$ for the ordinary realization; as
in~\eqref{eq:vertex-incl}, the natural inclusion
\[
  \Re\bigl(B_\hdot GL(r,F[\hdot])\bigr)\;\longrightarrow\;
  \Re_\aaa\bigl(B_\hdot GL(r,F[\hdot])\bigr)^{\mathrm{top},\sigma}
\]
is a homotopy equivalence.
\end{definition}

Realizing in the first variable recovers $BGL(r,F[t_1,\dots,t_m])$, so
Lemma~\ref{lem:htpy-constant} applies again and yields
\begin{equation}\label{eq:bisimp-to-BGL}
  \kappa\;:\;\Re\bigl(B_\hdot GL(r,F[\hdot])\bigr)\;\longrightarrow\;BGL(F)^+ .
\end{equation}

\subsubsection{The universal regulator class}

We first fix the definition of the class $r_{2p-1}$, since the comparison below
is meaningless until it is pinned down.  Recall that the plus construction
$BGL(R)^\delta\to BGL(R)^+$ is an \emph{acyclic} map: it induces an isomorphism
on integral homology, and hence on cohomology with arbitrary coefficients
\cite[IV.1.4]{Weibel}.

\begin{definition}\label{def:universal-regulator}
Let $F\subset\cc$ be a subfield and let $\mathbf{V}_{\mathrm{taut}}^{\,\delta}$
be the tautological flat rank-$r$ bundle on $BGL(r,F)^\delta$, i.e.\ the local
system attached to the identity representation, with its canonical flat
connection.  Its Cheeger-Simons class
$$
\CS_p(\mathbf{V}_{\mathrm{taut}}^{\,\delta})\in H^{2p-1}(BGL(r,F)^\delta,\cc/\zz)
$$
is compatible with the stabilization maps in $r$, because $\CS_p$ is additive
and vanishes on a trivial flat bundle
(Propositions~\ref{prop:CS-add-app} and \ref{prop:CS-funct-app}), and so
defines a class on $BGL(F)^\delta$.  The \emph{universal regulator class}
\[
  r_{2p-1}\;\in\;H^{2p-1}\bigl(BGL(F)^+,\cc/\zz\bigr)
\]
is the class corresponding to it under the isomorphism
$$
H^{2p-1}(BGL(F)^+,\cc/\zz)\xrightarrow{\sim}H^{2p-1}(BGL(F)^\delta,\cc/\zz)
$$
induced by the acyclic map.
\end{definition}

\begin{remark}\label{rem:regulator-borel}
With this definition, the imaginary part of $r_{2p-1}$ is the volume regulator
$\vol_p$ of \S\ref{app:diffchar}, since $\vol_p$ is \emph{defined} there as the
imaginary part of the Chern-Simons class under the splitting
$\cc/\zz=\rr/\zz\oplus i\rr$.  Its agreement with Borel's regulator,
defined through continuous cohomology and used in \S\ref{sec:torsion}, is
the Beilinson-Borel comparison; it holds up to an explicit non-zero rational
factor, by \cite{BG}.  Only the span of the classes
$\{\vol^\sigma_p\}_\sigma$ is used in Step~3 of \S\ref{sec:torsion}, so the
factor is immaterial for our purposes.
\end{remark}

\subsubsection{The comparison on the universal object}

We now prove the comparison unconditionally.  The mechanism is that homotopy
invariance of $K$-theory, combined with acyclicity of the plus construction,
gives an equivalence not merely after applying $(-)^+$ but already on
\emph{homology of the discrete classifying spaces}; and cohomology is
therefore rigid enough to detect the comparison.

Write $F[\Delta^m] := F[t_0,\dots,t_m]/(\textstyle\sum_i t_i-1)$ for the
coordinate ring of the algebraic $m$-simplex, so that $m\mapsto F[\Delta^m]$ is
a simplicial ring and $F[\Delta^m]\cong F[t_1,\dots,t_m]$ is regular.  Put
\[
  \Omega_r \;:=\; \bigl|\,m\mapsto BGL(r,F[\Delta^m])^\delta\,\bigr|,
  \qquad
  \Omega_\infty \;:=\; \varinjlim_r \Omega_r
  \;=\;\bigl|\,m\mapsto BGL(F[\Delta^m])^\delta\,\bigr| ,
\]
the realizations being taken in the simplicial direction $m$.  These are the
realizations $\Re\bigl(B_\hdot GL(r,F[\hdot])\bigr)$ of
Definition~\ref{def:bisimplicial-model}.  Let
\[
  \iota_r:BGL(r,F)^\delta\longrightarrow\Omega_r,
  \qquad
  \iota_\infty:BGL(F)^\delta\longrightarrow\Omega_\infty
\]
be induced by the inclusion of constants $F\hookrightarrow F[\Delta^\hdot]$,
which is a map of simplicial rings; on realizations it is the inclusion of the
constant simplicial space.

\begin{lemma}\label{lem:iota-homology-iso}
The map $\iota_\infty:BGL(F)^\delta\to\Omega_\infty$ is an isomorphism on
integral homology.  Consequently
\[
  \iota_\infty^*\;:\;H^{k}\bigl(\Omega_\infty,M\bigr)
  \;\xrightarrow{\ \sim\ }\;H^{k}\bigl(BGL(F)^\delta,M\bigr)
\]
is an isomorphism for every $k$ and every abelian group $M$, in particular for
$M=\cc/\zz$.
\end{lemma}

\begin{proof}
Fix $m$.  The ring $F[\Delta^m]\cong F[t_1,\dots,t_m]$ is regular Noetherian,
so homotopy invariance of algebraic $K$-theory \cite[\S6, Thm.~8]{Quillen}
gives $K_i(F)\xrightarrow{\sim}K_i(F[\Delta^m])$ for all $i\geq0$, whence
\[
  BGL(F)^+\;\xrightarrow{\ \simeq\ }\;BGL(F[\Delta^m])^+ .
\]
Since the plus construction is acyclic, the vertical maps in
\[
\begin{array}{ccc}
  BGL(F)^\delta & \longrightarrow & BGL(F[\Delta^m])^\delta\\
  \downarrow & & \downarrow\\
  BGL(F)^+ & \xrightarrow{\ \simeq\ } & BGL(F[\Delta^m])^+
\end{array}
\]
are integral homology isomorphisms, and therefore so is the top map.  Thus the
map of simplicial spaces from the constant object $BGL(F)^\delta$ to
$m\mapsto BGL(F[\Delta^m])^\delta$ is a levelwise homology isomorphism.  Both
simplicial spaces are proper (they are classifying spaces of simplicial
groups, hence realizations of simplicial sets in each degree), so comparison
of the skeletal-filtration spectral sequences
\[
  E^1_{s,q}=H_q(X_s,\zz)\;\Longrightarrow\;H_{s+q}(|X_\hdot|,\zz)
\]
shows that the induced map on realizations is a homology isomorphism.  The
realization of a constant simplicial space is the space itself, giving the
first assertion.  The second follows because a map inducing an isomorphism on
integral homology induces an isomorphism on cohomology with arbitrary
coefficients.
\end{proof}

\begin{remark}\label{rem:why-stabilize}
Stabilization in $r$ is essential in Lemma~\ref{lem:iota-homology-iso} and is
not a convenience.  Homotopy invariance of $K$-theory is a statement about
$K_i=\pi_i(BGL(-)^+)$, i.e.\ about the \emph{stable} groups; the individual
maps $BGL(r,F)^\delta\to BGL(r,F[t])^\delta$ are not homology isomorphisms,
since $H_*(GL_r(F[t]))$ genuinely differs from $H_*(GL_r(F))$ for finite $r$.
This is exactly why the naive argument---restrict to the fibre over
$0\in\Delta$ and invoke essential constancy---does not work at finite level:
before stabilizing, the simplicial space is not essentially constant in any
sense that cohomology in degree $2p-1$ can detect.
\end{remark}

\begin{lemma}\label{lem:universal-comparison}
Let $\mathbf{V}_{\mathrm{taut},2}$ be the tautological algebraic vector bundle
on $\Re_\aaa\bigl(B_\hdot GL(r,F[\hdot])\bigr)$ and let
$\kappa:\Omega_r\to BGL(F)^+$ be the map \eqref{eq:bisimp-to-BGL}.  Then
\[
  \kappa^*\bigl(r_{2p-1}\bigr)\;=\;
  c^\calD_p\bigl(\mathbf{V}_{\mathrm{taut},2}\bigr)
  \qquad\text{in }H^{2p-1}\bigl(\Omega_r,\cc/\zz\bigr),
\]
the right-hand side being read through the isomorphism of
Proposition~\ref{prop:DB-A-realization}.  No connection on
$\mathbf{V}_{\mathrm{taut},2}$ is used or assumed to exist.
\end{lemma}

\begin{proof}
\emph{Step 1: reduction to $\Omega_\infty$.}  The map $\kappa$ was constructed
in \eqref{eq:bisimp-to-BGL} by stabilizing in $r$ and applying the plus
construction, so it factors as
$\Omega_r\to\Omega_\infty\xrightarrow{\kappa_\infty}BGL(F)^+$.  On the other
side, $c^\calD_p$ is additive and vanishes on trivial summands, so the classes
$c^\calD_p(\mathbf{V}_{\mathrm{taut},2})$ for varying $r$ are compatible under the
stabilization maps $\Omega_r\to\Omega_{r+1}$ and are pulled back from a single
class on $\Omega_\infty$.  It therefore suffices to prove the identity on
$\Omega_\infty$.

\emph{Step 2: both sides are detected by $\iota_\infty$.}  By
Lemma~\ref{lem:iota-homology-iso}, $\iota_\infty^*$ is an isomorphism on
$H^{2p-1}(-,\cc/\zz)$.  So it suffices to prove that the two classes have the
same pullback along $\iota_\infty$.

\emph{Step 3: the left-hand side.}  The composite
$\kappa_\infty\circ\iota_\infty:BGL(F)^\delta\to BGL(F)^+$ is, by
construction, the canonical map: on the constant part of the simplicial
direction the identification of Lemma~\ref{lem:htpy-constant} is the identity.
Hence, by Definition~\ref{def:universal-regulator},
\[
  \iota_\infty^*\kappa_\infty^*(r_{2p-1})
  \;=\;\CS_p\bigl(\mathbf{V}^{\,\delta}_{\mathrm{taut}}\bigr).
\]

\emph{Step 4: the right-hand side.}  We must identify
$\iota_\infty^*\mathbf{V}_{\mathrm{taut},2}$.  Over the constant part of the
simplicial direction the transition matrices lie in
$GL(F)\subset GL(F[\Delta^m])$, so they are independent of the polynomial
parameters, and the restricted bundle is the pullback of
$\mathbf{V}^{\,\delta}_{\mathrm{taut}}$ along the projection killing the
$\aaa^m$ factor.  In particular it underlies a flat bundle: the local system
attached to the identity representation of $GL(F)^\delta$.  This is the one
point in the argument at which a flat connection is available, and it is
available because the bundle happens to be constant in the algebraic
directions, not because any connection was chosen on the universal object.  By
Notation~\ref{not:cD-vs-CS}(a) and functoriality of $c^\calD_p$,
\[
  \iota_\infty^*\,c^\calD_p\bigl(\mathbf{V}_{\mathrm{taut},2}\bigr)
  \;=\;c^\calD_p\bigl(\mathbf{V}^{\,\delta}_{\mathrm{taut}}\bigr)
  \;=\;\CS_p\bigl(\mathbf{V}^{\,\delta}_{\mathrm{taut}}\bigr).
\]

Steps 3 and 4 give equal pullbacks; Step 2 concludes.
\end{proof}

\begin{remark}\label{rem:classmap-gap}
It is worth isolating what makes the argument work.
Lemma~\ref{lem:iota-homology-iso} is not a statement that the
simplicial space $m\mapsto BGL(F[\Delta^m])^\delta$ is homotopically constant:
the individual spaces $BGL(F[\Delta^m])^\delta$ are not homotopy equivalent to
$BGL(F)^\delta$, only \emph{homologically} so, and only after stabilization
(Remark~\ref{rem:why-stabilize}).  What one gets is precisely enough: a
homology isomorphism on realizations, which is exactly what is needed to
detect classes in $H^{2p-1}(-,\cc/\zz)$, and no more.  The role of acyclicity
of the plus construction is to convert a statement about homotopy groups
($K_i(F)\cong K_i(F[t]))$ into one about homology of discrete groups, where
cohomology with $\cc/\zz$ coefficients can be compared.
\end{remark}

\subsubsection{Descent to the given bundle}

With Lemma~\ref{lem:universal-comparison} in hand, the passage to an arbitrary
$X_\hdot$ is straightforward.

\begin{proposition}\label{prop:classmap-comparison}
Let $X_\hdot$ be a simplicial set and $V_X$ a rank-$r$ algebraic vector bundle
on $\Re_\aaa X_\hdot$.  Then, under the homotopy equivalence
$|X_\hdot|\simeq(\Re_\aaa X_\hdot)^{\mathrm{top},\sigma}$
of~\eqref{eq:vertex-incl},
\[
  c(V_X)^*\bigl(r_{2p-1}\bigr)\;=\;c^\calD_p(V_X)
  \qquad\text{in } H^{2p-1}\bigl(|X_\hdot|,\cc/\zz\bigr),
\]
the right-hand side being the Deligne-Beilinson Chern class read through
Proposition~\ref{prop:DB-A-realization}.  Neither side involves a connection.
\end{proposition}

\begin{proof}
The classifying map $\varphi$ of Lemma~\ref{lem:univ-classifying} is well
defined only up to simplicial homotopy, since the trivializations there are
chosen simplex by simplex.  We first rectify it, by making the choice part of
the data.

\emph{The bisimplicial set $X'_{\hdot,\hdot}$.}  For $x\in X_n$ let $M(x)$
denote the $F[t_1,\dots,t_n]$-module of sections of $V_X$ over the cell
$\{x\}\times\aaa^n$ of $\Re_\aaa X_\hdot$, and let
\[
  {\mathcal T}(x)\;:=\;\bigl\{\,\theta:F[t_1,\dots,t_n]^r
  \xrightarrow{\ \sim\ }M(x)\,\bigr\}
\]
be its set of trivializations, a $GL_r(F[t_1,\dots,t_n])$-torsor, non-empty by
the Quillen-Suslin theorem (for the Rees bundle no appeal to Quillen-Suslin is
needed: Corollary~\ref{cor:F-trivial-affine} exhibits the trivialization
explicitly, by the sections $(1-y_i)^{a_i}$ on each affine cube).  Put
\[
  X'_{n,m}\;:=\;\coprod_{x\in X_n}{\mathcal T}(x)^{m+1},
\]
the \v{C}ech nerve of the torsor in the second direction, with $\pi$ the map
forgetting the trivializations.  Each vertical simplicial set
${\mathcal T}(x)^{\hdot+1}$ is the nerve of a codiscrete groupoid on a non-empty set,
hence contractible, so $\pi$ is a levelwise equivalence and $|\pi|$ is a
homotopy equivalence.

\emph{The map $\phi$.}  It is combinatorial, and only afterwards realized:
\begin{equation}\label{eq:phi-formula}
  \phi(x;\theta_0,\dots,\theta_m)\;:=\;
  \bigl(\theta_0^{-1}\theta_1,\ \theta_1^{-1}\theta_2,\ \dots,\
  \theta_{m-1}^{-1}\theta_m\bigr)
  \;\in\;B_mGL\bigl(r,F[t_1,\dots,t_n]\bigr),
\end{equation}
an $m$-tuple of matrices over the ring in $n$ variables.  This is a strict map
of bisimplicial sets, compatible with all face maps in both directions by
construction.

\emph{The two bundles.}  Applying $\Re_\aaa$, the $(n,m)$-cell of
$\Re_\aaa X'_{\hdot,\hdot}$ is $X'_{n,m}\times\aaa^m\times\aaa^n$, and $\pi$
realizes to the projection onto $\{x\}\times\aaa^n$ which forgets the
trivializations and the $\aaa^m$ factor; so $\pi^*V_X$ is, cell by cell,
$M(x)$ pulled back along $\aaa^m\times\aaa^n\to\aaa^n$, with the gluing data
of $V_X$ unchanged.  On the other side, $\phi^*\mathbf{V}_{\mathrm{taut},2}$
is the bundle glued along the $\aaa^m$ direction by the matrices
\eqref{eq:phi-formula}.  Since those matrices are precisely the comparisons
between the chosen trivializations of $M(x)$, the glued bundle is canonically
$M(x)$ again, whence
\[
  \phi^*\mathbf{V}_{\mathrm{taut},2}\;\cong\;\pi^*V_X
\]
as algebraic vector bundles over $\Re_\aaa X'_{\hdot,\hdot}$.

Now both classes are natural.  Pulling back along $\phi$ the identity of
Lemma~\ref{lem:universal-comparison} gives
\[
  \phi^*\kappa^*(r_{2p-1})\;=\;\phi^*c^\calD_p(\mathbf{V}_{\mathrm{taut},2})
  \;=\;c^\calD_p\bigl(\pi^*V_X\bigr)\;=\;\pi^*c^\calD_p(V_X).
\]
On the other hand $\kappa\circ\phi$ is, by construction, the composite
$c(V_X)\circ|\pi|$ up to homotopy, so the left-hand side is
$\pi^*c(V_X)^*(r_{2p-1})$.  As $|\pi|$ is a homotopy equivalence, $\pi^*$ is
injective, and the two classes on $|X_\hdot|$ agree.
\end{proof}

\begin{remark}[The transposition of indices]\label{rem:index-transpose}
It is worth recording which index of $X'_{\hdot,\hdot}$ maps to which index of
the target, since the two exchange r\^oles.  Writing
$B_{a,b}:=B_aGL(r,F[t_1,\dots,t_b])$, formula~\eqref{eq:phi-formula} says
\[
  \phi\;:\;X'_{n,m}\;\longrightarrow\;B_{m,n} :
\]
the \emph{bar} index of the target is the \v{C}ech direction $m$, while its
\emph{ring} index is the simplicial direction $n$ of $X_\hdot$.  This is
forced: the transition matrices live over the affine cube $\aaa^n$ on which
$V_X$ is trivialized, not over the resolution direction.  On realizations the
cells match after the same transposition,
\[
  X'_{n,m}\times\aaa^m\times\aaa^n
  \;\longrightarrow\;
  B_{m,n}\times\aaa^n\times\aaa^m ,
\]
and the two affine factors of Definition~\ref{def:bisimplicial-model} are thus
accounted for: $\aaa^n$ is the affine cube of $X_\hdot$, where the matrices
actually act, and $\aaa^m$ is the direction of the \v{C}ech resolution.  The
r\^ole of $X'_{\hdot,\hdot}$ is only to make $\phi$ strict; it is larger than
$X_\hdot$, but at no cost, the fibres of $\pi$ being contractible.
\end{remark}

\subsection{The Rees connection, projective completion, and the volume
regulator}
\label{subsec:degeneration}

The comparison of \S\ref{subsec:classmap} is a statement about bundles: it
identifies $c^\calD_p(F_K)$ with the pullback of the regulator class, and no
connection enters.  To extract from it a statement about the \emph{volume
regulator}---which is what Step~3 of the proof of Theorem~\ref{thm:main}
needs---one must evaluate $c^\calD_p$, and for that the finer structure of the
Rees bundle is used: it carries connections with logarithmic poles, and the
Dupont-Hain-Zucker and Burgos-Gil theories convert the Deligne class into a
Chern-Simons class whose imaginary part is the volume regulator.

\subsubsection{Why the projective completion is needed}

Both theories are formulated for a pair $(\bar Z,D)$ with $\bar Z$ smooth
projective and $D\subset\bar Z$ a normal crossings divisor---or, in the form
we actually use, for a cubical or simplicial scheme all of whose terms are of
that type: Theorem~\ref{thm:DHZ-app}(1) produces an $F^1$-connection because
each piece is an algebraic bundle on a projective variety, and the Burgos-Gil
arithmetic Chern character $\widehat{\ch}_p(E,h)$ of
Theorem~\ref{thm:BG-full} lives in an arithmetic Chow group of a smooth
complex projective variety.  Neither is available on an affine space as such.

Two distinct hypotheses are therefore in play, and it is worth keeping them
apart.  \emph{Projectivity} is what the present subsection is about, and its
failure on $\aaa_K\Cube A$ is fatal.  \emph{Smoothness} is a separate matter:
the glued scheme $\pp_K\Cube A$ is projective but singular along the faces
where its pieces meet, so the theories do not apply to it directly either.
That difficulty is resolved in \S\ref{subsec:BG-on-cubical} by reading
$\pp_K\Cube A$ as the cubical scheme $c\mapsto\pp(c)\cong(\pp^1)^{|c|}$, whose
terms are smooth projective, and equivalently by the Grassmannian classifying
map of Proposition~\ref{prop:grassmann-definition}; we use both freely below
and refer to that subsection rather than repeating the reduction.

Returning to projectivity: on $\aaa^n$ every bundle is free and carries a vast
supply of connections with no Hodge-theoretic constraint; the Bott-Chern
secondary forms of Theorem~\ref{thm:DHZ-app}(2) need not lie in $F^1$, and the
independence statement fails.  The apparatus must therefore be applied on a
projective completion, and the affine realization entered only afterwards, by
restriction.

This is why the Rees bundle was constructed on $\pp_K\Cube A$ in
Theorem~\ref{thm:global-rees} and not merely on $\aaa_K\Cube A$.  We write
\[
  \pp_K\Cube A \;\supset\; \aaa_K\Cube A,
  \qquad
  D_\infty \;:=\; \pp_K\Cube A \setminus \aaa_K\Cube A ,
\]
so that on each cube $\pp(c)\cong(\pp^1)^k \supset \aaa(c)\cong\aaa^k$ and
$D_\infty|_{\pp(c)}=\bigcup_i\{y_i=\infty\}$, a normal crossings divisor.  We
further write $Q_i=\{y_i=1\}\subset\pp(c)$ for the divisors appearing in
Theorem~\ref{thm:global-rees}, and put $D:=D_\infty+\sum_iQ_i$, again normal
crossings.

It is worth noting that this is the same completion already used implicitly in
\S\ref{subsubsec:deligne-A-realization}: Deligne-Beilinson cohomology of a
smooth variety $U$ is defined through a smooth compactification with normal
crossings boundary and is independent of the choice, so the computation of
$H^k_\calD(\aaa^n,\zz(p))$ in Lemma~\ref{lem:DB-affine} is already a
computation on $\pp^n$ relative to the hyperplane at infinity.  The two halves
of the appendix therefore use the same convention; what changes here is that
the completion carries a connection, which is exactly what the affine picture
cannot see.

\subsubsection{The Rees connection as an $F^1$-connection}

\begin{lemma}\label{lem:rees-log-poles}
The canonical Rees connection $\nabla^{\mathrm{can}}$, whose form on each
graded piece of $F(c)$ is $\sum_j n_j\,\dlog(1-y_j)$, extends to a connection
on $F_K$ over $\pp_K\Cube A$ with logarithmic poles along $D$.  It is an
$F^1$-connection in the sense of \S\ref{app:simplicial}, and it is compatible
with the face maps, hence simplicial.
\end{lemma}

\begin{proof}
On $\pp(c)\cong(\pp^1)^k$ the form $\dlog(1-y_j)$ has simple poles exactly on
$Q_j=\{y_j=1\}$ and on $\{y_j=\infty\}$, with residues $+1$ and $-1$; both are
components of $D$, so the form is a section of
$\Omega^1_{\pp(c)}(\log D)$.  By Theorem~\ref{thm:global-rees} the graded
pieces of $F(c)$ are $\Oo_{\pp(c)}(-j_1Q_1-\cdots-j_kQ_k)\otimes U_j(\mu(c))$,
and $\sum_jn_j\dlog(1-y_j)$ is precisely the connection form of the canonical
meromorphic connection on that line-bundle twist.  Hence
$\nabla^{\mathrm{can}}$ is a logarithmic connection on $F_K$ along $D$, so its
connection matrix lies in $F^1$ of the log de Rham complex; this is the
$F^1$-condition.  Compatibility with the face maps is the content of
Proposition~\ref{prop:rees-face}: setting $y_i=0$ or $y_i=1$ carries the
displayed form to the corresponding form on the face.
\end{proof}

\begin{proposition}\label{prop:rees-can-nilflat}
$\nabla^{\mathrm{can}}$ is locally nil-flat with respect to the filtrations
$W(c)$: it preserves each $W(c)$, and the induced connection on every graded
quotient is flat and, over $\iiii\Cube A$, agrees with the locally constant
structure determined by the trivialization $\tau(c)$.  Its curvature is
strictly upper-triangular and non-zero, so $\nabla^{\mathrm{can}}$ is
\emph{not} flat.  The $F^1$-connection $\nabla$ of Theorem~\ref{thm:F1-exists}
is a deformation of $\nabla^{\mathrm{can}}$ through $F^1$-connections, and
\[
  c^\calD_p(F_K,\nabla)\;=\;c^\calD_p(F_K,\nabla^{\mathrm{can}})
  \;=\;c^\calD_p(F_K)\;\in\;H^{2p}_\calD\bigl(\pp_K\Cube A,\zz(p)\bigr).
\]
\end{proposition}

\begin{proof}
The connection form is block upper-triangular for $W(c)$, with diagonal blocks
$\sum_jn_j\dlog(1-y_j)\cdot\mathrm{id}$, which are flat; so
$\nabla^{\mathrm{can}}$ preserves $W(c)$ with flat graded quotients, i.e.\ is
locally nil-flat, and its curvature is strictly upper-triangular.  It is
non-zero because the off-diagonal blocks of the form are not closed---this is
the failure of flatness, and it is the substance of the patching construction
rather than a defect: a globally flat connection would make $\CS_p$ an
invariant of a flat bundle and nothing would have been gained over the compact
case.  Over $\iiii(c)$ all $y_j\in[0,1]$, so $\dlog(1-y_j)$ is smooth there
and the graded structure is the locally constant one given by $\tau(c)$.

For the last statement, $\nabla$ and $\nabla^{\mathrm{can}}$ are both
$F^1$-connections on $F_K$ relative to $(\pp_K\Cube A, D)$---the first by
Theorem~\ref{thm:F1-exists}, the second by
Lemma~\ref{lem:rees-log-poles}---so Theorem~\ref{thm:DHZ-app}(2) applies and
they define the same Deligne class.  Concretely, the Bott-Chern secondary form
of the pair lies in $F^1A^{2p-1}(\log D)$ and is therefore a coboundary in the
Deligne complex.  The common value is $c^\calD_p(F_K)$, the Deligne-Beilinson
Chern class of the underlying bundle, by
Theorem~\ref{thm:DHZ-app}(1) together with
Theorem~\ref{thm:BG-full}(2).
\end{proof}

\subsubsection{From the Deligne class to the volume regulator}

\begin{theorem}\label{thm:rees-cD-equals-vol}
Let $\nabla^{\Del}$ be Deligne's patched connection on $F_K$ and let
$\sigma:K\hookrightarrow\cc$ be an embedding.  Then, under the identification
of Proposition~\ref{prop:DB-A-realization} and the homotopy equivalence
$X\simeq S=\iiii\Cube A$,
\[
  c^\calD_p(F_K^\sigma)\;=\;\CS_p(\nabla^{\Del})\;\in\;H^{2p-1}(X,\cc/\zz),
\]
and consequently, by Lemma~\ref{lem:CZ-splitting},
\[
  \Im_{\cc/\zz}\bigl(c^\calD_p(F_K^\sigma)\bigr)\;=\;\vol_p(\nabla^{\Del})
  \;\in\;H^{2p-1}(X,\rr).
\]
\end{theorem}

\begin{proof}
The three ingredients enter in the following order.

\emph{(i) Projectivity and existence (DHZ(1)).}  Read as a cubical scheme,
$\underline{\pp}_K\Cube A$ has smooth projective terms
$\pp(c)\cong(\pp^1)^{|c|}$ and $F_K$ is an algebraic bundle on it, so by
Theorem~\ref{thm:DHZ-app}(1) a simplicial $F^1$-connection exists.  Concretely
one may take $\phi_q^*\nabla_{\mathcal{S}}$, the pullback of the canonical
Hermitian connection on the tautological bundle of the Grassmannian along the
classifying map $\phi_q$ of Proposition~\ref{prop:grassmann-definition}, which
exists by the face-compatible presentation of
Lemma~\ref{lem:compatible-surjection}.  This is the step that fails on
$\aaa_K\Cube A$ alone, since a surjection from a twist of $\Oo$ is what
requires projectivity; and it is also the step for which the singularity of
the glued scheme $\pp_K\Cube A$ has to be circumvented, either by descent or
by $\phi_q$ (\S\ref{subsec:BG-on-cubical}).

\emph{(ii) Independence (DHZ(2)).}  By
Proposition~\ref{prop:rees-can-nilflat} and
Theorem~\ref{thm:DHZ-app}(2), the connections $\nabla$,
$\nabla^{\mathrm{can}}$, $\nabla^{\Del}$ and the Grassmannian connection of
(i) are all $F^1$-connections on $F_K$ relative to $(\pp_K\Cube A,D)$, hence
all define the same class $c^\calD_p(F_K)$, the hypotheses of
Theorem~\ref{thm:DHZ-app}(2) being verified on each smooth projective
$\pp(c)$ separately and the conclusion descending by naturality
(\S\ref{subsec:BG-on-cubical}).  In particular the Deligne class is computed
by whichever of them is convenient, and it is an invariant of the bundle.

\emph{(iii) Comparison (Burgos-Gil).}  Here $\widehat{\ch}_p$ is formed
levelwise, on each $\pp(c)$, for a levelwise metric $h=(h_c)$; equivalently it
is $\phi_q^*$ of the arithmetic Chern character of the tautological bundle on
the Grassmannian with its canonical metric.  By Theorem~\ref{thm:BG-full}(2)
the image of the arithmetic Chern character $\widehat{\ch}_p(F_K^\sigma,h)$
in
$H^{2p}_\calD$ is $c^\calD_p$ and is independent of the metric $h$.  Since
$\nabla^{\Del}$ is locally nil-flat, $\ch_p(\nabla^{\Del})=0$ as a form by
Proposition~\ref{prop:nilflat-ch-zero}, so by
Theorem~\ref{thm:CS-DB-app}(3) the character
$\widehat{\ch}_p(F_K^\sigma,\nabla^{\Del})$ is flat, defines
$\CS_p(\nabla^{\Del})\in H^{2p-1}(-,\cc/\zz)$, and satisfies
$\alpha(\CS_p(\nabla^{\Del}))=c^\calD_p(F_K^\sigma)$.  It remains to know that
$\alpha$ is injective here, and this is exactly what the $\aaa$-realization
supplies: by Proposition~\ref{prop:DB-A-realization} the map $\alpha$ is an
isomorphism, its kernel $F^pH^{2p-1}$ vanishing because the affine directions
carry no Deligne cohomology beyond degree one
(Remark~\ref{rem:alpha-kernel}).  Hence
$c^\calD_p(F_K^\sigma)=\CS_p(\nabla^{\Del})$.

Restricting to $\iiii\Cube A\simeq X$ and applying
Proposition~\ref{prop:DB-A-realization} places the identity in
$H^{2p-1}(X,\cc/\zz)$.  The final assertion is
Lemma~\ref{lem:CZ-splitting}: the imaginary component of a class in
$H^{2p-1}(X,\cc/\zz)$ is a well-defined element of $H^{2p-1}(X,\rr)$, and for
$\CS_p(\nabla^{\Del})$ it is by definition $\vol_p(\nabla^{\Del})$.
\end{proof}

\begin{remark}\label{rem:division-of-labour}
The two halves of the appendix use quite different mechanisms.
The classifying-map comparison is topological and
$K$-theoretic: it identifies $c^\calD_p$ of a \emph{bundle} with the pullback
of the regulator, using homotopy invariance of $K$-theory and acyclicity of the
plus construction, and it takes place on the affine realization where no
connection exists.  The present subsection is Hodge-theoretic: it evaluates
that same Deligne class, using the projective completion, the existence and
independence statements of Dupont-Hain-Zucker for $F^1$-connections with
logarithmic poles, and the Burgos-Gil comparison, and it produces the volume
regulator as the imaginary part.

Neither argument can be run in the other's setting.  One cannot evaluate
$c^\calD_p$ on the tautological bisimplicial object, because there is no
projective completion of $B_\hdot GL(r,F[\hdot])$ carrying a logarithmic
connection to which Dupont-Hain-Zucker and Burgos-Gil could be applied; and one
cannot compare with the regulator on $\pp_K\Cube A$, because the regulator is
defined through a classifying map to $BGL(K)^+$ and the comparison needs the
universal object.  It is the fact that both compute the same Deligne class of
the same bundle that joins them, and that is why $c^\calD_p$, rather than
$\CS_p$, has to be the common currency.
\end{remark}

\subsection{Synopsis in diagrams}
\label{subsec:rees-diagrams}

We close this appendix with three diagrams recording the objects, the bundles
they carry, and the groups in which the various classes live.
Figure~\ref{fig:spaces} collects the spaces; Figure~\ref{fig:regcohom} the
cohomology groups and the two places at which the Deligne class coincides with
the Chern-Simons class; and Figure~\ref{fig:torsion} the torsion argument of
\S\ref{sec:torsion}.

A word on degrees, since two realizations of the same simplicial object occur
and only one of them carries a bundle.  For a simplicial set $Y_\hdot$ the
bundle lives on the \emph{algebraic} realization $\Re_\aaa Y_\hdot$ and its
$p$-th Deligne-Beilinson Chern character lives there in degree $2p$,
\[
  \ch^{\calD}_p\;\in\;H^{2p}_\calD\bigl(\Re_\aaa Y_\hdot,\zz(p)\bigr)
  \;\cong\;H^{2p-1}\bigl(|Y_\hdot|,\cc/\zz\bigr),
\]
the right-hand group being the $\cc/\zz$-cohomology of the \emph{topological}
realization (Proposition~\ref{prop:DB-A-realization}).  The three instances
used are
\[
\begin{array}{lll}
Y_\hdot=B_\hdot GL(r,F[\hdot]): &
  H^{2p}_\calD\bigl(\Re_\aaa Y_\hdot,\zz(p)\bigr) &\cong\;
  H^{2p-1}(\Omega_\infty,\cc/\zz),\\[3pt]
Y_\hdot=X'_{\hdot,\hdot}: &
  H^{2p}_\calD\bigl(\Re_\aaa X'_{\hdot,\hdot},\zz(p)\bigr) &\cong\;
  H^{2p-1}(|X'_{\hdot,\hdot}|,\cc/\zz),\\[3pt]
Y_\hdot=\Cube A: &
  H^{2p}_\calD\bigl(\aaa_K\Cube A,\zz(p)\bigr) &\cong\;
  H^{2p-1}(X,\cc/\zz),
\end{array}
\]
each isomorphism having the same source, namely Lemma~\ref{lem:DB-affine}:
the Deligne cohomology of an affine cell is $\cc/\zz(p)$ in degree
\emph{one} and zero otherwise, so the degree $2p$ is produced by the
combinatorics of the covering and not by any single cell.  In particular
$\ch^{\calD}_p(\mathbf{V}_{\mathrm{taut},2})$ is a class on
$\Re_\aaa B_\hdot GL(r,F[\hdot])$ and not on $\Omega_\infty$, which carries no
bundle.  In the figures the $\cc/\zz$ groups are the ones drawn.

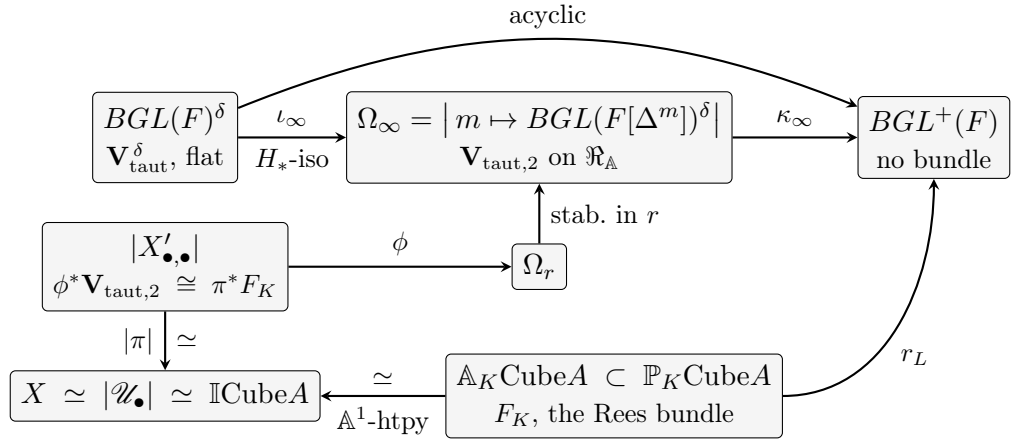
\begin{figure}[ht]
\centering
\begin{tikzpicture}[scale=0.9,
  every node/.style={font=\small},
  sp/.style={draw,rounded corners=2pt,inner sep=4pt,fill=gray!8,align=center},
  ar/.style={->,>=stealth,thick},
  lb/.style={font=\footnotesize}]

\node[sp] (Bd) at (0,2.0)
  {$BGL(F)^{\delta}$\\[1pt]{\footnotesize$\mathbf{V}^{\,\delta}_{\mathrm{taut}}$, flat}};
\node[sp] (Om) at (5.5,2.0)
  {$\Omega_{\infty}=\bigl|\,m\mapsto BGL(F[\Delta^m])^{\delta}\bigr|$\\[1pt]
   {\footnotesize$\mathbf{V}_{\mathrm{taut},2}$ on $\Re_\aaa$}};
\node[sp] (Bp) at (11.3,2.0)
  {$BGL^{+}(F)$\\[1pt]{\footnotesize no bundle}};

\draw[ar] (Bd) -- node[lb,above] {$\iota_{\infty}$}
                 node[lb,below] {$H_{*}$-iso} (Om);
\draw[ar] (Om) -- node[lb,above] {$\kappa_{\infty}$} (Bp);
\draw[ar] (Bd) to[out=22,in=158] node[lb,above] {acyclic} (Bp);

\node[sp] (Xp) at (0,0.1) {$|X'_{\hdot,\hdot}|$\\[1pt]
   {\footnotesize$\phi^{*}\mathbf{V}_{\mathrm{taut},2}\;\cong\;\pi^{*}F_K$}};
\node[sp] (Or) at (5.5,0.1) {$\Omega_r$};
\draw[ar] (Xp) -- node[lb,above] {$\phi$} (Or);
\draw[ar] (Or) -- node[lb,right] {stab.\ in $r$} (Om);

\node[sp] (X)   at (0,-1.8)
  {$X\;\simeq\;|\Uu_{\hdot}|\;\simeq\;\iiii\Cube A$};
\node[sp] (Aff) at (6.6,-1.8)
  {$\aaa_K\Cube A\;\subset\;\pp_K\Cube A$\\[1pt]
   {\footnotesize$F_K$, the Rees bundle}};
\draw[ar] (Xp) -- node[lb,left] {$|\pi|$} node[lb,right] {$\simeq$} (X);
\draw[ar] (Aff) -- node[lb,above] {$\simeq$}
                   node[lb,below] {$\aaa^1$-htpy} (X);
\draw[ar] (Aff.east) to[out=0,in=-90]
  node[lb,below right,pos=0.45] {$r_L$} (Bp.south);
\end{tikzpicture}
\caption{Spaces and tautological bundles.  Top row, the universal side: the
flat tautological local system exists only on the discrete classifying space
$BGL(F)^{\delta}$; the algebraic bundle $\mathbf{V}_{\mathrm{taut},2}$ lives on
the $\aaa$-realization $\Re_\aaa B_\hdot GL(r,F[\hdot])$ of the same
bisimplicial object, not on the topological realization $\Omega_\infty$ drawn
here; and $BGL^{+}(F)$ carries no bundle at all, only the universal class.
Bottom row, the geometric side: the Rees bundle $F_K$ on the cubical
realization, and the rectification $X'_{\hdot,\hdot}$ of
Proposition~\ref{prop:classmap-comparison} through which the classifying map is
made strict.  The isomorphism at $X'_{\hdot,\hdot}$ is one of algebraic bundles
over $\Re_\aaa X'_{\hdot,\hdot}$: this is the common site over which
$\mathbf{V}_{\mathrm{taut},2}$ and $F_K$ are compared.}
\label{fig:spaces}
\end{figure}

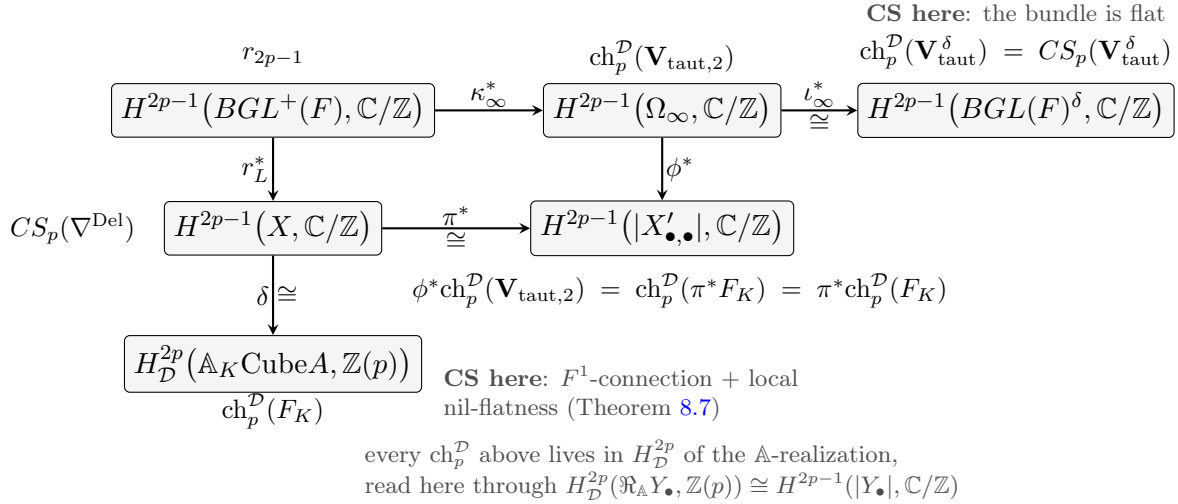
\begin{figure}[ht]
\centering
\begin{tikzpicture}[scale=0.78,
  every node/.style={font=\small},
  gp/.style={draw,rounded corners=2pt,inner sep=3.5pt,fill=gray!8},
  ar/.style={->,>=stealth,thick},
  el/.style={font=\footnotesize,inner sep=1pt},
  nt/.style={font=\scriptsize,inner sep=1pt,text=black!70}]

\node[gp] (Hp)  at (0,2.6)    {$H^{2p-1}\bigl(BGL^{+}(F),\cc/\zz\bigr)$};
\node[gp] (HOm) at (6.6,2.6)  {$H^{2p-1}\bigl(\Omega_{\infty},\cc/\zz\bigr)$};
\node[gp] (Hd)  at (12.6,2.6) {$H^{2p-1}\bigl(BGL(F)^{\delta},\cc/\zz\bigr)$};

\node[el] at (0,3.5)    {$r_{2p-1}$};
\node[el] at (6.6,3.5)  {$\ch^{\calD}_p(\mathbf{V}_{\mathrm{taut},2})$};
\node[el] at (12.6,3.6)
  {$\ch^{\calD}_p(\mathbf{V}^{\,\delta}_{\mathrm{taut}})
    \;=\;\CS_p(\mathbf{V}^{\,\delta}_{\mathrm{taut}})$};
\node[nt] at (12.6,4.25) {\textbf{CS here}: the bundle is flat};

\draw[ar] (Hp) -- node[el,above] {$\kappa_{\infty}^{*}$} (HOm);
\draw[ar] (HOm) -- node[el,above] {$\iota_{\infty}^{*}$}
                   node[el,below] {$\cong$} (Hd);

\node[gp] (HX)  at (0,0.6)   {$H^{2p-1}\bigl(X,\cc/\zz\bigr)$};
\node[gp] (HXp) at (6.6,0.6) {$H^{2p-1}\bigl(|X'_{\hdot,\hdot}|,\cc/\zz\bigr)$};

\draw[ar] (Hp)  -- node[el,left] {$r_L^{*}$} (HX);
\draw[ar] (HOm) -- node[el,right] {$\phi^{*}$} (HXp);
\draw[ar] (HX)  -- node[el,above] {$\pi^{*}$}
                   node[el,below] {$\cong$} (HXp);

\node[el] at (6.9,-0.45)
  {$\phi^{*}\ch^{\calD}_p(\mathbf{V}_{\mathrm{taut},2})
   \;=\;\ch^{\calD}_p(\pi^{*}F_K)\;=\;\pi^{*}\ch^{\calD}_p(F_K)$};

\node[gp] (HD) at (0,-1.7) {$H^{2p}_{\calD}\bigl(\aaa_K\Cube A,\zz(p)\bigr)$};
\draw[ar] (HX) -- node[el,left] {$\delta$} node[el,right] {$\cong$} (HD);
\node[el] at (-3.4,0.6) {$\CS_p(\nabla^{\Del})$};
\node[el] at (0,-2.5) {$\ch^{\calD}_p(F_K)$};
\node[nt,align=left] at (5.9,-2.2)
  {\textbf{CS here}: $F^1$-connection $+$ local\\
   nil-flatness (Theorem~\ref{thm:CS-regulator})};
\node[nt,align=left] at (6.6,-3.5)
  {every $\ch^{\calD}_p$ above lives in $H^{2p}_{\calD}$ of the
   $\aaa$-realization,\\
   read here through
   $H^{2p}_{\calD}(\Re_\aaa Y_\hdot,\zz(p))\cong H^{2p-1}(|Y_\hdot|,\cc/\zz)$};
\end{tikzpicture}
\caption{Where the classes are compared.  The square commutes and $\pi^{*}$ is
injective, which is what carries the universal identity of
Lemma~\ref{lem:universal-comparison} down to $X$.  Over
$X'_{\hdot,\hdot}$ the two tautological bundles become isomorphic, so their
$p$-th Deligne-Beilinson Chern characters agree there.  The Deligne class
coincides with the Chern-Simons class at exactly two places, both marked: on
$BGL(F)^{\delta}$, for free, because the tautological bundle is flat there;
and on $X$, by Theorem~\ref{thm:CS-regulator}, because $\nabla^{\Del}$ is
locally nil-flat and an $F^1$-connection is available.}
\label{fig:regcohom}
\end{figure}

\begin{figure}[ht]
\centering
\begin{tikzpicture}[scale=0.85,
  every node/.style={font=\small},
  gp/.style={draw,rounded corners=2pt,inner sep=3.5pt,fill=gray!8},
  ar/.style={->,>=stealth,thick},
  el/.style={font=\footnotesize,inner sep=1pt}]

\node[gp] (HX) at (0,1.3) {$H^{2p-1}\bigl(X,\cc/\zz\bigr)$};
\node[el] at (-3.9,1.3) {$\CS_p(\nabla^{\Del})=r_L^{*}(r_{2p-1})$};

\node[gp] (Hs) at (0,-0.5)
  {$H^{2p-1}(X,\rr/\zz)\;\oplus\;i\,H^{2p-1}(X,\rr)$};
\draw[ar] (HX) -- node[el,left] {$\cong$} (Hs);

\node[gp] (HB) at (8.6,-0.5)
  {$H^{2p-1}\bigl(BGL^{+}(F),\rr\bigr)\;\cong\;\bigoplus_{\sigma}\rr$};
\draw[ar] (HB) -- node[el,above] {$r_L^{*}=0$} (Hs);
\node[el] at (8.6,-1.35) {Borel; spanned by $\vol^{\sigma}_p$};
\node[el] at (0,-1.8) {torsion \ $\oplus$ \ $0$};
\end{tikzpicture}
\caption{The torsion argument of \S\ref{sec:torsion}.  The imaginary summand is
torsion free, so a torsion class must have vanishing volume regulator; the
hermitian argument kills $\vol^{\sigma}_p$ for each embedding $\sigma$
separately, and Borel's theorem \cite{Borel} says these span, so
$r_L^{*}$ annihilates the whole of the real cohomology.}
\label{fig:torsion}
\end{figure}
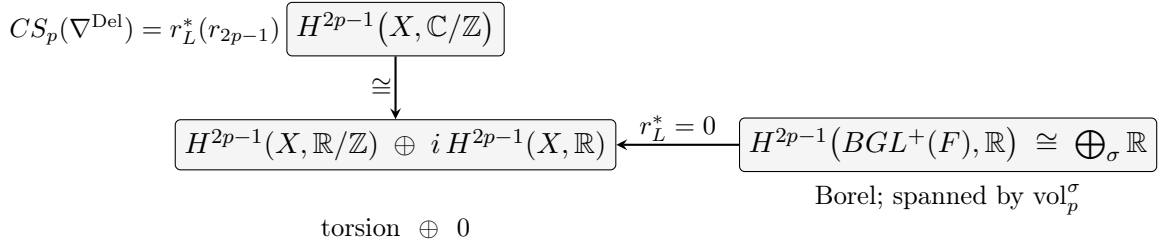

\clearpage

\section{The non-abelian Hodge correspondence and deformation to a VHS}
\label{sec:deform-to-vhs-detail}

The purpose of this section is to give a self-contained and detailed account
of Theorem~\ref{thm:deform-to-vhs}: every representation
$\zeta:\pi_1(X^*,x_0)\to GL_r(\cc)$ with unipotent monodromy deforms, within
$R^{\mathrm{nil}}$, to a representation underlying a polarized complex VHS.
We explain in full the roles of semisimplification, the pluri-harmonic metric,
the $\cc^*$-family of flat connections, and the 
decomposition theorem.

\subsection{Setup and notation}

Throughout this section $X$ is a smooth connected projective complex variety,
$D = D_1+\cdots+D_k\subset X$ a simple normal crossings divisor,
$X^* = X\setminus D$, and $x_0\in X^*$ a base-point.  We fix a rank $r$ and
consider the space
\[
  R^{\mathrm{nil}} \;:=\;
  \Bigl\{\,\zeta:\pi_1(X^*,x_0)\to GL_r(\cc) \;\Big|\;
  \zeta(\gamma_i)\text{ is unipotent for every local loop }\gamma_i
  \text{ around }D_i\Bigr\}.
\]
This is an affine algebraic variety over $\cc$.  The quotient
$\mathcal{M}^{\mathrm{nil}} := R^{\mathrm{nil}}/GL_r(\cc)$ (by simultaneous
conjugation) is the corresponding moduli space of flat bundles with unipotent
monodromy.

We also fix the following standard notation for a harmonic bundle.  A
$\mathcal{C}^\infty$ complex vector bundle $E$ on $X^*$ carries:
\begin{itemize}[leftmargin=2em]
  \item a Hermitian metric $h$ (a smooth positive-definite Hermitian form on
    the fibres);
  \item a holomorphic structure $\delbar_E$ (making $E$ into a holomorphic
    bundle $\mathcal{E}$);
  \item a Higgs field $\theta\in A^{1,0}(X^*,\End(E))$, holomorphic
    ($\delbar_E\theta = 0$) and satisfying $\theta\wedge\theta = 0$;
  \item the Chern connection $\nabla_h = \nabla_h^{1,0}+\delbar_E$, the
    unique connection on $(E,h,\mathcal{E})$ that is compatible with both $h$
    and the holomorphic structure;
  \item the adjoint Higgs field $\bar\theta := \theta^{*_h}\in A^{0,1}(X^*,\End(E))$
    (the $h$-adjoint of $\theta$, of type $(0,1)$).
\end{itemize}

\subsection{Step 1: Semisimplification}
\label{subsec:semisimplify}

\begin{definition}
A representation $\zeta:\pi_1(X^*,x_0)\to GL_r(\cc)$ is \emph{semisimple} if
the corresponding local system (equivalently, the corresponding flat bundle
$(E,\nabla)$) is a direct sum of irreducible local systems.  Equivalently,
$\zeta$ is isomorphic (up to conjugation) to a direct sum of irreducible
subrepresentations.
\end{definition}

\begin{definition}
The \emph{semisimplification} of $\zeta$ is the unique semisimple
representation $\zeta^{ss}$ whose Jordan-H\"{o}lder factors (the irreducible
sub-quotients of the Jordan-H\"{o}lder filtration) are isomorphic to those of
$\zeta$, and whose underlying $\mathcal{C}^\infty$ bundle coincides with that
of $\zeta$ (as graded).
\end{definition}

More concretely: choose a Jordan-H\"{o}lder filtration
$0 = E_0\subset E_1\subset\cdots\subset E_s = E$ of the flat bundle $(E,\nabla)$
by flat sub-bundles, with successive quotients $(L_j,\nabla_j) = E_j/E_{j-1}$
irreducible.  The semisimplification is
\[
  (E^{ss},\nabla^{ss}) \;:=\; \bigoplus_{j=1}^s (L_j,\nabla_j).
\]
It is independent of the choice of filtration, up to isomorphism.

\begin{lemma}\label{lem:ss-Rnil}
If $\zeta\in R^{\mathrm{nil}}$, then $\zeta^{ss}\in R^{\mathrm{nil}}$.
\end{lemma}

\begin{proof}
The unipotent monodromy condition says $\zeta(\gamma_i) = \Id + N_i$ with
$N_i$ nilpotent.  Since every Jordan-H\"{o}lder factor is a quotient of a
sub-bundle of $(E,\nabla)$, its monodromy $L_j(\gamma_i)$ is a sub-quotient of
$\zeta(\gamma_i) = \Id + N_i$.  A sub-quotient of a unipotent matrix is
unipotent (the sub-quotient of $\Id+N_i$ is $\Id + \bar N_i$ where $\bar N_i$
is the induced endomorphism on $L_j$, which is still nilpotent).  Hence each
$L_j$ has unipotent monodromy around $D_i$, so
$\zeta^{ss} = \bigoplus_j \zeta_{L_j}\in R^{\mathrm{nil}}$.
\end{proof}

\begin{lemma}\label{lem:ss-path}
There is a path inside $R^{\mathrm{nil}}$ from $\zeta$ to $\zeta^{ss}$.
\end{lemma}

\begin{proof}
The upper-triangular form of $\zeta$ (in a frame respecting the Jordan-H\"{o}lder
filtration) writes $\zeta = \zeta^{ss} + \zeta^{u}$ where $\zeta^u$ is strictly
upper-triangular.  The path
\[
  \zeta(t) \;:=\; \zeta^{ss} + (1-t)\,\zeta^u, \qquad t\in[0,1],
\]
lies in $R^{\mathrm{nil}}$: for every $t$, $\zeta(t)(\gamma_i)$ is upper-triangular
with $\Id$ on the diagonal (since $\zeta^{ss}(\gamma_i) = \Id$ and
$\zeta^u(\gamma_i)$ is strictly upper-triangular), hence unipotent.  The path
is algebraic (polynomial in $t$), so it is a morphism of algebraic varieties
$\aaa^1\to R^{\mathrm{nil}}$.
\end{proof}

\subsection{Step 2: The pluri-harmonic metric and the harmonic bundle}
\label{subsec:pluriharmonic}

The key analytic tool is the following existence theorem.

 In the compact case the existence statement is due to
Corlette and, in the algebraic formulation, to the second author; in the
quasi-projective tame case it is due to Jost-Zuo \cite{JostZuo} and, in the
generality used here, to Mochizuki.  

\begin{theorem}[Corlette {\cite{Corlette}}, {the second author}
{\cite{Simpson-IHES}}, {Jost-Zuo {\cite{JostZuo}}},
Mochizuki {\cite{Mochizuki,Mochizuki2}}]
\label{thm:harmonic-exists}
Let $(V,\nabla_{\mathrm{flat}})$ be a semisimple flat bundle on $X^*$ with
unipotent monodromy around each divisor component $D_i$.  Then there exists a
Hermitian metric $h$ on $V$ that is \emph{pluri-harmonic}: the associated
curvature equation
\begin{equation}\label{eq:self-dual}
  F(\nabla_h) + [\theta,\bar\theta] = 0
\end{equation}
holds, where $\nabla_h$ is the Chern connection of $(V,h)$,
$\theta = \nabla_{\mathrm{flat}}^{1,0} - \nabla_h^{1,0}$
is the Higgs field, and $\bar\theta = \theta^{*_h}$ is its $h$-adjoint.
Furthermore, the metric $h$ is unique up to the action of the automorphisms
of $(V,\nabla_{\mathrm{flat}})$, and the resulting triple $(V,\theta,h)$ is a
\emph{tame harmonic bundle} in the sense of Mochizuki: $h$ has at most
polynomial growth (moderate growth) near each $D_i$, with no exponential
singularities.
\end{theorem}

\begin{proof}
The theorem has three layers of increasing generality.

\textit{Step 1: Compact K\"{a}hler manifolds (no boundary).}
Corlette \cite{Corlette} proved that every semisimple flat bundle on a compact
K\"{a}hler manifold admits a pluri-harmonic metric, by solving the harmonic
map equation (which in this context is the Hermitian--Einstein equation on
the associated principal bundle).  Independently, in the algebraic setting,
Hitchin \cite{Hitchin} (for rank $2$ on curves) and the second author
\cite{Simpson-IHES} (in full generality for any smooth projective variety
$X$) established the correspondence.

\textit{Step 2: Noncompact curves.}
The second author in
\cite{Simpson-harmonic} extended the correspondence to the case where
$X^* = X\setminus D$ with $X$ a smooth projective curve and $D$ a finite set
of points.  That paper proves existence and uniqueness of the
pluri-harmonic metric
in the tame case (regular singular connections, i.e.\ unipotent or
quasi-unipotent monodromy) by a heat-flow argument, establishing precise
growth estimates for $h$ near the punctures.

\textit{Step 3: Quasi-projective varieties in higher dimensions (our setting).}
For $X^*$ a smooth quasi-projective variety of dimension $\geq 2$ with
$X = \bar X$ projective and $D = \bar X\setminus X^*$ a normal crossings
divisor, the existence and uniqueness of the pluri-harmonic metric for tame
flat bundles is due to \textbf{Mochizuki} \cite{Mochizuki,Mochizuki2}.
For the unipotent case considered in this paper, the existence of a
tame pluriharmonic metric had been obtained earlier by Jost-Zuo
\cite{JostZuo}.  Mochizuki's
two-volume Memoirs establish:
\begin{enumerate}[leftmargin=2em,label=\rm(\roman*)]
  \item \emph{Existence:} every semisimple tame flat bundle (i.e.\ with
    quasi-unipotent, in particular unipotent, monodromy around each
    component of $D$) admits a pluri-harmonic metric.
  \item \emph{Uniqueness:} the metric is unique up to the structure group
    of automorphisms.
  \item \emph{Asymptotic behaviour:} the metric $h$ has at most polynomial
    growth in a local coordinate $z_i$ near $D_i$, of the form
    $h\sim c_i\cdot|z_i|^{2\alpha_i}$ where $\alpha_i\in\rr$ (and
    $\alpha_i = 0$ in the unipotent case, giving moderate growth).
  \item \emph{Extension:} the harmonic bundle $(V,\theta,h)$ extends across
    $D$ as a \emph{tame harmonic bundle} in Mochizuki's sense, i.e.\ the
    Higgs field $\theta$ has at most a simple pole along each $D_i$.
\end{enumerate}

In our paper, $X^*$ is quasi-projective of arbitrary dimension $d\geq 1$ and
$D$ is a normal crossings divisor.  The flat bundle has unipotent (a special
case of quasi-unipotent) monodromy.  The existence of $h$ in this setting is
therefore a consequence of Mochizuki's theorem.
\end{proof}

\begin{definition}
A triple $(E, \theta, h)$ satisfying \eqref{eq:self-dual} is called a
\emph{harmonic bundle}.  The Higgs field $\theta$ and its adjoint $\bar\theta$
are determined by the flat connection $\nabla$ and the metric $h$:
\begin{equation}\label{eq:Higgs-decomp}
  \nabla = \nabla_h + \theta + \bar\theta,
\end{equation}
where $\nabla_h = \nabla_h^{1,0}+\delbar_E$ is of type
$(1,0)+(0,1)$ and $\theta$ is of type $(1,0)$, $\bar\theta$ of type $(0,1)$.
\end{definition}

\begin{remark}
The decomposition \eqref{eq:Higgs-decomp} is canonical once $h$ is fixed.
The flatness $\nabla^2 = 0$ combined with the type decomposition gives exactly
\eqref{eq:self-dual}, together with the holomorphicity condition
$\delbar_E\theta = 0$ and $\theta\wedge\theta=0$.  Thus $(E,\theta,h)$ is
a harmonic bundle if and only if $\nabla = \nabla_h+\theta+\bar\theta$
is a flat connection.
\end{remark}

\subsection{Step 3: The \texorpdfstring{$\cc^*$}{C*}-family of flat connections}
\label{subsec:Cstar-family}

Given a harmonic bundle $(E,\theta,h)$, one constructs a family of flat
connections parametrised by $\lambda\in\cc^*$.

\begin{definition}
For $\lambda\in\cc^*$, define the \emph{$\lambda$-connection}
\begin{equation}\label{eq:lambda-conn}
  \nabla_\lambda \;:=\; \nabla_h + \lambda^{-1}\theta + \lambda\,\bar\theta.
\end{equation}
\end{definition}

\begin{proposition}\label{prop:lambda-flat}
$\nabla_\lambda$ is flat for every $\lambda\in\cc^*$.
\end{proposition}

\begin{proof}
We compute the curvature directly.  Write
$\nabla_\lambda = (\nabla_h^{1,0} + \lambda^{-1}\theta)
+ (\delbar_E + \lambda\,\bar\theta)$, decomposed by type.
The $(2,0)$-part of the curvature is:
\[
  (\nabla_h^{1,0} + \lambda^{-1}\theta)^2
  = F^{2,0}(\nabla_h) + \lambda^{-1}\nabla_h^{1,0}\theta
  + \lambda^{-2}\theta\wedge\theta.
\]
Since $\nabla_h$ is the Chern connection of a holomorphic bundle,
$F^{2,0}(\nabla_h) = 0$.  Since $\theta$ is holomorphic,
$\nabla_h^{1,0}\theta = \partial\theta = 0$ in the holomorphic sense.
And $\theta\wedge\theta = 0$ by the Higgs condition.  So the $(2,0)$-part
vanishes.  By conjugation (or the same computation) the $(0,2)$-part also
vanishes.

The $(1,1)$-part is:
\[
  F^{1,1}(\nabla_h) + \lambda^{-1}[\delbar_E,\theta]
  + \lambda\,[\nabla_h^{1,0},\bar\theta]
  + \lambda^{-1}\lambda\,[\theta,\bar\theta].
\]
Since $\delbar_E\theta = 0$ (holomorphicity of $\theta$), the second term is
zero.  The third term is $\lambda[\nabla_h^{1,0},\bar\theta] = \lambda(\delbar_E\theta)^* = 0$.
The remaining terms give $F^{1,1}(\nabla_h) + [\theta,\bar\theta] = 0$
by the harmonic bundle equation \eqref{eq:self-dual}.

Hence $F(\nabla_\lambda) = 0$ for all $\lambda\in\cc^*$.
\end{proof}

\begin{proposition}\label{prop:lambda-one}
At $\lambda = 1$ one recovers the original flat connection:
$\nabla_1 = \nabla_h + \theta + \bar\theta = \nabla$.
\end{proposition}

\begin{proof}
This is immediate from the definition \eqref{eq:lambda-conn} and the
decomposition \eqref{eq:Higgs-decomp}.
\end{proof}

\begin{proposition}\label{prop:lambda-circle}
For $\lambda = e^{i\phi}\in S^1\subset\cc^*$, the connection $\nabla_\lambda$
has the same monodromy type as $\nabla$: its monodromy around each $D_i$ is
unipotent.  In particular, the circle
$\{e^{i\phi}\cdot(E,\theta,h) : \phi\in[0,2\pi)\}\subset\mathcal{M}^{\mathrm{nil}}$
lies inside $R^{\mathrm{nil}}$.
\end{proposition}

\begin{proof}
The monodromy of $\nabla_\lambda$ around a small loop $\gamma$ near $D_i$ is
computed using the connection $\nabla_\lambda = \nabla_h + \lambda^{-1}\theta
+ \lambda\,\bar\theta$.  In a suitable local frame near $D_i$ adapted to the
parabolic structure (the weight filtration $W(N_i)$), the Higgs field $\theta$
is strictly lower-triangular (it maps weight-$k$ to weight-$(k-1)$), and
$\bar\theta$ is strictly upper-triangular.  The metric connection $\nabla_h$
has logarithmic singularities with nilpotent residues $N_i$ of the same
upper-triangular form as $\nabla$.

The monodromy of $\nabla_\lambda$ is therefore $T_i^\lambda = \exp(2\pi i\,
\mathrm{Res}_{D_i}(\nabla_\lambda))$.  Since the residue of $\nabla_\lambda$
at $D_i$ is the same nilpotent matrix $N_i$ for all $\lambda\in S^1$ (the
phase $e^{i\phi}$ multiplies only the off-diagonal terms $\theta$ and
$\bar\theta$, whose residues at $D_i$ are zero or strictly lower/upper-triangular
below the diagonal), we get $T_i^\lambda = \exp(2\pi i N_i) = T_i$, which is
unipotent.  Hence $\nabla_\lambda\in R^{\mathrm{nil}}$ for all $\lambda\in S^1$.
\end{proof}

\subsection{Step 4: The Higgs bundle decomposition and the VHS}
\label{subsec:Higgs-to-VHS}

We now explain why the Higgs bundle $(\mathcal{E},\theta)$ underlying a
semisimple flat bundle with unipotent monodromy decomposes, via the 
 nonabelian Hodge
correspondence, into a VHS.  This is the heart of the ``Identifying the VHS''
step.

\subsubsection{Stability of Higgs bundles and the tame nonabelian Hodge
correspondence}

\begin{definition}
A Higgs bundle $(\mathcal{E},\theta)$ on $X^*$ is \emph{stable} (resp.\ 
\emph{semistable}) if for every $\theta$-invariant subsheaf
$0\neq\mathcal{F}\subsetneq\mathcal{E}$ (i.e.\ $\theta(\mathcal{F})\subset
\mathcal{F}\otimes\Omega^1_{X^*}$), the slope inequality
$\mu(\mathcal{F}) < \mu(\mathcal{E})$ (resp.\ $\leq$) holds, where
$\mu(\mathcal{F}) = \deg(\mathcal{F})/\mathrm{rank}(\mathcal{F})$.
\end{definition}

A key theorem in 
non-abelian Hodge theory, as 
extended to the
quasi-projective setting by Mochizuki, is:

\begin{theorem}[{\cite{Simpson-IHES}} for projective $X$;
  Mochizuki {\cite{Mochizuki,Mochizuki2}} for quasi-projective $X^*$]
\label{thm:simpson-correspondence}
There is a natural equivalence of categories (the \emph{non-abelian Hodge
correspondence}, in the tame quasi-projective form established by
Mochizuki)
\begin{equation}\label{eq:simpson-corr}
\left\{\begin{array}{c}
  \text{semisimple flat bundles}\\[2pt]
  (V,\nabla_{\mathrm{flat}})\text{ on }X^*
\end{array}\right\}
\;\xleftrightarrow[\displaystyle\sim]{\;\;\;\;\;\;\;\;\;\;\;\;\;\;\;\;\;\;\;}\;
\left\{\begin{array}{c}
  \text{polystable Higgs bundles }(\mathcal{E},\theta)\\[2pt]
  \text{on }X^*\text{ with }\mu(\mathcal{E})=0
\end{array}\right\}
\end{equation}
in which the two directions are:
\begin{enumerate}[leftmargin=2em,label=\rm(\arabic*)]
  \item \textbf{Flat to Higgs} (Corlette--Donaldson direction): given a
    semisimple flat bundle $(V,\nabla_{\mathrm{flat}})$, the pluri-harmonic
    metric $h$ of Theorem~\ref{thm:harmonic-exists} determines a Higgs
    bundle $(\mathcal{E},\theta)$ by setting $\mathcal{E} = (V,\delbar_h)$
    (the holomorphic structure induced by $h$) and
    $\theta = \nabla_{\mathrm{flat}}^{1,0} - \nabla_h^{1,0}$ (the $(1,0)$-part
    of the difference between $\nabla_{\mathrm{flat}}$ and the Chern connection
    $\nabla_h$).

  \item \textbf{Higgs to flat} (Hitchin direction): given a polystable Higgs
    bundle $(\mathcal{E},\theta)$ of slope zero, Theorem~\ref{thm:harmonic-exists}
    applied in the Higgs-to-flat direction gives the unique pluri-harmonic
    metric $h$ on $\mathcal{E}$ (solving the same equation
    $F(\nabla_h)+[\theta,\bar\theta]=0$), and the flat connection is recovered
    as $\nabla_{\mathrm{flat}} = \nabla_h + \theta + \bar\theta$.
\end{enumerate}

\noindent
Note that the underlying $\mathcal{C}^\infty$ bundles on both sides are the same
object: $V = \mathcal{E}$ as smooth bundles (the correspondence is the identity
on underlying smooth bundles).  The distinction is that the flat bundle uses the
flat connection $\nabla_{\mathrm{flat}}$, while the Higgs bundle uses the pair
$(\delbar_h,\theta)$.
\end{theorem}

\begin{remark}[Attribution of the correspondence]\label{rem:attribution}
The non-abelian Hodge correspondence was established in several stages, each
handling a more general geometric setting.

\begin{enumerate}[leftmargin=2em,label=\rm(\arabic*),itemsep=4pt]
  \item \textbf{Compact curves (Hitchin, 1987).}  Hitchin \cite{Hitchin}
    proved the correspondence for rank-$2$ bundles on smooth projective
    curves.  He introduced the Higgs bundle formalism and solved the
    self-duality equations on curves using gauge theory. 
    Donaldson treated the existence of harmonic maps in this setting 
    in an appendix to \cite{Hitchin}

  \item \textbf{Compact K\"{a}hler manifolds (Corlette 1988;
    the second author, 1992).}
    Corlette \cite{Corlette} proved the existence of the pluri-harmonic metric
    (flat-to-Higgs direction) on compact K\"{a}hler manifolds via the
    equivariant harmonic map theorem.  
    The second author
     \cite{Simpson-IHES} then proved
    both directions of the correspondence for arbitrary smooth projective
    varieties $X$ (in any dimension), establishing the equivalence between
    semisimple flat bundles and polystable Higgs bundles of slope zero.

  \item \textbf{Noncompact curves (the second author, 1990).}  
  The second author
    \cite{Simpson-harmonic} extended the correspondence to the case
    $X^* = X\setminus D$ where $X$ is a smooth projective \emph{curve} and
    $D$ is a finite set of points.  That paper treats the tame case (regular singular
    connections, which includes our unipotent monodromy hypothesis) and the
    wild case.  That paper also establishes the precise growth conditions
    (``Higgs bundles on noncompact curves'') which give the title.  However,
    this paper handles only the curve case, not higher-dimensional
    quasi-projective varieties.

  \item \textbf{Quasi-projective varieties, unipotent monodromy
    (Jost-Zuo, 1996).}  Jost and Zuo \cite{JostZuo} proved the existence of a
    tame pluriharmonic metric for every reductive representation of
    $\pi_1(X^*)$ with unipotent local monodromy, for $X^*=X\setminus D$ of
    arbitrary dimension with $D$ a normal crossings divisor.  This is exactly
    the hypothesis under which we work, so their theorem already suffices for
    the existence statement of Theorem~\ref{thm:harmonic-exists}; we quote
    Mochizuki below for the full equivalence of categories and for the
    asymptotic estimates.

  \item \textbf{Quasi-projective varieties in higher dimensions (Mochizuki,
    2002--2007).}  The extension to $X^*$ of arbitrary dimension is the main
    achievement of Mochizuki's two-volume Memoirs
    \cite{Mochizuki,Mochizuki2}.  He introduces the category of
    \emph{tame harmonic bundles} on $(X,D)$ (where $D$ is a normal crossings
    divisor) and proves:
    \begin{itemize}[leftmargin=2em]
      \item existence and uniqueness of the pluri-harmonic metric for any
        semisimple tame flat bundle (Theorem 1.1 in \cite{Mochizuki});
      \item the equivalence of categories between semisimple tame flat bundles
        and polystable ``tame'' Higgs bundles (i.e.\ those whose Higgs field
        has at most a simple pole along $D$ with nilpotent residue);
      \item precise asymptotic estimates for the metric near $D$, with
        polynomial growth in the tame case and exponential growth in the
        wild case.
    \end{itemize}

  \item \textbf{Our setting.}  In this paper $X^*$ is a smooth quasi-projective
    variety, $D$ is a normal crossings divisor, and the monodromy around each
    $D_i$ is \emph{unipotent} (the special case of tame monodromy with
    $\alpha_i = 0$).  The non-abelian Hodge correspondence we use is therefore
    Mochizuki's theorem, and the attribution in
    Theorem~\ref{thm:simpson-correspondence} reflects this; the existence of the
    pluriharmonic metric in our unipotent case is already contained in
    \cite{JostZuo}.  We note that \cite{Simpson-harmonic} does not treat
    the quasi-projective correspondence in dimension $\geq 2$; it covers only
    noncompact curves.
\end{enumerate}
\end{remark}  Since the flat bundle $(V,\nabla_{\mathrm{flat}})$ is semisimple
(a direct sum of irreducibles), the corresponding Higgs bundle is polystable
of slope zero.  Choosing a splitting into irreducible flat summands
$(V,\nabla_{\mathrm{flat}}) = \bigoplus_j(V_j,\nabla_j)$ with each $V_j$
irreducible, the correspondence gives a decomposition
\begin{equation}\label{eq:polystable-decomp}
  (\mathcal{E},\theta) \;=\; \bigoplus_{j=1}^m (\mathcal{E}_j,\theta_j),
\end{equation}
where each $(\mathcal{E}_j,\theta_j)$ is stable of slope zero and corresponds
to $(V_j,\nabla_j)$.  (Since each $V_j$ is irreducible as a flat bundle, the
corresponding Higgs bundle is stable, not merely polystable.)

\begin{remark}\label{rem:same-bundle}
The notation in \eqref{eq:simpson-corr} uses $(V,\nabla_{\mathrm{flat}})$ on
the left and $(\mathcal{E},\theta)$ on the right to emphasize the different
structures, even though $V = \mathcal{E}$ as $\mathcal{C}^\infty$ bundles.
In the literature one often writes $(E,\nabla)$ for both, understanding that
the same bundle $E$ appears with different additional structures on the two sides.
\end{remark}

\subsubsection{Stable Higgs bundles of slope zero and Hodge types}

The stable summands in the decomposition \eqref{eq:polystable-decomp} carry
a special structure.

\begin{theorem}[Mochizuki {\cite{Mochizuki}}]
\label{thm:VHS-decomp}
Let $X^*$ be a smooth quasi-projective variety.
A stable Higgs bundle $(\mathcal{E}_j,\theta_j)$ of slope zero on $X^*$ that
arises from an irreducible semisimple flat bundle via the nonabelian Hodge
correspondence decomposes as a direct sum of holomorphic sub-bundles
\begin{equation}\label{eq:Hodge-decomp}
  \mathcal{E}_j \;=\; \bigoplus_{p\in\zz} \mathcal{E}_j^p,
\end{equation}
such that the Higgs field satisfies the \emph{Griffiths transversality}
condition:
\begin{equation}\label{eq:Griffiths}
  \theta_j\bigl(\mathcal{E}_j^p\bigr) \;\subset\; \mathcal{E}_j^{p-1}
  \otimes\Omega^1_{X^*}.
\end{equation}
That is, $\theta_j$ strictly lowers the index $p$ by one.
\end{theorem}

\begin{remark}[Source of the deformation, and preservation of unipotent
monodromy]\label{rem:mochizuki-r-rho}
The decomposition of Theorem~\ref{thm:VHS-decomp} should not be read as a
consequence of stability of the Higgs bundle alone. Rather, it is obtained
from Mochizuki's deformation theorem
\cite[\S 10.1, Theorem 10.5]{MochizukiAsterisque}, and there are really two
logically separate claims involved, which we record separately.

\emph{(1) Deformation to a VHS.} After passing to the semisimple
representation and its associated tame harmonic bundle $(\mathcal E,\theta)$,
Mochizuki considers the $\cc^*$-deformation $(\mathcal E,\theta)\mapsto
(\mathcal E,t\theta)$, $t\in\cc^*$, and uses a descending induction on the
invariant $r(\rho)$ (the sum of the multiplicities of the stable
constituents of the associated polystable Higgs bundle) to obtain a limit
representation whose Higgs bundle is a \emph{system of Hodge bundles}: this
is exactly the graded object of \eqref{eq:Hodge-decomp}, satisfying
Griffiths transversality by construction. This first claim is what
\cite[Theorem 10.5]{MochizukiAsterisque} directly supplies.

\emph{(2) The deformation stays inside $R^{\mathrm{nil}}$.} This is the
additional fact our setting needs, and is not itself part of Mochizuki's
statement. Since the original representation lies in $R^{\mathrm{nil}}$,
its associated tame harmonic bundle has trivial parabolic weights and
nilpotent residue at each $D_i$. The deformation used in Mochizuki's
argument leaves the underlying parabolic filtration and its weights
unchanged; and the residue endomorphisms remain nilpotent under the
$\cc^*$-scaling, since $(tN)^k=t^kN^k$ shows $tN$ is nilpotent whenever $N$
is, for every $t\in\cc^*$. Thus the local monodromy remains unipotent
throughout the deformation, so the limiting graded Higgs bundle still has
trivial parabolic structure and nilpotent residues, and the corresponding
flat bundle lies in $R^{\mathrm{nil}}$, as needed for
Theorem~\ref{thm:simpson-correspondence} to apply to it.
\end{remark}

\begin{remark}
The index $p$ is the \emph{Hodge type} or \emph{Hodge degree}.  If we set
$q := n - p$ for a fixed integer $n$ (the \emph{weight}), then
$\mathcal{E}_j^p$ corresponds to the $(p,q) = (p,n-p)$ piece of a pure Hodge
structure of weight $n$.  The grading in \eqref{eq:Hodge-decomp} and the
condition \eqref{eq:Griffiths} are exactly the data of a complex VHS.
\end{remark}

\subsubsection{The polarization}

The pluri-harmonic metric $h$ provides the missing piece: a \emph{polarization}.

\begin{proposition}\label{prop:polarization}
Let $(E,\theta,h)$ be a harmonic bundle corresponding to a semisimple flat
bundle $\zeta^{ss}$ with unipotent monodromy.  Define the \emph{Hodge metric}
\[
  h_{\mathrm{Hodge}}(u,v) \;:=\; h(C\,u, v),
\]
where $C$ is the \emph{Weil operator}: $C|_{\mathcal{E}_j^p} = i^{p-q}
= i^{2p-n}$ on the $(p,q)$-piece.  Then $h_{\mathrm{Hodge}}$ is positive
definite on each fibre.

The sesquilinear form
\begin{equation}\label{eq:polarization}
  Q(u,v) \;:=\; h(u,\bar{v})
\end{equation}
(viewing $h$ without the Weil operator) is a flat indefinite Hermitian form
preserved by the flat connection $\nabla = \nabla_h+\theta+\bar\theta$.  It
satisfies $Q|_{\mathcal{E}_j^p\otimes\overline{\mathcal{E}_j^q}} = 0$ for
$p\neq n-q$ (Hodge-Riemann bilinear relations of the first kind) and
$(-1)^p Q(u,\bar{u})>0$ for $0\neq u\in\mathcal{E}_j^p$ (bilinear relations
of the second kind, after restriction to each Hodge summand).
\end{proposition}

\begin{proof}
Preservation by $\nabla$: the harmonic metric $h$ satisfies $\nabla h = 0$
(the connection is compatible with the metric in the sense that
$dh(u,v) = h(\nabla u,v)+h(u,\nabla v)$ for all sections $u,v$), which is
the flatness condition on the metric.  Hence $Q = h$ is a flat Hermitian form.

The bilinear relations: these are the standard Hodge-Riemann bilinear
relations for a VHS, which hold for the harmonic metric associated to a
semisimple representation.  The proof uses the fact that $h$ is the unique
$\mathrm{GL}_r(\cc)$-equivariant solution of the harmonic bundle equations,
and the equivariance forces the positivity properties on each Hodge summand.
See \cite{Simpson-IHES} and \cite{Mochizuki} for details.
\end{proof}

\subsubsection{The resulting VHS}

\begin{theorem}\label{thm:zeta0-VHS}
The flat connection $\zeta_0$ defined by the harmonic bundle data
$(E,\theta,h)$ of $\zeta^{ss}$ at $\lambda = 1$ underlies a polarized complex
VHS on $X^*$ of the following form:
\begin{enumerate}[leftmargin=2em,label=\rm(\arabic*)]
  \item \textbf{Local system}: $\zeta_0 = \bigoplus_j \zeta_j$ where each
    $\zeta_j$ is an irreducible local system with fibres
    $V_j = \bigoplus_p \mathcal{E}_j^p|_{x_0}$ (the Hodge decomposition on the
    fibre at $x_0$).
  \item \textbf{Hodge filtration}: the Hodge filtration on each $V_j$ is
    $F^k V_j = \bigoplus_{p\geq k}\mathcal{E}_j^p|_{x_0}$.
  \item \textbf{Griffiths transversality}: $\nabla(F^k)\subset F^{k-1}\otimes
    \Omega^1_{X^*}$, which holds because $\theta$ shifts degree by $-1$
    (equation \eqref{eq:Griffiths}).
  \item \textbf{Polarization}: the flat indefinite Hermitian form $Q$ of
    Proposition~\ref{prop:polarization} is the polarization.
\end{enumerate}
\end{theorem}

\begin{proof}
Items (1)--(3) are a restatement of Theorem~\ref{thm:VHS-decomp}: the
decomposition \eqref{eq:Hodge-decomp} with Griffiths transversality
\eqref{eq:Griffiths} is precisely the definition of a complex VHS.  The
polarization (4) is furnished by Proposition~\ref{prop:polarization}.

That $\zeta_0$ has the \emph{same underlying topological type} as $\zeta^{ss}$
follows because the non-abelian Hodge correspondence is a diffeomorphism on
the underlying $\mathcal{C}^\infty$ bundles: the flat bundle $(E,\nabla_1) =
(E,\nabla) = \zeta^{ss}$ and the VHS $(E,\nabla_1,F^\hdot,Q) = \zeta_0$ use
the \emph{same} bundle $E$ and the \emph{same} flat connection $\nabla_1 =
\nabla$ at $\lambda=1$.  In particular, $\zeta_0$ has the same monodromy
representation as $\zeta^{ss}$, hence the same unipotent monodromy around each
$D_i$.
\end{proof}

\subsection{Step 5: The path in \texorpdfstring{$R^{\mathrm{nil}}$}{Rnil}}
\label{subsec:path-in-Rnil}

We now assemble all steps into the proof of Theorem~\ref{thm:deform-to-vhs}.

\begin{theorem}[Detailed proof of Theorem~\ref{thm:deform-to-vhs}]
\label{thm:deform-detailed}
There is an explicit path $\zeta(t)$, $t\in[0,1]$, inside $R^{\mathrm{nil}}$
from $\zeta = \zeta(0)$ to $\zeta_0 = \zeta(1)$, a polarized complex VHS.
\end{theorem}

\begin{proof}
The path is the concatenation of two steps.

\textit{Step A: Path from $\zeta$ to $\zeta^{ss}$ (Lemma~\ref{lem:ss-path}).}
Define $\zeta_t^A := \zeta^{ss} + (1-t)\,\zeta^u$ for $t\in[0,1]$, where
$\zeta^u$ is the strictly upper-triangular part in the Jordan-H\"{o}lder frame.
By Lemma~\ref{lem:ss-Rnil} and Lemma~\ref{lem:ss-path}, this is a polynomial
path in $R^{\mathrm{nil}}$ from $\zeta^A_0 = \zeta$ to $\zeta^A_1 = \zeta^{ss}$.

\textit{Step B: Path from $\zeta^{ss}$ to $\zeta_0$ (the $\cc^*$-family).}
Let $(E,\theta,h)$ be the harmonic bundle of $\zeta^{ss}$
(Theorem~\ref{thm:harmonic-exists}).  Define $\zeta_s^B := \nabla_\lambda$
with $\lambda = e^{i\pi s/2}$ for $s\in[0,1]$.  Then:
\begin{itemize}[leftmargin=2em]
  \item At $s=0$: $\lambda=1$, so $\nabla_1 = \nabla = \nabla^{ss} = \zeta^{ss}$.
  \item At $s=1$: $\lambda = e^{i\pi/2} = i$, so
    $\nabla_i = \nabla_h + i^{-1}\theta + i\,\bar\theta
    = \nabla_h - i\theta + i\,\bar\theta$.
  \item All points lie in $R^{\mathrm{nil}}$ by Proposition~\ref{prop:lambda-circle}.
  \item The endpoint $\nabla_i$ is the flat connection of the VHS $\zeta_0$.
    To see this: at $\lambda = i$, the connection matrix in the Hodge frame
    \eqref{eq:Hodge-decomp} satisfies $i^{-1}\theta|_{\mathcal{E}_j^p}
    = -i\cdot\theta|_{\mathcal{E}_j^p}$ (going from degree $p$ to degree $p-1$)
    and $i\,\bar\theta|_{\mathcal{E}_j^{p-1}} = i\cdot\bar\theta$ (going from
    degree $p-1$ to degree $p$).  The resulting flat connection
    $\nabla_i$ preserves the Hodge filtration up to a phase, giving a
    connection compatible with the VHS structure.
\end{itemize}

\textit{Concatenation.}
Define the path
\[
  \zeta(t) := \begin{cases}
    \zeta^A_{2t} & 0\leq t\leq\tfrac{1}{2}, \\
    \zeta^B_{2t-1} & \tfrac{1}{2}\leq t\leq 1.
  \end{cases}
\]
This is a continuous path (piecewise smooth) in $R^{\mathrm{nil}}$ from
$\zeta(0) = \zeta$ to $\zeta(1) = \zeta_0$.  It is a path in $R^{\mathrm{nil}}$
by the results of Steps A and B.
\end{proof}

\subsection{Why this reduction suffices for the main theorem}
\label{subsec:reduction-suffices}

The full force of the deformation to a VHS (Theorem~\ref{thm:deform-detailed})
enters the proof of Theorem~\ref{thm:main} as follows.

\begin{enumerate}[leftmargin=2em,label=\rm(\arabic*),itemsep=6pt]

  \item \textbf{Deformation invariance.}  By Theorem~\ref{thm:deformation-app}
    (proved in Appendix~\ref{app:deligne-patch}), the Chern-Simons class
    \[
      \CS_p(\zeta)\;=\;\CS_p(\zeta_0)\;\in\; H^{2p-1}(X,\cc/\zz)
    \]
    because $\zeta$ and $\zeta_0$ lie in the same connected component of
    $R^{\mathrm{nil}}$ (they are connected by the path of Theorem
    \ref{thm:deform-detailed}).  It therefore suffices to prove torsion for
    $\zeta_0$.

  \item \textbf{The VHS gives a flat indefinite Hermitian form.}
    For $\zeta_0$ a polarized VHS, Theorem~\ref{thm:zeta0-VHS}(4) supplies a
    flat indefinite Hermitian form $Q$ on the underlying local system
    $L_0 = E|_{X^*}$ preserved by $\nabla_0 = \nabla^{\zeta_0}$.  The monodromy
    logarithms $N_i = \log T_i^{\zeta_0}$ are $Q$-self-adjoint:
    $Q(N_i u,v) = Q(u,N_i v)$ for all $u,v$.

  \item \textbf{$N$-isotropic filtrations exist.}
    Because the $N_i$ are $Q$-self-adjoint and commute (Cattani-Kaplan-Schmid
    theorem \cite{CKS}), the poset of $N_\hdot$-isotropic filtrations on each
    fibre is contractible (Proposition~\ref{prop:nisofilcontr}).  The
    \v{C}ech section theorem (Theorem~\ref{thm:cech-section}) then produces
    a hermitian pre-patching collection for the Deligne canonical extension
    of $\zeta_0$.

  \item \textbf{Volume regulator vanishes.}
    By Theorem~\ref{thm:main-hermitian} and Corollary~\ref{cor:maincor},
    the existence of the hermitian pre-patching collection forces
    $\vol_p(\nabla^{\Del}_{\zeta_0}) = 0$.

  \item \textbf{Torsion.}
    The vanishing of the volume regulator $\vol_p = 0$ (the imaginary component
    under $\cc/\zz = \rr/\zz\oplus i\rr$, Lemma~\ref{lem:CZ-splitting})
    combined with Borel's theorem, implies $\CS_p(\zeta_0)$ is torsion
    (Steps B and C of the proof of Theorem~\ref{thm:main}).
\end{enumerate}

The deformation to a VHS is thus the \emph{only point} where the specific
geometry of $\zeta$ is used.  All subsequent steps---the Hermitian patching, the
vanishing, and the torsion conclusion---depend only on the fact that the endpoint
$\zeta_0$ underlies a polarized VHS.  The non-abelian Hodge correspondence is the
bridge ensuring that such a VHS always exists within $R^{\mathrm{nil}}$.

\subsection{Diagram of the full reduction}
\label{subsec:diagram}

\begin{figure}[h]
\[
\xymatrix@R=1.8em@C=2.2em{
\zeta\in R^{\mathrm{nil}}
  \ar[r]^-{\text{Lem.~\ref{lem:ss-path}}}_-{\text{upper-tri.\ path}}
& \zeta^{ss}\in R^{\mathrm{nil}}
  \ar[r]^-{\text{Thm.~\ref{thm:harmonic-exists}}}_-{\text{Corlette-Mochizuki}}
& (E,\theta,h)\text{ harmonic}
  \ar[d]^-{\lambda\in S^1}_-{\text{Prop.~\ref{prop:lambda-circle}}} \\
& \CS_p(\zeta)=\CS_p(\zeta_0)
  \ar@{<-}[r]^-{\text{deform.\ inv.}}
& \zeta_0\in R^{\mathrm{nil}}
  \ar[d]^-{\text{Thm.~\ref{thm:zeta0-VHS}}}_-{\text{VHS decomp.}} \\
&& \zeta_0\text{ is a VHS with polar.\ }Q
  \ar[d]^-{\text{Prop.~\ref{prop:polarization}}} \\
&&
  N_i\ Q\text{-self-adj., }N\text{-iso.\ filts.\ contractible}
  \ar[d]^-{\text{Cor.~\ref{cor:maincor}}} \\
&&
  \vol_p(\nabla^{\Del}_{\zeta_0})=0
  \ar[d]^-{\text{Borel+splitting}} \\
&&
  \CS_p(\zeta)\text{ is torsion}
}
\]
\caption{The chain of deductions from $\zeta$ to the torsion of $\CS_p(\zeta)$.}
\label{fig:reduction}
\end{figure}
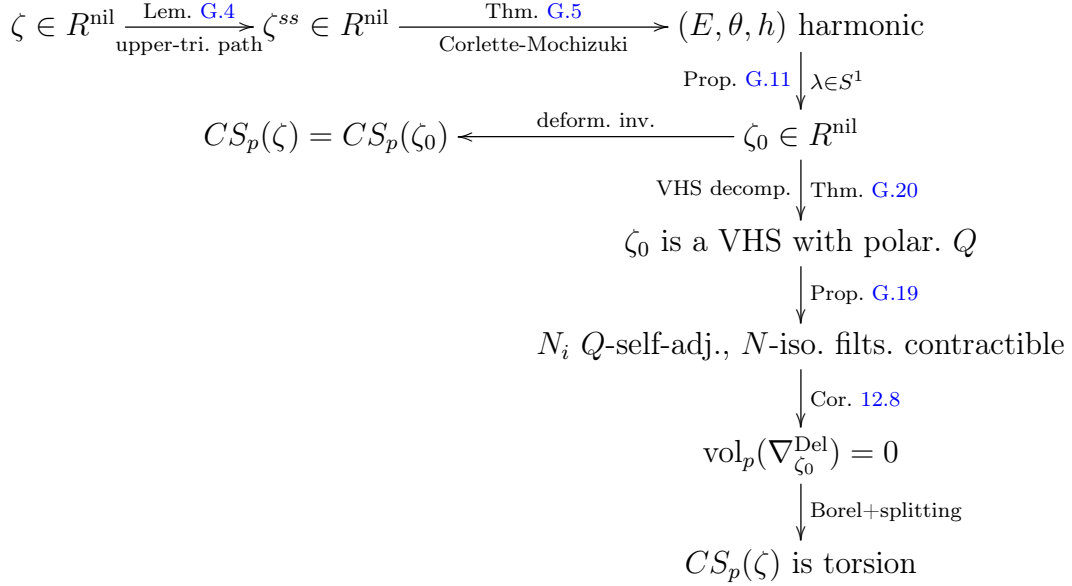

\begin{remark}
The passage $\zeta\to\zeta^{ss}\to(E,\theta,h)\to\zeta_0$ involves, at each
stage, objects of increasing geometric richness:
\begin{itemize}[leftmargin=2em]
  \item $\zeta\to\zeta^{ss}$: a purely algebraic (Jordan-H\"{o}lder) construction,
    working over any field.
  \item $\zeta^{ss}\to(E,\theta,h)$: an analytic construction using the
    solution of a non-linear elliptic PDE (the harmonic map equation), which
    requires $X$ to be K\"{a}hler (or quasi-projective in the tame setting).
  \item $(E,\theta,h)\to\zeta_0$: an algebro-geometric construction
    (the grading by Hodge type), made possible by the eigenspace decomposition
    of the harmonic metric.
\end{itemize}
Each step uses the specific geometry of the quasi-projective setting
$(X^*,D)$ in an essential way.
\end{remark}


\section{Deligne--Sullivan triviality of the canonical extension and the Rees model}
\label{sec:deligne-sullivan}

In this section we prove that the canonical extension of a flat bundle with
unipotent monodromy becomes $\mathcal{C}^\infty$-trivial over a finite cover
(Proposition~\ref{prop:DS-NCD} below), generalising to the full normal
crossings setting the two-divisor result of \cite{IS-2div}.  The key new
ingredient is that the multi-Rees construction of this paper (Section~\ref{sec:multrees})
plays the role played by the two-variable interpolation sheaf in
\cite{IS-2div}: in both cases the topological model of the canonical extension
is an explicit locally-free module that becomes trivial mod any prime once the
monodromy is trivialised.

\subsection{Setup and statement}

Throughout this section $X$ is a smooth projective variety over $\cc$,
$D = D_1\cup\cdots\cup D_k$ is a simple normal crossings divisor with smooth
irreducible components, and $X^* = X\setminus D$.  We fix a flat bundle
$(E,\nabla)$ on $X^*$ with \emph{unipotent monodromy} around each $D_i$
(i.e.\ the local monodromy operators $T_i = \rho(\gamma_i)$ are unipotent,
where $\gamma_i$ is a small loop around $D_i$).  Let $L = E^{\nabla}$ be the
associated local system, and let $(F,\nabla^{\Del})$ denote Deligne's canonical
extension to $X$ \cite{Deligne70} (a vector bundle on $X$ with logarithmic
connection having nilpotent residues along $D$).

\begin{proposition}[Deligne--Sullivan triviality, normal crossings case]%
\label{prop:DS-NCD}
In the above situation there exists a finite covering $\pi:\tilde X^*\to X^*$
such that, writing $\tilde X$ for the normalization of $X$ in $\tilde X^*$:
\begin{enumerate}[leftmargin=2em,label=\rm(\alph*)]
  \item $\tilde X$ is smooth and $\tilde X\to X$ is finite and \'etale over $X^*$,
    branched exactly along $D$, with ramification index $n_i\geq 1$ along the
    strict transform of $D_i$ for each $i$, and product ramification
    indices $(n_i)_{i\in I}$ in the normal directions along each stratum $D_I=
    \bigcap_{i\in I}D_i$;
  \item the canonical extension of $\pi^*E$ to $\tilde X$ is trivial as a
    $\mathcal{C}^\infty$-bundle.
\end{enumerate}
\end{proposition}

The proof occupies \S\S\ref{subsec:DS-arith}--\ref{subsec:DS-hasse} below.
It follows the four-step strategy of the two-divisor case \cite[Prop.~6.6]{IS-2div},
with the two-variable interpolation sheaf (= multi-Rees module) of that paper
replaced by the multi-Rees construction of Section~\ref{sec:multrees}.

\begin{remark}\label{rem:DS-why}
In the two-divisor case ($k=2$), two separate difficulties had to be overcome:
(A) failure of common refinement at the corner (two filtrations of a vector
space generically do not admit a common refinement), and (B) failure of the
codimension argument for the two-parameter deformation collar; see
\cite[\S1]{IS-2div}.  Both were resolved by keeping the \emph{pair}
$(W^{(1)},W^{(2)})$ at the corner and using the elementary two-filtration
splitting lemma of Deligne.  For $k\geq 3$ this pair strategy breaks down
since three filtrations generically have no simultaneous splitting; however,
the Rees construction of this paper sidesteps the issue entirely by always
placing the \emph{single} filtration $W(N_I) = W\!\left(\sum_{i\in I}N_i\right)$
at the stratum $D_I$ (on the cover $U_I$ of the adapted covering), and the
adapted covering ensures that on every non-empty intersection $U_I\cap U_J$
there is a containment of index sets $I\subset J$ or $J\subset I$, so that
only \emph{pairs of sequentially compatible} weight filtrations are ever compared.
For such pairs, the Cattani--Kaplan--Schmid theorem \cite{CKS} (and
Mochizuki's extension \cite{Mochizuki} in the quasi-projective case) guarantees
a common splitting, so the difficulty (A) does not arise.
\end{remark}

\subsection{Arithmetic model}
\label{subsec:DS-arith}

\begin{lemma}\label{lem:DS-arith}
Let $\{(W(I),\tau(I))\}_{I\in A}$ be the \v{C}ech patching data produced by
Corollary~\ref{cor:cech-exists} (or Corollary~\ref{cor:cech-exists-isotropic}
in the VHS case), over the adapted covering $\{U_I\}_{I\in A}$ of
Proposition~\ref{prop:covering}.  There exists a subring $A_0\subset\cc$,
finitely generated as a $\zz$-algebra, and a free $A_0$-module $V_{A_0}\cong A_0^r$
with an action of $\pi_1(X^*)$ through $\rho$, such that:
\begin{enumerate}[leftmargin=2em,label=\rm(\roman*)]
  \item the monodromy logarithms $N_i = \log T_i\in\End_{A_0}(V_{A_0})$;
  \item for every non-empty multi-index $I\in A$, the weight filtration
    $W(N_I)$ of $N_I = \sum_{i\in I}N_i$ has all steps as free $A_0$-direct
    summands of $V_{A_0}$;
  \item the successive quotients $\Gr^{W(N_I)}_j(V_{A_0})$ are free
    $A_0$-modules, and the trivialization $\tau(I)$ (a flat connection on
    $\Gr^{W(N_I)}(L|_{U_I^*})$) is defined over $A_0$;
  \item for each covering pair $I\subset J$, the functor $A\to\Xi_K$
    (Construction~\ref{constr:A-to-Xi}) at the morphism $I\to J$ is defined
    over $A_0$: the filtration $W(N_I\leq N_J):= \Gr^{W(N_I)}(W(N_J))$ on
    $V(I)_{A_0}:=\Gr^{W(N_I)}(V_{A_0})$ and the isomorphism
    $\Gr^{W(N_I\leq N_J)}(V(I)_{A_0})\cong V(J)_{A_0}$ are
    free $A_0$-modules and $A_0$-linear maps.
\end{enumerate}
\end{lemma}

\begin{proof}
Since $\pi_1(X^*)$ is finitely presented (quasi-projective varieties have
finitely presented fundamental groups), choose generators including small loops
$\gamma_i$ around each $D_i$.  Let $A_1 = \zz[\text{matrix entries of
}\rho(g)^{\pm 1}]_{g\in\text{generators}}$.  Since each $T_i = \rho(\gamma_i)$
is unipotent, the logarithm series $N_i = \log T_i$ is a finite sum and has
denominators dividing $(r-1)!$ (where $r=\mathrm{rk}\,E$); let $A_2 = A_1[(r-1)!^{-1}]$.
Over $A_2$ we have $N_i\in\End(V_{A_2})$.  Commutativity of the $N_i$ is
needed, and holds, \emph{locally at each stratum}.  For every non-empty $I$
with $D_I\neq\emptyset$ the normal crossings local model gives
$U_I\cong\Delta^n$ with
$U_I^*\cong(\Delta^*)^{|I|}\times\Delta^{\,n-|I|}$, so that
$\pi_1(U_I^*)\cong\zz^{|I|}$ is generated by the coordinate meridians
$\gamma_i$, $i\in I$.  Hence the local monodromies $T_i$, $i\in I$, commute in
this neighbourhood of $D_I$, and so do their logarithms $N_i$, $i\in I$.  This
is exactly the commutativity required to form $N_I=\sum_{i\in I}N_i$ and its
monodromy weight filtration $W(N_I)$.  No global commutativity is asserted:
meridians around arbitrary distinct components of $D$ need not commute in
$\pi_1(X^*)$, and the argument never uses that they do.

For each non-empty $I$, the weight filtration $W(N_I)$ has steps
$W(N_I)_j = \ker\bigl((N_I)^{|j|+1}\bigr)\cap\text{image conditions}$,
which are finitely generated $A_2$-submodules of $V_{A_2}$.  By \emph{generic
freeness} (Grothendieck: a finitely generated module over a finitely generated
$\zz$-algebra becomes free after inverting finitely many elements), applied
simultaneously to the finitely many modules $W(N_I)_j$, the quotients
$V_{A_2}/W(N_I)_j$, and the associated graded pieces $\Gr^{W(N_I)}_j$ (there
are only finitely many, since $A$ is finite and each weight filtration has
finitely many steps), there exists a \emph{single} localization
$A_0 = A_2[S^{-1}]$ over which all of these simultaneously become free
$A_0$-modules. Because each quotient $V_{A_0}/W(N_I)_j$ is then free, hence
projective, over $A_0$, the exact sequence
$0\to W(N_I)_j\to V_{A_0}\to V_{A_0}/W(N_I)_j\to0$ splits, so $W(N_I)_j$ is a
free \emph{direct summand} of $V_{A_0}$, as required (freeness of $W(N_I)_j$
and of $\Gr^{W(N_I)}_j$ alone would not by itself give this: a short exact
sequence of finitely generated modules over a general ring need not split
even when the sub- and quotient-modules are individually free, unless the
splitting is arranged for, as here, by choosing the quotient free/projective
before extracting the summand).  The morphisms $\Gr^{W(N_I\leq N_J)}$ in the
functor $A\to\Xi_K$ are then $A_0$-linear maps between free $A_0$-modules.
Since $A$ (the nerve poset) is finite, only finitely many elements of $A_2$
need to be inverted, so $A_0$ remains a finitely generated $\zz$-algebra.
\end{proof}

\subsection{Finite cover and lattice rectification}
\label{subsec:DS-finite-cover}

\begin{lemma}\label{lem:DS-finite-cover}
With $A_0$ as in Lemma~\ref{lem:DS-arith}, let $\mathfrak{q}_1, \mathfrak{q}_2$
be two maximal ideals of $A_0$ with distinct residue characteristics
$\ell_1\neq\ell_2$.  Then there exists a finite covering $\pi:\tilde X^*\to X^*$
such that:
\begin{enumerate}[leftmargin=2em,label=\rm(\roman*)]
  \item the representations $\rho$ and all the graded representations
    $\Gr^{W(N_I)}(\rho|_{\pi_1(U_I^*)})$ for $I\in A$, and all the morphisms
    in the functor $A\to\Xi_K$, become trivial modulo $\mathfrak{q}_1$ and
    modulo $\mathfrak{q}_2$ after pullback along $\pi$;
  \item for each stratum $D_I = \bigcap_{i\in I}D_i$ (with $I\neq\emptyset$),
    the induced finite-index sublattice of the central lattice
    $\zz^{|I|} = \langle\gamma_i\rangle_{i\in I}\subset\pi_1(U_I^*)$ is
    \emph{rectangular}: of the form $\bigoplus_{i\in I}n_i\zz$ for some
    $n_i\geq 1$, so that the normalization $\tilde X$ is smooth along the
    preimage of $D_I$.
\end{enumerate}
\end{lemma}

\begin{proof}
\textit{Part (i).}  Since $A_0$ is a finitely generated $\zz$-algebra and
$\mathfrak{q}_j$ is maximal, the residue field $A_0/\mathfrak{q}_j$ is finite,
so $\prod_j\mathrm{GL}_r(A_0/\mathfrak{q}_j)$ is a finite group.  The kernel of
$\pi_1(X^*)\to\prod_j\mathrm{GL}_r(A_0/\mathfrak{q}_j)$ (via $\rho$) is a normal
finite-index subgroup; let $\pi':\tilde X^{*\prime}\to X^*$ be the corresponding
covering.  Along each $U_I^*$, the representation $\Gr^{W(N_I)}(\rho)$ is
also an $A_0$-valued representation of a finitely generated group (a quotient
of $\pi_1(U_I^*)$); applying the same argument to each and intersecting
(finitely many) subgroups, we obtain a single covering along which all the
data of the functor $A\to\Xi_K$ are simultaneously trivial mod
$\mathfrak{q}_1,\mathfrak{q}_2$.

\textit{Part (ii).}  For each stratum $D_I$, the normal subgroup $\ker(\pi_1(U_I^*)\to
\pi_1(U_I))$ is normally generated by the $\gamma_i$, $i\in I$, which form the
\emph{central meridian subgroup of the chosen local normal crossings
neighbourhood} $U_I^*$: in the local model
$U_I^*\cong(\Delta^*)^{|I|}\times\Delta^{\,n-|I|}$ the coordinate meridians
generate $\pi_1(U_I^*)\cong\zz^{|I|}$, which is abelian.  Centrality is thus a
statement about the local group $\pi_1(U_I^*)$ only, and is not claimed for the
images of the $\gamma_i$ in the global group $\pi_1(X^*)$.  The covering $\pi'$ restricts, near $D_I$, to a
finite-index sublattice $\Lambda_I\subset\zz^{|I|}$ of the central torus.  By
exactly the argument of \cite[Lem.~6.8(ii)]{IS-2div}: any finite-index
sublattice $\Lambda\subset\zz^k$ contains $N\zz^k$ (with $N = [\zz^k:\Lambda]$)
as a rectangular sublattice; the further Kummer cover corresponding to
$N\zz^k\subset\Lambda$ is smooth at the corner stratum (since the toric variety
of the cone $\rr^k_{\geq 0}$ with the lattice $N\zz^k$ is a product of
$k$ copies of $\Delta\to\Delta$, $w_i\mapsto w_i^N$, which is smooth).  Taking
the intersection of all these Kummer covers over all strata $D_I$ gives the
final covering $\pi:\tilde X^*\to X^*$ satisfying both (i) and (ii).
\end{proof}

\begin{remark}
Part (ii) is the \emph{lattice rectification} step.  In the two-component case
the lattice is $\zz^2$ and the argument is \cite[Lem.~6.8(ii)]{IS-2div};
for general $k$, the same argument applies since $N\zz^k\subset\Lambda$ for
$N = [\zz^k:\Lambda]$ regardless of the shape of $\Lambda$.  The key point is
that the normal crossings hypothesis makes the $\gamma_i$, $i\in I$, into the
central meridian subgroup of the \emph{local} group $\pi_1(U_I^*)$, which is
what makes the local cover of a polydisk extend smoothly to a product of
branched covers.  Here too, no centrality in the global group $\pi_1(X^*)$ is
claimed or used.
\end{remark}

\subsection{The Rees topological model of the canonical extension}
\label{subsec:DS-rees-model}

We now show that the global Rees bundle $F_K$ of Theorem~\ref{thm:global-rees}
(or rather its base change to $\cc$, restricted to the interval realization
$\iiii\Cube A\simeq X$) provides the required topological model of the canonical
extension.

\begin{lemma}\label{lem:DS-rees-model}
The $\mathcal{C}^\infty$ complex vector bundle $F_K^{\cc}|_{\iiii\Cube A}$ on
$X\simeq\iiii\Cube A$ (the Rees bundle of Theorem~\ref{thm:global-rees},
restricted to the interval realization) is isomorphic, as a $\mathcal{C}^\infty$
bundle, to the canonical extension $F$ of $L$.  In particular, the classifying
map of $F$ factors as
\[
  X \;\simeq\; |\mathbb{I}\Cube A| \;\longrightarrow\; BGL_r(\cc)^+,
\]
where the second map is classified by the Rees local system of the functor
$A\to\Xi_K$ (Construction~\ref{constr:A-to-Xi}).

If the functor $A\to\Xi_K$ and all its data are defined over a subring
$A_0\subset\cc$ (as in Lemma~\ref{lem:DS-arith}), then the Rees bundle
descends to a locally free sheaf of $A_0$-modules on a model of $\iiii\Cube A$
over $\Spec A_0$, and the reduction modulo any ideal $\mathfrak{q}\subset A_0$
is a locally free $A_0/\mathfrak{q}$-module on the corresponding finite-field model.
\end{lemma}

\begin{proof}
Two assertions are involved in the first claim, and they are established by
different means; we separate them.

\emph{(i) Isomorphism of the underlying $\mathcal{C}^\infty$ bundles.}  This is
the explicit topological comparison provided by
Theorem~\ref{thm:global-rees} together with Proposition~\ref{prop:rL-CS}: the
Rees bundle on the projective cubical realization $\pp_K\Cube A$ carries the
canonical connection $\nabla^{\mathrm{can}}$ (the multi-variable Rees
connection, which is locally nil-flat but not flat), and its restriction to the
interval realization $\iiii\Cube A\simeq X$ agrees, up to a
$\mathcal{C}^\infty$ gauge transformation, with the canonical extension $F$.
The isomorphism $F_K^{\cc}|_{\iiii\Cube A}\cong F$ is obtained in this way,
directly from the Rees data.

\emph{(ii) Comparison of the Deligne classes.}  Granting (i), the
$F^1$-connection argument of Section~\ref{sec:F1} identifies the associated
characteristic data: the $F^1$-connection $\nabla$ on $F_K|_{\iiii\Cube A}$
produces the same Deligne cohomology class as $\nabla^{\Del}$
(Theorem~\ref{thm:BG-comparison}), and that class is independent of the choice
of $F^1$-connection (Theorem~\ref{thm:DHZ}).

The order of the two steps is essential.  Equality of Deligne cohomology
classes does not by itself imply that the underlying $\mathcal{C}^\infty$
bundles are isomorphic, so (i) is proved by the Rees comparison and is not
inferred from (ii).

For the arithmetic descent: the affine Rees module of
Construction~\ref{constr:A-to-Xi} is built from the $A_0$-lattice
$V_{A_0}$ and the weight filtrations $W(N_I)_{A_0}$ of Lemma~\ref{lem:DS-arith}.
On each affine cube $\aaa(c)$ of the cubical decomposition, the local Rees
module is
\[
  F_K(c)|_{\aaa(c)} \;=\;
  \bigoplus_{j\in\zz^{|c^{-1}(t)|}} \Gr^{W(c)}_j(V_{A_0})
  \otimes_{A_0}
  A_0[t_i : i\in c^{-1}(t)],
\]
which is a free $A_0[\mathbf{t}]$-module (by Lemma~\ref{lem:DS-arith}(ii)).
These modules glue along the face maps (by the functoriality of the $A\to\Xi_K$
morphisms), giving a locally free sheaf over $\Spec A_0$.  Reduction mod
$\mathfrak{q}$ is compatible with the face gluings, so the reduction is again
locally free, with fiber $(A_0/\mathfrak{q})^r$.
\end{proof}

\begin{remark}
The multi-Rees module
$\sum_{\mathbf{a}\in\zz^k} t_1^{a_1}\cdots t_k^{a_k}(W(N_I)_{a_1}\cap\cdots)$
that appears here is the natural generalization of the two-variable
interpolation sheaf of \cite[Lem.~6.9]{IS-2div}.  The key point noted in
\cite[Rmk.~6.11]{IS-2div} is that the two descriptions are the same object:
the Rees module of this paper and the two-variable interpolation sheaf of the
two-divisor case agree on their common domain ($k=2$ with a corner patch).
Proposition~\ref{prop:rL-CS} (Appendix~\ref{app:rees}) makes this explicit:
the local system of the Rees bundle over each affine cube $\aaa(c)$ is
precisely the deformation $\mathrm{Ad}(\psi^{c_1}_{t_1}\cdots\psi^{c_k}_{t_k})(\rho|)$
of the representation, the same formula used in \cite[\S 6.3]{IS-2div}.
\end{remark}

\subsection{Hasse principle and conclusion}
\label{subsec:DS-hasse}

\begin{lemma}\label{lem:DS-hasse}
With notation as in Lemmas~\ref{lem:DS-arith}--\ref{lem:DS-rees-model}, let
$\pi:\tilde X^*\to X^*$ and $\tilde X$ be as in Lemma~\ref{lem:DS-finite-cover}
for the ideals $\mathfrak{q}_1,\mathfrak{q}_2$.  Then the pullback of $F$ to
$\tilde X$ is trivial as a $\mathcal{C}^\infty$-bundle.
\end{lemma}

\begin{proof}
Let $n = r = \mathrm{rk}\,F$ and $d = \dim_\cc\tilde X$.  The classifying map
$f:\tilde X\to\mathbf{Gr}(n,\cc^{n+N})$ for the pullback of $F$
(for $N\geq d$, by a general position argument) factors through the
$d$-coskeleton.  Note that the stable-range condition is a condition on the
\emph{real} dimension: since $\dim_\rr\tilde X = 2d$, the requirement
$N\geq\tfrac12\dim_\rr\tilde X$ reads $N\geq d$, not $N\geq d/2$.  By Sullivan's Hasse principle for morphisms \cite{SullivanAdams},
the map $f$ is null-homotopic if and only if its $\ell$-adic completion
$f_{\hat\ell}$ is null-homotopic for every prime $\ell$.  Fix $\ell$.

Since $\mathfrak{q}_1,\mathfrak{q}_2$ have \emph{distinct} residue characteristics
$\ell_1\neq\ell_2$, at least one of them --- say $\mathfrak{q}$ --- has residue
characteristic $\neq\ell$.  By Lemma~\ref{lem:DS-finite-cover}(i), the monodromy
$\rho$ and all the data of the functor $A\to\Xi_K$ are trivial mod $\mathfrak{q}$
after pullback to $\tilde X^*$.  By Lemma~\ref{lem:DS-rees-model}, the Rees bundle
$F$ (hence its pullback to $\tilde X$) is a topological model for the canonical
extension $F$; and the reduction of the Rees bundle mod $\mathfrak{q}$ is,
by triviality of the monodromy data, isomorphic to the \emph{constant} bundle
$(A_0/\mathfrak{q})^r$ on the mod-$\mathfrak{q}$ model.  Hence $F|_{\tilde X}$
becomes trivial as a bundle of $A_0/\mathfrak{q}$-modules.  By a theorem of
Deligne \cite[Lemme]{DeSu} (as used in the analogous argument of
\cite[Lem.~6.10]{IS-2div}), a bundle that is $(\text{prime-to-}\ell)$-trivial
has trivial $\ell$-adic completion.  Hence $f_{\hat\ell}$ is null-homotopic.
As $\ell$ was arbitrary, $f$ is null-homotopic and $F|_{\tilde X}$ is
$\mathcal{C}^\infty$-trivial.
\end{proof}

\begin{remark}[Hypotheses of Sullivan's Hasse principle]
The invocation of \cite[Thm.~3.1 and \S 3]{SullivanAdams} above depends on
the precise form of Sullivan's arithmetic-square/Hasse-principle theorem
being applicable to the map $f$ in the stated range: in particular, on the
target $\mathbf{Gr}(n,\cc^{n+N})$ being of finite type and sufficiently
connected (or the argument being run one Postnikov stage at a time), on the
$d$-coskeletal truncation used above lying within the range where profinite
completion detects null-homotopy, on the passage from the finite
Grassmannian to its stable (rationalized/profinitely completed) model being
legitimate at this stage of the argument, and on the precise meaning of
``$\ell$-adic null-homotopy'' of $f$ (null-homotopy of the $\ell$-adic
completion $f_{\hat\ell}$ in Sullivan's sense, as opposed to, say, vanishing
after applying $H_*(-,\zz_\ell)$). We have not independently re-verified
each of these hypotheses against \cite{SullivanAdams} in the generality
used here, and flag this as a citation-level point that should be checked
against the precise statement of Sullivan's theorem before this appendix is
regarded as complete.
\end{remark}

\begin{proof}[Proof of Proposition~\ref{prop:DS-NCD}]
Let $A_0$ be the arithmetic model of Lemma~\ref{lem:DS-arith}.  Choose
two maximal ideals $\mathfrak{q}_1,\mathfrak{q}_2$ of $A_0$ with distinct
residue characteristics (they exist since $A_0$ is a finitely generated
$\zz$-algebra of positive Krull dimension, hence has infinitely many maximal
ideals with infinitely many distinct residue characteristics).  Let
$\pi:\tilde X^*\to X^*$ and $\tilde X$ be as in Lemma~\ref{lem:DS-finite-cover}.
By Lemma~\ref{lem:DS-finite-cover}(ii) and the smoothness criterion for
toric covers of polydisk corners, $\tilde X$ is smooth (part~(a) of the
Proposition).  By Lemma~\ref{lem:DS-hasse}, the canonical extension of
$\pi^*E$ to $\tilde X$ is $\mathcal{C}^\infty$-trivial (part~(b)).
\end{proof}

\subsection{The Chern-Simons classes via the Deligne--Sullivan covering}
\label{subsec:DS-CS}

Proposition~\ref{prop:DS-NCD} makes possible a second construction of the
extended Chern-Simons classes, independent of the theory of differential
characters.  On the covering $\tilde X$ the canonical extension becomes a
trivial $\mathcal{C}^\infty$ bundle, so the connection is a matrix-valued
one-form and the Chern-Simons class is represented by an explicit closed
differential form.  The content of the subsection is that this form-level
class descends, and agrees with the class $\CS_p(\nabla^{\Del})$ of
\S\ref{app:diffchar}.

\begin{construction}[Chern-Simons classes from the Deligne--Sullivan model]
\label{con:DS-CS}
Let $\pi:\tilde X\to X$ and
$\Phi:\pi^*F\overset{\sim}{\longrightarrow}\tilde X\times\cc^{\,r}$ be as in
Proposition~\ref{prop:DS-NCD}, and write
\[
  \Phi_*\bigl(\pi^*\nabla^{\Del}\bigr) \;=\; d+A,\qquad
  A\in A^1\bigl(\tilde X,\mathfrak{gl}_r(\cc)\bigr).
\]
By Proposition~\ref{prop:nilflat-ch-zero} the connection $\nabla^{\Del}$ is
locally nil-flat, so $\ch_p(\nabla^{\Del})=0$ as a form, and hence
$\ch_p(d+A)=\pi^*\ch_p(\nabla^{\Del})=0$.  Therefore the Chern-Simons form of
\eqref{eq:CS-form-normalised},
\[
  \mathrm{CS}_{2p-1}(A) \;=\; \frac{1}{p!\,(2\pi i)^p}\;p\int_0^1
  \Tr\bigl(A\wedge(t\,dA+t^2A\wedge A)^{p-1}\bigr)\,dt ,
\]
is \emph{closed}, and we may set
\[
  \widetilde{\CS}_p \;:=\;\bigl[\mathrm{CS}_{2p-1}(A)\bigr]
  \;\in\;H^{2p-1}(\tilde X,\cc/\zz).
\]
\end{construction}

\begin{proposition}\label{prop:DS-CS-welldef}
\leavevmode
\begin{enumerate}[leftmargin=2em,label=\rm(\arabic*)]
  \item $\widetilde{\CS}_p$ is independent of the trivialization $\Phi$.
  \item $\widetilde{\CS}_p = \pi^*\,\CS_p(\nabla^{\Del})$.
  \item $m\cdot\CS_p(\nabla^{\Del}) = \pi_*\widetilde{\CS}_p$, where
    $m=\deg\pi$ and $\pi_*$ is the Umkehr map.
\end{enumerate}
\end{proposition}

\begin{proof}
Both $X$ and $\tilde X$ are smooth projective varieties, hence closed oriented
$\mathcal{C}^\infty$ manifolds of the same real dimension $n=2\dim_\cc X$
(smoothness of $\tilde X$ is Proposition~\ref{prop:DS-NCD}(a)), and $\pi$ is a
morphism of degree $m$, in particular a smooth map.  Note that $\pi$ is
\emph{not} a covering: it is branched along $\pi^{-1}(D)$.  Only smoothness of
$\pi$ and its degree are used below, so this costs nothing.

\smallskip
\emph{(1) Independence of the trivialization.}
Two trivializations $\Phi,\Phi'$ differ by a smooth map
$g:\tilde X\to GL_r(\cc)$, and the corresponding connection forms are related
by the gauge transformation $A'=A^g=gAg^{-1}-(dg)g^{-1}$.  Both
$\mathrm{CS}_{2p-1}(A)$ and $\mathrm{CS}_{2p-1}(A^g)$ are closed, since
Chern-Weil forms are gauge invariant and vanish by
Construction~\ref{con:DS-CS}.  Lemma~\ref{lem:gauge-variation} gives
\[
  \mathrm{CS}_{2p-1}(A^g)-\mathrm{CS}_{2p-1}(A) \;=\; -\,g^*\mu_{2p-1}+d\beta ,
\]
with $g^*\mu_{2p-1}$ of integral periods; so the two forms define the same
class in $H^{2p-1}(\tilde X,\cc/\zz)$.

The statement has content: the set of trivializations is a torsor under
$\mathrm{Map}(\tilde X,GL_r(\cc))$, whose $\pi_0$ is $[\tilde X,GL_r(\cc)]$
and is in general non-trivial; so $g$ need not be null-homotopic, and
$g^*\mu_{2p-1}$ need not vanish in $H^{2p-1}(\tilde X,\cc)$.  What
Lemma~\ref{lem:gauge-variation} provides is that it vanishes modulo $\zz$.
This is precisely where the normalising factor of
\eqref{eq:CS-form-normalised} is used.

\smallskip
\emph{(2) Comparison with the differential character.}  Three steps.

\emph{(2a) Naturality.}  The Cheeger-Simons character is natural for smooth
maps: for $f:Y\to X$ smooth and $(E,\nabla)$ any bundle with connection,
$\widehat{\ch}_p(f^*E,f^*\nabla)=f^*\widehat{\ch}_p(E,\nabla)$.  This is
immediate from the definition in \S\ref{app:diffchar}, the classifying map of
$f^*E$ being the composite of $f$ with that of $E$.  Applied to $f=\pi$,
\[
  \widehat{\ch}_p\bigl(\pi^*F,\pi^*\nabla^{\Del}\bigr)
  \;=\;\pi^*\,\widehat{\ch}_p\bigl(F,\nabla^{\Del}\bigr)
  \;=\;\pi^*\,\CS_p(\nabla^{\Del}).
\]
The branching of $\pi$ is irrelevant here: naturality requires only that
$\pi$ be smooth.

\emph{(2b) Flatness.}  The curvature of
$\widehat{\ch}_p(\pi^*F,\pi^*\nabla^{\Del})$ is
$\pi^*\ch_p(\nabla^{\Del})=0$, so by \eqref{eq:diff-char-exact} this character
is flat and both sides of the asserted identity are classes in
$H^{2p-1}(\tilde X,\cc/\zz)$.

\emph{(2c) Evaluation.}  Transporting through $\Phi$ replaces
$(\pi^*F,\pi^*\nabla^{\Del})$ by $(\tilde X\times\cc^r,\,d+A)$, and
$\widehat{\ch}_p$ is an invariant of the isomorphism class of a bundle with
connection.  Using the straight-line homotopy $\nabla_t=d+tA$ from the trivial
connection and Proposition~\ref{prop:transgression}, together with
$F(\nabla_t)=t\,dA+t^2A\wedge A$,
\[
  \widehat{\ch}_p(d+A)-\widehat{\ch}_p(d)
  \;=\;\bigl[\mathrm{CS}_{2p-1}(A)\bigr] ,
\]
and $\widehat{\ch}_p(d)=0$ for $p\geq1$.  Combined with (2a) this is the
claim.

\smallskip
\emph{(3) The transfer identity.}  Define
$\pi_*:=\mathrm{PD}_X\circ\pi_*\circ\mathrm{PD}_{\tilde X}^{-1}$ on
$H^*(-,\cc/\zz)$, using Poincar\'e duality with coefficients in $\cc/\zz$,
available on any closed oriented manifold.  The projection formula
$\pi_*(\pi^*x\frown y)=x\frown\pi_*y$ with $y=[\tilde X]$ and
$\pi_*[\tilde X]=m[X]$ gives $\pi_*\pi^*=m\cdot\mathrm{id}$.  Applying $\pi_*$
to (2) yields the assertion.
\end{proof}

\begin{remark}\label{rem:DS-is-construction}
Proposition~\ref{prop:DS-CS-welldef} may be read as an independent
\emph{construction} of $\CS_p(\nabla^{\Del})$: one first defines
$\widetilde{\CS}_p\in H^{2p-1}(\tilde X,\cc/\zz)$ by an explicit closed form on
the Deligne--Sullivan cover, then observes that it is invariant under the deck
group $G$ of $\tilde X\to X$.  Deck invariance is immediate from part~(1): for
$\sigma\in G$ one has $\pi\sigma=\pi$, so $\sigma^*\Phi$ is another
trivialization of $\pi^*F$ and hence
$\sigma^*\widetilde{\CS}_p=\widetilde{\CS}_p$.  (No averaging of $\Phi$ over
$G$ is available, the space of trivializations being a torsor under
$\mathrm{Map}(\tilde X,GL_r(\cc))$, which is neither contractible nor convex.)

Granting deck invariance, $\widetilde{\CS}_p$ descends to a class in
$H^{2p-1}(X,\cc/\zz)\otimes\zz[1/m]$ whose $m$-fold multiple is
$\pi_*\widetilde{\CS}_p$.  The content of
Proposition~\ref{prop:DS-CS-welldef}(2) is that this descended class has no
denominators and coincides with the differential-character definition.  The
advantage of the route through $\tilde X$ is that it is manifestly arithmetic:
everything is computed from the $K$-rational Rees model $F_K$ of
Lemma~\ref{lem:DS-rees-model}.
\end{remark}

\begin{remark}\label{rem:DS-CS-no-order}
It bears repeating, in view of the explicit form $\mathrm{CS}_{2p-1}(A)$ now
available on $\tilde X$, that $\widetilde{\CS}_p$ is \emph{not} zero merely
because $\pi^*F$ is a trivial bundle.  The Chern-Simons class is a secondary
invariant of the connection, not of the underlying bundle; the standard
example is a flat $SU(2)$ connection on the trivial bundle over a spherical
space form $S^3/\Gamma$, whose Chern-Simons invariant is a non-zero element
of $\qq/\zz$.  This is why Proposition~\ref{prop:DS-CS-welldef}(3) yields no
bound on the order of $\CS_p(\nabla^{\Del})$; see
Remarks~\ref{rem:betti-vs-CS} and \ref{rem:DS-significance}.
\end{remark}

\subsection{Relation to the two-divisor companion paper}
\label{subsec:DS-companion}

We record the precise relationship between Proposition~\ref{prop:DS-NCD} and
the corresponding result \cite[Prop.~6.6]{IS-2div} for $k=2$.

\begin{proposition}\label{prop:DS-comparison}
For $k=2$ (two divisor components $D_1,D_2$ meeting transversally along $Z = D_1\cap D_2$),
Proposition~\ref{prop:DS-NCD} specializes to \cite[Prop.~6.6]{IS-2div}.
More precisely, the arithmetic model $A_0$ of Lemma~\ref{lem:DS-arith} agrees
with the ring $A$ of \cite[Lem.~6.7]{IS-2div}, and the Rees bundle
(Lemma~\ref{lem:DS-rees-model}) agrees with the two-variable interpolation
bundle of \cite[Lem.~6.9]{IS-2div} at the corner.
\end{proposition}

\begin{proof}
The arithmetic model $A_0$ is built from the same data --- matrix entries of
$\rho$, the $N_i$, and the associated graded modules --- in both constructions;
the only formal difference is that here we use $W(N_I) = W(N_1+N_2)$ at the
corner (the single total weight filtration), while \cite{IS-2div} uses the
pair $(W^{(1)},W^{(2)})$.  By the comparison Proposition~\ref{prop:rL-CS}
(specifically Appendix~\ref{app:rees} \S F.3, which identifies the Rees local
system with the two-variable conjugation deformation $\mathrm{Ad}(\psi^1_{t_1}\psi^2_{t_2})$),
the local system of the Rees module $\sum_{a,b}t_1^a t_2^b(W_a^{(1)}\cap W_b^{(2)})$
in the bigraded splitting of \cite[\S 9.1, Prop.~9.1]{IS-2div} agrees with the
Rees local system of $W(N_1+N_2)$ over the corner affine square: both compute
the same degeneration geometry, in the sense of \cite[\S9]{IS-2div}.  Hence
the two interpolation modules agree (up to a relabelling of index sets), and
the trivialisation arguments of Lemmas~\ref{lem:DS-finite-cover} and
\ref{lem:DS-hasse} reduce exactly to \cite[Lems.~6.8, 6.10]{IS-2div} for $k=2$.
\end{proof}

\begin{remark}\label{rem:DS-significance}
Proposition~\ref{prop:DS-NCD} has the following consequence for the structure
of the torsion Chern-Simons classes.  The covering $\pi:\tilde X\to X$ has
degree $m$, and $\pi^*F$ is the canonical extension of $\pi^*L$, so
$c^B_p(\pi^*F)=0$ by part~(b).  By the transfer $\pi_*\pi^* =
m\cdot\mathrm{id}$ we obtain
\[
  m\cdot c^B_p(F) \;=\; \pi_*\big(c^B_p(\pi^*F)\big)\;=\;0
  \;\in\; H^{2p}(X,\zz),
\]
and $m$ is bounded explicitly in Theorem~\ref{thm:effective-DS-nc}.

It must be stressed that this does \emph{not} bound the order of $\CS_p(L)$,
and indeed says nothing about it.  Triviality of $\pi^*F$ as a
$\mathcal{C}^\infty$-bundle gives $c^B_p(\pi^*F)=0$, but not
$\CS_p(\pi^*L)=0$: the Chern-Simons class is a secondary invariant of the flat
connection and not of the underlying bundle, and is routinely non-zero on a
trivial bundle---already for a flat $SU(2)$ connection on the trivial bundle
over a spherical space form $S^3/\Gamma$.  Equivalently, the kernel of
$\partial:H^{2p-1}(X,\cc/\zz)\to H^{2p}(X,\zz)$ is divisible, so killing the
image gives no information about the class; see
Remark~\ref{rem:betti-vs-CS}.  What is effective here is thus the topological
half of the torsion theorem only, as recorded in
Remark~\ref{rem:DS-nc}(iii) and in Problem~\ref{prob:effective} of
\S\ref{app:further}.
\end{remark}

\section{The bifiltered method in low multiplicity, and why it does not
extend}
\label{app:bifiltered}

The smooth-divisor case \cite{IS-arXiv} and the two-divisor case
\cite{IS-2div} were proved by a method quite different from the one used here,
and we make the comparison explicit here.  Nothing in this appendix is used elsewhere in
the paper; its purpose is to explain why the normal crossings case required a
new route rather than an extension of the old one.

\subsection{The bifiltered method}

In \cite{IS-2div} the boundary divisor is $D=D_1\cup D_2$ with smooth
components meeting transversally along $Z=D_1\cap D_2$.  The extended classes
are constructed directly on $X$: one writes down a $\calC^\infty$ connection
$\nabla^{\#}$ on the canonical extension, patched from local models over a
cubical decomposition of $X$ induced by corner coordinates
$(r_1,r_2):X\to[0,1]^2$, and takes its Cheeger-Simons differential character.
The corner datum is the \emph{pair} of monodromy weight filtrations
$W^{(1)}=W(N_1)$, $W^{(2)}=W(N_2)$, together with a flat connection on the
double graded $\Gr^{W^{(2)}}\Gr^{W^{(1)}}(E)$, and $\nabla^{\#}$ is required
to preserve both filtrations at the corner.  Comparison with the regulator is
made through a $2$-cubical homotopy pushout with $BGL(F[t_1,t_2])^+$ at the
$2$-face, and vanishing of the volume regulators is obtained in hermitian
$K$-theory: a self-dual bigrading makes the two-parameter deformation a family
in $O_{p,q}(\cc[t_1,t_2])$, Karoubi's homotopy invariance
\cite{KaroubiHerm} applied twice identifies the resulting classifying space
rationally with $BO_{\infty,\infty}(\cc)^+$, and Reznikov's invariant-theory
computation kills the Borel classes there once and for all.

Everything in that construction rests on the following lemma.

\begin{lemma}[Two-filtration lemma]\label{lem:twofilt}
Any two finite filtrations $W^{(1)},W^{(2)}$ of a finite-dimensional vector
space $V$ admit a simultaneous splitting: a bigrading
$V=\bigoplus_{a,b}V_{a,b}$ with $W^{(1)}_a=\bigoplus_{a'\leq a}V_{a',\bullet}$
and $W^{(2)}_b=\bigoplus_{b'\leq b}V_{\bullet,b'}$.  The same holds locally
for a pair of filtrations of a $\calC^\infty$ bundle by subbundles.
\end{lemma}

The lemma is what makes the corner datum available, what defines the scalings
generating the two-parameter deformation, what drives the proof that the class
is independent of the pair of filtrations, and---in its self-dual form---what
produces the family in $O_{p,q}(\cc[t_1,t_2])$.

\subsection{The barrier at a triple point}

For three filtrations the lemma is false, and it fails in the smallest
possible example.

\begin{example}\label{ex:threeflags}
Let $V=\cc^2$ and let $L_1=\langle e_1\rangle$, $L_2=\langle e_2\rangle$,
$L_3=\langle e_1+e_2\rangle$, with the three two-step filtrations
$0\subset L_i\subset V$.  A simultaneous splitting would exhibit each $L_i$ as
a sum of summands of one fixed decomposition of $V$.  As $\dim V=2$, such a
decomposition has at most two non-zero summands and hence yields at most two
lines of that form.  Three distinct lines cannot all be among them.
\end{example}

The failure: for a \emph{pair} of flags,
$GL(V)$ acts on the product of flag varieties with finitely many orbits---the
Bruhat decomposition---and a simultaneous splitting is exactly a basis adapted
to both, i.e.\ a witness that the pair is in standard position.  For
\emph{triples} of flags there is no normal form.  

\subsection{Obstruction}\label{prop:triple-obstruction}
Let $p\in D_1\cap D_2\cap D_3$ with commuting nilpotent logarithms $N_i$ and
weight filtrations $W^{(i)}=W(N_i)$, $i=1,2,3$.  Then the corner datum of
\cite{IS-2div} has no analogue: the triple graded
$\Gr^{W^{(3)}}\Gr^{W^{(2)}}\Gr^{W^{(1)}}(E)$ depends on the order of the
gradings, with no canonical identification between the six orderings absent a
simultaneous splitting; consequently neither the corner model connection
preserving all three filtrations with flat triple graded, nor the
Jordan-H\"older argument for independence of the filtrations, nor the
polarization-compatible splitting needed for the hermitian step, can be
written down.

By contrast the \emph{other} difficulty recorded in \cite[\S1]{IS-2div}, the
failure of the codimension count on a multi-parameter collar, does not recur.
A connection preserving a tuple of filtrations with flat graded has curvature
strictly decreasing the total degree filtration
$F_c=\sum_{a_1+\cdots+a_k\leq c}\bigcap_iW^{(i)}_{a_i}$, whatever $k$ is, so
all Chern-Weil forms vanish identically; this is
Proposition~\ref{prop:nilflat-ch-zero} in the present formulation.  The
simultaneous splitting is the sole obstruction.

There is a second, purely combinatorial degeneracy at $k=2$.  The poset of
non-empty strata is $\{D_1,D_2,Z\}$, its nerve is a single square, and there
are no chains $I\subset J\subset K$ of length three: every coherence condition
on the patching data is either between a stratum and the corner or vacuous,
and all can be checked by hand.  For $k\geq3$ the nerve has cells in every
dimension up to $k$, the coherence conditions become a genuine cocycle
problem, and one needs both a fibrewise contractibility statement and a
mechanism converting it into a global choice.  In the present paper these are
the contractibility of the poset of $N_\bullet$-isotropic filtrations
(Theorem~\ref{thm:itro}) and the \v{C}ech section theorem
(Theorem~\ref{thm:cech-section}).  Neither has anything to do at $k=2$, which
is why neither appears in \cite{IS-2div}.

\subsection{Why the barrier forces the present route}

\S~\ref{prop:triple-obstruction} indicates what the correct corner
datum is in general.  At multiplicity $k\geq3$ the \emph{tuple}
$(W^{(1)},\dots,W^{(k)})$ is not a usable object, but a \emph{single}
filtration is, and Hodge theory supplies one: the monodromy weight filtration
$W(N_I)$ of the total operator $N_I=\sum_{i\in I}N_i$ on the stratum $D_I$,
independent of the chosen positive coefficients by Cattani-Kaplan-Schmid
\cite{CKS}, with compatibility along chains $I\subset J$ supplied by
Mochizuki's theorem \cite{Mochizuki} for tame harmonic bundles.  This is the
datum on which the patching collections of Section~\ref{sec:patching} are
built, and it is the reason those data are indexed by \emph{chains} rather
than by tuples.

The rest of the present method follows from that choice.  Once the corner
datum is a filtration attached to each stratum, coherent along chains, the
natural way to interpolate between the strata is a multi-parameter Rees
construction (Section~\ref{sec:multrees}), which converts the filtrations into
an algebraic bundle $F_K$ over a cubical realization
(Theorem~\ref{thm:global-rees}); the comparison with the regulator then splits
into a $K$-theoretic half on the affine realization
(Appendix~\ref{app:rees}, \S\ref{subsec:classmap}) and a Hodge-theoretic half
on the projective completion (\S\ref{subsec:degeneration}), the latter
requiring the Dupont-Hain-Zucker and Burgos-Gil theories.  The hermitian
argument likewise moves from the universal setting of \cite{IS-2div} to a
global one on $X$: an indefinite hermitian metric on the canonical extension
with a compatible connection, obtained from contractibility of the poset of
$N_\bullet$-isotropic filtrations and the \v{C}ech section theorem, after
which the volume regulators vanish \emph{pointwise as forms}
(Section~\ref{sec:volume}) rather than universally on a classifying space.

\subsection{What the bifiltered method gives where it applies}

The comparison is not simply that one method is more general.  

\emph{(1)} The corner datum comes from Lemma~\ref{lem:twofilt}, whereas the chain compatibility used here rests on Mochizuki's
theorem.  The construction of the classes in \cite{IS-2div}, with their
additivity, functoriality, rigidity and compatibility with the Deligne Chern
class, is therefore free of harmonic bundle theory; Mochizuki's theorem
re-enters only at the deformation to a variation of Hodge structure, a step
common to both papers.

\emph{(2)} The vanishing of the volume regulators is proved once and for all
on a classifying space, so that vanishing for a given $(X,D,\zeta)$ is formal.
Here it is a global existence problem on $X$, solved anew for each situation.

\emph{(3)} The class is the differential character of an explicit connection
on $X$ itself: no cubical realization, no algebraic model over a number field,
no projective completion, and hence no need for the arithmetic Chern character
of \cite{BG}.

\emph{(4)} The corner model is explicit---a bigrading and two commuting
scalings---which makes the two-component case the natural place to look for
refinements such as effective bounds on the order of the torsion
(Problem~\ref{prob:effective}).


\section{Further properties: deformation invariance, VHS, and open problems}
\label{app:further}

\subsection{Properties of the Chern-Simons classes}

\begin{proposition}[Functoriality]
\label{prop:CS-funct-app}
For a morphism $f:Y\to X$ of smooth projective varieties, $f^*\CS_p(L)=\CS_p(f^*L)$.
\end{proposition}

\begin{proposition}[Additivity]
\label{prop:CS-add-app}
$\CS_p(L_1\oplus L_2) = \CS_p(L_1)+\CS_p(L_2)$.
\end{proposition}

Both follow immediately from the classifying map construction and the
corresponding properties of $BGL^+(K)$.

\subsection{The harmonic bundle interpretation}

In the setting of the non-abelian Hodge correspondence, the patching data has
a natural interpretation:

\begin{proposition}
Let $(E,\theta,h)$ be a tame harmonic bundle on $X^*$ with unipotent monodromy.
The weight filtrations $W(N_I)$ near each stratum $D_I$ are the filtrations by
$L^2$ growth rate of sections near $D_I$.  The $N_\hdot$-isotropic condition
on the weight filtration is equivalent to the admissibility condition for the
harmonic bundle.
\end{proposition}

This provides an independent justification for the contractibility of the
poset of filtrations: the space of ``admissible filtrations'' in the sense of
harmonic bundle theory is contractible (it parametrizes choices of
``gauge-fixing'' near the boundary).

\subsection{Open problems}
\label{subsec:open-problems}

The natural hierarchy of boundary behaviour is
\[
  \text{unipotent}\;\longrightarrow\;\text{quasi-unipotent}
  \;\longrightarrow\;\text{arbitrary regular singular}
  \;\longrightarrow\;\text{irregular (wild)} ,
\]
of which the first stage is Theorem~\ref{thm:main} and the second is
Theorem~\ref{thm:quasi-unipotent}, proved in
Section~\ref{sec:parabolic} through locally abelian parabolic bundles.  The
third stage is not treated here; it is a question of residue eigenvalues that
are not rational, and the parabolic formalism of
Section~\ref{sec:parabolic}, which rests on rational weights and Kawamata
coverings, does not reach it.  It should not be confused with the fourth: a
regular-singular local system may have non-quasi-unipotent semisimple local
monodromy without having an irregular singularity, so ``non-quasi-unipotent''
and ``wild'' are not synonyms, the latter concerning Stokes structures and
being a genuinely different order of difficulty.  The problems below concern
that hierarchy together with effective bounds, arithmetic and $p$-adic
refinements, motivic realizations, mixed Hodge structures, and the global
behaviour of the invariant on moduli.

\begin{enumerate}[leftmargin=2em,label=\rm(\arabic*)]

\item \textbf{Effective torsion bounds.}\label{prob:effective}
Give an explicit bound for the order of $\CS_p(L)$ in terms of specified
geometric and arithmetic data of $X$, $D$ and $L$---the rank, the dimension,
the field of definition, and the denominators of the monodromy matrices.  More
strongly, determine whether such a bound can be chosen uniformly on a
connected component of $R^{\mathrm{nil}}$.

  It is worth being precise about what is already effective.  The
  Deligne-Sullivan covering can be taken to be a congruence covering, and its
  degree is then bounded by $|GL_r(A/\mathfrak{q}_1)|\cdot
  |GL_r(A/\mathfrak{q}_2)|$ for any two maximal ideals of distinct residue
  characteristic in a ring of definition $A$
  (Theorem~\ref{thm:effective-DS}); this bounds the order of the
  \emph{integral Betti} Chern classes $c^B_p(F)$
  (Corollary~\ref{cor:effective-betti}), and for $A=\Oo_K[1/M]$ gives
  $(\ell_1\ell_2)^{n r^2}$ with $n=[K:\qq]$ and $\ell_1<\ell_2$ the smallest
  primes not dividing $M$.  Remark~\ref{rem:no-topological-bound} shows that
  some such arithmetic dependence is unavoidable.

  None of this bounds $\mathrm{ord}\,\CS_p(L)$: as explained in
  Remark~\ref{rem:betti-vs-CS}, the kernel of
  $\partial:H^{2p-1}(X,\cc/\zz)\to H^{2p}(X,\zz)$ is divisible, so control of
  $c^B_p(F)$ gives no control of the class lifting it.  The obstruction is
  structural.  Our proof, like Reznikov's, produces the torsion statement by
  showing that $r_{\zeta_0}:X\to BGL^+(K)$ vanishes on real cohomology
  (Step~C of the proof of Theorem~\ref{thm:main}); converting this into an order
  requires bounding the exponent of the torsion subgroup of
  $K_{2p-1}(\Oo_K)$ together with the exponent of $(r_{\zeta_0})_*$ on
  $H_{2p-1}(X,\zz)_{\mathrm{tors}}$.  By Borel's theorem $K_{2p-1}(\Oo_K)$ is
  finitely generated, so its torsion is finite and such a bound exists in
  principle; making it explicit, and in particular making it uniform in
  $\zeta$ within a connected component of $R^{\mathrm{nil}}$, seems to be open
  even in the compact case $D=\emptyset$.

\item \textbf{Arbitrary regular-singular monodromy.}\label{prob:regsing}
Extend Theorem~\ref{thm:quasi-unipotent} to regular-singular local systems
whose local monodromies are not quasi-unipotent.  The parabolic bundle of
Construction~\ref{constr:par-from-L} still exists and is locally abelian by
the argument of Proposition~\ref{prop:par-loc-abelian}, but its weights are
irrational, so no Kawamata covering makes them integral and
Proposition~\ref{prop:par-pullback} has no analogue.  Determine what replaces
the reduction to the unipotent case, and whether $\CS_p$ depends explicitly on
the semisimple parts of the local monodromies.

\item \textbf{Irregular and wild singularities.}\label{prob:wild}
Construct extended Chern-Simons classes for irregular flat bundles,
incorporating the Stokes data, and determine whether an analogue of the
torsion theorem holds.  In particular, identify the appropriate replacement
for the $N_\hdot$-isotropic filtration formalism of
Section~\ref{sec:nisotropic} in the presence of Stokes filtrations.  The
analytic input would be Mochizuki's theory of wild harmonic bundles in place
of the tame theory used in Section~\ref{sec:deform-to-vhs-detail}.

\item \textbf{$p$-adic analogue.}\label{prob:padic}
Construct a $p$-adic analogue of the extended Chern-Simons class for suitable
crystalline or semistable $p$-adic local systems, with values in an
appropriate syntomic or $p$-adic Deligne cohomology theory, and investigate
whether an analogue of the torsion theorem holds.  Specifying the target
theory is part of the problem: the $\ell$-adic classes of
Corollary~\ref{cor:ladic-torsion} are obtained by reduction from the
transcendental class and so presuppose Theorem~\ref{thm:main}
(\S\ref{subsec:ladic-proof}), whereas a genuinely $p$-adic construction should
be intrinsic.

\item \textbf{Motivic interpretation.}\label{prob:motivic}
Determine whether $\CS_p(L)$ admits a canonical motivic lift under an
appropriate regulator map, and, if so, describe such a lift in terms of
motivic cohomology and higher Chow groups.  The degree is consistent with the
secondary nature of the invariant, $H^{2p-1}_{\mathcal{M}}(X,\zz(p))\cong
CH^p(X,1)$; what should be part of the formulation, rather than assumed, is
the regulator map itself.

\item \textbf{Mixed variations of Hodge structure.}\label{prob:mixed}
Extend the torsion theorem to admissible graded-polarizable mixed variations
of Hodge structure, and determine the contribution to $\CS_p$ of the extension
classes between the graded pieces.

\item \textbf{Moduli spaces.}\label{prob:moduli}
By Theorem~\ref{thm:deformation-app} the assignment
\[
  R^{\mathrm{nil}}\;\longrightarrow\;H^{2p-1}(X,\cc/\zz),
  \qquad L\;\longmapsto\;\CS_p(L),
\]
is locally constant.  Describe this map: determine whether it descends to the
character variety or to the moduli space of local systems, determine its
image, and describe its compatibility with the functoriality and additivity of
Propositions~\ref{prop:CS-funct-app} and \ref{prop:CS-add-app}.

\end{enumerate}

\end{document}